\documentclass[a4paper,10pt]{extarticle}
 
\usepackage[utf8]{inputenc}
\usepackage[T1]{fontenc}
\usepackage[english]{babel}
\usepackage{lmodern}
\usepackage{geometry}
\usepackage{url,xspace}
\usepackage{amsmath,amssymb,stmaryrd,spalign,bbm,mathtools,mathrsfs,dsfont,thmtools}
\usepackage[thmmarks,amsmath,framed,thref,hyperref]{ntheorem}
\usepackage{comment}
\usepackage[colorlinks=true,citecolor=teal,linkcolor=blue]{hyperref}
\usepackage{bm}
\usepackage{array}
\usepackage{enumitem}
\usepackage{changepage}
\usepackage{euscript}
\usepackage{cite}
\usepackage{graphicx}
\usepackage[toc,page]{appendix}
\usepackage{etoolbox}
\usepackage{pgfplots}
\pgfplotsset{width=10cm,compat=1.9}

\usepackage{tikz, xcolor, float}

\usepackage{makecell}
\usepackage{tikz-cd}

\usepgfplotslibrary{external}

\counterwithin{equation}{section}

\setlist[enumerate,1]{label={\textit{(\roman*)}}}

\newcommand{\zerodisplayskips}{%
  \setlength{\abovedisplayskip}{5pt}%
  \setlength{\belowdisplayskip}{5pt}%
  \setlength{\abovedisplayshortskip}{5pt}%
  \setlength{\belowdisplayshortskip}{5pt}}
\appto{\normalsize}{\zerodisplayskips}
\appto{\small}{\zerodisplayskips}
\appto{\footnotesize}{\zerodisplayskips}

\newtheorem{theorem}{Theorem}[section]
    \newtheorem{corollary}[theorem]{Corollary}
    \newtheorem{lemma}[theorem]{Lemma}
    \newtheorem{proposition}[theorem]{Proposition}

    \theorembodyfont{\upshape}
    \newtheorem{definition}[theorem]{Definition}
    
    \newtheorem{remark}[theorem]{Remark}
    
\counterwithin{equation}{section}

\setlist[enumerate,1]{label={\textit{(\roman*)}}}

\theoremstyle{nonumberplain}
\theoremheaderfont{\itshape}
\theorembodyfont{\normalfont}
\theoremseparator{.\,—}
\theoremsymbol{}
\newtheorem{proof-wo}{Proof}
\theoremsymbol{\ensuremath{\blacksquare}}
\newtheorem{proof}{Proof}

\renewcommand{\thesection}{\arabic{section}}

\makeatletter
\def\@cite#1#2{[\textbf{#1}\if@tempswa , #2\fi]}
\def\@biblabel#1{[\textbf{#1}]}
\makeatother

\newcommand{\bigslant}[2]{{\raisebox{.2em}{$#1$}\left/\raisebox{-.2em}{$#2$}\right.}}

\newcommand{\eus}{\EuScript}

\newcommand{\llb}{\llbracket}
\newcommand{\rrb}{\rrbracket}

\let\div\undefined
\DeclareMathOperator{\div}{div\,}

\DeclareMathOperator{\supp}{supp\,}

\title{Homogeneous Sobolev and Besov spaces on special Lipschitz domains and their traces }
\author{Anatole \textsc{Gaudin}\thanks{{Aix-Marseille Université, CNRS, I2M, Marseille, France} - \textbf{email:} anatole.gaudin@univ-amu.fr}}

\date{}

\begin{document}

%page de garde
\maketitle

\begin{abstract}
    We aim to contribute to the folklore of function spaces on Lipschitz domains. We prove the boundedness of the trace operator for homogeneous Sobolev and Besov spaces on a special Lipschitz domain with sharp regularity. To achieve this, we provide appropriate definitions and properties, ensuring our construction of these spaces is suitable for non-linear partial differential equations and boundary value problems. The trace theorem holds with the sharp range $s \in (\frac{1}{p}, 1 + \frac{1}{p})$. While the case of inhomogeneous function spaces is well-known, the case of homogeneous function spaces appears to be new, even for a smooth half-space. We refine several arguments from a previous paper on function spaces on the half-space and include a treatment for the endpoint cases $p=1$ and $p=+\infty$.
\end{abstract}

\addtocontents{toc}{\protect\thispagestyle{empty}}
\tableofcontents
%ne pas numéroter le sommaire

\pagestyle{plain}
\thispagestyle{empty}
%espacement entre les lignes d'un tableau
\renewcommand{\arraystretch}{1.5}

%====================== INCLUSION DES PARTIES ======================

%----------------------------------------------------
%------------------- Section 1 ----------------------
%----------------------------------------------------
%----------------------------------------------------
%------------------- Section 1 ----------------------
%----------------------------------------------------

\section{Introduction}

We establish in this paper a proper and user-friendly construction of homogeneous Sobolev and Besov spaces\footnote{See the beginning of Subsection~\ref{subsec:SobolevBesovRn}, in particular Definition~\ref{def:HomFuncSpacesRn}, for the appropriate definitions in the case of the whole space.} on some unbounded Lipschitz domains $\Omega \subset \mathbb{R}^n$, defined (up to a rotation) by the epigraph of a uniformly Lipschitz function, more precisely
\begin{align*}
    \Omega =\{\,(x',x_n)\in\mathbb{R}^{n-1}\times\mathbb{R}\,|\, x_n>\phi(x')\,\}\text{,}
\end{align*}
for a globally and uniformly Lipschitz continuous function $\phi\,:\,\mathbb{R}^{n-1}\longrightarrow\mathbb{R}$.

The introduced technical biases also appear to be fundamental. For instance, the trace theorem with optimal regularity, stated with homogeneous estimates, is obtained through a trick based on the $\mathrm{L}^q$-maximal regularity result obtained in \cite{Gaudin2023}, which allows us to circumvent significant difficulties in a simple way.

Additionally, complementary tricks and techniques for complex interpolation of homogeneous inequalities, as well as new results for interpolation of normed spaces, are introduced, particularly addressing the absence of completeness for the considered normed spaces without considering their completions; see Remark~\ref{rmk:HomInterpLift} and the Appendix~\ref{Appendix:Interp}. We note that the loss of completeness is necessary to define such function spaces with high regularities; see Subsection~\ref{sec:HomSpacesNecessary}.

As an intermediate step for an in-depth investigation in the case of special Lipschitz domains $\Omega$, it was necessary to obtain new and sharpened results for homogeneous function spaces on the whole space $\mathbb{R}^n$.

Finally, the construction of these homogeneous Sobolev and Besov spaces, as normed vector spaces of actual distributions on special Lipschitz domains, along with the optimal trace theorem and related properties, has been an open problem for more than a decade. Even analogs in the case of a $\mathrm{C}^\infty$ bent half-space do not appear to be known until now.

The main goals and strategies of this paper are highlighted at the end of Subsections~\ref{sec:HomSpacesNecessary} and \ref{sec:tracthmIntro}. The most significant results of this paper are highlighted at the end of Subsection~\ref{sec:HomSpacesNecessary} and in Subsection~\ref{sec:KatoHarSobDiag}.

\subsection{Necessity of homogeneous function spaces and their issues}\label{sec:HomSpacesNecessary}

Homogeneous Sobolev and Besov spaces are function spaces that arise very naturally whil solving the most standard partial differential equations on unbounded domains, such as the whole or the half-space. The fundamental examples being the Laplace and the heat equations,

\noindent\begin{minipage}{0.4\textwidth}
\begin{equation*}\tag{$\mathcal{L}$}
-\Delta u =f \text{, in } \mathbb{R}^n, \label{eq:LapEqRn}
\end{equation*} 
\end{minipage}%
\begin{minipage}{0.2\textwidth}\centering
    and 
\end{minipage}%
\begin{minipage}{0.4\textwidth}
\begin{equation*}\tag{$\mathcal{HE}$}
\left\{\begin{array}{l}
\partial_t u-\Delta u=f, \text{ in } \mathbb{R}_+\times \mathbb{R}^n\\
u(0)=0,\end{array}\label{eq:HeatEqRn}\right. 
\end{equation*}
\end{minipage}

\noindent for which we cannot control the solution $u$ itself, but only the amount of derivatives of the solution $u$ that appears in the corresponding equations. More precisely:
\begin{enumerate}
    \item In the case of the Laplace equation \eqref{eq:LapEqRn}, one cannot expect $u\in\mathrm{L}^2(\mathbb{R}^n)$, for arbitrary $f\in\mathrm{L}^2(\mathbb{R}^n)$. The only control we can have is for the (semi-)norm defined by $$\lVert u\rVert_{\dot{\mathrm{H}}^{2,2}(\mathbb{R}^n)} := \lVert\nabla^2u\rVert_{\mathrm{L}^2(\mathbb{R}^n)} = \lVert f \rVert_{\mathrm{L}^2(\mathbb{R}^n)}.$$
    \item In the case of the heat equation \eqref{eq:HeatEqRn}, one cannot expect $u\in\mathrm{L}^2(\mathbb{R}_+\times\mathbb{R}^n)$, for arbitrary $f\in\mathrm{L}^2(\mathbb{R}_+\times\mathbb{R}^n)$. One can only establish a control of the form $$\lVert \partial_t u\rVert_{{\mathrm{L}}^2(\mathbb{R}_+\times\mathbb{R}^n)}^2 + \lVert\nabla^2u\rVert_{\mathrm{L}^2(\mathbb{R}_+\times\mathbb{R}^n)}^2 = \lVert f \rVert_{\mathrm{L}^2(\mathbb{R}_+\times\mathbb{R}^n)}^2.$$
\end{enumerate}
Thus, we need function spaces whose norms precisely control the regularity and integrability arising from the equations, namely, the homogeneous Sobolev (and Besov) spaces. For the reader's convenience, we mention that these two facts\footnote{that in both cases, $u$ itself cannot lie in $\mathrm{L}^2$, provided $f\in\mathrm{L}^2$ is arbitrary.} can be proven by contradiction using the Closed Graph Theorem and a dilation argument. These phenomena are due to the Laplacian not being invertible on $\mathrm{L}^2(\mathbb{R}^n)$. In the case of the heat equation in $\mathrm{L}^q(\mathrm{L}^p)$-spaces, a similar phenomenon occurs, but it also holds for more general operators. See, for example, the discussion in \cite[Introduction, Subsection~1.1.1]{Gaudin2023}, as well as the results \cite[Definition~1.1 \& Proposition~2.2]{Monniaux2009MaxReg} and \cite[Corollary~3.5.3]{PrussSimonett2016}.

\medbreak

In the last two decades, the use of homogeneous function spaces has been central to the global-in-time study of various problems in fluid mechanics. For more details, see \cite{bookBahouriCheminDanchin,DanchinMucha2009,DanchinMucha2015,DanchinHieberMuchaTolk2020,OgawaShimizu2016,OgawaShimizu2021,OgawaShimizu2022,OgawaShimizu2024} and the references therein.

In the context of non-linear partial differential equations (even on the whole space), one has to be clear about the definition of homogeneous function spaces. Typically, elements of homogeneous function spaces, such as homogeneous Sobolev and Besov spaces, are defined as equivalence classes of tempered distributions modulo polynomials, denoted by $\bigslant{\mathcal{S}'(\mathbb{R}^n)}{\mathcal{P}(\mathbb{R}^n)}$, as in \cite[Chapter~6,~Section~6.3]{BerghLofstrom1976}, \cite[Chapter~5]{bookTriebel1983}, or \cite[Chapter~2]{bookSawano2018}.

This construction, modulo polynomials, proves unsuitable for nonlinear partial differential equations for several reasons, as discussed in \cite[Introduction]{Gaudin2022}. We also present further reasons in this section. Notably, according to \cite{Bourdaud1988, Bourdaud2013}, \cite[Chapter~2,~Section~2.4]{Triebel2015}, and \cite[Chapter~2,~Section~2.4.3]{bookSawano2018}, it is not clear that one can canonically choose the polynomial part to obtain an element of $\mathcal{S}'(\mathbb{R}^n)$ in an independent way of the Sobolev index of the considered function spaces.

\begin{itemize}
    \item \textbf{Issues for the definition of (para-)products.}
    
    Given $[u],[v]\in\bigslant{\mathcal{S}'(\mathbb{R}^n)}{\mathcal{P}(\mathbb{R}^n)}$ with representatives $u+P, v+Q, u+\Tilde{P}, v+\Tilde{Q} \in\mathcal{S}'(\mathbb{R}^n)$,, we have
\begin{align*}
(u+P) (v+Q) - (u+\Tilde{P}) (v+\Tilde{Q}) =& (P- \Tilde{P})v + (Q-\Tilde{Q})u+ P Q - \Tilde{P}\Tilde{Q}.
\end{align*}
Therefore, even with meaningful bilinear (para)product estimates, while $P Q - \Tilde{P}\Tilde{Q}$ is a polynomial, $(P-\Tilde{P})v + (Q-\Tilde{Q})u$ is not, meaning the product ultimately \textit{depends on the choice of representatives!}\footnote{This issue persists even when $P,Q,\Tilde{P},\Tilde{Q}$ are constants.}
    \item \textbf{Issues with the definition on domains by restriction.}
    
    Defining the restriction to a domain $\Omega$, of an element that belongs to $\bigslant{\mathcal{S}'(\mathbb{R}^n)}{\mathcal{P}(\mathbb{R}^n)}$ seems ambiguous. For instance, consider $[f]\in \bigslant{\mathcal{S}'(\mathbb{R})}{\mathcal{P}(\mathbb{R})}$ such that has at least one representative $f\in\mathrm{L}^1_{\text{loc}}(\mathbb{R})\cap\mathcal{S}'(\mathbb{R})$.
    
    Then $x\mapsto f(x)-x^2$ and $x\mapsto f(x)-x^3$ are two distinct representative that admit different restrictions in $\mathcal{D}'((a,b))$ for any $-\infty<a<b\leqslant +\infty$. Thus, the restriction in the distributional sense seems meaningless with regard to for the quotient structure. In particular, it appears unclear on how to define the support, having in mind the quotient structure.

    \item \textbf{Issues with the composition with a diffeomorphism.}

    Assuming that $u+P, u+Q\in\mathcal{S}'(\mathbb{R}^n)$ are two representatives of $[u]\in \bigslant{\mathcal{S}'(\mathbb{R}^n)}{\mathcal{P}(\mathbb{R}^n)}$, and $\Psi$ is a smooth diffeomorphism of $\mathbb{R}^n$, the meaning of the expression
\begin{align*}
    u\circ \Psi + P\circ\Psi- (u\circ \Psi + Q\circ\Psi) = P\circ\Psi - Q\circ\Psi
\end{align*}
is ambiguous: even assuming $u\circ \Psi\in\mathcal{S}'(\mathbb{R}^n)$, $P\circ\Psi$ and $Q\circ\Psi$ might not even qualify as tempered distributions. Even if it is, it should not depend on the choice of $P$ and $Q$ which is again unclear. This creates a significant challenge in transferring properties of homogeneous function spaces from the whole space or half-space to a bent half-space via a change of coordinates, particularly in defining traces on the boundary.

    \item \textbf{Issues for extension (and projection) operators.}

    To properly maintain properties such as interpolation identities, one typically uses extension operators $\mathcal{E}\,:\, \mathcal{S}'(\mathbb{R}^n)_{|_{\Omega}}\longrightarrow \mathcal{S}'(\mathbb{R}^n)$, such that $[\mathcal{E} \cdot]_{|_{\Omega}}=\mathrm{I}$ on $\mathcal{S}'(\mathbb{R}^n)_{|_{\Omega}}$. However, for homogeneous function spaces built modulo polynomials, preserving the quotient structure requires that $\mathcal{E}(\mathcal{P}(\mathbb{R}^n)_{|_\Omega}) \subset \mathcal{P}(\mathbb{R}^n)$, by linearity. However, since polynomials are entirely determined by their values taken on a finite number of points, we should necessarily have
\begin{align*}
\mathcal{E} P = P\text{, } \forall P\in\mathcal{P}(\mathbb{R}^n).
\end{align*}
This condition is overly restrictive and is likely not met by most standard extension operators, even for $\Omega=\mathbb{R}^n_+$. For instance, this condition is not met for the extension operator by higher-order reflection around the boundary used in \cite[Chapter~2]{DanchinHieberMuchaTolk2020} and \cite{Gaudin2022}, which was initially introduced and studied for smooth functions and standard Sobolev spaces in \cite{Hestenes1941,Babich1953}.

Although some extension operators can meet the polynomial-preserving condition, we want our construction to accommodate a wider variety of extension operators. This is in order to have the most flexible construction of these function spaces in concrete cases, \textit{i.e.}, when one wants to study non-linear and boundary-value problems. Therefore, we need function spaces and ambient spaces where we do not have to distinguish between different representatives of their elements.

    \item \textbf{Issues concerning distribution theory and the completion of function spaces.}

    Any completion of some function spaces may not be canonically identified as subspaces of $\mathcal{D}'(\mathbb{R}^n)$. For example, it can be verified that
    \begin{align*}
    \mathrm{C}^\infty_c(\mathbb{R}^n)\not\subset \dot{\mathrm{H}}^{-{n}/{2},2}(\mathbb{R}^n).
    \end{align*}
    Consequently, completion of Schwartz functions for the $\dot{\mathrm{H}}^{{n}/{2},2}(\mathbb{R}^n)$-norm cannot be \textit{canonically} embedded in $\mathcal{D}'(\mathbb{R}^n)$, as the (pre-)dual space lacks test functions.

    This implies that if one takes any abstract completion to define these spaces, one would encounter elements that cannot be \textit{canonically} identified with actual distributions.
\end{itemize}
In particular, the realization of homogeneous function spaces on special Lipschitz domains provided by Costabel, M${}^\text{c}$Intosh and Taggart \cite{CostabelTaggartMcIntosh2013}, built on $\bigslant{\mathcal{S}'(\mathbb{R}^n)}{\mathcal{P}(\mathbb{R}^n)}$, appears to be inapplicable for linear problems with boundary values and unsuitable for non-linear problems.

To circumvent those issues, Bahouri, Chemin, and Danchin proposed in \cite[Chapter~2]{bookBahouriCheminDanchin} to consider a subspace of $\mathcal{S}'(\mathbb{R}^n)$ consisting solely of tempered distributions without any (non-zero) polynomial part, see \cite[Examples,~p.23]{bookBahouriCheminDanchin}. This subspace, denoted $\mathcal{S}'_h(\mathbb{R}^n)$, is defined at the beginning of Subsection \ref{subsec:SobolevBesovRn}. Using $\mathcal{S}'_h(\mathbb{R}^n)$ as an ambient space, Bahouri, Chemin, and Danchin constructed homogeneous Besov spaces $\dot{\mathrm{B}}^{s}_{p,q}(\mathbb{R}^n)$. With their construction, the homogeneous Besov spaces are complete whenever $(s,p,q)\in\mathbb{R}\times[1,+\infty]^2$ satisfies
\begin{align*}\tag{$\mathcal{C}_{s,p,q}$}
    \left[ s<\frac{n}{p} \right]\text{ or }\left[q=1\text{ and } s\leqslant\frac{n}{p} \right]\text{. }
\end{align*}
This also led Danchin and Mucha to consider homogeneous Besov spaces on $\mathbb{R}^n_+$ and exterior domains, as discussed in \cite{DanchinMucha2009,DanchinMucha2015}. Additionally,  Danchin, Hieber, Mucha, and Tolksdorf examined briefly homogeneous Sobolev spaces $\dot{\mathrm{H}}^{m,p}$ on $\mathbb{R}^n$ and $\mathbb{R}^n_+$ for $m\in\mathbb{N}$, $p\in(1,+\infty)$, in \cite[Chapter~2]{DanchinHieberMuchaTolk2020}. 

The author extended this construction in a previous work \cite{Gaudin2022}, encompassing the entire scale of homogeneous Sobolev spaces, in the reflexive range, on the half-space $\mathbb{R}^n_+$. This extension was achieved by investigating interpolation properties and defining the meaning of traces on the boundary. The corresponding properties for homogeneous Besov spaces on the half-space have also been reviewed.

\medbreak

\textbf{The first main goal of the present paper} is to construct homogeneous Sobolev and Besov spaces on special Lipschitz domains, enhancing, sometime by a lot, certain arguments presented in \cite[Subsection~3B]{Gaudin2022} for the flat upper half-space. The structure of extension and projection operators used here forces us to consider two families of regularity indices: $(-1+\frac{1}{p},1]$ and $[0,+\infty)$ with a common overlap $[0,1]$. We also note that the simple argument employed in \cite[Proposition~3.1]{Gaudin2022} is not applicable to the whole family of non-negative regularity indices. Furthermore, we address the cases $p=1,+\infty$\footnote{Only partially for the case $p=+\infty$}. We no longer require the completeness assumption or the intersection with a complete space to achieve the main results, marking a significant improvement over the work in \cite{Gaudin2022}. We also improve several results in the case of the whole space in Subsection \ref{subsec:SobolevBesovRn}.

\begin{itemize}
    \item \textbf{New density results} on $\mathbb{R}^n$: \textbf{Proposition~\ref{prop:DensitySobBesovRn}}; and on special Lipschitz domains $\Omega$: \textbf{Propositions~\ref{prop:densityCcinftyinhomHspspeLip}} and \textbf{\ref{prop:densityCcinftyBspq0SpeLip}}.
    \item \textbf{New interpolation results} that does not require the completeness of involved spaces on $\mathbb{R}^n$: \textbf{Theorem~\ref{thm:InterpHomSpacesRn}}; and on special Lipschitz domains $\Omega$: \textbf{Proposition~\ref{prop:InterpHomSobSpeLip}} and \textbf{Theorem \ref{thm:RealInterpHomSpacesSpeLip}}.
    \item \textbf{New duality results} on $\mathbb{R}^n$: \textbf{Proposition \ref{prop:DualityBesovRn}}; and on special Lipschitz domains $\Omega$: \textbf{Proposition~\ref{prop:DualitySobolevDomain}} and \textbf{Theorem~\ref{thm:dualityBesovSpeLip}}.
\end{itemize}

We note that while the results highlighted above are nearly optimal, the intermediate results--such as the boundedness of extension operators, which are necessary to achieve these main results--are not entirely optimal, though they are close. This primarily concerns the case $p=+\infty$. See, for instance, Lemma \ref{lem:GlobalchangeCoordBesov} and Proposition \ref{prop:ExtProj0HomBspinfty}, then the beginning of Subsections \ref{sec:HomSobSpacesSpeLip} and \ref{sec:HomBesovSpacesSpeLip}, as well as Remark~\ref{rmk:RemarkNonOptimalitypequalinfty} for some brief discussions.  There are also many other related obstructions such as the loss of completeness, or the change of extension operators depending on the regularity index.

\subsection{Trace theorems}\label{sec:tracthmIntro}

Trace theorems with sharp regularity are fundamental for studying boundary value problems in the field of partial differential equations. The standard trace theorems for Besov or Sobolev functions defined on $\mathbb{R}^n_+:=\mathbb{R}^{n-1}\times (0,+\infty)$ and on bounded, sufficiently regular, domains are detailed for example, in \cite[Subsection~6.6]{BerghLofstrom1976}, \cite[Theorem~7.43,~Remark~7.45]{adamsFournier2003sobolev}, \cite[Theorems~3.16~\&~3.19]{Schneider2010}. A nearly exhaustive result for traces on subsets of $\mathbb{R}^n$ with minimal geometric assumptions can be found in \cite[Chapters~VI~\&~VII]{JonssonWallin1984}. The standard trace theorem for Lipschitz domains is given below. Refer to Definition \ref{def:TraceOperator} below, for a definition of the trace operator.

\begin{theorem}\label{thm:genericinhomTrace} Let $p\in(1,+\infty)$, $q\in[1,+\infty]$, $s\in(\frac{1}{p},1+\frac{1}{p})$, and $\Omega$ be either a special or a bounded Lipschitz domain,
\begin{enumerate}[label=($\roman*$)]
    \item the trace operator $[\cdot]_{|_{\partial\Omega}}\,:\, \mathrm{H}^{s,p}(\Omega)\longrightarrow \mathrm{B}^{s-\frac{1}{p}}_{p,p}(\partial\Omega)$ is a bounded surjection, in particular for all $u\in \mathrm{H}^{s,p}(\Omega)$,
    \begin{align*}
        \lVert  u_{|_{\partial\Omega}} \rVert_{\mathrm{B}^{s-\frac{1}{p}}_{p,p}(\partial\Omega)}  \lesssim_{s,p,n} \lVert u \rVert_{\mathrm{H}^{s,p}(\Omega)} ;
    \end{align*}
    \item The trace operator $[\cdot]_{|_{\partial\Omega}}\,:\, \mathrm{B}^{s}_{p,q}(\Omega)\longrightarrow \mathrm{B}^{s-\frac{1}{p}}_{p,q}(\partial\Omega)$ is a bounded surjection, in particular for all $u\in \mathrm{B}^{s}_{p,q}(\Omega)$, 
    \begin{align*}
        \lVert u_{|_{\partial\Omega}} \rVert_{\mathrm{B}^{s-\frac{1}{p}}_{p,q}(\partial\Omega)} \lesssim_{s,p,n} \lVert u \rVert_{\mathrm{B}^{s}_{p,q}(\Omega)}  ;
    \end{align*}
    \item The trace operator $[\cdot]_{|_{\partial\Omega}}\,:\, \mathrm{B}^{\frac{1}{p}}_{p,1}(\Omega)\longrightarrow \mathrm{L}^p(\partial\Omega)$ is a bounded surjection, in particular for all $u\in \mathrm{B}^{\frac{1}{p}}_{p,1}(\Omega)$, 
    \begin{align*}
        \lVert u_{|_{\partial\Omega}} \rVert_{\mathrm{L}^p(\partial\Omega)} \lesssim_{p,n} \lVert u \rVert_{\mathrm{B}^{\frac{1}{p}}_{p,1}(\Omega)}  .
    \end{align*}
\end{enumerate}
Moreover, the trace operator $[\cdot]_{|_{\partial\Omega}}$ admits a right bounded inverse (not necessarily linear) for each of the above cases.
\end{theorem}

Roughly speaking, the goal here is, up to technical modifications, to add dots to every $\mathrm{H}$ and $\mathrm{B}$ symbol in Theorem~\ref{thm:genericinhomTrace} for a special Lipschitz domain. Our focus on special Lipschitz domains is motivated by two main reasons. First, in bounded Lipschitz domains, localization aspects indicate minimal differences between inhomogeneous and homogeneous function spaces, as exemplified by Poincaré-Wirtinger-Sobolev type inequalities. Second, special Lipschitz domains appear to be the one of the rare class that currently allows for effective extension operators with homogeneous estimates. For more general unbounded Lipschitz or smoother domains, a naive localization argument with smooth cut-offs fails to preserve homogeneity.

In the context of inhomogeneous function spaces, simpler proofs for constructing the trace operator on the boundary can be found, see \textit{e.g.}, \cite[Theorem~1,~p.273]{Evansbook2ndEd2010} or \cite[Theorem~18.1,~Corollary~18.4]{Leoni2017}, seeing the trace operator as a (compact) operator with value in $\mathrm{L}^p(\partial \Omega)$ (when $\Omega$ has compact boundary). Similar results are also available for partial traces of vector fields, with compactness property in the case of compact boundary, \textit{e.g.}, see \cite{Monniaux2015,Denis2020} and the references therein.
A general result for the existence of partial traces on bounded Lipschitz domains is provided by Mitrea, Mitrea, and Shaw in \cite[Section~4]{MitreaMitreaShaw2008} for differential forms, which contains the result for vector fields.

Theorem \ref{thm:genericinhomTrace} and the properties of simple and double layer potentials have been extensively used to study the regularity and well-posedness of elliptic boundary value problems, as well as to derive functional analytic properties of the involved elliptic operators (see, for instance, \cite{JerisonKenig1995,FabesMendezMitrea1998,MendezMitrea2001,MitreaMitreaTaylor2001}).

However, when addressing boundary value problems in unbounded domains, such as bent half-spaces, there may be a lack of control over lower-order derivatives. In such cases, homogeneous estimates are necessary, requiring the use of homogeneous Sobolev and Besov spaces, $\dot{\mathrm{H}}^{s,p}(\Omega)$ and $\dot{\mathrm{B}}^{s}_{p,q}(\Omega)$. For instance, in the case of the flat upper half-space (or the whole space with a trace on a hyperplane), the trace theorem holds, as shown in \cite{Jawerth78}. The conventional definition of these function spaces, involving the restriction of tempered distributions modulo polynomials, is not well-suited for the usual strategy that requires pointwise composition to flatten the boundary. For a suitable realization of homogeneous function spaces on the flat upper half-space, the expected results can be easily derived from the case of inhomogeneous function spaces, as shown by the author in \cite[Section~4]{Gaudin2022}.

Regarding traces with homogeneous estimates, we refer to the recent work of Leoni \cite[Chapter~9]{Leoni2023}, and the work of Leoni and Tice \cite{LeoniTice2019}.

For special Lipschitz domains, a proof of the appropriate (homogeneous) trace estimate is given for Sobolev-Soblodeckij type spaces $\dot{\mathrm{W}}^{s,p}(\Omega)$\footnote{Note that when $0<s<1$ the (semi-)norm involved to define $\dot{\mathrm{W}}^{s,p}(\Omega)$ coincides with the one for the homogeneous Besov space  $\dot{\mathrm{B}}^{s}_{p,p}(\Omega)$ by finite differences (the so called Gagliardo (semi-)norm.) }, $s\in(1/p,1]$, with \cite[Theorem~18.34]{Leoni2017} and \cite[Theorem~9.29]{Leoni2023}.

A trace theorem has also been established for infinite strips with Lipschitz boundaries in \cite{LeoniTice2019}, concerning the so-called 'screened' Sobolev spaces. However, the framework is different in many ways and requires additional subtleties. In this case, the involved Sobolev spaces cannot coincide with the corresponding restriction of Sobolev spaces over $\mathbb{R}^n$. This discrepancy arises from the specific structure of the trace space, which includes additional features absent in the standard case of bounded and half-space type domains. Again, the estimates are only shown for Sobolev-Slobodeckij type spaces $\dot{\mathrm{W}}^{s,p}(\Omega)$ with $s\in(1/p,1]$, using the adapted but unusual trace space.

\medbreak

\textbf{The second main goal of this paper} is to give a proof of the expected trace estimates for scalar-valued homogeneous Sobolev and Besov spaces on special Lipschitz domains with optimal range of regularity:
\begin{align*}
[\cdot]_{|_{\partial\Omega}}\,:\,\dot{\mathrm{H}}^{s,p}(\Omega),\,\dot{\mathrm{B}}^{s}_{p,q}(\Omega)\longrightarrow \dot{\mathrm{B}}^{s-\frac{1}{p}}_{p,q}(\partial\Omega)\text{, }\quad s\in(1/p,1+1/p),\,p\in(1,+\infty),\,q\in[1,+\infty]\text{.}
\end{align*}
In order to prove the homogeneous version of Theorem \ref{thm:genericinhomTrace}, we aim to follow the strategy exhibited in \cite{Ding1996}, and initially described in \cite{Costabel1988}. Costabel and Ding gave a very simple proof for the boundedness of the trace operator
\begin{align*}
[\cdot]_{|_{\partial\Omega}}\,:\,{\mathrm{H}}^{s,2}(\Omega)\longrightarrow {\mathrm{H}}^{s-\frac{1}{2},2}(\partial\Omega)\text{, }\quad s\in(1/2,3/2)
\end{align*}
when $\Omega$ is a special or a bounded Lipschitz domain. In order to obtain such a result, they used function spaces with anisotropic regularity. However, the use of the Fourier transform, and the overall strategy restrict everything to the case of inhomogeneous and $\mathrm{L}^2$-based Sobolev spaces. The idea we present here is to use the global-in-time $\dot{\mathrm{H}}^{\alpha,p}(\mathrm{L}^p)$-maximal regularity for the Poisson semigroup on $\mathbb{R}^{n-1}$, allowed by \cite[Theorem~4.7]{Gaudin2023}, and interpolation theory to replace the use of $\mathrm{L}^2$ techniques.

\subsection{Description of some results through a Hardy-Littlewood-Sobolev-Kato type diagram.}\label{sec:KatoHarSobDiag}

Sobolev and Besov spaces, such as $\dot{\mathrm{W}}^{s,p}(\Omega)$, $\dot{\mathrm{H}}^{s,p}(\Omega)$, and $\dot{\mathrm{B}}^{s}_{p,q}(\Omega)$, can be represented by corresponding points $(\frac{1}{p},s)\in [0,1]\times\mathbb{R}$.

This representation is particularly effective for visually conveying key information, especially within the context of interpolation theory. It is also useful for exhibiting families of function spaces with shared properties.

The following figure illustrates a selection of key results from this paper, particularly the trace theorem and its implications.

\begin{figure}[H]
\centering
\begin{tikzpicture}[yscale=0.6,xscale=6]
  \draw[->] (-0.1,0) -- (1.1,0) node[right,yshift=-2mm] {$1/p$};
  \draw[->] (0,-4) -- (0,4) node[above] {$s$};

  \draw[domain=0:1,smooth,variable=\x,blue] plot ({\x},{3*\x}) node[right] {$s=n/p$};
  \fill[blue!30,opacity=0.3] (0,0) -- plot[domain=0:1] (0,0) -- (0,-4) -- (1,-4) -- (1,0) -- cycle;
  \fill[blue!30,opacity=0.3] (0,0) -- plot[domain=0:1] (\x,{3*\x}) -- (1,3) -- (1,0) -- cycle;
  \draw[domain=0:1,smooth,variable=\x,red] plot ({\x},{-1+\x}) node[right,yshift=2mm] {$s=-1+1/p$};
  \draw[domain=0:1,smooth,variable=\x,red] plot ({\x},{\x}) node[right] {$s=1/p$};
  \fill[red!30,opacity=0.3] (0,0) -- plot[domain=0:1] (\x,\x) -- (1,0) -- cycle;
  \fill[red!30,opacity=0.3] (0,0) -- plot[domain=0:1] (\x,{-1+\x}) -- (1,0) -- cycle;
  \draw[domain=0:1,smooth,variable=\x,orange!80!black] plot ({\x},{1+\x}) node[right] {$s=1+1/p$};
    \fill[orange!80!black,opacity=0.3] (0,0) -- plot[domain=0:1] (\x,{1+\x}) -- (1,1) -- cycle;
  \draw[dashed] (1,3) -- (0,3)  node[left] {$s=n$};
  \draw[dashed] (1,1) -- (0,1)  node[left] {$s=1$};
  \draw[dashed] (1,-3) -- (0,-3) node[left] {$s=-n$};
  \node[circle,fill,inner sep=1.5pt,label=below:$\mathrm{L}^2$] at (0.5,0) {};
  \node[circle,fill,inner sep=1.5pt,label=above:$\dot{\mathrm{H}}^1$] at (0.5,1) {};
  \draw[circle,fill,inner sep=1pt] (0,0) node[below left] {$0$};
  \draw[circle,fill,inner sep=1pt] (1,0) node[below right] {$1$};
\end{tikzpicture}
\caption{Representation of Sobolev and Besov spaces and their properties : a Hardy-Littlewood-Sobolev-Kato diagram. (with $n=3$)}
\end{figure}

\begin{itemize}
    \item The region corresponding to {{ {\eqref{AssumptionCompletenessExponents}} }} for Besov spaces, $s<n/p$ for Sobolev spaces. These are the function spaces that are complete;
    \item The region \textcolor{red}{\bm{$s\in(-1+1/p,1/p)$}} corresponds to spaces where elements of $\dot{\mathrm{H}}^{s,p}(\Omega)$ (or  $\dot{\mathrm{B}}^{s}_{p,q}(\Omega)$) can be extended by zero to the whole space $\mathbb{R}^n$ to obtain an element of $\dot{\mathrm{H}}^{s,p}(\mathbb{R}^n)$ (or $\dot{\mathrm{B}}^{s}_{p,q}(\mathbb{R}^n)$). Additionally, $\mathrm{C}_c^\infty(\Omega)$ is dense when $q<+\infty$.  See \textbf{Corollary \ref{cor:Hdp=Hsp0SpeLip}} and \textbf{Proposition \ref{prop:Bspq=Bspq0SpeLip}};
    \item The region \textcolor{orange!100!black}{\bm{$s\in(1/p,1+1/p)$}} corresponds to spaces where optimal homogeneous trace estimates are established, see \textbf{Theorem \ref{thm:TraceSpeLipopti}}. More precisely, provided $1<p<+\infty$, for all $u\in\dot{\mathrm{H}}^{s,p}(\Omega)$, we have $u_{|_{\partial\Omega}}\in\dot{\mathrm{B}}^{s-\frac{1}{p}}_{p,p}(\partial\Omega)$ with the estimate
    \begin{align*}
        \lVert u_{|_{\partial\Omega}} \rVert_{\dot{\mathrm{B}}^{s-\frac{1}{p}}_{p,p}(\partial\Omega)}\lesssim_{p,s,n}\lVert u \rVert_{\dot{\mathrm{H}}^{s,p}(\Omega)}.
    \end{align*}
    Moreover, \textbf{Lemma~\ref{lem:IdentHsp0and0trace}} shows that the null space of the trace operator can be identified with $\dot{\mathrm{H}}^{s,p}_0(\Omega)$, in which $\mathrm{C}_c^\infty(\Omega)$ is dense in it by \textbf{Proposition~\ref{prop:densityCcinftyinhomHspspeLip}}.
    
    Corresponding results for Besov spaces and their variants are also obtained, with appropriate modifications.
\end{itemize}

\subsection{Notations, definitions, and review of usual concepts}

Throughout this paper, the dimension is $n\geqslant 2$, and $\mathbb{N}$ is the set of non-negative integers. For $a,b\in\mathbb{R}$ with $a\leqslant b$, we write $\llb a,b\rrb:=[a,b]\cap\mathbb{Z}$.

For two real numbers $A,B\in\mathbb{R}$, $A\lesssim_{a,b,c} B$ means that there exists a constant $C>0$ depending on ${a,b,c}$ such that $A\leqslant C B$. When both $A\lesssim_{a,b,c} B$ and $B \lesssim_{a,b,c} A$ hold, we simply write $A\sim_{a,b,c} B$. When there are many indices, we may write $A\lesssim_{a,b,c}^{d,e,f} B$ instead of $A\lesssim_{a,b,c,d,e,f} B$.

\subsubsection{Smooth and measurable functions on open sets}

Let $\mathcal{S}(\mathbb{R}^n,\mathbb{C})$ be the space of complex-valued Schwartz functions, and $\mathcal{S}'(\mathbb{R}^n,\mathbb{C})$ its dual, called the space of tempered distributions. The Fourier transform on $\mathcal{S}'(\mathbb{R}^n,\mathbb{C})$ is written $\mathcal{F}$, and it is pointwise defined for any $f\in\mathrm{L}^1(\mathbb{R}^n,\mathbb{C})$ as
\begin{align*}
  \mathcal{F}f(\xi) :=\int_{\mathbb{R}^n} f(x)\,e^{-ix\cdot\xi}\,\mathrm{d}x\text{, } \xi\in\mathbb{R}^n\text{. }
\end{align*}
Additionally, for $p\in[1+\infty]$, the quantity $p'=\tfrac{p}{p-1}$ is its \textit{\textbf{H\"{o}lder conjugate}}.

For any $m\in\mathbb{N}$, the map $\nabla^m\,:\,\mathcal{S}'(\mathbb{R}^n,\mathbb{C})\longrightarrow \mathcal{S}'(\mathbb{R}^n,\mathbb{C}^{n^m})$ is defined as $\nabla^m u := (\partial^\alpha u)_{|\alpha|=m}$. The Laplace operator on $\mathbb{R}^n$ is given as $\Delta= \partial_{x_1}^2+\partial_{x_2}^2+\ldots+\partial_{x_{n-1}}^2+\partial_{x_n}^2$.

 We introduce the operator $\nabla'$ which stands for the gradient on $\mathbb{R}^{n-1}$ identified with the first $n-1$ variables of $\mathbb{R}^n$, \textit{i.e.} $\nabla'=(\partial_{x_1}, \ldots, \partial_{x_{n-1}})$. Similarly, one defines $\Delta' = \partial_{x_1}^2+\partial_{x_2}^2+\ldots+\partial_{x_{n-1}}^2$. We denote by $(e^{-t(-\Delta')^\frac{1}{2}})_{t\geqslant0}$ the Poisson semigroup on $\mathbb{R}^{n-1}$.

If $\Omega$ is an open set in $\mathbb{R}^n$, $\mathrm{C}_c^\infty(\Omega,\mathbb{C})$ denotes the set of smooth, compactly supported functions in $\Omega$, and $\mathcal{D}'(\Omega,\mathbb{C})$ is its topological dual. For $p\in[1,+\infty)$, $\mathrm{L}^p(\Omega,\mathbb{C})$  is the normed vector space of complex-valued (Lebesgue-) measurable functions whose $p$-th power is integrable with respect to the Lebesgue measure. $\mathcal{S}(\overline{\Omega},\mathbb{C})$ (\textit{resp.} $\mathrm{C}_c^\infty(\overline{\Omega},\mathbb{C})$) stands for functions which are restrictions on $\Omega$ of elements of $\mathcal{S}(\mathbb{R}^n,\mathbb{C})$ (\textit{resp.} $\mathrm{C}_c^\infty(\mathbb{R}^n,\mathbb{C})$). Unless the contrary is explicitly stated, we always identify $\mathrm{L}^p(\Omega,\mathbb{C})$ (resp. $\mathrm{C}_c^\infty(\Omega,\mathbb{C})$) as the subspace of functions in $\mathrm{L}^p(\mathbb{R}^n,\mathbb{C})$ (resp. $\mathrm{C}_c^\infty(\mathbb{R}^n,\mathbb{C})$) supported in $\overline{\Omega}$ through the extension by $0$ outside $\Omega$. $\mathrm{L}^\infty(\Omega,\mathbb{C})$ stands for the space of essentially bounded (Lebesgue-) measurable functions.

For $s\in\mathbb{R}$, $p\in[1,+\infty)$, $\ell^p_s(\mathbb{Z},\mathbb{C})$ denotes the normed vector space of $p$-summable sequences of complex numbers with respect to the counting measure $2^{ksp}\mathrm{d}k$;  $\ell^\infty_s(\mathbb{Z},\mathbb{C})$ denotes the vector space of sequences $(x_k)_{k\in\mathbb{Z}}$ such that $(2^{ks}x_k)_{k\in\mathbb{Z}}$  is bounded.
More generally, when $X$ is a Banach space, for $p\in[1,+\infty]$, one can also consider $\mathrm{L}^p(\Omega,X)$ denotes the space of (Bochner-)measurable functions $u\,:\,\Omega\longrightarrow X$, such that $t\mapsto\lVert u(t)\rVert_X \in \mathrm{L}^p(\Omega,\mathbb{R})$. Similarly, one can consider $\ell^p_s(\mathbb{Z},X)$. Finally, $\mathrm{C}^0(\Omega,X)$ denotes the space of continuous functions on $\Omega\subset \mathbb{R}^n$ with values in $X$. The subspace $\mathrm{C}^0_b(\mathbb{R},X)$  consists of uniformly bounded continuous functions. $\mathrm{C}^0_0(\mathbb{R},X)$ denotes the subspace of continuous functions that vanish at infinity. For $\mathcal{C}\in\{\mathrm{C},\mathrm{C}_c,\mathrm{C}_b,\mathrm{C}_0\}$, we define $\mathcal{C}^0(\overline{\Omega}, X)$ as the vector space of continuous functions on $\overline{\Omega}$ which are restrictions of elements that belongs to $\mathcal{C}^0(\mathbb{R}^n, X)$. The same goes for the function spaces of continuously differentiable functions up to the order $k\in\mathbb{N}$: $\mathcal{C}^k(\overline{\Omega}, X):=\{\,u\in \mathcal{C}^{k-1}(\overline{\Omega}, X)\,|\,\nabla u \in \mathcal{C}^{k-1}(\overline{\Omega}, X)\,\}$, and the space of continuously differentiable functions up to the order $k$, such that the derivatives of order $k$ are Lipschitz $\mathcal{C}^{k,1}(\overline{\Omega}, X):=\{\,u\in \mathcal{C}^{k}(\overline{\Omega}, X)\,|\,\nabla^{k+1} u \in \mathrm{L}^{\infty}({\Omega}, X)\,\}$. We keep similar notations with $\Omega$ instead of $\overline{\Omega}$.

For $\Omega$ an open set of $\mathbb{R}^n$, we say that $\Omega$ is a special Lipschitz domain, if there exists, up to a rotation, a globally Lipschitz function $\phi\,:\,\mathbb{R}^{n-1}\longrightarrow\mathbb{R}$, such that
\begin{align*}
    \Omega =\{\,(x',x_n)\in\mathbb{R}^{n-1}\times\mathbb{R}\,|\, x_n>\phi(x')\,\}\text{.}
\end{align*}
In other words, a special Lipschitz domain of $\mathbb{R}^n$ is the epigraph of real valued Lipschitz function defined on $\mathbb{R}^{n-1}$.

\subsubsection{Interpolation of normed vector spaces}

For more details, the interested reader is invited to consult \cite{BerghLofstrom1976,bookTriebel1978,bookLunardiInterpTheory} or \cite[Section~1.3]{EgertPhDThesis2015}.

Let $(X,\left\lVert\cdot\right\rVert_X)$ and $(Y,\left\lVert\cdot\right\rVert_Y)$ be two normed vector spaces. We write $X\hookrightarrow Y$ to indicate that $X$ embeds continuously into $Y$. We briefly recall the basics of interpolation theory. If there exists a Hausdorff topological vector space $Z$, such that $X,Y\subset Z$, then $X\cap Y$ and $X+Y$ are normed vector spaces with their canonical norms. One can define the $K$-functional of $z\in X+Y$, for any $t>0$ by
\begin{align*}
    K(t,z,X,Y) := \underset{\substack{(x,y)\in X\times Y,\\ z=x+y}}{\inf}\left({\left\lVert{x}\right\rVert_{X}+t\left\lVert{y}\right\rVert_{Y}}\right)\text{. }
\end{align*}
This enables us to construct, for any $\theta\in(0,1)$, $q\in[1,+\infty]$, the real interpolation spaces between $X$ and $Y$ with indices $\theta,q$ as
\begin{align*}
    (X,Y)_{\theta,q} := \left\{\, x\in X+Y\,\Big{|}\,t\longmapsto t^{-\theta}K(t,x,X,Y)\in\mathrm{L}^q_\ast(\mathbb{R}_+)\,\right\}\text{, }
\end{align*}
where $\mathrm{L}^q_\ast(\mathbb{R}_+):=\mathrm{L}^q((0,+\infty),\mathrm{d}t/t)$. When $q=+\infty$, one can consider
\begin{align*}
    (X,Y)_{\theta} := \left\{\, x\in X+Y\,\Big{|}\,\lim_{t,t^{-1}\rightarrow0_+} t^{-\theta}K(t,x,X,Y)=0\,\right\}\text{, }
\end{align*}
endowed with the norm $\lVert \cdot \rVert_{(X,Y)_{\theta,\infty}}$.

If moreover we assume that $X$ and $Y$ are complex Banach spaces, one can consider $\mathrm{F}(X,Y)$ the set of all continuous functions $f:\overline{S}\longmapsto X+Y$, $S$ being the strip of complex numbers whose real part is between $0$ and $1$, with $f$ holomorphic in $S$, and such that
\begin{align*}
    t\longmapsto f(it)\in \mathrm{C}^0_b(\mathbb{R},X) \quad\text{ and }\quad t\longmapsto f(1+it)\in \mathrm{C}^0_b(\mathbb{R},Y)\text{. }
\end{align*}
We can endow the space $\mathrm{F}(X,Y)$ with the norm
\begin{align*}
    \left\lVert{f}\right\rVert_{\mathrm{F}(X,Y)}:=\max\left(\underset{t\in\mathbb{R}}{\mathrm{sup}} \left\lVert {f(it)}\right\rVert_{X},\underset{t\in\mathbb{R}}{\mathrm{sup}} \left\lVert {f(1+it)}\right\rVert_{Y}\right)\text{, }
\end{align*}
which makes $\mathrm{F}(X,Y)$ a Banach space since it is a closed subspace of $\mathrm{C}^0(\overline{S},X+Y)$.
Hence, for $\theta\in(0,1)$, the normed vector space given by
\begin{align*}
    [X,Y]_{\theta} &:= \left\{\,f(\theta)\,\big{|}\,f\in \mathrm{F}(X,Y)\,\right\} 
    \text{, }\\
    \left\lVert{x}\right\rVert_{[X,Y]_{\theta}} &:= \underset{\substack{f\in \mathrm{F}(X,Y),\\ f(\theta)=x}}{\inf} \left\lVert{f}\right\rVert_{\mathrm{F}(X,Y)}\text{, }
\end{align*}
is a Banach space called the complex interpolation space between $X$ and $Y$ associated with $\theta$.

%----------------------------------------------------
%------------------- Section 2 ----------------------
%----------------------------------------------------
%----------------------------------------------------
%------------------- Section 2 ----------------------
%----------------------------------------------------

\section{Around the standard theory of Sobolev and Besov spaces}

This section provides a review of the standard construction of inhomogeneous and homogeneous Sobolev spaces. We will also refine known properties for our construction of homogeneous function spaces on $\mathbb{R}^n$, including the endpoint cases $p=1,+\infty$.

\subsection{Sobolev and Besov spaces on the whole space}\label{subsec:SobolevBesovRn}

We start the section as in \cite[Subsection~2A]{Gaudin2022}: to address Sobolev and Besov spaces on the whole space, we introduce the Littlewood-Paley decomposition with $\varphi\in \mathrm{C}_c^\infty(\mathbb{R}^n)$, a radial, real-valued, non-negative function such that
\begin{itemize}[label={$\bullet$}]
    \item $\supp \varphi \subset B(0,4/3)$;
    \item ${\varphi}_{|_{B(0,3/4)}}=1$;
\end{itemize}
Thus, we define the following functions for any $j\in\mathbb{Z}$ and for all $\xi\in\mathbb{R}^n$,
\begin{align*}
    \varphi_j(\xi):=\varphi(2^{-j}\xi)\text{, }\qquad \psi_j(\xi) := \varphi_{j}(\xi/2)-\varphi_{j}(\xi)\text{,}
\end{align*}
and the family $(\psi_j)_{j\in\mathbb{Z}}$ has the following properties
\begin{itemize}[label={$\bullet$}]
    \item $\mathrm{supp}(\psi_j)\subset \{\,\xi\in\mathbb{R}^n\,|\,3\cdot 2^{j-2}\leqslant\left\lvert{\xi}\right\rvert \leqslant 2^{j+3}/3\,\}$;
    \item $\forall\xi\in\mathbb{R}^n\setminus\{0\}$, $\sum\limits_{j=-M}^N{\psi_j}(\xi)\xrightarrow[N,M\rightarrow+\infty]{} 1$.
\end{itemize}
Such a family $(\varphi,(\psi_j)_{j\in\mathbb{Z}})$  is called a Littlewood-Paley family. We consider the following two families of operators associated with their Fourier multipliers:
\begin{itemize}[label={$\bullet$}]
    \item The \textit{\textbf{homogeneous}} family of Littlewood-Paley dyadic decomposition operators $(\dot{\Delta}_j)_{j\in\mathbb{Z}}$, where
    \begin{align*}
        \dot{\Delta}_j := \mathcal{F}^{-1}\psi_j\mathcal{F}\text{,}
    \end{align*}
    \item The \textit{\textbf{inhomogeneous}} family of Littlewood-Paley dyadic decomposition operators $({\Delta}_k)_{k\in\mathbb{Z}}$, where
    \begin{align*}
       {\Delta}_{-1} := \mathcal{F}^{-1}\varphi\mathcal{F}\text{,}
    \end{align*}
    $\Delta_k:=\dot{\Delta}_k$ for any $k\geqslant 0$, and $\Delta_k:=0$ for any $k\leqslant-2$.
    \item The low frequency cut-off operators $(\dot{S}_k)_{k\in\mathbb{Z}}$, where
    \begin{align*}
       \dot{S}_k := \mathcal{F}^{-1}\varphi_k\mathcal{F}\text{.}
    \end{align*}
    Note that $\dot{S}_0 = \Delta_{-1}$.
\end{itemize}
As a direct application of Young's inequality for the convolution, they are all uniformly bounded families of operators on $\mathrm{L}^p(\mathbb{R}^n)$, $p\in[1,+\infty]$.

Both families of operators lead to the following quantities, for $s\in\mathbb{R}$, $p,q\in[1,+\infty]$ and $u\in \mathcal{S}'(\mathbb{R}^n)$,
\begin{align*}
    \left\lVert u \right\rVert_{\mathrm{B}^{s}_{p,q}(\mathbb{R}^n)}= \left\lVert(2^{ks}\left\lVert {\Delta}_k u \right\rVert_{\mathrm{L}^{p}(\mathbb{R}^n)})_{k\in\mathbb{Z}}\right\rVert_{\ell^{q}(\mathbb{Z})}\text{ and }
    \left\lVert u \right\rVert_{\dot{\mathrm{B}}^{s}_{p,q}(\mathbb{R}^n)} = \left\lVert \big(2^{js}\lVert \dot{\Delta}_j u \rVert_{\mathrm{L}^{p}(\mathbb{R}^n)}\big)_{j\in\mathbb{Z}}\right\rVert_{\ell^{q}(\mathbb{Z})}\text{.}
\end{align*}
These are respectively named the inhomogeneous and homogeneous Besov norms. However, the homogeneous norm is not a true norm since $\left\lVert u \right\rVert_{\dot{\mathrm{B}}^{s}_{p,q}(\mathbb{R}^n)}=0$ does not imply $u=0$. Thus, following \cite[Chapter~2]{bookBahouriCheminDanchin} and \cite[Chapter~3]{DanchinHieberMuchaTolk2020}, we introduce a subspace of tempered distributions where $\left\lVert \cdot \right\rVert_{\dot{\mathrm{B}}^{s}_{p,q}(\mathbb{R}^n)}$ is point--separating, say
\begin{align*}
   \mathcal{S}'_h(\mathbb{R}^n) &:= \left\{ u\in \mathcal{S}'(\mathbb{R}^n)\,\Big{|}\,\forall \Theta \in \mathrm{C}_c^\infty(\mathbb{R}^n),\,  \left\lVert \Theta(\lambda \mathfrak{D}) u \right\rVert_{\mathrm{L}^\infty(\mathbb{R}^n)} \xrightarrow[\lambda\rightarrow+\infty]{} 0\right\}\text{,}
\end{align*}
where for $\lambda>0$, $\Theta(\lambda \mathfrak{D})u = \mathcal{F}^{-1}{\Theta}(\lambda\cdot)\mathcal{F}u$. Note that $\mathcal{S}'_h(\mathbb{R}^n)$ does not contain any polynomials, so that $\left\lVert u \right\rVert_{\dot{\mathrm{B}}^{s}_{p,q}(\mathbb{R}^n)}=0$ does imply that $u=0$ when $u\in\mathcal{S}'_h(\mathbb{R}^n)$.

One can also define the following quantities, referred to as the potential norms of inhomogeneous and homogeneous Sobolev spaces:
\begin{align*}
    \left\lVert {u} \right\rVert_{\mathrm{\mathrm{H}}^{s,p}(\mathbb{R}^n)} := \left\lVert {(\mathrm{I}-\Delta)^\frac{s}{2} u} \right\rVert_{\mathrm{L}^{p}(\mathbb{R}^n)}\text{ and } \left\lVert {u} \right\rVert_{\dot{\mathrm{H}}^{s,p}(\mathbb{R}^n)} := \Big\lVert \sum_{j\in\mathbb{Z}} (-\Delta)^\frac{s}{2}\dot{\Delta}_{j} u  \Big\rVert_{{\mathrm{L}}^{p}(\mathbb{R}^n)}\text{, }
\end{align*}
where $(-\Delta)^\frac{s}{2}$ is understood on $u\in \mathcal{S}'_h(\mathbb{R}^n)$ by the action on its dyadic decomposition, \textit{i.e.},
\begin{align*}
    (-\Delta)^\frac{s}{2}\dot{\Delta}_j u := \mathcal{F}^{-1}|\xi|^s\mathcal{F}\dot{\Delta}_j u\text{.}
\end{align*}
This gives a family of bounded $\mathrm{C}^\infty$ functions, which follows from \cite[Lemma~2.2]{bookBahouriCheminDanchin} and the fact that $u\in\mathcal{S}'_h(\mathbb{R}^n)$. We recall the definition of the standard Sobolev (semi-)norms, provided $k\in\mathbb{N}$, $p\in[1,+\infty]$, and $u\in\mathcal{S}'(\mathbb{R}^n)$
\begin{align*}
    \lVert {u} \rVert_{\dot{\mathrm{W}}^{k,p}(\mathbb{R}^n)}:=\lVert\nabla^k u\rVert_{\mathrm{L}^p(\mathbb{R}^n)}\text{ and } \lVert {u} \rVert_{{\mathrm{W}}^{k,p}(\mathbb{R}^n)}:=\lVert u\rVert_{\mathrm{L}^p(\mathbb{R}^n)} +\lVert\nabla^k u\rVert_{\mathrm{L}^p(\mathbb{R}^n)}.
\end{align*}

\begin{definition}\label{def:HomFuncSpacesRn} For any $p,q\in[1,+\infty]$, $s\in\mathbb{R}$, and $k\in\mathbb{N}$, we define
\begin{itemize}[label={$\bullet$}]
    \item the inhomogeneous and homogeneous Sobolev (Bessel and Riesz potential) spaces,\\
    \resizebox{0.90\textwidth}{!}{$
        \mathrm{\mathrm{H}}^{s,p}(\mathbb{R}^n)=\left\{\, u\in\mathcal{S}'(\mathbb{R}^n) \,\big{|}\, \left\lVert {u} \right\rVert_{\mathrm{\mathrm{H}}^{s,p}(\mathbb{R}^n)}<+\infty \,\right\}\text{, }\dot{\mathrm{H}}^{s,p}(\mathbb{R}^n)=\left\{\, u\in\mathcal{S}'_h(\mathbb{R}^n) \,\big{|}\, \left\lVert {u} \right\rVert_{\dot{\mathrm{H}}^{s,p}(\mathbb{R}^n)}<+\infty \,\right\}\text{ ; }$}
    \item the standard inhomogeneous and homogeneous Sobolev spaces,\\
    \resizebox{0.90\textwidth}{!}{$
        \mathrm{\mathrm{W}}^{k,p}(\mathbb{R}^n)=\left\{\, u\in\mathcal{S}'(\mathbb{R}^n) \,\big{|}\, \left\lVert {u} \right\rVert_{{\mathrm{W}}^{k,p}(\mathbb{R}^n)}<+\infty \,\right\}\text{, }\dot{\mathrm{W}}^{k,p}(\mathbb{R}^n)=\left\{\, u\in\mathcal{S}'_h(\mathbb{R}^n) \,\big{|}\, \lVert {u} \rVert_{\dot{\mathrm{W}}^{k,p}(\mathbb{R}^n)}<+\infty \,\right\}\text{ ; }$}
    \item and the inhomogeneous and homogeneous Besov spaces,\\
    \resizebox{0.90\textwidth}{!}{$\mathrm{B}^{s}_{p,q}(\mathbb{R}^n)=\left\{\, u\in\mathcal{S}'(\mathbb{R}^n) \,\big{|}\, \left\lVert {u} \right\rVert_{\mathrm{B}^{s}_{p,q}(\mathbb{R}^n)}<+\infty \,\right\}\text{, }\dot{\mathrm{B}}^{s}_{p,q}(\mathbb{R}^n)=\left\{\, u\in\mathcal{S}'_h(\mathbb{R}^n) \,\big{|}\, \left\lVert {u} \right\rVert_{\dot{\mathrm{B}}^{s}_{p,q}(\mathbb{R}^n)}<+\infty \,\right\}\text{. }
    $}
\end{itemize}
These are all normed vector spaces. For all $k\in\mathbb{N}$, we set
\begin{itemize}
    \item $\mathrm{L}^\infty_h(\mathbb{R}^n):= \mathrm{L}^\infty(\mathbb{R}^n)\cap \mathcal{S}'_h(\mathbb{R}^n$)($=\dot{\mathrm{W}}^{0,\infty}(\mathbb{R}^n)$),
    \item $\dot{\mathrm{C}}^{k}_{0}(\mathbb{R}^n)$ is the space ${\mathrm{C}}^{k}_{0}(\mathbb{R}^n)$ equipped with the (semi-)norm $\lVert {\cdot} \rVert_{\dot{\mathrm{W}}^{k,\infty}(\mathbb{R}^n)}$,
    \item $\dot{\mathrm{C}}^{k}_{b,h}(\mathbb{R}^n)$ is the space ${\mathrm{C}}^{k}_{b,h}(\mathbb{R}^n):={\mathrm{C}}^{k}_{b}(\mathbb{R}^n)\cap\mathcal{S}'_h(\mathbb{R}^n)$ endowed with the (semi-)norm $\lVert {\cdot} \rVert_{\dot{\mathrm{W}}^{k,\infty}(\mathbb{R}^n)}$.
\end{itemize}

We can consider, for $p,q\in[1,+\infty]$, $s\in\mathbb{R}$, the following spaces
\begin{align*}
    \dot{\mathcal{B}}^{s}_{p,\infty}(\mathbb{R}^n):=&\Big\{\,u\in \dot{\mathrm{B}}^{s}_{p,\infty}(\mathbb{R}^n)\,\Big|\, \lim\limits_{|j|\rightarrow+\infty} 2^{js} \lVert\dot{\Delta}_{j}u\rVert_{\mathrm{L}^p(\mathbb{R}^n)}=0\,\Big\},\\
    \dot{\mathrm{B}}^{s,0}_{\infty,q}(\mathbb{R}^n):=&\Big\{\,u\in \dot{\mathrm{B}}^{s}_{\infty,q}(\mathbb{R}^n)\,\Big|\, (\dot{\Delta}_ju)_{j\in\mathbb{Z}}\subset\mathrm{C}_0^0(\mathbb{R}^n)\,\Big\},\\
    \dot{\mathcal{B}}^{s,0}_{\infty,\infty}(\mathbb{R}^n):=& \dot{\mathcal{B}}^{s}_{\infty,\infty}(\mathbb{R}^n)\cap\dot{\mathrm{B}}^{s,0}_{\infty,\infty}(\mathbb{R}^n).
\end{align*}
Finally, one can define their inhomogeneous counterparts ${\mathcal{B}}^{s}_{p,\infty}(\mathbb{R}^n)$, ${\mathrm{B}}^{s,0}_{\infty,q}(\mathbb{R}^n)$, and ${\mathcal{B}}^{s,0}_{\infty,\infty}(\mathbb{R}^n)$ replacing $(\dot{\mathcal{B}},\dot{\mathrm{B}},\dot{\Delta}_j)$ by $({\mathcal{B}},{\mathrm{B}},{\Delta}_j)$.
\end{definition}

The treatment of homogeneous Besov spaces $\dot{\mathrm{B}}^{s}_{p,q}(\mathbb{R}^n)$, $s\in\mathbb{R}$, $p,q\in[1,+\infty]$, defined on $\mathcal{S}'_h(\mathbb{R}^n)$ has been extensively covered in \cite[Chapter~2]{bookBahouriCheminDanchin}. The corresponding construction for homogeneous Sobolev spaces $\dot{\mathrm{H}}^{s,p}(\mathbb{R}^n)$, $s\in\mathbb{R}$, $p\in(1,+\infty)$ has been achieved only recently by the author, \cite[Section~2]{Gaudin2022}. This follows partial constructions such a \cite[Chapter~1]{bookBahouriCheminDanchin} for the case $p=2$, \cite[Chapter~3]{DanchinHieberMuchaTolk2020} for the case $s\in\mathbb{N}$.

The inhomogeneous spaces  $\mathrm{L}^p(\mathbb{R}^n)$, $\mathrm{\mathrm{H}}^{s,p}(\mathbb{R}^n)$, and $\mathrm{B}^{s}_{p,q}(\mathbb{R}^n)$ are all complete for all $p,q\in [1,+\infty]$, $s\in\mathbb{R}$. However, homogeneous spaces are not always complete in this setting (see \cite[Proposition~1.34,~Remark~2.26]{bookBahouriCheminDanchin}). It is shown (see \cite[Theorem~2.25]{bookBahouriCheminDanchin}) that homogeneous Besov spaces $\dot{\mathrm{B}}^{s}_{p,q}(\mathbb{R}^n)$ are complete whenever $(s,p,q)\in\mathbb{R}\times[1,+\infty]\times[1,+\infty]$ satisfies
\begin{align*}\tag{$\mathcal{C}_{s,p,q}$}\label{AssumptionCompletenessExponents}
    \left[ s<\frac{n}{p} \right]\text{ or }\left[q=1\text{ and } s\leqslant\frac{n}{p} \right]\text{, }
\end{align*}
For the remainder of this paper, we denote by $(\mathcal{C}_{s,p})$ the statement $(\mathcal{C}_{s,p,p})$. Similarly, one may show that for $p\in(1,+\infty)$, $\dot{\mathrm{H}}^{s,p}(\mathbb{R}^n)$ is complete whenever $(\mathcal{C}_{s,p})$ is satisfied, see \cite[Proposition~2.2]{Gaudin2022}. See Theorem \ref{thm:CompletenessSpacesRn} below for a more general statement with a proof that includes the Sobolev spaces with Lebesgue index $p=1$.

For all $s>0$, $(p,q)\in(1,+\infty)\times[1,+\infty]$, we have $\mathrm{L}^p(\mathbb{R}^n)\cap \dot{\mathrm{H}}^{s,p}(\mathbb{R}^n) = \mathrm{\mathrm{H}}^{s,p}(\mathbb{R}^n)$, and $\mathrm{L}^p(\mathbb{R}^n)\cap \dot{\mathrm{B}}^{s}_{p,q}(\mathbb{R}^n) =\mathrm{B}^{s}_{p,q}(\mathbb{R}^n)$ with equivalent norms (even when $p=1$ in the case of Besov spaces). See \cite[Theorem~6.3.2]{BerghLofstrom1976} for more details.

We also have the useful and very well-known equivalences of norms for the Riesz potential Sobolev spaces and for the homogeneous Besov spaces. See \cite[Proposition~2.2]{Gaudin2022}, \cite[Proposition~3.7]{DanchinHieberMuchaTolk2020} and \cite[Lemmas~2.1~\&~2.2]{bookBahouriCheminDanchin}.
\begin{proposition}\label{lem:Triebelnormeq} Let $p,q\in[1,+\infty]$ and $s\in\mathbb{R}$.
\begin{enumerate}
    \item  If $p\in(1,+\infty)$, for all $u\in\mathcal{S}'_h(\mathbb{R}^n)$ it holds
\begin{align*}
    \big\lVert(\dot{\Delta}_j u )_{j\in\mathbb{Z}}\big\rVert_{\mathrm{L}^{p}(\mathbb{R}^n,\ell^{2}_s(\mathbb{Z}))}\sim_{p,n,s} \left\lVert{u}\right\rVert_{\dot{\mathrm{H}}^{s,p}(\mathbb{R}^n)}.
\end{align*}
Furthermore, for $m\in\mathbb{N}$, one has 
\begin{align*}
    \lVert{u}\rVert_{\dot{\mathrm{H}}^{s+m,p}(\mathbb{R}^n)} \sim_{p,n,m,s} \sum_{k=1}^n\lVert{ \partial_{x_k}^mu}\rVert_{\dot{\mathrm{H}}^{s,p}(\mathbb{R}^n)} \sim_{p,n,m,s} \lVert{ \nabla^m u}\rVert_{\dot{\mathrm{H}}^{s,p}(\mathbb{R}^n)}.
\end{align*}
In particular, when $s=0$, we have $\dot{\mathrm{W}}^{m,p}(\mathbb{R}^n)= \dot{\mathrm{H}}^{m,p}(\mathbb{R}^n)$ as a set and with equivalence of norms.

\item For $\beta\geqslant 0$, $m\in\mathbb{N}$, all $u\in\mathcal{S}'_h(\mathbb{R}^n)$ it holds
\begin{gather*}
    \lVert{u}\rVert_{\dot{\mathrm{B}}^{s+\beta}_{p,q}(\mathbb{R}^n)} \sim_{p,n,\beta,s}  \lVert{ (-\Delta)^\frac{\beta}{2} u}\rVert_{\dot{\mathrm{B}}^{s}_{p,q}(\mathbb{R}^n)},\\
    \lVert{u}\rVert_{\dot{\mathrm{B}}^{s+m}_{p,q}(\mathbb{R}^n)} \sim_{p,n,m,s} \sum_{k=1}^n\lVert{ \partial_{x_k}^mu}\rVert_{\dot{\mathrm{B}}^{s}_{p,q}(\mathbb{R}^n)} \sim_{p,n,m,s} \lVert{ \nabla^m u}\rVert_{\dot{\mathrm{B}}^{s}_{p,q}(\mathbb{R}^n)}.
\end{gather*}
\end{enumerate}
\end{proposition}

The preceding lemma indicates that one can focus mainly on the treatment of potential Sobolev spaces $\dot{\mathrm{H}}^{s,p}(\mathbb{R}^n)$, $s\in\mathbb{R}$, $1<p<+\infty$, while dealing separately with the cases $p=1,+\infty$.

Now, we state several facts that are, for the most part, well-known. The next proposition is standard.

\begin{proposition}\label{prop:EmbeddSobBesovRn} Let $p,q\in[1,+\infty]$. The following embeddings hold:
\begin{enumerate}
    \item For $1\leqslant r\leqslant q<+\infty$, $s\in\mathbb{R}$,
    \begin{align*}
        {\dot{\mathrm{B}}^{s}_{p,q}(\mathbb{R}^n)}\hookrightarrow{\dot{\mathrm{B}}^{s}_{p,r}(\mathbb{R}^n)}\hookrightarrow{\dot{\mathcal{B}}^{s}_{p,\infty}(\mathbb{R}^n)}\hookrightarrow{\dot{\mathrm{B}}^{s}_{p,\infty}(\mathbb{R}^n)}.
    \end{align*}
    \item For $k\in\mathbb{N}$,
    \begin{align*}
        {\dot{\mathrm{B}}^{k}_{p,1}(\mathbb{R}^n)}\hookrightarrow{\dot{\mathrm{W}}^{k,p}(\mathbb{R}^n)}\hookrightarrow {\dot{\mathrm{B}}^{k}_{p,\infty}(\mathbb{R}^n)}.
    \end{align*}
    \item For $s\in\mathbb{R}$,
    \begin{align*}
        {\dot{\mathrm{B}}^{s}_{p,1}(\mathbb{R}^n)}\hookrightarrow{\dot{\mathrm{H}}^{s,p}(\mathbb{R}^n)}\hookrightarrow {\dot{\mathrm{B}}^{s}_{p,\infty}(\mathbb{R}^n)}.
    \end{align*}
    \item For $s>0$, $k\in\mathbb{N}^\ast$,
    \begin{align*}
        {{\mathrm{W}}^{k,1}(\mathbb{R}^n)}&\hookrightarrow{\dot{\mathrm{W}}^{k,1}(\mathbb{R}^n)} &&\text{ and } &&{{\mathrm{W}}^{k,\infty}(\mathbb{R}^n)}\cap\mathrm{L}^\infty_h(\mathbb{R}^n)\hookrightarrow{\dot{\mathrm{W}}^{k,\infty}(\mathbb{R}^n)};\\
        \text{if } 1<p<+\infty,\,{{\mathrm{H}}^{s,p}(\mathbb{R}^n)}&\hookrightarrow{\dot{\mathrm{H}}^{s,p}(\mathbb{R}^n)} &&\text{ and } &&{\dot{\mathrm{H}}^{-s,p}(\mathbb{R}^n)}\hookrightarrow{{\mathrm{H}}^{-s,p}(\mathbb{R}^n)};\\
        {{\mathrm{B}}^{s}_{p,q}(\mathbb{R}^n)}\cap \mathcal{S}'_h(\mathbb{R}^n)&\hookrightarrow{\dot{\mathrm{B}}^{s}_{p,q}(\mathbb{R}^n)} &&\text{ and } 
        &&\dot{\mathrm{B}}^{-s}_{p,q}(\mathbb{R}^n)\hookrightarrow{{\mathrm{B}}^{-s}_{p,q}(\mathbb{R}^n)}.
    \end{align*}
    \item For $s=0$,
    \begin{align*}
        {\dot{\mathrm{B}}^{0}_{p,1}(\mathbb{R}^n)}\hookrightarrow{{\mathrm{B}}^{0}_{p,1}(\mathbb{R}^n)}\text{ and }{{\mathrm{B}}^{0}_{p,\infty}(\mathbb{R}^n)}\hookrightarrow{\dot{\mathrm{B}}^{0}_{p,\infty}(\mathbb{R}^n)}.
    \end{align*}
\end{enumerate}
\end{proposition}

At the first glance, the nature of the elements in those homogeneous function spaces is unclear. The next lemma states that, in particular, for positive regularity indices, we deal at least with measurable functions. However, using $\mathcal{S}'_h(\mathbb{R}^n)$ as an ambient space is at the cost of a pathological behavior, especially for $\mathrm{L}^\infty$-type spaces. This pathological behavior can be even worse in the case of a weaker definition of $\mathcal{S}'_h(\mathbb{R}^n)$, see \cite[Theorems~21~\&~22]{Cobb24Sprimeh}.

The next lemma will be of a paramount importance throughout the whole paper.

\begin{lemma}[Structure lemma]\label{lem:structLemma} We have the following properties.
\begin{enumerate}
    \item We have the set inclusion $\mathrm{L}^p(\mathbb{R}^n)+\mathrm{C}^0_0(\mathbb{R}^n)\subset \mathcal{S}'_h(\mathbb{R}^n)$, $1\leqslant p <+\infty$.
    \item The spaces $\mathrm{L}^\infty_h(\mathbb{R}^n)$ and $\mathrm{C}_{b,h}^0(\mathbb{R}^n)$ are respectively strict closed subspaces of $\mathrm{L}^\infty(\mathbb{R}^n)$ and $\mathrm{C}_{b}^0(\mathbb{R}^n)$.
    \item The following set inclusions hold, provided $k_s:=\max(0,\lceil s\rceil)$:
    \begin{enumerate}
        \item $\dot{\mathrm{H}}^{s,p}(\mathbb{R}^n) \subset \mathrm{H}^{s,p}(\mathbb{R}^n)+\mathrm{C}^{k_s}_0(\mathbb{R}^n)$, $1\leqslant p <+\infty$, $s\in\mathbb{R}$, 
        \item $\dot{\mathrm{W}}^{k,1}(\mathbb{R}^n) \subset \mathrm{W}^{k,1}(\mathbb{R}^n)+\mathrm{C}^k_0(\mathbb{R}^n)$, $k\in\mathbb{N}$,
        \item $\dot{\mathrm{W}}^{k,\infty}(\mathbb{R}^n) \subset \mathrm{W}^{k,\infty}(\mathbb{R}^n)\cap \mathcal{S}'_h(\mathbb{R}^n)+\mathrm{C}^{k}_{b,h}(\mathbb{R}^n)$, $k\in\mathbb{N}^\ast$,
        \item $\dot{\mathrm{B}}^{s}_{p,q}(\mathbb{R}^n) \subset \mathrm{B}^{s}_{p,q}(\mathbb{R}^n)+\mathrm{C}^{k_s}_0(\mathbb{R}^n)$, $1\leqslant p <+\infty$, $1\leqslant q\leqslant +\infty$, $s\in\mathbb{R}$,
        \item $\dot{\mathrm{B}}^{s}_{\infty,q}(\mathbb{R}^n) \subset {\mathrm{B}}^{s}_{\infty,q}(\mathbb{R}^n)\cap \mathcal{S}'_h(\mathbb{R}^n)+\mathrm{C}^{k_s}_{b,h}(\mathbb{R}^n)$, $1\leqslant q\leqslant +\infty$, $s\in\mathbb{R}$,
        \item $\dot{\mathrm{B}}^{s,0}_{\infty,q}(\mathbb{R}^n) \subset {\mathrm{B}}^{s,0}_{\infty,q}(\mathbb{R}^n)+\mathrm{C}^{k_s}_{0}(\mathbb{R}^n)$, $1\leqslant q \leqslant +\infty$, $s\in\mathbb{R}$.
    \end{enumerate}
    In particular, the following set equalities hold when $s>0$, $q\in[1,+\infty]$,
    \begin{align*}
        \dot{\mathrm{B}}^{s}_{\infty,q}(\mathbb{R}^n) &= {\mathrm{B}}^{s}_{\infty,q}(\mathbb{R}^n)\cap\mathcal{S}'_h(\mathbb{R}^n),\\
        \dot{\mathrm{B}}^{s,0}_{\infty,q}(\mathbb{R}^n) &= {\mathrm{B}}^{s,0}_{\infty,q}(\mathbb{R}^n),
    \end{align*}
    and for all $k\in\mathbb{N}$,
    \begin{align*}
        \dot{\mathrm{W}}^{k,\infty}(\mathbb{R}^n)={\mathrm{W}}^{k,\infty}(\mathbb{R}^n)\cap\mathcal{S}'_h(\mathbb{R}^n).
    \end{align*}
\end{enumerate}
\end{lemma}

\begin{remark} The preceding lemma has two fundamental consequences that will be frequently used in the following sections:
\begin{itemize}
\item  By point \textit{(i)}, for an element $u\in\mathcal{S}'(\mathbb{R}^n)$, a sufficient condition for the membership to $\mathcal{S}'_h(\mathbb{R}^n)$ is  $\dot{S}_0u\in\mathrm{L}^p(\mathbb{R}^n)+\mathrm{C}^0_0(\mathbb{R}^n)$ for some $p\in[1,+\infty)$.
\item In particular, for all $p\in[1,+\infty)$, all $q\in[1,+\infty]$, all $k\in\mathbb{N}$, and all $s\in\mathbb{R}$, considering the inhomogeneous function spaces, one has
\begin{align*}
    \mathrm{W}^{k,1}(\mathbb{R}^n)\text{, }\mathrm{H}^{s,p}(\mathbb{R}^n)\text{, }\mathrm{B}^{s}_{p,q}(\mathbb{R}^n)\text{, }\mathrm{B}^{s,0}_{\infty,q}(\mathbb{R}^n) \subset \mathcal{S}'_h(\mathbb{R}^n).
\end{align*}
\end{itemize}
\end{remark}

\begin{proof} $\bullet$ The point \textit{(i)} follows from Bernstein's inequality \cite[Lemma~2.1]{bookBahouriCheminDanchin} for $\mathrm{L}^p(\mathbb{R}^n)$, $1\leqslant p<+\infty$, and from \cite[Example~9]{Cobb24Sprimeh} for $\mathrm{C}^0_0(\mathbb{R}^n)$.

$\bullet$ For the point \textit{(ii)}, this follows from the strong continuity of the convolution $\mathrm{L}^\infty\ast\mathrm{L}^1\hookrightarrow \mathrm{L}^\infty$, and the fact that $1\notin\mathcal{S}'_h(\mathbb{R}^n)$.

$\bullet$ We now proceed to prove point \textit{(iii)}. First, we address \textit{(d)}, \textit{(e)}, and \textit{(f)}, noting that point \textit{(a)} can be demonstrated in a similar manner.

Let $u\in\dot{\mathrm{B}}^{s}_{p,q}(\mathbb{R}^n)\subset \mathcal{S}'_h(\mathbb{R}^n)$, provided $p,q\in[1,+\infty]$, $s\in\mathbb{R}$, then we can write
\begin{align}\label{eq:LowHighfreqdecomp}
    u=\sum_{j\in\mathbb{Z}} \dot{\Delta}_ju = \sum_{ j<0} \dot{\Delta}_ju +\sum_{j\geqslant0} \dot{\Delta}_ju.
\end{align}
By \cite[Theorem~6.3.2]{BerghLofstrom1976}, one obtains by construction
\begin{align*}
    [\mathrm{I}-\dot{S}_{0}]u=\sum_{j\geqslant0} \dot{\Delta}_ju\in\mathrm{B}^{s}_{p,q}(\mathbb{R}^n).
\end{align*}
Now, the condition $u\in \mathcal{S}'_h(\mathbb{R}^n)$ implies that
\begin{align*}
   \dot{S}_{0}u= \sum_{ j<0} \dot{\Delta}_ju\text{, }\sum_{-N\leqslant j<0} \dot{\Delta}_ju\in\mathrm{L}^\infty(\mathbb{R}^n)
\end{align*}
for all $N\in\mathbb{N}$, with
\begin{align*}
    \Big\lVert \sum_{ j<0} \dot{\Delta}_ju -  \sum_{-N\leqslant j<0} \dot{\Delta}_ju\Big\rVert_{\mathrm{L}^\infty(\mathbb{R}^n)} \xrightarrow[N\rightarrow+\infty]{} 0.
\end{align*}
But since $\dot{\Delta}_ju\in\mathrm{L}^p(\mathbb{R}^n)$ for all $j\in\mathbb{Z}$, for $\chi\in\mathcal{S}(\mathbb{R}^n)$, such that $\mathcal{F}\chi=1$ on $B(0,4)$, we have
\begin{align*}
    \sum_{-N\leqslant j<0} \dot{\Delta}_j u = \chi\ast \left(\sum_{-N\leqslant j<0} \dot{\Delta}_ju \right)\in\mathrm{C}^0_b(\mathbb{R}^n),
\end{align*}
due to the inclusion $\mathcal{S}(\mathbb{R}^n)\ast\mathrm{L}^p(\mathbb{R}^n)\subset\mathrm{C}^0_b(\mathbb{R}^n)$, and even $\mathrm{C}^0_0(\mathbb{R}^n)$ if $p<+\infty$ (or $u\in\dot{\mathrm{B}}^{s,0}_{\infty,q}(\mathbb{R}^n)$). The closedness of $\mathrm{C}^0_b(\mathbb{R})$ and $\mathrm{C}^0_0(\mathbb{R}^n)$ in $\mathrm{L}^\infty(\mathbb{R}^n)$ ensures that
\begin{align*}
    \dot{S}_0u \in\mathrm{C}_b^0(\mathbb{R}^n)\quad\Big( \mathrm{C}^{0}_0(\mathbb{R}^n)\text{, if either }p<+\infty\text{, or }u\in\dot{\mathrm{B}}^{s,0}_{\infty,q}(\mathbb{R}^n)\Big).
\end{align*}
If $p<+\infty$ (or $u\in\dot{\mathrm{B}}^{s,0}_{\infty,q}(\mathbb{R}^n)$), the decomposition \eqref{eq:LowHighfreqdecomp} can be applied to $\nabla^k u\in\mathcal{S}'_h(\mathbb{R}^n)$. By the argument in the proof of \cite[Lemma~3.3]{DanchinHieberMuchaTolk2020}, and reiterating the present argument, we obtain that the low frequency part $\nabla^k \dot{S}_0u$ belongs to $\mathrm{C}^0_0(\mathbb{R}^n)$. Hence $\dot{S}_0u\in\mathrm{C}^k_0(\mathbb{R}^n)$. For $p=+\infty$, we similarly obtain $\dot{S}_0u\in\mathrm{C}^k_b(\mathbb{R}^n)$.

$\bullet$ Now, we prove \textit{(b)} and \textit{(c)}. Let $u\in\dot{\mathrm{W}}^{k,r}(\mathbb{R}^n)$, $k\in\mathbb{N}^\ast$, $r=1,+\infty$. One starts again with the decomposition \eqref{eq:LowHighfreqdecomp}. As before, one obtains that the block of low frequencies lies in $\mathrm{C}^{k}_b(\mathbb{R}^n)$, even in $\mathrm{C}^{k}_0(\mathbb{R}^n)$ when $r=1$. So we just have to show
\begin{align*}
     \sum_{j\geqslant0} \dot{\Delta}_ju\in\mathrm{L}^{r}(\mathbb{R}^n).
\end{align*}
This is true and follows from the fact that, thanks to Proposition \ref{prop:EmbeddSobBesovRn} and \cite[Theorem~6.3.2]{BerghLofstrom1976},
\begin{align*}
    \sum_{j\geqslant0} \dot{\Delta}_ju \in\mathrm{B}^{k}_{r,\infty}(\mathbb{R}^n)\subset \mathrm{L}^{r}(\mathbb{R}^n).
\end{align*}

$\bullet$ The equalities of sets follow directly from points \textit{(ii)-(d)} and \textit{(ii)-(e)}.
\end{proof}

We collect several density results in the next proposition.
\begin{proposition}\label{prop:DensitySobBesovRn} Let $p,q\in[1,+\infty)$, $s\in\mathbb{R}$, $k\in\mathbb{N}^\ast$,
\begin{enumerate}
    \item Let $u\in\mathcal{S}'_h(\mathbb{R}^n)$, if $u\in \dot{\mathrm{B}}^{s}_{p,q}(\mathbb{R}^n)$, then
    \begin{align*}
        \Big\lVert u -  \sum_{|j|\leqslant N} \dot{\Delta}_ju\Big\rVert_{\dot{\mathrm{B}}^{s}_{p,q}(\mathbb{R}^n)} \xrightarrow[N\rightarrow+\infty]{} 0.
    \end{align*}
    The result still holds if we replace $\dot{\mathrm{B}}^{s}_{p,q}$, by $\dot{\mathcal{B}}^{s}_{p,\infty}$, $\dot{\mathrm{B}}^{s,0}_{\infty,q}$, $\dot{\mathcal{B}}^{s}_{\infty,\infty}$, $\dot{\mathrm{H}}^{s,p}$, $\dot{\mathrm{W}}^{k,1}$, $\dot{\mathrm{C}}^k_{0}$ or $\mathrm{C}^0_0$.
    \item The following subspace of $\mathcal{S}(\mathbb{R}^n)$,
    \begin{align*}
        \mathcal{S}_0(\mathbb{R}^n):=\Big\{\,u\in\mathcal{S}(\mathbb{R}^n)\,\Big|\, 0\notin\supp(\mathcal{F}u)\text{ and }\supp(\mathcal{F}u)\text{ is compact.}\,\Big\},
    \end{align*}
    is a dense subspace of $\dot{\mathrm{B}}^{s}_{p,q}(\mathbb{R}^n)$, $\dot{\mathcal{B}}^{s}_{p,\infty}(\mathbb{R}^n)$, $\dot{\mathrm{B}}^{s,0}_{\infty,q}(\mathbb{R}^n)$, $\dot{\mathcal{B}}^{s,0}_{\infty,\infty}(\mathbb{R}^n)$, $\dot{\mathrm{H}}^{s,p}(\mathbb{R}^n)$, $\dot{\mathrm{W}}^{k,1}(\mathbb{R}^n)$, $\dot{\mathrm{C}}^k_0(\mathbb{R}^n)$, and $\mathrm{C}^0_0(\mathbb{R}^n)$.
    \item The space of smooth compactly supported functions, $\mathrm{C}_c^\infty(\mathbb{R}^n)$,
    is a dense subspace of
    \begin{enumerate}
        \item the Besov spaces: $\dot{\mathrm{B}}^{s}_{p,q}(\mathbb{R}^n)$, $s>-n/p'$,  $\dot{\mathrm{B}}^{s,0}_{\infty,q}(\mathbb{R}^n)$, $s>-n$, $\dot{\mathcal{B}}^{s,0}_{\infty,\infty}(\mathbb{R}^n)$, $s\geqslant-n$, $\dot{\mathcal{B}}^{s}_{p,\infty}(\mathbb{R}^n)$, $s\geqslant-n/p'$;
        \item the Sobolev spaces: $\dot{\mathrm{H}}^{s,p}(\mathbb{R}^n)$, $s>-n/p'$, and $\dot{\mathrm{W}}^{k,1}(\mathbb{R}^n)$.
    \end{enumerate}
    \item The space $\mathrm{C}_c^\infty(\mathbb{R}^n)\cap\dot{\mathrm{B}}^{s}_{p,q}(\mathbb{R}^n)$ is a dense subspace of $\dot{\mathrm{B}}^{s}_{p,q}(\mathbb{R}^n)$. 
    
    The result remains true if we replace $\dot{\mathrm{B}}^{s}_{p,q}$ by $\dot{\mathrm{B}}^{s,0}_{\infty,q}$, $\dot{\mathcal{B}}^{s}_{p,\infty}(\mathbb{R}^n)$, $\dot{\mathcal{B}}^{s,0}_{\infty,\infty}(\mathbb{R}^n)$ or $\dot{\mathrm{H}}^{s,p}(\mathbb{R}^n)$.
\end{enumerate}
\end{proposition}

\begin{remark}$\bullet$ An important fact is that the density results presented in \textit{(i)}-\textit{(iii)} are universal: the approximation procedure is the same in all function spaces.

$\bullet$ The density results in \textit{(iii)} and \textit{(iv)} are sharp, in the sense that $$\mathrm{C}_c^\infty(\mathbb{R}^n)\nsubseteq \dot{\mathrm{B}}^{0}_{1,1}(\mathbb{R}^n), \dot{\mathrm{H}}^{-{n}/{2},2}(\mathbb{R}^n).$$
Indeed, one can check that a smooth compactly supported function $u$ which satisfies
\begin{align*}
    \int_{\mathbb{R}^n}u(x) \mathrm{d}x =1,
\end{align*}
does not belong to these spaces. The proof below includes the arguments for the embedding of $\mathrm{C}_c^\infty(\mathbb{R}^n)$ into the appropriate homogeneous function spaces, when such embedding is permitted. See the Steps 3.3, 3.4 and 3.5 in the proof below.
\end{remark}

\begin{proof}\textbf{Step 1:} $\bullet$ Let $u\in\dot{\mathcal{B}}^{s}_{p,\infty}(\mathbb{R}^n)$, where $p\in[1,+\infty]$, we set for $N\in\mathbb{N}$,
\begin{align*}
    u_{N} = \sum_{|j|\leqslant N} \dot{\Delta}_ju.
\end{align*}
Since for all $\ell,m\in\mathbb{Z}$ such that $|\ell-m|\geqslant 2$, we have $\dot{\Delta}_\ell\dot{\Delta}_m=0$, we are able to estimate
\begin{align*}
    \lVert u -u_{N}\rVert_{\dot{\mathrm{B}}^{s}_{p,\infty}(\mathbb{R}^n)} &\leqslant \left(\sup_{|j|\geqslant N+2}\, 2^{js} \lVert \dot{\Delta}_ju\rVert_{\mathrm{L}^p(\mathbb{R}^n)}\right) + 2^{\pm s(N+1)}\lVert  \dot{\Delta}_{\pm(N+1)}u - \dot{\Delta}_{\pm(N+1)} \dot{\Delta}_{\pm N}u\rVert_{\mathrm{L}^p(\mathbb{R}^n)}\\ &\quad + 2^{\pm sN}\lVert  \dot{\Delta}_{\pm N}u - \dot{\Delta}_{\pm N}[ \dot{\Delta}_{ \pm(N-1)}+\dot{\Delta}_{\pm N}]u\rVert_{\mathrm{L}^p(\mathbb{R}^n)}\\
    &\lesssim_{s}\left(\sup_{|j|\geqslant N+2}\, 2^{js} \lVert \dot{\Delta}_ju\rVert_{\mathrm{L}^p(\mathbb{R}^n)}\right) + 2^{\pm s(N+1)}\lVert  \dot{\Delta}_{\pm (N+1)}u\rVert_{\mathrm{L}^p(\mathbb{R}^n)}\\&\quad+ 2^{\pm sN}\lVert  \dot{\Delta}_{\pm N}u\rVert_{\mathrm{L}^p(\mathbb{R}^n)}.
\end{align*}
We recall that $u\in\dot{\mathcal{B}}^{s}_{p,\infty}(\mathbb{R}^n)$ so that it implies $2^{js} \lVert \dot{\Delta}_j u\rVert_{\mathrm{L}^p(\mathbb{R}^n)} \xrightarrow[|j|\rightarrow+\infty]{}0$. We deduce
\begin{align*}
     \lVert u -u_{N}\rVert_{\dot{\mathrm{B}}^{s}_{p,\infty}(\mathbb{R}^n)}\xrightarrow[N\rightarrow+\infty]{}0.
\end{align*}

$\bullet$ For $u\in\dot{\mathrm{B}}^{s}_{p,q}(\mathbb{R}^n)\subset \dot{\mathcal{B}}^{s}_{p,\infty}(\mathbb{R}^n)$, we may estimate similarly
\begin{align*}
    \lVert u -u_{N}\rVert_{\dot{\mathrm{B}}^{s}_{p,q}(\mathbb{R}^n)} &\lesssim_{s} \left(\sum_{|j|\geqslant N+2}\, 2^{jsq} \lVert \dot{\Delta}_ju\rVert_{\mathrm{L}^p(\mathbb{R}^n)}^q\right)^\frac{1}{q} + 2^{\pm j(N+1)}\lVert  \dot{\Delta}_{\pm (N+1)}u\rVert_{\mathrm{L}^p(\mathbb{R}^n)}\\ &\quad + 2^{\pm jN}\lVert  \dot{\Delta}_{\pm N}u\rVert_{\mathrm{L}^p(\mathbb{R}^n)}.
\end{align*}
The results follows from the Dominated Convergence Theorem.

$\bullet$ For the case of $\dot{\mathrm{H}}^{s,p}(\mathbb{R}^n)$, $p\in(1,+\infty)$, and $\dot{\mathrm{W}}^{k,1}(\mathbb{R}^n)$, we use the fact that one can write
\begin{align*}
    u_{N}= [\dot{S}_N- \dot{S}_{-N}]u,
\end{align*}
where $(\dot{S}_N- \dot{S}_{-N})_{N\in\mathbb{N}}$ is a mollifying sequence, as given in \cite[Lemma~E.5.2]{bookHaase2006}, which already contains the case $\mathrm{C}^0_0(\mathbb{R}^n)$ and  $\mathrm{L}^p(\mathbb{R}^n)$, $1<p<+\infty$. Therefore, for $u\in\dot{\mathrm{H}}^{s,p}(\mathbb{R}^n)$
\begin{align*}
     \lVert u-u_N\rVert_{\dot{\mathrm{H}}^{s,p}(\mathbb{R}^n)}=\lVert (-\Delta)^{\frac{s}{2}}u-[\dot{S}_N- \dot{S}_{-N}](-\Delta)^{\frac{s}{2}}u\rVert_{{\mathrm{L}}^{p}(\mathbb{R}^n)}\xrightarrow[N\rightarrow+\infty]{}0. 
\end{align*}
For the case of $u\in\dot{\mathrm{W}}^{k,1}(\mathbb{R}^n)$, from the point \textit{(iii)-(b)} of Lemma \ref{lem:structLemma}, we obtain for all $|\alpha|=k$,
\begin{align*}
    \int_{\mathbb{R}^n}\partial^{\alpha}u(x)\mathrm{~d}x =0.
\end{align*}
Thus, we can apply \cite[Lemma~E.5.2]{bookHaase2006} as before to obtain
\begin{align*}
     \lVert u-u_N\rVert_{\dot{\mathrm{W}}^{k,1}(\mathbb{R}^n)}=\lVert \nabla^k u-[\dot{S}_N- \dot{S}_{-N}]\nabla^k u\rVert_{{\mathrm{L}}^{1}(\mathbb{R}^n)}\xrightarrow[N\rightarrow+\infty]{}0. 
\end{align*}
The same proof applies for $\dot{\mathrm{C}}^k_0(\mathbb{R}^n)$ starting from the case ${\mathrm{C}}^0_0(\mathbb{R}^n)$.

Again, the proof remains valid for $\dot{\mathrm{H}}^{s,1}(\mathbb{R}^n)$ instead of $\dot{\mathrm{W}}^{k,1}(\mathbb{R}^n)$, provided $s>0$, since for $f\in\dot{\mathrm{H}}^{s,1}(\mathbb{R}^n)$, one has $\mathcal{F}[(-\Delta)^{\frac{s}{2}}f](0) = |0|^s\mathcal{F}f(0)=0$, and then
\begin{align*}
    \int_{\mathbb{R}^n} (-\Delta)^{\frac{s}{2}}f(x)\mathrm{~d}x =0.
\end{align*}
\textbf{Step 2:} To prove point \textit{(ii)}, it suffices to follow \cite[Proposition~2.27]{bookBahouriCheminDanchin}, \cite[Lemma~2.5]{Gaudin2022}, with minor modifications. Therefore, we only prove the density result for the function space $\dot{\mathrm{W}}^{k,1}(\mathbb{R}^n)$, $k\geqslant1$. We start similarly, and with the notations introduced in the previous step. For a fixed $\varepsilon>0$, let $N\in\mathbb{N}$ be large enough such that
\begin{align*}
    \lVert u - u_{N}\rVert_{\dot{\mathrm{W}}^{k,1}(\mathbb{R}^n)} < \varepsilon\text{. }
\end{align*}
One has $u_N\in{\mathrm{L}}^{1}(\mathbb{R}^n)\cap \mathrm{C}^\infty_0(\mathbb{R}^n)$.
For $M\geqslant N+1$, $R>0$, provided $\Theta\in \mathrm{C}_c^\infty(\mathbb{R}^n)$, real valued, supported in $B(0,2)$, such that $\Theta_{|_{B(0,1)}}=1$, and $\Theta_R:=\Theta(\cdot /R)$, we introduce
\begin{align*}
    u_{N,M}^R:= (\dot{S}_{M}-\dot{S}_{-M})[\Theta_{R}u_{N}]\text{.}
\end{align*}
By construction, $u_{N,M}^R\in \mathcal{S}_0(\mathbb{R}^n)$.
Since $\dot{\Delta}_{k} u_{N} = 0$, $k\leqslant -M-1$, we have $\dot{S}_{-M}u_{N}=0$, and $\dot{S}_{M}u_{N}=u_{N}$, then
\begin{align*}
    u_{N,M}^R - u_{N} = (\dot{S}_{M}-\dot{S}_{-M})[(\Theta_{R}-1)u_{N}]\text{.}
\end{align*}
By Bernstein's inequalities \cite[Lemma~2.1]{bookBahouriCheminDanchin}, since $\supp \mathcal{F}(u_{N,M}^R - u_{N})\subset \{\,\xi\in\mathbb{R}^n\,|\,3\cdot 2^{-M-2}\leqslant\left\lvert{\xi}\right\rvert \leqslant 2^{M+3}/3\,\}$, we obtain
\begin{align*}
    \lVert u_{N,M}^R - u_{N}\rVert_{\dot{\mathrm{W}}^{k,1}(\mathbb{R}^n)} &\lesssim_{M,k} \lVert u_{N,M}^R - u_{N}\rVert_{{\mathrm{L}}^{1}(\mathbb{R}^n)}\\
    &\lesssim_{M,k} \lVert (\dot{S}_M-\dot{S}_{-M})[(\Theta_{R}-1)u_{N}]\rVert_{{\mathrm{L}}^{1}(\mathbb{R}^n)}\\
    &\lesssim_{M,k} \lVert [(\Theta_{R}-1)u_{N}]\rVert_{{\mathrm{L}}^{1}(\mathbb{R}^n)}\text{. }
\end{align*}
The Dominated Convergence Theorem yields
\begin{align*}
    \lVert u_{N,M}^R - u_{N}\rVert_{\dot{\mathrm{W}}^{k,1}(\mathbb{R}^n)} \xrightarrow[R\rightarrow +\infty]{} 0 \text{. }
\end{align*}
Thus, for $R>0$ big enough,
\begin{align*}
    \lVert u - u_{N,M}^R\rVert_{\dot{\mathrm{W}}^{k,1}(\mathbb{R}^n)} < 2\varepsilon\text{. }
\end{align*}

\textbf{Step 3:} We focus on the case of Besov spaces. For $u\in\dot{\mathrm{B}}^{s}_{p,q}(\mathbb{R}^n)$, we have $u_{N,M}^R\in\mathcal{S}_0(\mathbb{R}^n)$. For a fixed $\varepsilon>0$, we choose $N,M,R$ large enough so that
\begin{align*}
    \lVert u - u_{N,M}^R\rVert_{\dot{\mathrm{B}}^{s}_{p,q}(\mathbb{R}^n)} < \varepsilon\text{. }
\end{align*}
We introduce, for all $\eta>0$, $u_{N,M}^{R,\eta}:=\Theta(\cdot/\eta)u_{N,M}^R\in \mathrm{C}_c^\infty(\mathbb{R}^n)$. By the triangle inequality,
\begin{align*}
    \lVert u - u_{N,M}^{R,\eta}\rVert_{\dot{\mathrm{B}}^{s}_{p,q}(\mathbb{R}^n)} &\leqslant  \lVert u - u_{N,M}^R\rVert_{\dot{\mathrm{B}}^{s}_{p,q}(\mathbb{R}^n)} + \lVert u_{N,M}^R - u_{N,M}^{R,\eta}\rVert_{\dot{\mathrm{B}}^{s}_{p,q}(\mathbb{R}^n)}.
\end{align*}

So it suffices to prove that
\begin{align*}
    \lVert u_{N,M}^R - u_{N,M}^{R,\eta}\rVert_{\dot{\mathrm{B}}^{s}_{p,q}(\mathbb{R}^n)} \xrightarrow[\eta\rightarrow +\infty]{} 0.
\end{align*}

We deal with six subcases.

\textbf{Step 3.1:} The case $s>0$. This is the simplest case. We set $m=\lceil s\rceil+1$. We use the fact the $\mathrm{W}^{m,p}(\mathbb{R}^n)\hookrightarrow{\mathrm{B}}^{s}_{p,q}(\mathbb{R}^n)\hookrightarrow\dot{\mathrm{B}}^{s}_{p,q}(\mathbb{R}^n)$,
\begin{align*}
     \lVert u - u_{N,M}^{R,\eta}\rVert_{\dot{\mathrm{B}}^{s}_{p,q}(\mathbb{R}^n)} &\leqslant  \lVert u - u_{N,M}^R\rVert_{\dot{\mathrm{B}}^{s}_{p,q}(\mathbb{R}^n)} + \lVert u_{N,M}^R - u_{N,M}^{R,\eta}\rVert_{\dot{\mathrm{B}}^{s}_{p,q}(\mathbb{R}^n)} \\
     &\lesssim_{p,s,n} \lVert u - u_{N,M}^R\rVert_{\dot{\mathrm{B}}^{s}_{p,q}(\mathbb{R}^n)} + \lVert u_{N,M}^R - u_{N,M}^{R,\eta}\rVert_{{\mathrm{W}}^{m,p}(\mathbb{R}^n)}.
\end{align*}
By the Leibniz rule and the Dominated Convergence Theorem,
\begin{align*}
    \lVert u_{N,M}^R - u_{N,M}^{R,\eta}\rVert_{{\mathrm{W}}^{m,p}(\mathbb{R}^n)}\xrightarrow[\eta\rightarrow +\infty]{} 0.
\end{align*}
So for $\eta$ large enough,
\begin{align*}
     \lVert u - u_{N,M}^{R,\eta}\rVert_{\dot{\mathrm{B}}^{s}_{p,q}(\mathbb{R}^n)} \lesssim_{p,s,n} \varepsilon. 
\end{align*}

\textbf{Step 3.2:} The case $\dot{\mathcal{B}}^{0}_{p,\infty}(\mathbb{R}^n)$, $\dot{\mathcal{B}}^{0,0}_{\infty,\infty}(\mathbb{R}^n)$. By Proposition \ref{prop:EmbeddSobBesovRn}, we have $\mathrm{L}^p(\mathbb{R}^n)\hookrightarrow\dot{\mathcal{B}}^{0}_{p,\infty}(\mathbb{R}^n)$. As before, we obtain
\begin{align*}
     \lVert u_{N,M}^R - u_{N,M}^{R,\eta}\rVert_{\dot{\mathrm{B}}^{0}_{p,\infty}(\mathbb{R}^n)} &\lesssim_{p,s,n} \lVert u_{N,M}^R - u_{N,M}^{R,\eta}\rVert_{{\mathrm{L}}^{p}(\mathbb{R}^n)}.
\end{align*}
Again, we apply the Dominated Convergence Theorem to take the limit as $\eta$ tends to infinity.

\textbf{Step 3.3:} The case $\dot{\mathcal{B}}^{-{n}/{p'}}_{p,\infty}(\mathbb{R}^n)$, $\dot{\mathcal{B}}^{-n,0}_{\infty,\infty}(\mathbb{R}^n)$. By Bernstein's inequality \cite[Lemma~2.1]{bookBahouriCheminDanchin}, for all $v\in\mathrm{L}^1(\mathbb{R}^n)$, for all $j\in\mathbb{Z}$,
\begin{align*}
    \lVert\dot{\Delta}_j v\rVert_{\mathrm{L}^{p}(\mathbb{R}^n)} \lesssim_{n,p} 2^{j(n-\frac{n}{p})}\lVert\dot{\Delta}_j v\rVert_{\mathrm{L}^{1}(\mathbb{R}^n)}\lesssim_{n,p} 2^{j(n-\frac{n}{p})}\lVert v\rVert_{\mathrm{L}^{1}(\mathbb{R}^n)}.
\end{align*}
Thus, we have obtained $\mathrm{L}^1(\mathbb{R}^n)\hookrightarrow\dot{\mathrm{B}}^{-{n}/{p'}}_{p,\infty}(\mathbb{R}^n)$, so that again it suffices to consider
\begin{align*}
      \lVert u_{N,M}^R - u_{N,M}^{R,\eta}\rVert_{\dot{\mathrm{B}}^{-n/p'}_{p,\infty}(\mathbb{R}^n)}\lesssim_{p,n}\lVert u_{N,M}^R - u_{N,M}^{R,\eta}\rVert_{{\mathrm{L}}^{1}(\mathbb{R}^n)}.
\end{align*}
We can conclude as in the previous step.

\textbf{Step 3.4:} The remaining cases when $-n/p'<s<0$. By Proposition \ref{prop:EmbeddSobBesovRn} and the interpolation inequality \cite[Proposition~2.22]{bookBahouriCheminDanchin}, since we can write $s=(1-\theta)(-n/p')$, one obtains
\begin{align*}
    \lVert u_{N,M}^R - u_{N,M}^{R,\eta}\rVert_{\dot{\mathrm{B}}^{s}_{p,q}(\mathbb{R}^n)} &\leqslant\lVert u_{N,M}^R - u_{N,M}^{R,\eta}\rVert_{\dot{\mathrm{B}}^{s}_{p,1}(\mathbb{R}^n)}\\&\lesssim_{p,n,\theta} \lVert u_{N,M}^R - u_{N,M}^{R,\eta}\rVert_{\dot{\mathrm{B}}^{-n/p'}_{p,\infty}(\mathbb{R}^n)}^{1-\theta}\lVert u_{N,M}^R - u_{N,M}^{R,\eta}\rVert_{\dot{\mathrm{B}}^{0}_{p,\infty}(\mathbb{R}^n)}^{\theta}\\&\lesssim_{p,n,\theta}\lVert u_{N,M}^R - u_{N,M}^{R,\eta}\rVert_{{\mathrm{L}}^{1}(\mathbb{R}^n)}^{1-\theta}\lVert u_{N,M}^R - u_{N,M}^{R,\eta}\rVert_{{\mathrm{L}}^{p}(\mathbb{R}^n)}^{\theta}.
\end{align*}
This allows to conclude by the Dominated Convergence Theorem.

\textbf{Step 3.5:} The case $\dot{\mathrm{B}}^{0}_{p,q}(\mathbb{R}^n),\dot{\mathrm{B}}^{0,0}_{\infty,q}(\mathbb{R}^n)$. For $-n/p'<\zeta_0<0<\zeta_1$, and $\theta\in(0,1)$ such that $(1-\theta)\zeta_0+\theta\zeta_1 =0$, by the interpolation inequality \cite[Proposition~2.22]{bookBahouriCheminDanchin}, one has
\begin{align*}
\lVert u_{N,M}^R - u_{N,M}^{R,\eta} \rVert_{\dot{\mathrm{B}}^{0}_{p,q}(\mathbb{R}^n)} &\lesssim_{p,\zeta_0} \lVert u_{N,M}^R - u_{N,M}^{R,\eta} \rVert_{\dot{\mathrm{B}}^{\zeta_0}_{p,1}(\mathbb{R}^n)}^{1-\theta}\lVert u_{N,M}^R - u_{N,M}^{R,\eta} \rVert_{\dot{\mathrm{B}}^{\zeta_1}_{p,1}(\mathbb{R}^n)}^\theta.
\end{align*}
Then, thanks to the previous Step 3.1 and Step 3.4, that correspond to respectively $\zeta_1=s>0$ and $\zeta_0=s<0$, we can take the limit as $\eta$ tends to infinity.

\textbf{Step 3.6:} For the spaces $\dot{\mathrm{H}}^{s,p}(\mathbb{R}^n)$, $s>0$, and $\dot{\mathrm{W}}^{k,1}(\mathbb{R}^n)$, $k\geqslant 1$, the proof is similar to the case of Besov spaces with $s>0$.

For the spaces $\dot{\mathrm{H}}^{s,p}(\mathbb{R}^n)$, $-n/p'<s<0$, the proof is similar to the case of Besov spaces with $-n/p'<s<0$, since Proposition \ref{prop:EmbeddSobBesovRn} gives us $\dot{\mathrm{B}}^{s}_{p,1}(\mathbb{R}^n)\hookrightarrow\dot{\mathrm{H}}^{s,p}(\mathbb{R}^n)$.

\textbf{Step 4:} The proof of \textit{(iv)} follows from slight modifications of the arguments presented in Step 3. Due to \textit{(iii)}, we can assume without loss of generality that $s\leqslant -n/p'$.

We start from the beginning of Step 3, with $\varepsilon>0$ and $R,N,M>0$ such that
\begin{align*}
    \lVert u - u_{N,M}^R\rVert_{\dot{\mathrm{B}}^{s}_{p,q}(\mathbb{R}^n)} < \varepsilon\text{. }
\end{align*}
  Let $\mathfrak{m}\in\mathbb{N}$ such that $ s+2\mathfrak{m}>0$. Since $u_{N,M}^R\in\mathcal{S}_0(\mathbb{R}^n)$, we do have $(-\Delta)^{-\mathfrak{m}}u_{N,M}^R\in\mathcal{S}_0(\mathbb{R}^n)$. Therefore, for any $\eta>0$, we introduce
  \begin{align*}
      \tilde{u}_{N,M}^{R,\eta} := (-\Delta)^{\mathfrak{m}} \Theta(\cdot/\eta) (-\Delta)^{-\mathfrak{m}}u_{N,M}^R\in\mathrm{C}_c^\infty(\mathbb{R}^n)\cap \dot{\mathrm{B}}^{s}_{p,q}(\mathbb{R}^n).
  \end{align*}
From this, by the arguments in Step 3, since $(-\Delta)^{-\mathfrak{m}}u_{N,M}^R\in\mathcal{S}_0(\mathbb{R}^n)$ and $s+2\mathfrak{m}>0$, we can conclude
\begin{align*}
    \lVert \tilde{u}_{N,M}^{R,\eta} - u_{N,M}^R\rVert_{\dot{\mathrm{B}}^{s}_{p,q}(\mathbb{R}^n)} \lesssim_{p,\mathfrak{m},n} \lVert [\Theta(\cdot/\eta)-1 ](-\Delta)^{-\mathfrak{m}}u_{N,M}^R\rVert_{\dot{\mathrm{B}}^{s+2\mathfrak{m}}_{p,q}(\mathbb{R}^n)}  \xrightarrow[\eta\longrightarrow+\infty]{}0\text{. }
\end{align*}
This finishes the proof.
\end{proof}

\begin{theorem}\label{thm:CompletenessSpacesRn} Let $p,q,r,\kappa\in[1,+\infty]$, $s,\alpha\in\mathbb{R}$ and $m\in\mathbb{N}$. The following normed vector spaces are complete:
\begin{enumerate}
    \item $\dot{\mathrm{B}}^{s}_{p,q}(\mathbb{R}^n)$ and $\dot{\mathrm{B}}^{s}_{p,q}(\mathbb{R}^n)\cap\dot{\mathrm{B}}^{\alpha}_{r,\kappa}(\mathbb{R}^n)$, whenever \eqref{AssumptionCompletenessExponents} is satisfied;
    \item $\dot{\mathrm{H}}^{s,p}(\mathbb{R}^n)$ and $\dot{\mathrm{H}}^{s,p}(\mathbb{R}^n)\cap\dot{\mathrm{H}}^{\alpha,r}(\mathbb{R}^n)$, whenever $p,r<+\infty$, $s<n/p$;
    \item $\dot{\mathrm{W}}^{k,1}(\mathbb{R}^n)$ and $\dot{\mathrm{W}}^{k,1}(\mathbb{R}^n)\cap\dot{\mathrm{W}}^{m,1}(\mathbb{R}^n)$, whenever $k\in\llb 0,n\rrb$.
\end{enumerate}
The result remains valid exchanging the roles of $\dot{\mathrm{B}}^{\alpha}_{r,\kappa}(\mathbb{R}^n)$, $\dot{\mathrm{H}}^{\alpha,r}(\mathbb{R}^n)$ and $\dot{\mathrm{W}}^{m,1}(\mathbb{R}^n)$.
\end{theorem}

\begin{proof}  Point \textit{(i)} is just \cite[Theorem~2.25]{bookBahouriCheminDanchin}. 

\textbf{Step 1:} Now, for point \textit{(ii)}, let $(u_{\ell})_{\ell\in\mathbb{N}}\subset \dot{\mathrm{H}}^{s,p}(\mathbb{R}^n)\cap\dot{\mathrm{H}}^{\alpha,r}(\mathbb{R}^n)$ be a Cauchy sequence in $\dot{\mathrm{H}}^{s,p}(\mathbb{R}^n)\cap\dot{\mathrm{H}}^{\alpha,r}(\mathbb{R}^n)$. By Proposition \ref{prop:EmbeddSobBesovRn}, $(u_{\ell})_{\ell\in\mathbb{N}}$ is a Cauchy sequence in $\dot{\mathrm{B}}^{s}_{p,\infty}(\mathbb{R}^n)\cap\dot{\mathrm{B}}^{\alpha}_{r,\infty}(\mathbb{R}^n)\subset\dot{\mathrm{B}}^{s}_{p,\infty}(\mathbb{R}^n)$. Therefore, there exists $u\in \dot{\mathrm{B}}^{s}_{p,\infty}(\mathbb{R}^n)\cap\dot{\mathrm{B}}^{\alpha}_{r,\infty}(\mathbb{R}^n)\subset\dot{\mathrm{B}}^{s}_{p,\infty}(\mathbb{R}^n)\subset\mathcal{S}'_h(\mathbb{R}^n)$, $f\in\mathrm{L}^p(\mathbb{R}^n)$, $g\in\mathrm{L}^r(\mathbb{R}^n)$, such that
\begin{enumerate}
    \item $u_m\xrightarrow[m\longrightarrow+\infty]{}u$ in $\mathcal{S}'(\mathbb{R}^n)$,
    \item $(-\Delta)^\frac{s}{2}u_m\xrightarrow[m\longrightarrow+\infty]{}f$ in $\mathrm{L}^p(\mathbb{R}^n)$,
    \item $(-\Delta)^\frac{\alpha}{2}u_m\xrightarrow[m\longrightarrow+\infty]{}g$ in $\mathrm{L}^r(\mathbb{R}^n)$.
\end{enumerate}
By uniqueness of the limit in $\mathcal{S}'(\mathbb{R}^n)$ and continuity of the Littlewood-Paley decomposition, we deduce for all $j\in\mathbb{Z}$,
\begin{align*}
    \dot{\Delta}_jf = (-\Delta)^\frac{s}{2}\dot{\Delta}_j u, \text{ and }\dot{\Delta}_jg = (-\Delta)^\frac{\alpha}{2}\dot{\Delta}_j u.
\end{align*}
Since $u,f,g\in\mathcal{S}'_h(\mathbb{R}^n)$, it follows that, in $\mathcal{S}'(\mathbb{R}^n)$,
\begin{align*}
    u=\sum_{j\in\mathbb{Z}}\dot{\Delta}_j u = \sum_{j\in\mathbb{Z}} (-\Delta)^{-\frac{s}{2}}\dot{\Delta}_jf = \sum_{j\in\mathbb{Z}} (-\Delta)^{-\frac{\alpha}{2}}\dot{\Delta}_j g.
\end{align*}
Then, we have
\begin{enumerate}
    \item $(-\Delta)^\frac{s}{2}u_m\xrightarrow[m\longrightarrow+\infty]{}(-\Delta)^\frac{s}{2}u$ in $\mathrm{L}^p(\mathbb{R}^n)$,
    \item $(-\Delta)^\frac{\alpha}{2}u_m\xrightarrow[m\longrightarrow+\infty]{}(-\Delta)^\frac{\alpha}{2}u$ in $\mathrm{L}^r(\mathbb{R}^n)$.
\end{enumerate}

\textbf{Step 2:} For point \textit{(iii)}, we only show the completeness of $\dot{\mathrm{W}}^{n,1}(\mathbb{R}^n)$, otherwise one could proceed as in point \textit{(ii)} to reach the completeness of $\dot{\mathrm{W}}^{k,1}(\mathbb{R}^n)\cap\dot{\mathrm{W}}^{m,1}(\mathbb{R}^n)$, provided $k\in \llb0,n-1\rrb$, $m\in\mathbb{N}$.

Let $u\in\dot{\mathrm{W}}^{n,1}(\mathbb{R}^n)$, by Lemma \ref{lem:structLemma}, we obtain $u\in{\mathrm{W}}^{n,1}(\mathbb{R}^n)+\mathrm{C}^{n}_0(\mathbb{R}^n) \subset \mathrm{C}^{0}_0(\mathbb{R}^n)$, so for all $x\in\mathbb{R}^n$,
\begin{align*}
    u(x) = \int_{-\infty}^{x_1}\int_{-\infty}^{x_2}\ldots \int_{-\infty}^{x_n} \partial_{x_1}\partial_{x_2}\ldots\partial_{x_n} u (t_1,t_2,\ldots,t_n)\mathrm{~d}t_1\mathrm{d}t_2\ldots \mathrm{d}t_n.
\end{align*}
This is straightforward and standard to deduce the inequality
\begin{align*}
    \lVert u\rVert_{\mathrm{L}^\infty(\mathbb{R}^n)}\leqslant \lVert u\rVert_{\dot{\mathrm{W}}^{n,1}(\mathbb{R}^n)}.
\end{align*}
Now, let $(u_\ell)_{\ell\in\mathbb{N}}\subset \dot{\mathrm{W}}^{n,1}(\mathbb{R}^n)$ be a Cauchy sequence in $\dot{\mathrm{W}}^{n,1}(\mathbb{R}^n)$. Therefore, $(u_\ell)_{\ell\in\mathbb{N}}$ is a Cauchy sequence in $\mathrm{C}^0_0(\mathbb{R}^n)$, and $(\nabla^n u_\ell)_{\ell\in\mathbb{N}}$ is a Cauchy sequence in $\mathrm{L}^1(\mathbb{R}^n)$. By completeness, there exists $u\in \mathrm{C}^0_0(\mathbb{R}^n)$, and  for all multi-index $|\alpha|=n$, there exists $f_\alpha\in\mathrm{L}^1(\mathbb{R}^n)$ such that
\begin{enumerate}
    \item $u_\ell\xrightarrow[\ell\longrightarrow+\infty]{}u$ in $\mathrm{C}^0_0(\mathbb{R}^n)\subset\mathcal{S}'_h(\mathbb{R}^n)$,
    \item $\partial^{\alpha}u_\ell\xrightarrow[\ell\longrightarrow+\infty]{}f_\alpha$ in $\mathrm{L}^1(\mathbb{R}^n)\subset\mathcal{S}'_h(\mathbb{R}^n)$.
\end{enumerate}
By uniqueness of the limit in $\mathcal{S}'(\mathbb{R}^n)$, we obtain $\partial^\alpha u =f_\alpha$, which yields the result.
\end{proof}

We recall also the usual interpolation properties. It is of interest to bring to the reader's attention that, except in the case of complex interpolation, the result below does not require any completeness assumption on function spaces involved in any of the real interpolation identities.
\begin{theorem}\label{thm:InterpHomSpacesRn}Let $1\leqslant p_0,p_1,p,q,q_0,q_1\leqslant +\infty$, $s_0,s_1\in\mathbb{R}$, such that $s_0\neq s_1$, and for $\theta\in(0,1)$, let
\begin{align*}
    \left(s,\frac{1}{p_\theta},\frac{1}{q_\theta}\right):= (1-\theta)\left(s_0,\frac{1}{p_0},\frac{1}{q_0}\right)+ \theta\left(s_1,\frac{1}{p_1},\frac{1}{q_1}\right)\text{. }
\end{align*}
If $s_0,s_1\in\mathbb{N}$, we write $k_0:=s_0$ and $k_1:=s_1$.
We have the following interpolation identities with equivalence of norms
\begin{enumerate}
    \item $(\dot{\mathrm{B}}^{s_0}_{p,q_0}(\mathbb{R}^n),\dot{\mathrm{B}}^{s_1}_{p,q_1}(\mathbb{R}^n))_{\theta,q}=(\dot{\mathrm{H}}^{s_0,p}(\mathbb{R}^n),\dot{\mathrm{H}}^{s_1,p}(\mathbb{R}^n))_{\theta,q}=\dot{\mathrm{B}}^{s}_{p,q}(\mathbb{R}^n)$\label{eq:realInterpHomBspqRn};
    \item $(\dot{\mathrm{B}}^{s_0,0}_{\infty,q_0}(\mathbb{R}^n),\dot{\mathrm{B}}^{s_1,0}_{\infty,q_1}(\mathbb{R}^n))_{\theta,q}=(\dot{\mathrm{C}}^{k_0}_0(\mathbb{R}^n),\dot{\mathrm{C}}^{k_1}_{0}(\mathbb{R}^n))_{\theta,q}=\dot{\mathrm{B}}^{s,0}_{\infty,q}(\mathbb{R}^n)$\label{eq:realInterpHomBsinftyqRn};
    \item $(\dot{\mathrm{W}}^{k_0,1}(\mathbb{R}^n),\dot{\mathrm{W}}^{k_1,1}(\mathbb{R}^n))_{\theta,q}=\dot{\mathrm{B}}^{s}_{1,q}(\mathbb{R}^n)$\label{eq:realInterpHomBspqRnL1};
    \item $(\dot{\mathrm{C}}^{k_0}_{b,h}(\mathbb{R}^n),\dot{\mathrm{C}}^{k_1}_{b,h}(\mathbb{R}^n))_{\theta,q}=(\dot{\mathrm{W}}^{k_0,\infty}(\mathbb{R}^n),\dot{\mathrm{W}}^{k_1,\infty}(\mathbb{R}^n))_{\theta,q}=\dot{\mathrm{B}}^{s}_{\infty,q}(\mathbb{R}^n)$\label{eq:realInterpHomBspqRnLinfty};
    \item $[\dot{\mathrm{H}}^{s_0,p_0}(\mathbb{R}^n),\dot{\mathrm{H}}^{s_1,p_1}(\mathbb{R}^n)]_{\theta} = \dot{\mathrm{H}}^{s,p_\theta}(\mathbb{R}^n)$, if $1<p_0,p_1<+\infty$, and $(\mathcal{C}_{s_j,p_j})$ is true for $j\in\{0,1\}$\label{eq:complexInterpHomSobRn};
    \item $[\dot{\mathrm{B}}^{s_0}_{p_0,q_0}(\mathbb{R}^n),\dot{\mathrm{B}}^{s_1}_{p_1,q_1}(\mathbb{R}^n)]_{\theta} = \dot{\mathrm{B}}^{s}_{p_\theta,q_\theta}(\mathbb{R}^n)$, if  $(\mathcal{C}_{s_j,p_j,q_j})$ is satisfied for $j\in\{0,1\}$, $q_\theta\neq+\infty$.\label{eq:complexInterpHomBspqRn}
\end{enumerate}
Moreover, 
\begin{itemize}
    \item Identity \ref{eq:realInterpHomBspqRn} still holds if we replace $\dot{\mathrm{B}}^{s_j}_{p,q_j}(\mathbb{R}^n)$ by $\dot{\mathcal{B}}^{s_j}_{p,\infty}(\mathbb{R}^n)$, $j\in\{0,1\}$, see Proposition~\ref{prop:EmbeddSobBesovRn};
    \item When we replace $(\cdot,\cdot)_{\theta,q}$ by $(\cdot,\cdot)_{\theta}$, all the real interpolation identities still hold with $\dot{\mathcal{B}}^{s}_{p,\infty}$ as an output space ($\dot{\mathcal{B}}^{s,0}_{\infty,\infty}$ in \ref{eq:realInterpHomBsinftyqRn}).
    \item Identity \ref{eq:complexInterpHomBspqRn} remains valid for $q_0=q_1=q_\theta=+\infty$ with the output space $\dot{\mathcal{B}}^{s_\theta}_{p_{\theta},\infty}(\mathbb{R}^n)$.
\end{itemize}
\end{theorem}

\begin{proof} The idea is to follow the strategy presented in the proof of \cite[Theorem~2.6]{Gaudin2022} were the interpolation identities \ref{eq:realInterpHomBspqRn}, \ref{eq:complexInterpHomSobRn} and \ref{eq:complexInterpHomBspqRn} are already proved whenever $1<p,p_0,p_1<+\infty$, $q,q_0,q_1\in[1,+\infty]$, under the assumption that one of any involved space is complete.  However, we want to reach endpoint cases $p=1,+\infty$, with possibly $q=+\infty$, and to remove the completeness assumption. Hence, we need some additional knowledge and to proceed differently, while staying close to \cite[Theorem~2.6]{Gaudin2022}.

\textbf{Step 1:} For $(w_j)_{j\in\mathbb{Z}}\subset{}\mathcal{S}'(\mathbb{R}^n)$, say, for simplicity, with finite support in the discrete variable, we define the map
\begin{align}\label{eq:retractionMapsLp(l2s)intoHsp}
    \tilde{\Sigma}((w_j)_{j\in\mathbb{Z}}) := \sum_{j=-\infty}^{+\infty} \dot{\Delta}_j [w_{j-1}+w_{j}+w_{j+1}] \text{,}
\end{align}
and it satisfies for $v\in \mathcal{S}'_h(\mathbb{R}^n)$
\begin{align*}
    \tilde{\Sigma}((\dot{\Delta}_j v)_{j\in\mathbb{Z}})=v\text{.}
\end{align*}
We aim to show that for any $p,q\in[1,+\infty]$, $s\in\mathbb{R}$, the operator
\begin{align*}
(\dot{\Delta}_j\tilde{\Sigma})_{j\in\mathbb{Z}}\,:\,\ell^q_s(\mathbb{Z},\mathrm{L}^p(\mathbb{R}^n))\longrightarrow \ell^q_s(\mathbb{Z},\mathrm{L}^p(\mathbb{R}^n))
\end{align*}
is well-defined and bounded. Let $(u_m)_{m\in\mathbb{Z}}\subset \ell^1_s(\mathbb{Z},\mathrm{L}^p(\mathbb{R}^n))$, with finite support with respect to the discrete variable. Let $j\in\mathbb{Z}$,
\begin{align*}
    2^{js}\lVert \dot{\Delta}_j\tilde{\Sigma}[(u_m)_{m\in\mathbb{Z}}]\rVert_{\mathrm{L}^p(\mathbb{R}^n)} \leqslant  & 4^{s}2^{(j-2)s} \lVert u_{j-2} \rVert_{\mathrm{L}^p(\mathbb{R}^n)}+2^{1+s} 2^{(j-1)s} \lVert u_{j-1} \rVert_{\mathrm{L}^p(\mathbb{R}^n)}\\ &+3\cdot 2^{js} \lVert u_{j} \rVert_{\mathrm{L}^p(\mathbb{R}^n)}  +2^{1-s}2^{(j+1)s} \lVert u_{j+1} \rVert_{\mathrm{L}^p(\mathbb{R}^n)} \\ &+4^{-s}2^{(j+2)s} \lVert u_{j+2} \rVert_{\mathrm{L}^p(\mathbb{R}^n)} \\
    &\leqslant 2\cdot 4^{1+|s|}\lVert (u_m)_{m\in\mathbb{Z}}\rVert_{\ell^1_s(\mathbb{Z},\mathrm{L}^p(\mathbb{R}^n))}.
\end{align*}
Taking the supremum over $j\in\mathbb{Z}$, we have obtained a bounded map, for all $p\in[1,+\infty]$, all $s\in\mathbb{R}$,
\begin{align*}
(\dot{\Delta}_j\tilde{\Sigma})_{j\in\mathbb{Z}}\,:\,\ell^1_s(\mathbb{Z},\mathrm{L}^p(\mathbb{R}^n))\longrightarrow \ell^\infty_s(\mathbb{Z},\mathrm{L}^p(\mathbb{R}^n)).
\end{align*}
By real interpolation, see \cite[Theorem~5.6.1]{BerghLofstrom1976}, we obtain for all $p,q\in[1,+\infty]$, $s\in\mathbb{R}$, the desired boundedness property
\begin{align*}
(\dot{\Delta}_j\tilde{\Sigma})_{j\in\mathbb{Z}}\,:\,\ell^q_s(\mathbb{Z},\mathrm{L}^p(\mathbb{R}^n))\longrightarrow \ell^q_s(\mathbb{Z},\mathrm{L}^p(\mathbb{R}^n)).
\end{align*}

\textbf{Step 2:} We prove that, for all $p,q\in[1,+\infty]$, $p\neq+\infty$, and $s\in\mathbb{R}$ that satisfy \eqref{AssumptionCompletenessExponents},
\begin{align*}
\tilde{\Sigma}\,:\,\ell^q_s(\mathbb{Z},\mathrm{L}^p(\mathbb{R}^n))\longrightarrow \dot{\mathrm{B}}^{s}_{p,q}(\mathbb{R}^n).
\end{align*}
is well-defined, and bounded.

\textbf{Step 2.1:} When $p,q\in[1,+\infty)$, the result follows by density of elements with finite support with respect to the discrete variable in $\ell^q_s(\mathbb{Z},\mathrm{L}^p(\mathbb{R}^n))$ and the completeness of $\dot{\mathrm{B}}^{s}_{p,q}(\mathbb{R}^n)$.

\textbf{Step 2.2:} We prove the interpolation identity \ref{eq:realInterpHomBspqRn} in the case of Besov spaces, for $p\in[1,+\infty)$, $q\in[1,+\infty]$, and under the additional assumption $(\mathcal{C}_{s_0,p,q_0})$. Let $q_0,q_1\in[1,+\infty)$, $s_0<s<s_1$, such that $(\mathcal{C}_{s_0,p,q_0})$ is satisfied. Let $u\in \dot{\mathrm{B}}^{s_0}_{p,q_0}(\mathbb{R}^n)+\dot{\mathrm{B}}^{s_1}_{p,q_1}(\mathbb{R}^n)$. For $(a,b)\in \dot{\mathrm{B}}^{s_0}_{p,q_0}(\mathbb{R}^n)\times\dot{\mathrm{B}}^{s_1}_{p,q_1}(\mathbb{R}^n)$, such that $u=a+b$, by definition, we have
\begin{align*}
    (\dot{\Delta}_j u)_{j\in\mathbb{Z}} = (\dot{\Delta}_j a)_{j\in\mathbb{Z}} + (\dot{\Delta}_j b)_{j\in\mathbb{Z}} \in \ell^{q_0}_{s_0}(\mathbb{Z},\mathrm{L}^p(\mathbb{R}^n)) + \ell^{q_1}_{s_1}(\mathbb{Z},\mathrm{L}^p(\mathbb{R}^n))\text{.}
\end{align*}
Therefore, by the definition of the $K$-functional and Besov norms, for $t>0$,
\begin{align*}
    K(t,(\dot{\Delta}_j u)_{j\in\mathbb{Z}},\mathrm{L}^p(\mathbb{R}^n,\ell^{q_0}_{s_0}(\mathbb{Z},\mathrm{L}^p(\mathbb{R}^n)),\ell^{q_1}_{s_1}(\mathbb{Z},\mathrm{L}^p(\mathbb{R}^n)))&\leqslant \lVert a \lVert_{\dot{\mathrm{B}}^{s_0}_{p,q_0}(\mathbb{R}^n)} + t \lVert b \lVert_{\dot{\mathrm{B}}^{s_1}_{p,q_1}(\mathbb{R}^n)}\text{.}
\end{align*}
We take the infimum on an all such pairs $(a,b)$, then for all $t>0$,
\begin{align}\label{eq:KfunclqsLpControledByBspq}
    K(t,(\dot{\Delta}_j u)_{j\in\mathbb{Z}},\ell^{q_0}_{s_0}(\mathbb{Z},\mathrm{L}^p(\mathbb{R}^n)),\ell^{q_1}_{s_1}(\mathbb{Z},\mathrm{L}^p(\mathbb{R}^n)))\leqslant K(t,u,\dot{\mathrm{B}}^{s_0}_{p,q_0}(\mathbb{R}^n),\dot{\mathrm{B}}^{s_1}_{p,q_1}(\mathbb{R}^n))\text{.}
\end{align}
We have obtained $(\dot{\Delta}_j u)_{j\in\mathbb{Z}}\in \ell^{q_0}_{s_0}(\mathbb{Z},\mathrm{L}^p(\mathbb{R}^n)) + \ell^{q_1}_{s_1}(\mathbb{Z},\mathrm{L}^p(\mathbb{R}^n))$, so let $(\mathfrak{a},\mathfrak{B})\in \ell^{q_0}_{s_0}(\mathbb{Z},\mathrm{L}^p(\mathbb{R}^n))\times\ell^{q_1}_{s_1}(\mathbb{Z},\mathrm{L}^p(\mathbb{R}^n))$ such that $(\Delta_{j}u)_{j\in\mathbb{Z}}= \mathfrak{a}+\mathfrak{b}$. Since $u\in\mathcal{S}'_h(\mathbb{R}^n)$, and $\mathfrak{a}\in \ell^{q_0}_{s_0}(\mathbb{Z},\mathrm{L}^p(\mathbb{R}^n))$ under the condition $(\mathcal{C}_{s_0,p,q_0})$, by Step 2.1, we do have
\begin{align*}
    u=\tilde{\Sigma}{(\Delta_{j}u)_{j\in\mathbb{Z}}}\text{ and } \tilde{\Sigma}\mathfrak{a}\in\dot{\mathrm{B}}^{s_0}_{p,q_0}(\mathbb{R}^n)\subset\mathcal{S}'_h(\mathbb{R}^n).
\end{align*}
Therefore, the series $\tilde{\Sigma}\mathfrak{b}=\tilde{\Sigma}{(\Delta_{j}u)_{j\in\mathbb{Z}}}-\tilde{\Sigma}\mathfrak{a}$ converges to an element of $\mathcal{S}'_h(\mathbb{R}^n)$, and by Step 1,
\begin{align*}
    \lVert \dot{\Delta}_j\tilde{\Sigma}\mathfrak{b}\rVert_{\ell^{q_1}_{s_1}(\mathbb{Z},\mathrm{L}^p(\mathbb{R}^n))}\lesssim_{s_1}\lVert \mathfrak{b}\rVert_{\ell^{q_1}_{s_1}(\mathbb{Z},\mathrm{L}^p(\mathbb{R}^n))}.
\end{align*}
This yields $\tilde{\Sigma}\mathfrak{b}\in \dot{\mathrm{B}}^{s_1}_{p,q_1}(\mathbb{R}^n)$. Thus for $t>0$, by Step 1,
\begin{align*}
    K(t,u,\dot{\mathrm{B}}^{s_0}_{p,q_0}(\mathbb{R}^n),\dot{\mathrm{B}}^{s_1}_{p,q_1}(\mathbb{R}^n))&\leqslant  \lVert \dot{\Delta}_j\tilde{\Sigma}\mathfrak{a}\rVert_{\ell^{q_0}_{s_0}(\mathbb{Z},\mathrm{L}^p(\mathbb{R}^n))}+  t\lVert \dot{\Delta}_j\tilde{\Sigma}\mathfrak{b}\rVert_{\ell^{q_1}_{s_1}(\mathbb{Z},\mathrm{L}^p(\mathbb{R}^n))} \\
    &\leqslant 4^{2+|s_0|+|s_1|}\left( \lVert \mathfrak{a}\rVert_{\ell^{q_0}_{s_0}(\mathbb{Z},\mathrm{L}^p(\mathbb{R}^n))}+  t\lVert \mathfrak{b}\rVert_{\ell^{q_1}_{s_1}(\mathbb{Z},\mathrm{L}^p(\mathbb{R}^n))}\right).
\end{align*}
One can take the infimum on all such pairs $(\mathfrak{a},\mathfrak{b})$, combined with \eqref{eq:KfunclqsLpControledByBspq}, we obtain
\begin{align*}
    K(t,(\dot{\Delta}_j u)_{j\in\mathbb{Z}},\ell^{q_0}_{s_0}(\mathbb{Z},\mathrm{L}^p(\mathbb{R}^n)),\ell^{q_1}_{s_1}(\mathbb{Z},\mathrm{L}^p(\mathbb{R}^n)))\sim_{s_0,s_1} K(t,u,\dot{\mathrm{B}}^{s_0}_{p,q_0}(\mathbb{R}^n),\dot{\mathrm{B}}^{s_1}_{p,q_1}(\mathbb{R}^n)).
\end{align*}
We multiply by $t^{-\theta}$, we take the $\mathrm{L}^q_\ast$-norm on both sides, so that by \cite[Theorem~5.6.1]{BerghLofstrom1976},
\begin{align*}
     \lVert  u\rVert_{(\dot{\mathrm{B}}^{s_0}_{p,q_0}(\mathbb{R}^n),\dot{\mathrm{B}}^{s_1}_{p,q_1}(\mathbb{R}^n))_{\theta,q}}&\sim_{s_0,s_1,\theta}\lVert \dot{\Delta}_j u\rVert_{(\ell^{q_0}_{s_0}(\mathbb{Z},\mathrm{L}^p(\mathbb{R}^n)),\ell^{q_1}_{s_1}(\mathbb{Z},\mathrm{L}^p(\mathbb{R}^n)))_{\theta,q}} \\ &\sim_{s_0,s_1,\theta} \lVert \dot{\Delta}_j u\rVert_{\ell^{q}_{s}(\mathbb{Z},\mathrm{L}^p(\mathbb{R}^n))} =\lVert u \lVert_{\dot{\mathrm{B}}^{s}_{p,q}(\mathbb{R}^n)}.
\end{align*}

\textbf{Step 2.3:} By Step 2.1, for $p\in[1,+\infty)$, the operator
\begin{align*}
\tilde{\Sigma}\,:\,\ell^1_{s_j}(\mathbb{Z},\mathrm{L}^p(\mathbb{R}^n))\longrightarrow \dot{\mathrm{B}}^{s_j}_{p,1}(\mathbb{Z},\mathrm{L}^p(\mathbb{R}^n)).
\end{align*}
is well-defined and bounded for $s_j\leqslant \frac{n}{p}$, $j\in\{0,1\}$. But Step 2.2 gives the interpolation identity
\begin{align*}
    (\dot{\mathrm{B}}^{s_0}_{p,1}(\mathbb{R}^n),\dot{\mathrm{B}}^{s_1}_{p,1}(\mathbb{R}^n))_{\theta,\infty} = \dot{\mathrm{B}}^{s}_{p,\infty}(\mathbb{R}^n).
\end{align*}
By real interpolation, and \cite[Theorem~5.6.1]{BerghLofstrom1976}, 
\begin{align*}
\tilde{\Sigma}\,:\,\ell^\infty_{s}(\mathbb{Z},\mathrm{L}^p(\mathbb{R}^n))\longrightarrow \dot{\mathrm{B}}^{s}_{p,\infty}(\mathbb{R}^n).
\end{align*}
is well-defined and bounded. Then, for $q_0,q_1=+\infty$, one can prove \ref{eq:realInterpHomBspqRn} exactly as in Step 2.2. 

\textbf{Step 3:} We aim to show the interpolation identity \ref{eq:realInterpHomBspqRn} when $p=+\infty$. 

\textbf{Step 3.1:} We want to prove that for all $s\leqslant0$,
\begin{align*}
    \tilde{\Sigma}\,:\,\ell^1_{s}(\mathbb{Z},\mathrm{L}^\infty(\mathbb{R}^n))\longrightarrow \dot{\mathrm{B}}^{s}_{\infty,1}(\mathbb{R}^n),
\end{align*}
is well-defined and bounded.

Let $\mathfrak{f}=(f_\ell)_{\ell\in\mathbb{Z}}\in\ell^1_{s}(\mathbb{Z},\mathrm{L}^\infty(\mathbb{R}^n))$, by definition, we have for all $m\in\mathbb{N}$,
\begin{align*}
    \sum_{j\in\mathbb{Z}} 2^{js}\lVert\dot{\Delta}_jf_{j\pm m} \rVert_{\mathrm{L}^\infty(\mathbb{R}^n)} \leqslant 2^{m|s|} \sum_{j\in\mathbb{Z}} 2^{js} \lVert f_{j} \rVert_{\mathrm{L}^\infty(\mathbb{R}^n)}<+\infty.
\end{align*}
Hence, we just have to prove that $\tilde{\Sigma} [\mathfrak{f}]$ converges to an element of $\mathcal{S}'_h(\mathbb{R}^n)$. We prove first that it converges to an element of $\mathcal{S}'(\mathbb{R}^n)$.
For all $N\in\mathbb{N}^\ast$, we introduce the finite sums
\begin{align*}
    \tilde{\Sigma}_+^{N}[\mathfrak{f}]:=\dot{\Delta}_{-1}\dot{\Delta}_{0}[f_{-2}+f_{-1}+f_0]+ \dot{\Delta}_{0}[\dot{\Delta}_{0}+\dot{\Delta}_{1}][f_{-1}+f_0+f_1]+\sum_{j= 1}^N\dot{\Delta}_j [f_{j-1}+f_{j}+f_{j+1}].
\end{align*}
For all $N\in\mathbb{N}$, one can check that $\tilde{\Sigma}_+^{N}[\mathfrak{f}]\in {\mathrm{B}}^{s}_{\infty,1}(\mathbb{R}^n)$, and for $M>N\geqslant 1$,
\begin{align*}
    \lVert \tilde{\Sigma}_+^{N}[\mathfrak{f}] -\tilde{\Sigma}_+^{M}[\mathfrak{f}]\rVert_{{\mathrm{B}}^{s}_{\infty,1}(\mathbb{R}^n)} \lesssim_{s} \sum_{j=N-1}^{M+1} 2^{js} \lVert f_{j} \rVert_{\mathrm{L}^\infty(\mathbb{R}^n)}\xrightarrow[N,M\rightarrow+\infty]{} 0.
\end{align*}
Then, one obtains the convergence of the series
\begin{align*}
    \tilde{\Sigma}_+[\mathfrak{f}]:=\dot{\Delta}_{-1}\dot{\Delta}_{0}[f_{-2}+f_{-1}+f_0]+ \dot{\Delta}_{0}[\dot{\Delta}_{0}+\dot{\Delta}_{1}][f_{-1}+f_0+f_1]+\sum_{j\geqslant 1}\dot{\Delta}_j [f_{j-1}+f_{j}+f_{j+1}] 
\end{align*}
in the space ${\mathrm{B}}^{s}_{\infty,1}(\mathbb{R}^n)$. Since $s\leqslant 0$, the series
\begin{align*}
    \tilde{\Sigma}_-[\mathfrak{f}]:=\dot{\Delta}_{0}\dot{\Delta}_{-1}[f_{-1}+f_0+f_1] + \dot{\Delta}_{-1}(\dot{\Delta}_{-1}+\dot{\Delta}_{-2})[f_{-2}+f_{-1}+f_0] +  \sum_{j\leqslant-2} \dot{\Delta}_j [f_{j-1}+f_{j}+f_{j+1}]
\end{align*}
converges absolutely in $\mathrm{L}^\infty(\mathbb{R}^n)$. Hence, 
\begin{align*}
    \tilde{\Sigma}[\mathfrak{f}] = \tilde{\Sigma}_+[\mathfrak{f}]+ \tilde{\Sigma}_-[\mathfrak{f}] \subset \mathrm{B}^{s}_{\infty,1}(\mathbb{R}^n)+  \mathrm{L}^\infty(\mathbb{R}^n) \subset \mathcal{S}'(\mathbb{R}^n).
\end{align*}
It remains to show that $\tilde{\Sigma}[\mathfrak{f}]\in\mathcal{S}'_h(\mathbb{R}^n)$. Let $\Theta\in\mathrm{C}^\infty_c(\mathbb{R}^n)$, real-valued, such that $\supp \Theta\subset B(0,R)$, for some fixed $R>0$, for any $\lambda > 0$, one has
\begin{align*}
    \lVert \Theta(\lambda \mathfrak{D}) \tilde{\Sigma}[\mathfrak{f}]\rVert_{\mathrm{L}^\infty(\mathbb{R}^n)} &\leqslant 3 \lVert \mathcal{F}^{-1}\Theta\rVert_{\mathrm{L}^1(\mathbb{R}^n)} \sum_{j\leqslant -\log_{3}(\lambda)+\log_3(R)+4} \lVert f_{j} \rVert_{\mathrm{L}^\infty(\mathbb{R}^n)}\\ &\leqslant 3 \lVert \mathcal{F}^{-1}\Theta\rVert_{\mathrm{L}^1(\mathbb{R}^n)}  \sum_{j\leqslant -\log_{3}(\lambda)+\log_{3}(R)+4} 2^{js}\lVert f_{j} \rVert_{\mathrm{L}^\infty(\mathbb{R}^n)}\xrightarrow[\lambda\rightarrow+\infty]{} 0,
\end{align*}
which proves $\tilde{\Sigma}[\mathfrak{f}]\in\mathcal{S}'_h(\mathbb{R}^n)$. Moreover, if $q=1$ and $s=0$, one has  $\tilde{\Sigma}[\mathfrak{f}]\in\mathrm{C}_{b,h}^0(\mathbb{R}^n)$.

\textbf{Step 3.2:} We prove here the interpolation identity \ref{eq:realInterpHomBspqRn} for $p=+\infty$, $q\in[1,+\infty]$, as well as the boundedness property
\begin{align*}
    \tilde{\Sigma}\,:\,\ell^q_{s}(\mathbb{Z},\mathrm{L}^\infty(\mathbb{R}^n))\longrightarrow \dot{\mathrm{B}}^{s}_{\infty,q}(\mathbb{R}^n),\qquad q\in(1,+\infty].
\end{align*}

First, the case $q_0,q_1=1$. Let $s_0\leqslant 0$, $s_0<s<s_1$, we let $u\in\dot{\mathrm{B}}^{s_0}_{\infty,1}(\mathbb{R}^n)+\dot{\mathrm{B}}^{s_1}_{\infty,1}(\mathbb{R}^n)$, as in Step 2.2, we obtain for all $t>0$,
\begin{align*}
    K(t,(\dot{\Delta}_j u)_{j\in\mathbb{Z}},\ell^{1}_{s_0}(\mathbb{Z},\mathrm{L}^\infty(\mathbb{R}^n)),\ell^{1}_{s_1}(\mathbb{Z},\mathrm{L}^\infty(\mathbb{R}^n)))\sim_{s_0,s_1} K(t,u,\dot{\mathrm{B}}^{s_0}_{\infty,1}(\mathbb{R}^n),\dot{\mathrm{B}}^{s_1}_{\infty,1}(\mathbb{R}^n)).
\end{align*}
Again, by \cite[Theorem~5.6.1]{BerghLofstrom1976},
\begin{align*}
     \lVert  u\rVert_{(\dot{\mathrm{B}}^{s_0}_{\infty,1}(\mathbb{R}^n),\dot{\mathrm{B}}^{s_1}_{\infty,1}(\mathbb{R}^n))_{\theta,q}}&\sim_{s_0,s_1,\theta}\lVert (\dot{\Delta}_j u)_{j\in\mathbb{Z}}\rVert_{(\ell^{1}_{s_0}(\mathbb{Z},\mathrm{L}^\infty(\mathbb{R}^n)),\ell^{1}_{s_1}(\mathbb{Z},\mathrm{L}^\infty(\mathbb{R}^n)))_{\theta,q}} \\ &\sim_{s_0,s_1,\theta} \lVert (\dot{\Delta}_j u)_{j\in\mathbb{Z}}\rVert_{\ell^{q}_{s}(\mathbb{Z},\mathrm{L}^\infty(\mathbb{R}^n))} =\lVert u \lVert_{\dot{\mathrm{B}}^{s}_{\infty,q}(\mathbb{R}^n)}.
\end{align*}

In particular, choosing $s_0,s_1\leqslant0$, we can interpolate
\begin{align*}
    \tilde{\Sigma}\,:\,\ell^1_{s_j}(\mathbb{Z},\mathrm{L}^\infty(\mathbb{R}^n))\longrightarrow \dot{\mathrm{B}}^{s_j}_{\infty,1}(\mathbb{R}^n),\, j\in\{0,1\},
\end{align*}
in order to obtain the boundedness of
\begin{align*}
     \tilde{\Sigma}\,:\,\ell^q_{s}(\mathbb{Z},\mathrm{L}^\infty(\mathbb{R}^n))\longrightarrow \dot{\mathrm{B}}^{s}_{\infty,q}(\mathbb{R}^n),\, s<0, \, q\in[1,+\infty].
\end{align*}
Therefore, we may reproduce the present step replacing $q_0,q_1=1$, by any value $q_0,q_1\in[1,+\infty]$. This concludes the proof of identity \ref{eq:realInterpHomBspqRn} for all $p,q,q_0,q_1\in[1,+\infty]$, $s,s_0,s_1\in\mathbb{R}$, such that $(\mathcal{C}_{s_0,p,q_0})$.

\textbf{Step 4:} We have proved the identities \ref{eq:realInterpHomBspqRn} and \ref{eq:realInterpHomBsinftyqRn} under the assumption $(\mathcal{C}_{s_0,p,q_0})$. We want to remove this additional condition and to prove the identity \ref{eq:realInterpHomBspqRn} in the case of Sobolev spaces, as well as the identities \ref{eq:realInterpHomBspqRnL1} and \ref{eq:realInterpHomBspqRnLinfty}.

To do so, let $s_0\in\mathbb{R}$ be arbitrary with $s_0<s_1$. By previous steps, one can find $\tilde{s_0}<s_0$, and $r\in[1,+\infty]$ such that $(\mathcal{C}_{\tilde{s_0},p,r})$ is satisfied, by previous steps one has
\begin{align*}
    (\dot{\mathrm{B}}^{\tilde{s_0}}_{p,r}(\mathbb{R}^n),\dot{\mathrm{B}}^{s_1}_{p,q_1}(\mathbb{R}^n))_{\frac{s_0-\Tilde{s_0}}{s_1-\Tilde{s_0}},q_0}=\dot{\mathrm{B}}^{s_0}_{p,q_0}(\mathbb{R}^n).
\end{align*}
Therefore, by the extremal reiteration property, Lemma \ref{lem:ExtrmReitPropnonComp}, we deduce
\begin{align*}
(\dot{\mathrm{B}}^{s_0}_{p,q_0}(\mathbb{R}^n),\dot{\mathrm{B}}^{s_1}_{p,q_1}(\mathbb{R}^n))_{\theta,q} &= \left( (\dot{\mathrm{B}}^{\tilde{s_0}}_{p,r}(\mathbb{R}^n),\dot{\mathrm{B}}^{s_1}_{p,q_1}(\mathbb{R}^n))_{\frac{s_0-\Tilde{s_0}}{s_1-\Tilde{s_0}},q_0} , \dot{\mathrm{B}}^{s_1}_{p,q_1}(\mathbb{R}^n)\right)_{\frac{s-\Tilde{s_0}}{s_1-\Tilde{s_0}},q}\\
&=(\dot{\mathrm{B}}^{\tilde{s_0}}_{p,r}(\mathbb{R}^n),\dot{\mathrm{B}}^{s_1}_{p,q_1}(\mathbb{R}^n))_{\frac{s-\Tilde{s_0}}{s_1-\Tilde{s_0}},q} = \dot{\mathrm{B}}^{s}_{p,q}(\mathbb{R}^n).
\end{align*}
 
The interpolation identities \ref{eq:realInterpHomBspqRn}, \ref{eq:realInterpHomBspqRnL1} and \ref{eq:realInterpHomBspqRnLinfty} follow from Proposition \ref{prop:EmbeddSobBesovRn}.

\textbf{Step 5:} The real interpolation identity \ref{eq:realInterpHomBsinftyqRn}. Since $\dot{\Delta}_j\,:\,\mathrm{C}^0_0(\mathbb{R}^n)\rightarrow\mathrm{C}^0_0(\mathbb{R}^n)$, the operators $(\dot{\Delta}_j)_{j\in\mathbb{Z}}$, $\tilde{\Sigma}$ and $(\dot{\Delta}_j\tilde{\Sigma})_{j\in\mathbb{Z}}$ restrict as bounded operators
\begin{align}
    \bullet \, &\tilde{\Sigma}\,:\,\ell^q_s(\mathbb{Z},\mathrm{C}^0_0(\mathbb{R}^n))\longrightarrow \dot{\mathrm{B}}^{s,0}_{\infty,q}(\mathbb{R}^n)\text{, $q\in[1,+\infty]$, $s\in\mathbb{R}$ such that $(\mathcal{C}_{s,\infty,q})$,}\nonumber\\
    \bullet \, &(\dot{\Delta}_j)_{j\in\mathbb{Z}}\,:\,\dot{\mathrm{B}}^{s,0}_{\infty,q}(\mathbb{R}^n)\longrightarrow \ell^q_s(\mathbb{Z},\mathrm{C}^0_0(\mathbb{R}^n))\text{,  $q\in[1,+\infty]$, $s\in\mathbb{R}$,}\label{eq:liftOperatorLinftyRealInterpRn}\\
    \bullet \, &(\dot{\Delta}_j\tilde{\Sigma})_{j\in\mathbb{Z}}\,:\,\ell^q_s(\mathbb{Z},\mathrm{C}^0_0(\mathbb{R}^n))\longrightarrow \ell^q_s(\mathbb{Z},\mathrm{C}^0_0(\mathbb{R}^n))\text{, $q\in[1,+\infty]$, $s\in\mathbb{R}$.}\nonumber
\end{align}
When $q=+\infty$, we have the same boundedness properties with $(c^0_{s},\dot{\mathcal{B}}^{s,0}_{\infty,\infty})$ instead of $(\ell^\infty_s,\dot{\mathrm{B}}^{s,0}_{\infty,\infty})$. Following Step 3.1, we obtain $\dot{\mathrm{B}}^{0,0}_{\infty,1}(\mathbb{R}^n)\hookrightarrow\mathrm{C}^0_0(\mathbb{R}^n)$. Therefore, it suffices to reproduce Step 3.2 and Step 4.

\textbf{Step 6:} The complex interpolation identity \ref{eq:complexInterpHomBspqRn} follows from a standard retraction and co-retraction argument, since for all $s\in\mathbb{R}$, $p,q\in[1,+\infty]$ such that \eqref{AssumptionCompletenessExponents}, we have the well-defined bounded maps
\begin{align*}
     (\dot{\Delta}_{j})_{j\in\mathbb{Z}}\,:\,\dot{\mathrm{B}}^{s}_{p,q}(\mathbb{R}^n) \longrightarrow\ell^q_{s}(\mathbb{Z},\mathrm{L}^p(\mathbb{R}^n)) \text{, and }\tilde{\Sigma}\,:\,\ell^q_{s}(\mathbb{Z},\mathrm{L}^p(\mathbb{R}^n))\longrightarrow \dot{\mathrm{B}}^{s}_{p,q}(\mathbb{R}^n),
\end{align*}
that satisfy $\tilde{\Sigma}[(\dot{\Delta}_{j})_{j\in\mathbb{Z}}]=\mathrm{I}$. The proof is now complete.
\end{proof}

\begin{remark} The Steps 3.1 and 3.2 in the proof above do imply that, provided $s\leqslant 0$ and $q\in[1,+\infty]$ satisfy the condition $(\mathcal{C}_{s,\infty,q})$, the Besov spaces
\begin{align*}
    \dot{\mathfrak{B}}^{s}_{\infty,q}(\mathbb{R}^n):=\{\,u\in\mathcal{S}'(\mathbb{R}^n)\,|\,\lVert u\rVert_{\dot{\mathrm{B}}^{s}_{\infty,q}(\mathbb{R}^n)}<+\infty\,\},
\end{align*}
are such that one has the inclusion
\begin{align*}
    \dot{\mathfrak{B}}^{s}_{\infty,q}(\mathbb{R}^n)\subset \mathcal{S}'_h(\mathbb{R}^n).
\end{align*}
Then necessarily $ \dot{\mathfrak{B}}^{s}_{\infty,q}(\mathbb{R}^n)= \dot{\mathrm{B}}^{s}_{\infty,q}(\mathbb{R}^n)$. 
\end{remark}

\begin{proposition}[ {\cite[Proposition~2.7]{Gaudin2022}} ]\label{prop:dualityRieszpotential} For any $s\in\mathbb{R}$, $p\in(1,+\infty)$,
\begin{align*}
\left\{\begin{array}{cl}
 \dot{\mathrm{H}}^{s,p}\times \dot{\mathrm{H}}^{-s,p'} &\longrightarrow \mathbb{C}\\
(u, v) &\longmapsto \sum\limits_{\left|j-j'\right| \leq 1} \left\langle\dot{\Delta}_{j} u, \dot{\Delta}_{j'} v\right\rangle_{\mathbb{R}^n}
\end{array}\right.
\end{align*}
defines a continuous bilinear functional on $\dot{\mathrm{H}}^{s,p}(\mathbb{R}^n)\times \dot{\mathrm{H}}^{-s,p'}(\mathbb{R}^n)$. Denote by $\mathcal{V}^{-s,p'}$ the set of functions $v\in\mathcal{S}(\mathbb{R}^n)\cap\dot{\mathrm{H}}^{-s,p'}(\mathbb{R}^n)$ such that $\left\lVert{v}\right\rVert_{\dot{\mathrm{H}}^{-s,p'}(\mathbb{R}^n)}\leqslant 1$. If $u\in\mathcal{S}'_h(\mathbb{R}^n)$, then we have
\begin{align*}
    \left\lVert{u}\right\rVert_{\dot{\mathrm{H}}^{s,p}(\mathbb{R}^n)} = \sup\limits_{\substack{v\in \mathcal{V}^{-s,p'}}} \big\lvert\big\langle u,  v\big\rangle_{\mathbb{R}^n}\big\rvert\text{. }
\end{align*}
Moreover, if $(\mathcal{C}_{s,p})$ is satisfied, $\dot{\mathrm{H}}^{s,p}(\mathbb{R}^n)$ is reflexive and we have
\begin{align}\label{eq:dualityRieszPotential}
    (\dot{\mathrm{H}}^{-s,p'}(\mathbb{R}^n))' = \dot{\mathrm{H}}^{s,p}(\mathbb{R}^n)\text{. }
\end{align}
\end{proposition}

\begin{proposition}[ {\cite[Proposition~2.29]{bookBahouriCheminDanchin}} ]\label{prop:DualityBesovRn} For any $s\in\mathbb{R}$, $p,q\in[1,+\infty]$,
\begin{align}\label{eq:BilinFormDualityBesovRn}
\left\{\begin{array}{cl}
 \dot{\mathrm{B}}^{s}_{p,q}\times \dot{\mathrm{B}}^{-s}_{p',q'} &\longrightarrow \mathbb{C}\\
(u, v) &\longmapsto \sum\limits_{\left|j-j'\right| \leq 1} \left\langle\dot{\Delta}_{j} u, \dot{\Delta}_{j'} v\right\rangle_{\mathbb{R}^n}
\end{array}\right.
\end{align}
defines a continuous bilinear functional on $\dot{\mathrm{B}}^{s}_{p,q}(\mathbb{R}^n)\times \dot{\mathrm{B}}^{-s}_{p',q'}(\mathbb{R}^n)$. Denote by $\mathcal{Q}^{-s}_{p',q'}$ the set of functions $v\in\mathcal{S}(\mathbb{R}^n)\cap\dot{\mathrm{B}}^{-s}_{p',q'}(\mathbb{R}^n)$ such that $\left\lVert{v}\right\rVert_{\dot{\mathrm{B}}^{-s}_{p',q'}(\mathbb{R}^n)}\leqslant 1$. If $u\in\mathcal{S}'_h(\mathbb{R}^n)$, then we have
\begin{align*}
    \left\lVert{u}\right\rVert_{\dot{\mathrm{B}}^{s}_{p,q}(\mathbb{R}^n)} \sim_{s,p,q} \sup\limits_{\substack{v\in \mathcal{Q}^{-s}_{p',q'}}} \big\lvert\big\langle u,  v\big\rangle_{\mathbb{R}^n}\big\rvert\text{. }
\end{align*}
Moreover, we have the duality identities
\begin{enumerate}
    \item $\dot{\mathrm{B}}^{s}_{p,q}(\mathbb{R}^n) = (\dot{\mathrm{B}}^{-s}_{p',q'}(\mathbb{R}^n))'$, $p,q\in(1,+\infty]$, $s<n/p$; \label{eq:dualityBesovRn1}
    \item $\dot{\mathrm{B}}^{s}_{p,1}(\mathbb{R}^n) = (\dot{\mathcal{B}}^{-s}_{p',\infty}(\mathbb{R}^n))'$, $p\in(1,+\infty]$, $s\leqslant n/p$;\label{eq:dualityBesovRn2}
    \item $\dot{\mathrm{B}}^{s}_{1,q}(\mathbb{R}^n) = (\dot{\mathrm{B}}^{-s,0}_{\infty,q'}(\mathbb{R}^n))'$, $q\in(1,+\infty]$, $s< n$;\label{eq:dualityBesovRn3}
    \item $\dot{\mathrm{B}}^{s}_{1,1}(\mathbb{R}^n) = (\dot{\mathcal{B}}^{-s,0}_{\infty,\infty}(\mathbb{R}^n))'$, $s\leqslant n$.\label{eq:dualityBesovRn4}
\end{enumerate}
\end{proposition}

\begin{proof} Since the first part of the statement is given by \cite[Proposition~2.29]{bookBahouriCheminDanchin}. We only prove the duality identities.
We use the bounded maps $\tilde{\Sigma}$, $(\dot{\Delta}_j\tilde{\Sigma})_{j\in\mathbb{Z}}$, and $(\dot{\Delta}_j)_{j\in\mathbb{Z}}$ as introduced in \eqref{eq:retractionMapsLp(l2s)intoHsp}.
We do have the same boundedness properties with the operator
\begin{align}\label{eq:ResconstructOpAst}
    \tilde{\Sigma}_{\ast}[(w_j)_{j\in\mathbb{Z}}] := \sum_{j\in\mathbb{Z}} \dot{\Delta}_j w_j\text{, } (w_j)_{j\in\mathbb{Z}}\subset\mathcal{S}'(\mathbb{R}^n)
\end{align}
replacing $\tilde{\Sigma}$.

We focus on \ref{eq:dualityBesovRn3}. By the boundedness properties of the bilinear form \eqref{eq:BilinFormDualityBesovRn}, we have a canonical embedding $\dot{\mathrm{B}}^{s}_{1,q}(\mathbb{R}^n) \hookrightarrow (\dot{\mathrm{B}}^{-s,0}_{\infty,q'}(\mathbb{R}^n))'$. We prove the embedding is surjective.

Let $U\in(\dot{\mathrm{B}}^{-s,0}_{\infty,q'}(\mathbb{R}^n))'$. We use the boundedness properties \eqref{eq:liftOperatorLinftyRealInterpRn}: for all $(f_j)_{j\in\mathbb{Z}}\subset\mathrm{C}^0_0(\mathbb{R}^n)$ with finite support with respect to the discrete variable, we do have
\begin{align*}
    \tilde{\Sigma}_{\ast}[(f_j)_{j\in\mathbb{Z}}] \in \dot{\mathrm{B}}^{-s,0}_{\infty,q'}(\mathbb{R}^n).
\end{align*}
Therefore, by density of finitely supported sequences in $\ell^{q'}_{s}(\mathbb{Z},\mathrm{C}^0_0(\mathbb{R}^n))$, $U$ induces an element $\mathfrak{U}$ that belongs to $\big(\ell^{q'}_{s}(\mathbb{Z},\mathrm{C}^0_0(\mathbb{R}^n))\big)'$ by the formula
\begin{align*}
    \mathfrak{U}\,:\, (f_j)_{j\in\mathbb{Z}} \longmapsto \langle U, \Tilde{\Sigma}(f_j)_{j\in\mathbb{Z}}\rangle.
\end{align*}
By Lemma \ref{lem:DualitySequencespaces}, there exists $\mathfrak{u} =(\mathfrak{u}_j)_{j\in\mathbb{Z}}\in{\ell^q_s(\mathbb{Z},\mathcal{M}(\mathbb{R}^n))}$, such that for all $\varphi\in\dot{\mathrm{B}}^{-s,0}_{\infty,q'}(\mathbb{R}^n)$,
\begin{align*}
    \langle U, \varphi\rangle=\langle \mathfrak{U}, (\dot{\Delta}_j\varphi)_{j\in\mathbb{Z}} \rangle= \sum_{j\in\mathbb{Z}}\langle \mathfrak{u}_j, \dot{\Delta}_j \varphi \rangle_{\mathbb{R}^n}.
\end{align*}
We would like to set $u:=\sum_{j\in\mathbb{Z}} {\dot{\Delta}_j}\mathfrak{u}_j = \tilde{\Sigma}_{\ast}[(\mathfrak{u}_j)_{j\in\mathbb{Z}}]$. By standard properties of the convolution, the operator $\dot{\Delta}_j\,:\,\mathcal{M}(\mathbb{R}^n)\longrightarrow \mathrm{L}^1(\mathbb{R}^n)$ is bounded with norm at most $1$. For $s<n$, $\dot{\mathrm{B}}^{s}_{1,q}(\mathbb{R}^n)$ is a complete space, and
\begin{align}\label{eq:}
    \lVert \tilde{\Sigma}_{\ast}[(\mathfrak{u}_j)_{j\in\llb-N,N\rrb}]\rVert_{\dot{\mathrm{B}}^{s}_{1,q}(\mathbb{R}^n)}&\leqslant (1+2^{|s|+1}) 
    \big\lVert (\lVert{\dot{\Delta}_j}\mathfrak{u}_j\rVert_{\mathrm{L}^1(\mathbb{R}^n)})_{j\in\llb-N,N\rrb}\big\rVert_{\ell^{q}_s(\mathbb{Z})}\\
    &\lesssim_{s} \big\lVert (\lVert\mathfrak{u}_j \rVert_{\mathcal{M}(\mathbb{R}^n)})_{j\in\llb-N,N\rrb} \big\rVert_{\ell^{q}_s(\mathbb{Z})} = \big\lVert (\mathfrak{u}_j)_{j\in\llb-N,N\rrb}\big\rVert_{\ell^{q}_s(\mathbb{Z},\mathcal{M}(\mathbb{R}^n))}, 
\end{align}
so that $\tilde{\Sigma}_{\ast}$ extends uniquely by completeness as an operator defined on $\ell^{q}_s(\mathbb{Z},\mathcal{M}(\mathbb{R}^n))$ with values in $\dot{\mathrm{B}}^{s}_{1,q}(\mathbb{R}^n)$. Hence, the series $u:=\sum_{j\in\mathbb{Z}} {\dot{\Delta}_j}\mathfrak{u}_j = \tilde{\Sigma}_{\ast}[(\mathfrak{u}_j)_{j\in\mathbb{Z}}]$ is well-defined as an element of $\dot{\mathrm{B}}^{s}_{1,q}(\mathbb{R}^n)$, and we have
\begin{align*}
    \langle U, \varphi\rangle=\langle u, \varphi \rangle_{\mathbb{R}^n}\text{, }\forall \varphi\in\dot{\mathrm{B}}^{-s,0}_{\infty,q'}(\mathbb{R}^n) \text{, and }  \lVert u \rVert_{\dot{\mathrm{B}}^{s}_{1,q}(\mathbb{R}^n)}\lesssim_{s} \lVert (\mathfrak{u}_j)_{j\in\mathbb{Z}}\rVert_{\ell^{q}_s(\mathbb{Z},\mathcal{M}(\mathbb{R}^n))}.
\end{align*}
The remaining duality identities can be proven similarly.
\end{proof}

\begin{proposition}[ {\cite[Proposition~2.9]{Gaudin2022}, \cite[Lemma~12]{DanchinMucha2009}} ]\label{prop:SobolevMultiplier} For all $p,q\in[1,+\infty]$, for all $s\in (-1+\frac{1}{p},\frac{1}{p})$, for all $u\in\dot{\mathrm{B}}^{s}_{p,q}(\mathbb{R}^{n})$ (resp. $\dot{\mathrm{H}}^{s,p}(\mathbb{R}^n)$, $p\neq1,+\infty$), one has
\begin{align*}
    \lVert \mathbbm{1}_{\mathbb{R}^n_+} u \rVert_{\dot{\mathrm{B}}^{s}_{p,q}(\mathbb{R}^{n})} \lesssim_{s,p,n} \lVert u \rVert_{\dot{\mathrm{B}}^{s}_{p,q}(\mathbb{R}^{n})}\text{ }\text{ (resp. }  \lVert \mathbbm{1}_{\mathbb{R}^n_+} u \rVert_{\dot{\mathrm{H}}^{s,p}(\mathbb{R}^{n})} \lesssim_{s,p,n} \lVert u \rVert_{\dot{\mathrm{H}}^{s,p}(\mathbb{R}^{n})}   \text{ ). }
\end{align*}
The same result still holds with $\{{\mathrm{H}},{\mathrm{B}}\}$ instead of $\{\dot{\mathrm{H}},\dot{\mathrm{B}}\}$.
\end{proposition}

\subsection{Function spaces by restriction}\label{sec:FunctionSpacesRestrict}
Let $s\in\mathbb{R}$, $k\in\mathbb{N}$, $p,q\in[1,+\infty]$ and $\Omega$ be an open set of $\mathbb{R}^n$.

For $\mathrm{X}\in\{ \mathrm{B}^{s}_{p,q}, \dot{\mathrm{B}}^{s}_{p,q},  {\mathrm{H}}^{s,p}, \dot{\mathrm{H}}^{s,p},  {\mathrm{W}}^{k,p}, \dot{\mathrm{W}}^{k,p}\}$\footnote{Of course, in the case of Besov spaces, one can also consider $\mathrm{B}^{s,0}_{\infty,q}$, $\mathcal{B}^{s}_{p,\infty}$, etc.}, we define
\begin{align*}
    \mathrm{X}(\Omega):= \mathrm{X}(\mathbb{R}^n)_{|_{\Omega}}\text{, }
\end{align*}
with the quotient norm $\lVert u \rVert_{\mathrm{X}(\Omega)}:= \inf\limits_{\substack{\Tilde{u}\in \mathrm{X}(\mathbb{R}^n),\\ \tilde{u}_{|_{\Omega}}=u\, .}} \lVert \Tilde{u} \rVert_{\mathrm{X}(\mathbb{R}^n)}$. A direct consequence of the definition of those spaces is the density of $\mathcal{S}_0(\overline{\Omega})\subset\mathcal{S}(\overline{\Omega})$ in each of them, the completeness and reflexivity when their counterpart on $\mathbb{R}^n$ are complete and reflexive. We can also define
\begin{align*}
    \mathrm{X}_0(\Omega):= \left\{\,u\in \mathrm{X}(\mathbb{R}^n) \,\Big{|}\, \supp u \subset \overline{\Omega} \right\}\text{, }
\end{align*}
with its natural induced norm $\lVert  u \rVert_{\mathrm{X}_0(\Omega)}:= \lVert  u \rVert_{\mathrm{X}(\mathbb{R}^n)}$. We always have the canonical continuous injection,
\begin{align*}
    \mathrm{X}_0(\Omega)\hookrightarrow \mathrm{X}(\Omega) \text{. }
\end{align*}

Since there is a natural embedding $\mathcal{S}'(\mathbb{R}^n)\hookrightarrow \mathcal{D}'(\mathbb{R}^n)$, we also have the inclusion,
\begin{align*}
    \mathrm{X}(\Omega) \subset \mathcal{D}'(\Omega)\text{,}
\end{align*}
where $\mathcal{D}'(\Omega)=(\mathrm{C}_c^\infty(\Omega))'$ is the topological vector space of distributions on $\Omega$.

If $\mathrm{X}$ and $\mathrm{Y}$ are different function spaces
\begin{itemize}
    \item  if one has a continuous embedding,
\begin{align*}
    \mathrm{Y}(\mathbb{R}^n)\hookrightarrow \mathrm{X}(\mathbb{R}^n) \text{. }
\end{align*}
A direct consequence of the definition is that
\begin{align*}
    \mathrm{Y}(\Omega)\hookrightarrow \mathrm{X}(\Omega) \text{, }
\end{align*}
and similarly with $\mathrm{X}_0$ and $\mathrm{Y}_0$.
    
    \item We denote $[\mathrm{X}\cap \mathrm{Y}](\Omega)$ the restriction of $\mathrm{X}(\mathbb{R}^n)\cap \mathrm{Y}(\mathbb{R}^n)$ to $\Omega$. In general, without any further information, one only has the continuous embedding: 
    \begin{align*}
        [\mathrm{X}\cap \mathrm{Y}](\Omega)\hookrightarrow \mathrm{X}(\Omega)\cap \mathrm{Y}(\Omega) \text{. }
    \end{align*}
\end{itemize}

The results corresponding to those obtained for the whole space $\mathbb{R}^n$ in the previous subsection are usually carried over by the existence of an appropriate
extension operator
\begin{align*}
    \mathcal{E}\,:\, \mathcal{S}'(\Omega)\longrightarrow \mathcal{S}'(\mathbb{R}^n)\text{, }
\end{align*}
bounded from $\mathrm{X}(\Omega)$ to $\mathrm{X}(\mathbb{R}^n)$.

Following the proofs of \cite[Proposition~3.22]{DanchinHieberMuchaTolk2020} and \cite[Lemma~3.13]{Gaudin2022}, the definition of function spaces by restriction yields the next result.

\begin{lemma}\label{lem:EmbeddingInterpHomSobspacesSpeLip} Let $\mathrm{X}(\mathbb{R}^n)$, $\mathrm{Y}(\mathbb{R}^n)$ and $\mathrm{Z}(\mathbb{R}^n)$  be three compatible function spaces over $\mathbb{R}^n$ such that for some $q\in[1,+\infty]$, $\theta\in(0,1)$,
\begin{align*}
    (\mathrm{X}(\mathbb{R}^n),\mathrm{Y}(\mathbb{R}^n))_{\theta,q}=\mathrm{Z}(\mathbb{R}^n).
\end{align*}
We have the continuous embeddings
\begin{align}
    \mathrm{Z}(\Omega)\hookrightarrow&(\mathrm{X}(\Omega),\mathrm{Y}(\Omega))_{\theta,q}\text{, }\label{eq:HomInterpEmbedding1}\\
    \mathrm{Z}_0(\Omega)\hookleftarrow&(\mathrm{X}_0(\Omega),\mathrm{Y}_0(\Omega))_{\theta,q} \text{.}\label{eq:HomInterpEmbedding3}
\end{align}
\end{lemma}

\subsection{Quick overview of inhomogeneous function spaces on (special) Lipschitz domains}\label{sec:InhomSpacesSpeLip}

From now on, and until the end of the paper, $\Omega$ will be a fixed special Lipschitz domain given by a fixed uniformly Lipschitz function $\phi\,:\,\mathbb{R}^{n-1}\longrightarrow \mathbb{R}$, \textit{i.e.},
\begin{align*}
    \Omega :=\{\,(x',x_n)\in\mathbb{R}^{n-1}\times\mathbb{R}\,|\, x_n>\phi(x')\,\}\text{.}
\end{align*}
We also set the following global bi-Lipschitz map of $\mathbb{R}^n=\mathbb{R}^{n-1}\times\mathbb{R}$,
\begin{align}\label{eq:globalChangecoordSpeLip}
    \Psi \,:\, (x',x_n)\longmapsto (x',x_n+\phi(x'))\text{.} 
\end{align}
For which, we have
\begin{align}\label{eq:globalChangecoordJac}
    \Psi(\mathbb{R}^n_+) = \Omega\text{, } \Psi^{-1}(\Omega) = \mathbb{R}^n_+ \text{ and } \mathrm{det}(\nabla \Psi) =  \mathrm{det}(\nabla(\Psi^{-1})) = 1\text{.} 
\end{align}

This subsection is dedicated to recall few selected facts about inhomogeneous function spaces on Lipschitz domains. A substantial part of the presented results is used in the next sections to carry over the corresponding ones for the homogeneous scales of function spaces. One may also see this subsection as a roadmap for the results we aim to reproduce. We follow closely the presentation given in \cite[Section~3A]{Gaudin2022}.

 $\bullet$ \textbf{Extension operators:} For a suitable extension operator in the case of inhomogeneous function spaces on a (special) Lipschitz domain, a notable approach was achieved by Stein in \cite[Chapter~VI,~Section~3]{Stein1970}, for Sobolev spaces with non-negative index, and Besov spaces of positive index of regularity (this follows by real interpolation). A full and definitive result for the inhomogeneous case on Lipschitz domains, and even in a more general case (allowing $p,q$ to be less than $1$ considering the whole Besov and Triebel-Lizorkin scales), was given by Rychkov in \cite{Rychkov1999} where the extension operator is known to be universal and to cover even negative regularity index. 

 $\bullet$ \textbf{Interpolation property:} The extension operator provided by Rychkov can be used to prove, thanks to \cite[Theorem~6.4.2]{BerghLofstrom1976}, if $(\mathfrak{H},\mathfrak{B})\in\{(\mathrm{H},\mathrm{B}), (\mathrm{H}_0,\mathrm{B}_{\cdot,\cdot,0})\}$,
\begin{align}
    [\mathfrak{H}^{s_0,p_0}(\Omega),\mathfrak{H}^{s_1,p_1}(\Omega)]_\theta=\mathfrak{H}^{s,p_\theta}(\Omega)\text{, }&\qquad (\mathfrak{B}^{s_0}_{p,q_0}(\Omega),\mathfrak{B}^{s_1}_{p,q_1}(\Omega))_{\theta,q} = \mathfrak{B}^{s}_{p,q}(\Omega)\text{, }\\
    (\mathfrak{H}^{s_0,p}(\Omega),\mathfrak{H}^{s_1,p}(\Omega))_{\theta,q}= \mathfrak{B}^{s}_{p,q}(\Omega)\text{, }&\qquad [\mathfrak{B}^{s_0}_{p_0,q_0}(\Omega),\mathfrak{B}^{s_1}_{p_1,q_1}(\Omega)]_{\theta} = \mathfrak{B}^{s}_{p_\theta,q_\theta}(\Omega)\text{, }
\end{align}
whenever $(p_0,q_0),(p_1,q_1),(p,q)\in[1,+\infty]^2$($p_j\neq 1,+\infty$, for the complex interpolation of Sobolev (Bessel potential) spaces), $s_0\neq s_1$ two real numbers, and $\theta\in(0,1)$ such that
\begin{align*}
    \left(s,\frac{1}{p_\theta},\frac{1}{q_\theta}\right):= (1-\theta)\left(s_0,\frac{1}{p_0},\frac{1}{q_0}\right)+ \theta\left(s_1,\frac{1}{p_1},\frac{1}{q_1}\right)\text{,}
\end{align*}
with $q_\theta<+\infty$.

 $\bullet$ \textbf{Subspaces of functions supported in the domain:} A nice property is that the description of the boundary yields the following density results, for all $p\in(1,+\infty)$, $q\in[1,+\infty)$, $s\in\mathbb{R}$,
\begin{align}
    \mathrm{\mathrm{H}}^{s,p}_0(\Omega)= \overline{\mathrm{C}_c^\infty(\Omega)}^{\lVert \cdot \rVert_{\mathrm{\mathrm{H}}^{s,p}(\mathbb{R}^n)}}\text{,}\quad\text{and}\quad \mathrm{B}^{s}_{p,q,0}(\Omega)= \overline{\mathrm{C}_c^\infty(\Omega)}^{\lVert \cdot \rVert_{\mathrm{B}^{s}_{p,q}(\mathbb{R}^n)}}\text{. }
\end{align}
One may check \cite[Section~2]{JerisonKenig1995} for the treatment in the case of Sobolev spaces, the case of Besov spaces follows by an interpolation argument, see \cite[Theorem~3.4.2]{BerghLofstrom1976}. As a direct consequence, one has from \cite[Proposition~2.9]{JerisonKenig1995} and \cite[Theorem~3.7.1]{BerghLofstrom1976}, that for all $s\in\mathbb{R}$, $p\in(1,+\infty)$, $q\in[1,+\infty)$,
\begin{align}
    (\mathrm{\mathrm{H}}^{s,p}(\Omega))' = &\mathrm{H}^{-s,p'}_0(\Omega)\text{, }\, (\mathrm{B}^{s}_{p,q}(\Omega))'=\mathrm{B}^{-s}_{p',q',0}(\Omega) \text{, }\\
    &(\mathrm{B}^{s}_{p,q,0}(\Omega))'=\mathrm{B}^{-s}_{p',q'}(\Omega)\text{. }
\end{align}

And finally, thanks to a modified version of Proposition \ref{prop:SobolevMultiplier}, we also have a particular case of equality of Sobolev spaces, with equivalent norms, for all $p\in(1,+\infty)$, $q\in[1,+\infty]$, $s\in(-1+\frac{1}{p},\frac{1}{p})$,
\begin{align}
    \mathrm{\mathrm{H}}^{s,p}(\Omega) = \mathrm{\mathrm{H}}^{s,p}_0(\Omega)\text{, }\, \mathrm{B}^{s}_{p,q}(\Omega)=\mathrm{B}^{s}_{p,q,0}(\Omega) \text{. }
\end{align}
In the case of Besov spaces, everything remains true when $p=1$.

 $\bullet$ \textbf{Further construction and properties of function spaces:} The interested reader will find an explicit and way more general (and still valid, for the most part of it, in the case of a special Lipschitz domain) treatment for bounded Lipschitz domains in \cite{KaltonMayborodaMitrea2007}, where the Triebel-Lizorkin scale, including Hardy spaces, and other endpoint function spaces are also considered.

A recent and accessible exposition is available in \cite[Chapters~8~\&~11]{Leoni2023}. It deals with inhomogeneous Sobolev-Slobodeckij spaces $\mathrm{W}^{s,p}(\Omega)$, which coincides with usual Sobolev spaces when $s\in\mathbb{Z}$, and with diagonal Besov spaces $\mathrm{B}^{s}_{p,p}(\Omega)$ when $s\in\mathbb{R}\setminus \mathbb{Z}$. The case of indices $s\in[0,1]$ is treated in the case of Lipschitz domains, and $s\in[0,m+1]$ in the case where $\Omega$ is a $\mathrm{C}^{m,1}$ domain.

All the results presented above will be used without being mentioned and are assumed to be well-known to the reader.

%----------------------------------------------------
%------------------- Section 3 ----------------------
%----------------------------------------------------
%----------------------------------------------------
%------------------- Section 3 ----------------------
%----------------------------------------------------

\section{Homogeneous Sobolev and Besov spaces on special Lipschitz domains.}\label{sec:ConstrucHomFunctionSpacesSpeLip}

As in \cite[Subsection~3B]{Gaudin2022}, one may expect to recover similar results for the scale of homogeneous Sobolev and Besov as mentioned in Section \ref{sec:InhomSpacesSpeLip}. However, this approach still faces the same issues raised in the introduction of \cite[Subsection~3B]{Gaudin2022}: the lack of completeness for the whole scale of function spaces, and whether Rychkov's extension operator, given in \cite{Rychkov1999}, satisfies homogeneous estimates remains unknown.

The extension method employed in \cite[Chapter~3]{DanchinHieberMuchaTolk2020} and \cite[Subsection~3B1]{Gaudin2022} -- by means of the global change of coordinates and the extension operator by higher order reflection around the boundary -- will fail for high regularities. Indeed, the global bi-Lipschitz map $\Psi$ is obviously not sufficiently regular, making it impossible to take its derivatives more than once. Moreover, even if it were a smooth global diffeomorphism, higher order derivatives would produce inhomogeneous parts with lower order terms. However, this method still makes sense for regularity indices $s\in(-1+1/p,1]$.

The main idea here is to use Stein's universal extension operator introduced in \cite[Chapter~VI]{Stein1970}, for which we have \textbf{homogeneous} estimates for non-negative integers indices of regularity. More precisely, we have the next statement.

\begin{theorem}[ {\cite[Chapter~VI,~Section~3,~Theorem~5']{Stein1970}} ]\label{thm:SteinsExtensionOp} For $\Omega$ a special Lipschitz domain, there exists a well-defined linear operator
\begin{align*}
    \mathcal{E}\,:\,\mathrm{L}^1(\Omega)+\mathrm{L}^\infty(\Omega) \longrightarrow \mathrm{L}^1(\mathbb{R}^n)+\mathrm{L}^\infty(\mathbb{R}^n)
\end{align*}
such that for all $u\in\mathrm{L}^1(\Omega)+\mathrm{L}^\infty(\Omega)$,
\begin{align*}
    [{\mathcal{E}u}]_{|_{\Omega}} = u\text{,}
\end{align*}
and for all $p\in[1,+\infty]$, $m\in\mathbb{N}$, if $u\in\mathrm{W}^{m,p}(\Omega)$, we have $\mathcal{E}u\in \mathrm{W}^{m,p}(\mathbb{R}^n)$, with the following estimates
\begin{equation}\label{eq:SteinExtOpEstInteger}
    \lVert \nabla^k (\mathcal{E}u)\rVert_{\mathrm{L}^p(\mathbb{R}^n)}\lesssim_{p,n,k,\partial\Omega}\lVert \nabla^k u\rVert_{\mathrm{L}^p(\Omega)},\quad k\in\llb 0,m\rrb.
\end{equation}
We also have $\mathcal{E}[\mathrm{C}^{m}_0(\overline{\Omega})] \subset \mathrm{C}^{m}_0(\mathbb{R}^n)$ and $\mathcal{E}[\mathrm{C}^{m}_b(\overline{\Omega})] \subset \mathrm{C}^{m}_b(\mathbb{R}^n)$, for all $m\in\mathbb{N}$.
\end{theorem}

From there, the goal is to fall in a setting so that one just has to use the proofs in \cite[Section~3]{Gaudin2022}, or at least reproduce it with appropriate modifications. Indeed, these proofs mainly depend on the existence of good extension operator with appropriate homogeneous estimates, and in some case, on the reflexivity of considered Sobolev spaces. Our first goal and starting point is to obtain \eqref{eq:SteinExtOpEstInteger} for homogeneous Sobolev norms of fractional order.

\subsection{Homogeneous Sobolev spaces on special Lipschitz domains}\label{sec:HomSobSpacesSpeLip}

The following preliminary lemma is a direct consequence of the definition of function spaces by restriction, Lemma \ref{lem:structLemma} and Theorem \ref{thm:SteinsExtensionOp}.

\begin{lemma}\label{lem:SumSpacesDomainSteinExtOp}We have the following independent properties:
\begin{enumerate}
    \item For $s \in\mathbb{R}$, $p\in(1,+\infty)$, $\dot{\mathrm{H}}^{s,p}(\Omega)\subset \mathrm{H}^{s,p}(\Omega)+\mathrm{C}^{k_s}_0(\overline{\Omega})$, where $k_s=\max(0,\lceil s\rceil)$.
    \item For $k\in\mathbb{N}^\ast$, we have the set inclusions
    \begin{align*}
        \dot{\mathrm{W}}^{k,1}(\Omega) \subset \mathrm{W}^{k,1}(\Omega)+\mathrm{C}^k_0(\overline{\Omega}),\\
        \dot{\mathrm{W}}^{k,\infty}(\Omega) \subset \mathrm{W}^{k,\infty}(\Omega).
    \end{align*}
    \item The extension operator $\mathcal{E}$ introduced in Theorem \ref{thm:SteinsExtensionOp} is such that it satisfies the boundedness properties
    \begin{align*}
        \mathcal{E}\,&:\,\mathrm{H}^{s,p}(\Omega)+\mathrm{C}^{m}_0(\overline{\Omega}) \longrightarrow \mathrm{H}^{s,p}(\mathbb{R}^n)+\mathrm{C}^{m}_0(\mathbb{R}^n),\\
        \mathcal{E}\,&:\,\mathrm{W}^{k,1}(\Omega)+\mathrm{C}^{m}_0(\overline{\Omega}) \longrightarrow \mathrm{W}^{k,1}(\mathbb{R}^n)+\mathrm{C}^{m}_0(\mathbb{R}^n),
    \end{align*}
    for all $p\in(1,+\infty)$, all $s\geqslant0$, all $k,m\in\mathbb{N}$.
\end{enumerate}
\end{lemma}

\medbreak

For the homogeneous Sobolev spaces over $\mathrm{L}^\infty$, Theorem \ref{thm:SteinsExtensionOp} already gives the desired homogeneous estimates, without being clear that the range is a subset of $\mathcal{S}'_h(\mathbb{R}^n)$ or not. We mean that in general, it is not known if $\mathcal{E}({\mathrm{L}}^{\infty}_h(\Omega))\subset {\mathrm{L}}^{\infty}_h(\mathbb{R}^n)$, nor if $\mathcal{E}(\mathrm{C}^k_{h,b}(\overline{\Omega}))\subset \mathrm{C}^k_{b,h}(\mathbb{R}^n)$. However, we do have
\begin{align*}
   \mathcal{E}(\mathrm{C}^{m}_0(\overline{\Omega}))\subset \mathrm{C}^{m}_0(\mathbb{R}^n)\subset \mathcal{S}'_h(\mathbb{R}^n).
\end{align*}
 
Hence, we will take a focus on the case $p\in[1,+\infty)$, and more especially the case of fractional regularity indices $s\in\mathbb{R}$ when $p\in(1,+\infty)$. This is because, according to Lemma \ref{lem:SumSpacesDomainSteinExtOp}, $\mathcal{E}$ is entirely and uniquely well-defined on $\dot{\mathrm{H}}^{s,p}(\Omega)$ whenever $s\geqslant 0$, and one has
\begin{align*}
     \mathcal{E}\left(\dot{\mathrm{H}}^{s,p}(\Omega)\right)\subset\mathcal{E}\left(\mathrm{H}^{s,p}(\Omega)+\mathrm{C}^{\lceil s\rceil}_0(\overline{\Omega})\right)\subset\mathrm{H}^{s,p}(\mathbb{R}^n)+\mathrm{C}^{\lceil s\rceil}_0(\mathbb{R}^n)\subset \mathcal{S}'_h(\mathbb{R}^n),
\end{align*}
and similarly for ${\mathrm{W}}^{k,1}$ instead of ${\mathrm{H}}^{s,p}$.

Upon initial examination, it might appear that obtaining boundedness with respect to the norms of fractional-order (Riesz potential) homogeneous Sobolev spaces could be achieved by simply applying complex interpolation to \eqref{eq:SteinExtOpEstInteger}. While the idea is morally correct, one cannot directly employ complex interpolation: our definition makes these function spaces incomplete at high regularity, and completeness is essential for defining complex interpolation. While, the standard complex interpolation procedure is unavailable for the normed spaces we consider, we can still interpolate the inequalities by other means.

\medbreak

\begin{proposition}\label{prop:ExtOpHomSobSpaces} The extension operator $\mathcal{E}$ introduced in Theorem \ref{thm:SteinsExtensionOp} is such that for all $p\in(1,+\infty)$, all $s\geqslant 0$, and all $u\in\dot{\mathrm{H}}^{s,p}(\Omega)$,
\begin{align*}
    {\mathcal{E}u}_{|_{\Omega}} = u\text{,}
\end{align*}
with the estimate
\begin{align*}
    \left\lVert \mathcal{E}u \right\rVert_{\dot{\mathrm{H}}^{s,p}(\mathbb{R}^n)}\lesssim_{p,s,n,\partial\Omega} \left\lVert u \right\rVert_{\dot{\mathrm{H}}^{s,p}(\Omega)}\text{.}
\end{align*}

The same result holds for the Sobolev spaces $\dot{\mathrm{W}}^{k,1}(\Omega)$, $k\in\mathbb{N}$.
\end{proposition}

\begin{proof} Let $\mathcal{E}$ denote Stein's extension operator given in Theorem \ref{thm:SteinsExtensionOp}.

\textbf{Step 1:} The case $1< p< +\infty$. For all $m\in\mathbb{N}$, $p\in(1,+\infty)$, $u\in \mathrm{H}^{m,p}(\Omega)$, by the definition of function spaces by restriction, one obtains that
\begin{align*}
    \left\lVert \nabla^m u\right\rVert_{\mathrm{L}^p(\Omega)} \leqslant \lVert u \rVert_{\dot{\mathrm{H}}^{m,p}(\Omega)}&\lesssim_{p,n,m}\inf\limits_{\substack{\Tilde{u}\in \dot{\mathrm{H}}^{m,p}(\mathbb{R}^n),\\ \tilde{u}_{|_\Omega}=u\, .}} \lVert \nabla^m\Tilde{u} \rVert_{\mathrm{L}^p(\mathbb{R}^n)}\\ &\lesssim_{p,n,m} \left\lVert \nabla^m (\mathcal{E}u)\right\rVert_{\mathrm{L}^p(\mathbb{R}^n)}\lesssim_{p,n,m,\partial\Omega}\left\lVert \nabla^m u\right\rVert_{\mathrm{L}^p(\Omega)}\text{.}
\end{align*}
Hence, Stein's extension operator $\mathcal{E}$ satisfies for all $u\in \mathrm{H}^{m,p}(\Omega)$
\begin{align}\label{eq:estimateSteinExtOphomHmpOmegaRn}
    \lVert \mathcal{E}u \rVert_{\dot{\mathrm{H}}^{m,p}(\mathbb{R}^n)}\lesssim_{p,n,m,\partial\Omega} \lVert u \rVert_{\dot{\mathrm{H}}^{m,p}(\Omega)}\text{ . }
\end{align}
So that $\mathcal{E}\,:\,\dot{\mathrm{H}}^{m,p}(\Omega)\longrightarrow \dot{\mathrm{H}}^{m,p}(\mathbb{R}^n)$ is bounded on subspace $\mathrm{H}^{m,p}(\Omega)$.

The estimate \eqref{eq:estimateSteinExtOphomHmpOmegaRn} implies, that for all $U\in\mathrm{H}^{m,p}(\mathbb{R}^n)$, by the definition of function spaces by restriction
\begin{align*}
    \lVert \mathcal{E} [\mathbbm{1}_\Omega U] \rVert_{\dot{\mathrm{H}}^{m,p}(\mathbb{R}^n)}\lesssim_{p,n,m,\partial\Omega} \lVert U \rVert_{\dot{\mathrm{H}}^{m,p}(\mathbb{R}^n)}\text{ . }
\end{align*}
Therefore, if one uses Proposition \ref{lem:Triebelnormeq},
\begin{align}\label{eq:estimateSteinExtOphomFmp2RnRn}
    \lVert (\dot{\Delta}_j\mathcal{E} [\mathbbm{1}_\Omega U])_{j\in\mathbb{Z}} \rVert_{\mathrm{L}^{p}(\mathbb{R}^n,\ell^2_m(\mathbb{Z}))}\lesssim_{p,n,m,\partial\Omega} \lVert  (\dot{\Delta}_j U)_{j\in\mathbb{Z}} \rVert_{\mathrm{L}^{p}(\mathbb{R}^n,\ell^2_m(\mathbb{Z}))}\text{ . }
\end{align}

For $v=(v_j)_{j\in\mathbb{Z}}\in\mathrm{L}^p(\mathbb{R}^n,\ell^2_m(\mathbb{Z}))$ with finite support with respect to the discrete variable, we set
\begin{align*}
    \Xi_\Omega v:= \Big( \dot{\Delta}_j \mathcal{E}\big[\mathbbm{1}_{\Omega}\big(\sum_{k\in\mathbb{Z}} \dot{\Delta}_k\big[v_{k-1} + v_k +v_{k+1}\big]\big)\big]\Big)_{j\in\mathbb{Z}}
\end{align*}
and since $v$ has finite support with respect to the discrete variable we may define the auxiliary function $V:=\sum_{k\in\mathbb{Z}} \dot{\Delta}_k\big[v_{k-1} + v_k +v_{k+1}\big] \in \mathrm{H}^{m,p}(\mathbb{R}^n)$, and we obtain, by \cite[Proposition~6.1.4]{bookGrafakos2014Classical},
\begin{align*}
    \lVert \Xi_\Omega v\rVert_{\mathrm{L}^{p}(\mathbb{R}^n,\ell^2_m(\mathbb{Z}))} & \lesssim_{p,n,m,\partial\Omega} \lVert  (\dot{\Delta}_j V)_{j\in\mathbb{Z}} \rVert_{\mathrm{L}^{p}(\mathbb{R}^n,\ell^2_m(\mathbb{Z}))}\lesssim_{p,n,m,\partial\Omega} \lVert  v \rVert_{\mathrm{L}^{p}(\mathbb{R}^n,\ell^2_m(\mathbb{Z}))} \text{ . }
\end{align*}

It follows that $\Xi_\Omega$ extends uniquely as a bounded linear operator on $\mathrm{L}^{p}(\mathbb{R}^n,\ell^2_m(\mathbb{Z}))$ for all $p\in(1,+\infty)$, $m\in\mathbb{N}$, which is consistent on elements whose support in the discrete variable is finite. It remains consistent for all elements of the form $(\dot{\Delta}_j U)_{j\in\mathbb{Z}}$, provided $U\in\mathrm{H}^{m,p}(\mathbb{R}^n)+\mathrm{C}^{\lceil s\rceil}_0(\mathbb{R}^n)$, we have by construction and uniqueness of the continuous extension provided by Theorem \ref{thm:SteinsExtensionOp},
\begin{align*}
    \Xi_\Omega [(\dot{\Delta}_j U)_{j\in\mathbb{Z}}]  = \big(\dot{\Delta}_j \mathcal{E}\big[\mathbbm{1}_{\Omega}U\big]\big)_{j\in\mathbb{Z}}\text{.}
\end{align*}

The complex interpolation of mixed weighted Lebesgue spaces, see \cite[Theorems~5.1.2~\&~5.6.3]{BerghLofstrom1976}, yields that
\begin{align*}
    \Xi_\Omega\,:\,\mathrm{L}^{p}(\mathbb{R}^n,\ell^2_s(\mathbb{Z}))\longrightarrow\mathrm{L}^{p}(\mathbb{R}^n,\ell^2_s(\mathbb{Z}))
\end{align*}
is a well-defined bounded linear operator for all $s\geqslant0$, $p\in(1,+\infty)$. Then, the map
\begin{align}\label{eq:LiftedExtOPFspLpl2s}
    \Xi_\Omega[(\dot{\Delta}_j [\cdot])_{j\in\mathbb{Z}}]\,:\,\dot{\mathrm{H}}^{s,p}(\mathbb{R}^n)\longrightarrow\mathrm{L}^{p}(\mathbb{R}^n,\ell^2_s(\mathbb{Z}))
\end{align}
is also well-defined and bounded by Proposition \ref{lem:Triebelnormeq}. Provided $s\geqslant0$, $\mathcal{E}(v+w)\in\mathrm{H}^{s,p}(\mathbb{R}^n)+\mathrm{C}^0_{\lceil s\rceil}(\mathbb{R}^n)$ is already entirely determined for all $v\in\mathrm{H}^{s,p}(\Omega)$, $w\in\mathrm{C}^{\lceil s\rceil}_0(\overline{\Omega})$ by Theorem \ref{thm:SteinsExtensionOp}, and does not depend on the choice of representatives. Hence, for $U\in \dot{\mathrm{H}}^{s,p}(\mathbb{R}^n)\subset{\mathrm{H}}^{s,p}(\mathbb{R}^n)+\mathrm{C}^{\lceil s\rceil}_0(\mathbb{R}^n)$ such that $U_{|_{\Omega}}=u$, by uniqueness of the continuous extension of $\Xi_\Omega$, and the mapping properties of $\mathcal{E}$ provided by Lemma \ref{lem:SumSpacesDomainSteinExtOp}, we have
\begin{align*}
     \Xi_\Omega[(\dot{\Delta}_j U)_{j\in\mathbb{Z}}] = \big(\dot{\Delta}_k \mathcal{E}[\mathbbm{1}_{\Omega}U]\big)_{k\in\mathbb{Z}} = \big(\dot{\Delta}_k \mathcal{E}u\big)_{k\in\mathbb{Z}}\text{.}
\end{align*}
Thus, one may use the estimate \eqref{eq:LiftedExtOPFspLpl2s} and Proposition \ref{lem:Triebelnormeq} to deduce
\begin{align*}
    \left\lVert \mathcal{E}u \right\rVert_{\dot{\mathrm{H}}^{s,p}(\mathbb{R}^n)}\lesssim_{p,s,n,\partial\Omega} \left\lVert U \right\rVert_{\dot{\mathrm{H}}^{s,p}(\mathbb{R}^n)}\text{.}
\end{align*}
However, $U$ is an arbitrary $\dot{\mathrm{H}}^{s,p}$-extension of $u$, so that by the definition of function spaces by restriction, it holds that
\begin{align*}
    \left\lVert \mathcal{E}u \right\rVert_{\dot{\mathrm{H}}^{s,p}(\mathbb{R}^n)}\lesssim_{p,s,n,\partial\Omega} \left\lVert u \right\rVert_{\dot{\mathrm{H}}^{s,p}(\Omega)}\text{.}
\end{align*}

\textbf{Step 2:} The case $p=1$. Let $u\in\dot{\mathrm{W}}^{k,1}(\Omega)\subset \mathrm{W}^{k,1}(\Omega)+\mathrm{C}^{k}_0(\overline{\Omega})$, then by Lemma \ref{lem:SumSpacesDomainSteinExtOp}, $\mathcal{E}u$ is well-defined. It remains to show the continuity with respect to the $\dot{\mathrm{W}}^{k,1}$-norm. Let $U\in\dot{\mathrm{W}}^{k,1}(\mathbb{R}^n)$, be an arbitrary extension of $u$, \textit{i.e.} such that $U_{|_{\Omega}}=u$. We write $U$ as
\begin{align*}
    U = \dot{S}_0 U +[\mathrm{I}-\dot{S}_0] U =: U_-+U_+,
\end{align*}
and we introduce for all $N\in\mathbb{N}$,
\begin{align*}
    U_{-,N} := [\dot{S}_0-\dot{S}_{-N}]U.
\end{align*}
We follow the proof of Lemma \ref{lem:structLemma}: since $U\in\dot{\mathrm{W}}^{k,1}(\mathbb{R}^n)\subset\mathcal{S}'_h(\mathbb{R}^n)$, we have for all $N\in\mathbb{N}$,
\begin{align*}
    U_{-,N},\,U_{-}\in\mathrm{C}^k_0(\mathbb{R}^n)\text{ and }U_{-,N}, U_+\in\mathrm{W}^{k,1}(\mathbb{R}^n),
\end{align*}
with the property
\begin{align*}
    \lVert U_{-}-U_{-,N}\rVert_{\mathrm{W}^{k,\infty}(\mathbb{R}^n)}\xrightarrow[N\rightarrow+\infty]{}0.
\end{align*}
Thus, $(U_{-,N})_{N\in\mathbb{N}}$ converges to $U_{-}$ everywhere as well as the sequences of all its derivatives up to the order $k$. By Theorem \ref{thm:SteinsExtensionOp}, the same goes for $(\mathcal{E}[\mathbbm{1}_{\Omega}U_{-,N}])_{N\in\mathbb{N}}$ converging to $\mathcal{E}[\mathbbm{1}_{\Omega}U_{-}]$ as well as the derivatives up to the order $k$. By the triangle inequality and the Fatou Lemma, Theorem \ref{thm:SteinsExtensionOp}, and the definition of function space by restriction, 
\begin{align*}
    \lVert \nabla^k \mathcal{E}u\rVert_{\mathrm{L}^1(\mathbb{R}^n)}&\leqslant \liminf_{N\rightarrow+\infty }\,\lVert \nabla^k \mathcal{E}[\mathbbm{1}_{\Omega}U_{-,N}]\rVert_{\mathrm{L}^1(\mathbb{R}^n)} + \lVert \nabla^k \mathcal{E}[\mathbbm{1}_{\Omega}U_{+}]\rVert_{\mathrm{L}^1(\mathbb{R}^n)}\\
    &\lesssim_{k,n,\partial\Omega} \liminf_{N\rightarrow+\infty }\,\lVert \nabla^k U_{-,N}\rVert_{\mathrm{L}^1(\mathbb{R}^n)} + \lVert \nabla^k U_{+}\rVert_{\mathrm{L}^1(\mathbb{R}^n)}\\
    &\lesssim_{k,n,\partial\Omega}\lVert \nabla^k U\rVert_{\mathrm{L}^1(\mathbb{R}^n)}.
\end{align*}
We take then the infimum on all such $U$, this yields the result.
\end{proof}

\begin{remark}\label{rmk:HomInterpLift}The method employed here is quite general, and could be adapted to the interpolation of many other kinds of linear operators.

The general idea is to lift the operator at a level for which we can take \textit{a} completion without losing any ambient structure information: here at the level of anisotropic Lebesgue spaces $\mathrm{L}^p(\ell^2_s)$ instead of taking abstract completion of our Sobolev spaces $\dot{\mathrm{H}}^{s,p}$. From this point, one perform  the complex interpolation, then one may hope to get back to a subset of those spaces for which we can compute explicitly the operator, which was exactly what we have done. This is a key point for complex interpolation of operators in the case of non-complete spaces, when one wants to preserve homogeneous estimates.

To be more explicit, for the specific proof above, we applied complex interpolation on the top level of the following family of densely defined commutative diagrams :

\begin{equation*}
\begin{tikzcd}[ampersand replacement=\&]
    {\mathrm{L}^{p}(\ell_{m}^2)} \arrow[d, "\mathrm{R}_{\Omega}\circ\Tilde{\Sigma}"'] \arrow[r, "\Xi_{\Omega}"] \& {\mathrm{L}^{p}(\ell_{m}^2)}\\
    {\dot{\mathrm{H}}^{m,p}(\Omega)} \arrow[r,"\mathcal{E}"'] \& {\dot{\mathrm{H}}^{m,p}(\mathbb{R}^n)} \arrow[u, "(\dot{\Delta}_j)_{j\in\mathbb{Z}}"'] 
\end{tikzcd}\text{ , }\qquad m\in\mathbb{N}.
\end{equation*}

Similar ideas appear in the work of Auscher and Amenta \cite[Chapter~4,~Sections~4.2~\&~4.3]{AmentaAuscher2018}, where interpolation for realizations of abstract operator-adapted Hardy spaces is concerned.
\end{remark}

Proposition \ref{prop:ExtOpHomSobSpaces} is already a powerful enough tool to carry many results. However, this Stein's extension operator has its use restricted to non-negative indices of regularity for the Sobolev scale and positive indices of regularity for the Besov scale. It would be of interest to be able to look at similar properties for regularity indices $s\in(-1+1/p,1/p)$.

We need to carry over the behavior of the global change of coordinates on the homogeneous scale.

For any measurable function $u$ on either $\Omega$ or $\mathbb{R}^n$, and any measurable function $v$ on either $\mathbb{R}^n_+$ or $\mathbb{R}^n$,  we introduce the maps
\begin{align}\label{eq:Tmapsglobalchangeofcoodinates}
    T_\phi u := u\circ \Psi\text{,} \quad\text{ and }\quad T_\phi^{-1} v = v\circ \Psi^{-1} \text{.}
\end{align}
We recall that the definition of $\Psi$ and its basic properties were given in \eqref{eq:globalChangecoordSpeLip} and \eqref{eq:globalChangecoordJac}.

From now on, whenever a result explicitly demonstrates the boundedness of $T_\phi$, or related operator, on homogeneous function spaces, we will sometimes mention the validity of the concerned result in the case of inhomogeneous spaces as well. This is important because there may be an impact on the well-definedness and boundedness of subordinate operators in the context of homogeneous function spaces. See Lemma \ref{lem:structLemma}, Lemma \ref{lem:SumSpacesDomainSteinExtOp}, \textit{(i)} \& \textit{{(ii)}}, and Lemma \ref{lem:structLemBesovSpeLip}.

We mention now that $T_\phi$, and $T_\phi^{-1}$, act as a bijective isometry on $\mathrm{L}^p(\mathbb{R}^n)$, $p\in[1,+\infty]$, on $\mathrm{C}_0^0(\mathbb{R}^n)$ and on $\mathrm{C}_b^0(\mathbb{R}^n)$. However, we do not know whether $T_\phi(\mathrm{C}_{b,h}^0(\mathbb{R}^n))$ is contained in $\mathrm{C}_{b,h}^0(\mathbb{R}^n)$. 

\begin{proposition}\label{prop:globalMapChangeHsp} Let $p\in(1,+\infty)$, $s\in[-1,1]$ and $\mathcal{T}\in\{T_\phi,T_\phi^{-1}\}$. For all $u\in \dot{\mathrm{H}}^{s,p}(\mathbb{R}^n)$, we have $\mathcal{T}u \in \dot{\mathrm{H}}^{s,p}(\mathbb{R}^n)$ with the estimate,
\begin{align*}
    \lVert \mathcal{T}u\rVert_{\dot{\mathrm{H}}^{s,p}(\mathbb{R}^n)} \lesssim_{s,n,\partial\Omega}\lVert u\rVert_{\dot{\mathrm{H}}^{s,p}(\mathbb{R}^n)}\text{.}
\end{align*}
The result also holds for $\dot{\mathrm{W}}^{k,1}(\mathbb{R}^n)$, $k\in\{0,1\}$, instead of $\dot{\mathrm{H}}^{s,p}(\mathbb{R}^n)$.

The result remains true with $\{\mathrm{H},\, \mathrm{W}\}$ instead of $\{\dot{\mathrm{H}},\, \dot{\mathrm{W}}\}$.
\end{proposition}

\begin{proof}We set $\mathcal{T}=T_\phi$ and $\mathcal{T}^{\ast}=T_\phi^{-1}$. First, we check that $\mathcal{T}$, is well-defined on $\dot{\mathrm{H}}^{s,p}(\mathbb{R}^n)$, $\dot{\mathrm{W}}^{k,1}(\mathbb{R}^n)$ and that $\mathcal{T}(\dot{\mathrm{H}}^{s,p}(\mathbb{R}^n)),\mathcal{T}(\dot{\mathrm{W}}^{k,1}(\mathbb{R}^n))\subset\mathcal{S}'_h(\mathbb{R}^n)$. To achieve this, we examine the well-known behavior in the inhomogeneous case. Let $u\in{\mathrm{W}}^{1,p}(\mathbb{R}^n)$, we recall that the following equalities hold almost everywhere
\begin{align*}
    \partial_{x_k}(\mathcal{T}u) = \mathcal{T}(\partial_{x_k}u) \pm \partial_{x_k}\phi \mathcal{T}(\partial_{x_n}u)\text{, }\qquad \partial_{x_n}(\mathcal{T}u)=\mathcal{T}(\partial_{x_n}u)\text{ , } k\in\llb 1,n-1\rrb\text{.}
\end{align*}
We recall that $\mathcal{T}$ is bounded on $\mathrm{L}^p(\mathbb{R}^n)$, and that, moreover, the Jacobian determinant of $\Psi$ is $1$, see \eqref{eq:globalChangecoordJac}. Therefore, we obtain
\begin{align*}
    \lVert \nabla \mathcal{T}u \lVert_{\mathrm{L}^p(\mathbb{R}^n)} \leqslant \lVert \partial_{x_n} u &\lVert_{\mathrm{L}^p(\mathbb{R}^n)} + \sum_{k=1}^{n-1} \lVert \partial_{x_k} u \lVert_{\mathrm{L}^p(\mathbb{R}^n)} + \lVert \partial_{x_k}\phi \lVert_{\mathrm{L}^{\infty}}\lVert \partial_{x_n} u \lVert_{\mathrm{L}^p(\mathbb{R}^n)}\\
    &\leqslant (1+(n-1)\lVert \nabla'\phi \lVert_{\mathrm{L}^{\infty}(\mathbb{R}^{n-1})})\lVert \nabla u \lVert_{\mathrm{L}^p(\mathbb{R}^n)}\text{.}
\end{align*}
Similar computations yield,
\begin{align*}
    \lVert \nabla \mathcal{T}^{\ast}u \lVert_{\mathrm{L}^p(\mathbb{R}^n)} \leqslant (1+(n-1)\lVert \nabla'\phi \lVert_{\mathrm{L}^{\infty}(\mathbb{R}^{n-1})})\lVert \nabla u \lVert_{\mathrm{L}^p(\mathbb{R}^n)}\text{.}
\end{align*}
This estimate has been established regardless of the value of $p\in[1,+\infty]$. When $p\in(1,+\infty)$, $s\in(0,1)$, the boundedness of $\mathcal{T}$ and $\mathcal{T}^\ast$ on $\mathrm{H}^{s,p}(\mathbb{R}^n)$ follows by complex interpolation, the case $s\in[-1,0)$ holds by duality. One can also check, that $\mathcal{T}\,:\,\mathrm{C}^{1}_0(\mathbb{R}^n)\longrightarrow \mathrm{C}_0^{0,1}(\mathbb{R}^n)$. Hence, for all $p\in(1,+\infty)$, we have obtained the boundedness of
\begin{align*}
    \mathcal{T}\,:\,\mathrm{H}^{s,p}(\mathbb{R}^n)+\mathrm{C}^{1}_0(\mathbb{R}^n)\longrightarrow \mathrm{H}^{s,p}(\mathbb{R}^n)+\mathrm{C}_0^{0,1}(\mathbb{R}^n)\subset \mathcal{S}'_h(\mathbb{R}^n).
\end{align*}

Thus, $\mathcal{T}$ is well-defined on  $\mathrm{H}^{s,p}(\mathbb{R}^n)+\mathrm{C}^{1}_0(\mathbb{R}^n)$, and therefore on $\dot{\mathrm{H}}^{s,p}(\mathbb{R}^n)$, $s\in[-1,1]$, with values in $\mathcal{S}_h'(\mathbb{R}^n)$. The same goes with ${\mathrm{W}}^{k,1}(\mathbb{R}^n)$, $k\in\{0,1\}$, instead of ${\mathrm{H}}^{s,p}(\mathbb{R}^n)$. The same statements hold for $\mathcal{T}^\ast$. It remains to show continuity with respect to the homogeneous Sobolev norms.

The boundedness with respect to homogeneous norms in the case $s=k=0,1$, $p\in[1,+\infty)$ can be achieved as before, there is nothing to do.

Now, we assume $p\in(1,+\infty)$. For $v\in\mathrm{L}^2(\mathbb{R}^n)\cap \dot{\mathrm{H}}^{-1,p}(\mathbb{R}^n)$, by Proposition \ref{prop:dualityRieszpotential}, and since the Jacobian determinant of $\Psi^{-1}$ is $1$,
\begin{align*}
    \lVert \mathcal{T}v \lVert_{\dot{\mathrm{H}}^{-1,p}(\mathbb{R}^n)} &= \sup\limits_{\substack{u\in \mathcal{S}(\mathbb{R}^n)\\\lVert u \rVert_{\dot{\mathrm{H}}^{1,p'}(\mathbb{R}^n)}\leqslant 1}} \left\lvert \int_{\mathbb{R}^n} \mathcal{T}v(x)\, u(x) \,\mathrm{d}x \right\rvert\\
    &=  \sup\limits_{\substack{u\in \mathcal{S}(\mathbb{R}^n)\\\lVert u \rVert_{\dot{\mathrm{H}}^{1,p'}(\mathbb{R}^n)}\leqslant 1}} \left\lvert \int_{\mathbb{R}^n} v(x)\, \mathcal{T}^{\ast}u(x) \,\mathrm{d}x \right\rvert\\
    &\leqslant \lVert v \lVert_{\dot{\mathrm{H}}^{-1,p}(\mathbb{R}^n)} \left( \sup\limits_{\substack{u\in \mathcal{S}(\mathbb{R}^n)\\\lVert u \rVert_{\dot{\mathrm{H}}^{1,p'}(\mathbb{R}^n)}\leqslant 1}} \lVert \mathcal{T}^{\ast} u \lVert_{\dot{\mathrm{H}}^{1,p'}(\mathbb{R}^n)} \right)\\
    &\lesssim_{n,\partial\Omega} \lVert v \lVert_{\dot{\mathrm{H}}^{-1,p}(\mathbb{R}^n)}\text{.}
\end{align*}
The same goes for $\mathcal{T}^\ast$. Hence, $\mathcal{T}$ (resp. $\mathcal{T}^\ast$) extends uniquely as a bounded linear operator on $\dot{\mathrm{H}}^{-1,p}(\mathbb{R}^n)$. But since $\mathcal{T}$ (resp. $\mathcal{T}^\ast$) is known to be bounded on $\mathrm{L}^p(\mathbb{R}^n)$, by complex interpolation given in Theorem~\ref{thm:InterpHomSpacesRn}, $\mathcal{T}$ (resp. $\mathcal{T}^\ast$) is then a bounded linear operator on $\dot{\mathrm{H}}^{s,p}(\mathbb{R}^n)$, for all $s\in[-1,0]$. One may repeat the duality argument, thanks to the boundedness on $\dot{\mathrm{H}}^{-s,p'}(\mathbb{R}^n)$ we just proved, to obtain for $s\in[0,1]$, $u\in\dot{\mathrm{H}}^{s,p}(\mathbb{R}^n)$,
\begin{align*}
    \lVert \mathcal{T}u \lVert_{\dot{\mathrm{H}}^{s,p}(\mathbb{R}^n)}\lesssim_{s,n,\partial\Omega} \lVert u \lVert_{\dot{\mathrm{H}}^{s,p}(\mathbb{R}^n)}\text{.}
\end{align*}
We proceed similarly for $\mathcal{T}^\ast$.
\end{proof}

\begin{remark} Everything still holds for more general bi-Lipschitz transformations with constant Jacobian determinants. One may probably want to generalize Proposition \ref{prop:globalMapChangeHsp} to obtain a result similar to \cite[Lemma~2.1.1]{DanchinMucha2015}.
\end{remark}

We can deduce from Proposition \ref{prop:globalMapChangeHsp} several interesting corollaries.

\begin{corollary}\label{cor:HomSobolevMultiplierSpeLip}For all $p\in(1,+\infty)$, for all $s\in (-1+\frac{1}{p},\frac{1}{p})$, for all $u\in\dot{\mathrm{H}}^{s,p}(\mathbb{R}^n)$,
\begin{align*}
    \lVert \mathbbm{1}_{\Omega} u \rVert_{\dot{\mathrm{H}}^{s,p}(\mathbb{R}^{n})} \lesssim_{s,p,n,\partial\Omega} \lVert u \rVert_{\dot{\mathrm{H}}^{s,p}(\mathbb{R}^{n})}.
\end{align*}
The same result still holds with ${\mathrm{H}}$ instead of $\dot{\mathrm{H}}$.
\end{corollary}

\begin{proof} It suffices to write $\mathbbm{1}_{\Omega}u=T_\phi^{-1}\mathbbm{1}_{\mathbb{R}^n_+}T_\phi u$ then to apply Propositions \ref{prop:globalMapChangeHsp}~and~\ref{prop:SobolevMultiplier}.
\end{proof}

\begin{corollary}\label{cor:IsomHomSobSpacesRn+SpeLip} Let $p\in(1,+\infty)$, $s\in[-1,1]$. For all $u\in \dot{\mathrm{H}}^{s,p}(\Omega)$, one has $T_\phi u\in \dot{\mathrm{H}}^{s,p}(\mathbb{R}^n_+)$ with the estimate
\begin{align*}
    \lVert T_\phi u \rVert_{\dot{\mathrm{H}}^{s,p}(\mathbb{R}^n_+)}\lesssim_{p,s,n,\partial\Omega} \left\lVert u \right\rVert_{\dot{\mathrm{H}}^{s,p}(\Omega)}\text{.}
\end{align*}
In particular, $T_\phi\,:\, \dot{\mathrm{H}}^{s,p}(\Omega) \longrightarrow \dot{\mathrm{H}}^{s,p}(\mathbb{R}^n_+)$ is an isomorphism. The result still holds if we replace $(\Omega,\mathbb{R}^n_+,T_\phi)$ by $(\mathbb{R}^n_+,\Omega,T_\phi^{-1})$, and also holds if we replace $\dot{\mathrm{H}}^{s,p}$ by $\dot{\mathrm{W}}^{k,1}$, $k\in\{0,1\}$.

The result remains true with $\{\mathrm{H},\, \mathrm{W}\}$ instead of $\{\dot{\mathrm{H}},\, \dot{\mathrm{W}}\}$.
\end{corollary}

\begin{proof}Follows from the definition of function spaces by restriction and Proposition \ref{prop:globalMapChangeHsp}.
\end{proof}

\begin{corollary}\label{cor:ExtOpHomSobSpacesrestric} There exists a well-defined linear operator $\mathrm{E}$, such that for all $p\in(1,+\infty)$, $s\in(-1+1/p,1]$, for all $u\in \dot{\mathrm{H}}^{s,p}(\mathbb{R}^n)$, we have
\begin{align*}
    [{\mathrm{E}u}]_{|_{\Omega}} = u\text{,}
\end{align*}
with the estimate
\begin{align*}
    \left\lVert \mathrm{E}u \right\rVert_{\dot{\mathrm{H}}^{s,p}(\mathbb{R}^n)}\lesssim_{p,s,n,\partial\Omega} \left\lVert u \right\rVert_{\dot{\mathrm{H}}^{s,p}(\Omega)}\text{.}
\end{align*}
 The result also holds if we replace $\dot{\mathrm{H}}^{s,p}$, by $\dot{\mathrm{W}}^{k,1}$, $k\in\{0,1\}$. The result remains true with $\{\mathrm{H},\, \mathrm{W}\}$ instead of $\{\dot{\mathrm{H}},\, \dot{\mathrm{W}}\}$.
\end{corollary}

\begin{proof} We introduce the extension operator on the half space by even reflection, for any measurable function $u\,:\,\mathbb{R}^n_+\longrightarrow \mathbb{C}$, and for almost every $(x',x_n)\in\mathbb{R}^{n-1}\times\mathbb{R}$,
\begin{align}\label{eq:FlatExtensionOpRn+}
    \tilde{\mathrm{E}}u (x',x_n) &:= \begin{cases}
  u(x',x_n)\,&\text{, if } (x',x_n)\in\mathbb{R}^{n-1}\times(0,+\infty)\text{, }\\    
  u(x',-x_n)\,&\text{, if } (x',x_n)\in\mathbb{R}^{n-1}\times(-\infty,0) \text{. }
\end{cases}
\end{align}
The operator $\tilde{\mathrm{E}}$ is known to have the desired properties when $\Omega=\mathbb{R}^n_+$, see, \textit{e.g.}, \cite[Proposition~3.1]{Gaudin2022}, it can be improved as
\begin{align*}
    \tilde{\mathrm{E}}\,:\,\mathrm{H}^{s,p}(\mathbb{R}^n_+)+\mathrm{C}^{1}_0(\mathbb{R}^n_+)\longrightarrow \mathrm{H}^{s,p}(\mathbb{R}^n)+\mathrm{C}_0^{0,1}(\mathbb{R}^n)\subset \mathcal{S}'_h(\mathbb{R}^n),
\end{align*}
with the desired homogeneous estimates on $\dot{\mathrm{H}}^{s,p}$, $s\in(-1+1/p,1]$: to see it, one can reproduce the Step 2 in the proof of Proposition \ref{prop:ExtOpHomSobSpaces}.

Now, it suffices to set
\begin{align*}
    {\mathrm{E}}:=T_{\phi}^{-1}\tilde{\mathrm{E}}T_{\phi}\text{.}
\end{align*}
The boundedness properties follows from Propositions~\ref{prop:globalMapChangeHsp}~and~\ref{prop:SobolevMultiplier} when $s\in(-1+1/p,1/p)$. When $s\geqslant 0$, it suffices to apply Corollary \ref{cor:IsomHomSobSpacesRn+SpeLip}.
\end{proof}

The next proposition is of paramount importance for the rest of the paper, it will carry over the interpolation properties as well as many density results.

\begin{proposition}\label{prop:ExtOpIntersecHomHspSpeLip}Let $p_j\in(1,+\infty)$, $s_j> -1+\tfrac{1}{p_j}$, $j\in\{0,1\}$. Assume  that one of the two following conditions is satisfied
\begin{enumerate}
    \item $s_0,s_1\leqslant1$ and $\mathbb{E}=\mathrm{E}$, given by Corollary \ref{cor:ExtOpHomSobSpacesrestric};\label{firstcaseExtIntersecHsp} or
    \item $s_0, s_1 \geqslant 0$ and $\mathbb{E}=\mathcal{E}$, given by Proposition \ref{prop:ExtOpHomSobSpaces}.
\end{enumerate}

Then for all $u\in \dot{\mathrm{H}}^{s_0,p_0}(\Omega)\cap \dot{\mathrm{H}}^{s_1,p_1}(\Omega)$, we have  $\mathbb{E}u\in \dot{\mathrm{H}}^{s_j,p_j}(\mathbb{R}^n)$, $j\in\{0,1\}$, with the estimate
\begin{align}\label{eq:decoupledEstExtOp}
    \lVert \mathbb{E} u \rVert_{\dot{\mathrm{H}}^{s_j,p_j}(\mathbb{R}^n)} \lesssim_{s_j,p_j,n}   \lVert  u \rVert_{\dot{\mathrm{H}}^{s_j,p_j}(\Omega)}\text{. }
\end{align}
A similar result holds if, for $j\in\{0,1\}$, we replace $\dot{\mathrm{H}}^{s_j,p_j}$, by $\dot{\mathrm{W}}^{k_j,1}$ $k_j\in\mathbb{N}$.
\end{proposition}

\begin{proof}Let $p_j\in(1,+\infty)$, $s_j>-1+1/p_j$. One may expect to reproduce the proof of \cite[Proposition~3.3]{Gaudin2022}, but this possible only when $\mathbb{E}=\mathrm{E}$, $s_j\leqslant1$.

Now, we set $\mathbb{E}=\mathcal{E}$, $s_j\geqslant 0$, and we consider the operator $\Xi_{\Omega}$, as introduced in \eqref{eq:LiftedExtOPFspLpl2s}. For $u\in\dot{\mathrm{H}}^{s_0,p_0}(\Omega)\cap \dot{\mathrm{H}}^{s_1,p_1}(\Omega)$, and $U\in\dot{\mathrm{H}}^{s_1,p_1}(\mathbb{R}^n)$ such that $U_{|_{\Omega}}=u$, we recall that we have $\mathcal{E}u \in\dot{\mathrm{H}}^{s_0,p_0}(\mathbb{R}^n)\subset \mathcal{S}'_h(\mathbb{R}^n)$. One also has
\begin{align*}
    \big(\dot{\Delta}_k \mathcal{E}u\big)_{k\in\mathbb{Z}} = \big(\dot{\Delta}_k \mathcal{E}[\mathbbm{1}_{\Omega}U]\big)_{k\in\mathbb{Z}} = \Xi_\Omega[(\dot{\Delta}_k U)_{k\in\mathbb{Z}}] \in {\mathrm{L}}^{p_1}(\mathbb{R}^n,\ell^2_{s_1}(\mathbb{Z})) \text{.}
\end{align*}
Therefore, by Proposition \ref{lem:Triebelnormeq}, since $\mathcal{E}u\in\mathcal{S}'_h(\mathbb{R}^n)$,
\begin{align*}
    \lVert \mathcal{E}u\rVert_{\dot{\mathrm{H}}^{s_1,p_1}(\mathbb{R}^n)}&\sim_{p_1,s_1,n} \big\lVert (\dot{\Delta}_k \mathcal{E}u)_{k\in\mathbb{Z}} \big\rVert_{{\mathrm{L}}^{p_1}(\mathbb{R}^n,\ell^2_{s_1}(\mathbb{Z}))}\\
    &\sim_{p_1,s_1,n}\lVert \Xi_\Omega[(\dot{\Delta}_k U)_{k\in\mathbb{Z}}] \rVert_{{\mathrm{L}}^{p_1}(\mathbb{R}^n,\ell^2_{s_1}(\mathbb{Z}))}\\
    &\lesssim_{p_1,s_1,n,\partial\Omega} \lVert U\rVert_{\dot{\mathrm{H}}^{s_1,p_1}(\mathbb{R}^n)}\text{.}
\end{align*}
As in the proof of Proposition \ref{prop:ExtOpHomSobSpaces}, since $U$ is an arbitrary extension of $u$ in $\dot{\mathrm{H}}^{s_1,p_1}(\mathbb{R}^n)$, taking the infimum on all such $U$ yields
\begin{align*}
    \lVert \mathcal{E}u\rVert_{\dot{\mathrm{H}}^{s_1,p_1}(\mathbb{R}^n)}\lesssim_{p_1,s_1,n,\partial\Omega} \lVert u\rVert_{\dot{\mathrm{H}}^{s_1,p_1}(\Omega)}\text{.}
\end{align*}
Thus for $u\in \dot{\mathrm{H}}^{s_0,p_0}(\Omega)\cap \dot{\mathrm{H}}^{s_1,p_1}(\Omega)$, and by the definition of restriction spaces,
\begin{align*}
    \lVert u \rVert_{[\dot{\mathrm{H}}^{s_0,p_0}\cap \dot{\mathrm{H}}^{s_1,p_1}](\Omega)} \leqslant \lVert \mathcal{E} u \rVert_{\dot{\mathrm{H}}^{s_0,p_0}(\mathbb{R}^n)} + \lVert \mathcal{E} u \rVert_{\dot{\mathrm{H}}^{s_1,p_1}(\mathbb{R}^n)} \lesssim_{s_0,s_1,\partial\Omega}^{p_0,p_1,n}   \lVert  u \rVert_{\dot{\mathrm{H}}^{s_0,p_0}(\Omega)} + \lVert  u \rVert_{\dot{\mathrm{H}}^{s_1,p_1}(\Omega)}\text{. }
\end{align*}
This yields the result.
\end{proof}

Now, we want to work with homogeneous Sobolev spaces whose elements are supported in $\overline{\Omega}$.

\begin{proposition}\label{prop:ProOpIntersecHomHsp0SpeLip}Let $p_j\in(1,+\infty)$, $s_j> -1+\tfrac{1}{p_j}$, $j\in\{0,1\}$. We assume that either
\begin{enumerate}
    \item $s_0,s_1\leqslant 1$; or
    \item $s_0,s_1\geqslant 0$.
\end{enumerate}
Then there exists a linear operator $\mathcal{P}_0$ such that for all $u\in \dot{\mathrm{H}}^{s_0,p_0}(\mathbb{R}^n)\cap \dot{\mathrm{H}}^{s_1,p_1}(\mathbb{R}^n)$, we have  $\mathcal{P}_0u\in \dot{\mathrm{H}}^{s_j,p_j}_0(\Omega)$, $j\in\{0,1\}$, with the estimate
\begin{align*}
    \lVert \mathcal{P}_0 u \rVert_{\dot{\mathrm{H}}^{s_j,p_j}(\mathbb{R}^n)} \lesssim_{s_j,p_j,n,\partial\Omega}   \lVert  u \rVert_{\dot{\mathrm{H}}^{s_j,p_j}(\mathbb{R}^n)}\text{. }
\end{align*}
The result also holds if, for $j\in\{0,1\}$, we replace $\dot{\mathrm{H}}^{s_j,p_j}$, by either $\dot{\mathrm{W}}^{k_j,1}$, $k_j\in\mathbb{N}$. 
\end{proposition}

\begin{proof} One mentions $\overline{\Omega}^c$ is also a special Lipschitz domain. If $\mathbb{E}$ is an extension operator for $\Omega$ provided by Theorem \ref{prop:ExtOpIntersecHomHspSpeLip}, we denote by $\mathbb{E}^{-}$ the extension operator from $\overline{\Omega}^c$ to $\mathbb{R}^n$, and we set  for all $u\in \dot{\mathrm{H}}^{s_0,p_0}(\mathbb{R}^n)\cap \dot{\mathrm{H}}^{s_1,p_1}(\mathbb{R}^n)$
\begin{align*}
    \mathcal{P}_0 u := u - \mathbb{E}^{-}[\mathbbm{1}_{\overline{\Omega}^c} u]\text{.}
\end{align*}
In this case, the boundedness properties follow from Theorem \ref{prop:ExtOpIntersecHomHspSpeLip}.
\end{proof}

The next proposition improves the similar result \cite[Proposition~3.9]{Gaudin2022} with a new proof. 
\begin{proposition}\label{prop:densityCcinftyinhomHspspeLip} Let $s_0,s_1\in\mathbb{R}$, $p_0,p_1\in(1,+\infty)$. The space $\mathrm{C}_c^\infty(\Omega)$ is a dense subspace of
\begin{enumerate}
    \item $\dot{\mathrm{H}}^{s_0,p_0}_0(\Omega)\cap \dot{\mathrm{H}}^{s_1,p_1}_0(\Omega)$, if $s_j>-1+1/p_j$, $j\in\{0,1\}$;\label{eq:densityCcinftyInterHsp}
    \item $\dot{\mathrm{H}}^{s_0,p_0}_0(\Omega)$, $-n/p_0'<s_0<0$;\label{eq:densityCcinftyHspNeg}
    \item $\dot{\mathrm{W}}^{k_0,1}_0(\Omega)\cap \dot{\mathrm{W}}^{k_1,1}_0(\Omega)$, for $k_0,k_1\in\mathbb{N}$.\label{eq:densityCcinftyWk1}
\end{enumerate}
\end{proposition}

\begin{proof}\textbf{Step 1:} We prove \ref{eq:densityCcinftyInterHsp} through three subcases: first the case $s_0,s_1\geqslant 0$, second the case $s_j\in(-1+1/p_j,1]$, $j\in\{0,1\}$, and third the case $-1+1/p_0<s_0<0$ and $s_1>1$.

\textbf{Step 1.1:} We assume $s_0,s_1\geqslant 0$. Let $u\in \dot{\mathrm{H}}^{s_0,p_0}_0(\Omega)\cap \dot{\mathrm{H}}^{s_1,p_1}_0(\Omega)\subset\dot{\mathrm{H}}^{s_0,p_0}(\mathbb{R}^n)\cap \dot{\mathrm{H}}^{s_1,p_1}(\mathbb{R}^n)$. Following Step 3 in the proof of Proposition \ref{prop:DensitySobBesovRn}, for $\varepsilon>0$ fixed, let $M>N\in\mathbb{N}$, $\eta,R>0$ such that
\begin{align*}
    \lVert u - u_{N,M}^{R,\eta}\rVert_{\dot{\mathrm{H}}^{s_0,p_0}(\mathbb{R}^n)\cap \dot{\mathrm{H}}^{s_1,p_1}(\mathbb{R}^n)} <\varepsilon.
\end{align*}
We consider the projection operator $\mathcal{P}_0$ given by Proposition \ref{prop:ProOpIntersecHomHsp0SpeLip} for $s_0,s_1\geqslant0$, and associated with Stein's extension operator. Since $\mathcal{P}_0 u =u$,
\begin{align*}
    \lVert u -  \mathcal{P}_0[u_{N,M}^{R,\eta}]\rVert_{\dot{\mathrm{H}}^{s_0,p_0}(\mathbb{R}^n)\cap \dot{\mathrm{H}}^{s_1,p_1}(\mathbb{R}^n)} = \lVert  \mathcal{P}_0[u] -  \mathcal{P}_0[u_{N,M}^{R,\eta}]\rVert_{\dot{\mathrm{H}}^{s_0,p_0}(\mathbb{R}^n)\cap \dot{\mathrm{H}}^{s_1,p_1}(\mathbb{R}^n)} \lesssim_{p_0,p_1,s_0,s_1,n} \varepsilon.
\end{align*}
By construction, one has $\mathcal{P}_0[u_{N,M}^{R,\eta}]\in{\mathrm{H}}^{m,p_0}_0(\Omega)\cap{\mathrm{H}}^{m,p_1}_0(\Omega)\cap\mathrm{C}^{\infty}_0(\mathbb{R}^n)$, for all $m\in\mathbb{N}$. For $\Theta\in\mathrm{C}^\infty_c(\mathbb{R}^n)$, real valued, non-negative, such that $\supp \Theta \subset B(0,2)$, with $\Theta=1$ on $B(0,1)$, and $\zeta,h>0$, we set 
\begin{align*}
    u_{N,M,0}^{R,\eta,\zeta,h}:=\tau_{-h}\left(\Theta(\cdot/\zeta)\mathcal{P}_0[u_{N,M}^{R,\eta}]\right)
\end{align*}
where $\tau_{-h}f(x',x_n) := f(x',x_n-h)$, for all measurable functions $f\,:\,\mathbb{R}^n\longrightarrow\mathbb{C}$. We have
\begin{equation*}
    \supp (u_{N,M,0}^{R,\eta,\zeta,h}) \subset \overline{\Omega_{h,\zeta}}:= \overline{\{\,(x', x_n)\in\mathbb{R}^{n-1}\times \mathbb{R}\,|\, x_n>\phi(x')+h\,\}}\cap\overline{B(0,2\zeta)},
\end{equation*}
so that $\overline{\Omega_{h,\zeta}}$ is compact with $\overline{\Omega_{h,\zeta}}\cap \Omega^c=\emptyset$. For $m =\lfloor s_0\rfloor+\lfloor s_1\rfloor+2$, $j\in\{0,1\}$,
\begin{align*}
    \lVert \mathcal{P}_0[u_{N,M}^{R,\eta}] - u_{N,M,0}^{R,\eta,\zeta,h}\rVert_{\dot{\mathrm{H}}^{s_j,p_j}(\mathbb{R}^n)} &\leqslant \lVert \mathcal{P}_0[u_{N,M}^{R,\eta}] - u_{N,M,0}^{R,\eta,\zeta,0}\rVert_{{\mathrm{H}}^{m,p_j}(\mathbb{R}^n)}+ \lVert u_{N,M,0}^{R,\eta,\zeta,0} - u_{N,M,0}^{R,\eta,\zeta,h}\rVert_{{\mathrm{H}}^{m,p_j}(\mathbb{R}^n)}\\
    &\leqslant \lVert \mathcal{P}_0[u_{N,M}^{R,\eta}] - u_{N,M,0}^{R,\eta,\zeta,0}\rVert_{{\mathrm{H}}^{m,p_j}(\mathbb{R}^n)}+ \lVert u_{N,M,0}^{R,\eta,\zeta,0} - u_{N,M,0}^{R,\eta,\zeta,h}\rVert_{{\mathrm{H}}^{m,p_j}(\mathbb{R}^n)}.
\end{align*}
By strong continuity of translations, as $h$ tends to $0$,
\begin{align*}
    \limsup_{h\rightarrow 0_+} \,\lVert \mathcal{P}_0[u_{N,M}^{R,\eta}] - u_{N,M,0}^{R,\eta,\zeta,h}\rVert_{\dot{\mathrm{H}}^{s_j,p_j}(\mathbb{R}^n)} \leqslant \lVert \mathcal{P}_0[u_{N,M}^{R,\eta}] - u_{N,M,0}^{R,\eta,\zeta,0}\rVert_{{\mathrm{H}}^{m,p_j}(\mathbb{R}^n)}.
\end{align*}
By the Dominated Convergence Theorem, as $\zeta$ tends to infinity, we can choose $\zeta$ large enough and $h$ close enough to $0$, so that
\begin{align*}
     \,\lVert \mathcal{P}_0[u_{N,M}^{R,\eta}] - u_{N,M,0}^{R,\eta,\zeta,h}\rVert_{\dot{\mathrm{H}}^{s_j,p_j}(\mathbb{R}^n)} <\varepsilon.
\end{align*}
Hence, the triangle inequality yields
\begin{align*}
    \lVert u -  u_{N,M,0}^{R,\eta,\zeta,h}\rVert_{\dot{\mathrm{H}}^{s_0,p_0}(\mathbb{R}^n)\cap \dot{\mathrm{H}}^{s_1,p_1}(\mathbb{R}^n)} \lesssim_{p_0,p_1,s_0,s_1,n} \varepsilon.
\end{align*}

\textbf{Step 1.2:}  We assume $s_j\in(-1+1/p_j,1]$, for $j\in\{0,1\}$. Let $u\in \dot{\mathrm{H}}^{s_0,p_0}_0(\Omega)\cap \dot{\mathrm{H}}^{s_1,p_1}_0(\Omega)\subset\dot{\mathrm{H}}^{s_0,p_0}(\mathbb{R}^n)\cap \dot{\mathrm{H}}^{s_1,p_1}(\mathbb{R}^n)$. As for the previous step, let $\varepsilon>0$ fixed, $M>N\geqslant0$, $R,\eta>0$ sufficiently large, and we also consider the projection operator $\mathcal{P}_0$ given by Proposition~\ref{prop:ProOpIntersecHomHsp0SpeLip} for $s_j\in(-1+1/p_j,1]$ for $j\in\{0,1\}$. The projection operator $\mathcal{P}_0$ is the one associated with the extension operator by rough reflection around the boundary, introduced in Corollary~\ref{cor:ExtOpHomSobSpacesrestric}. Since $u^{R,\eta}_{N,M}\in\mathrm{C}^\infty_c(\mathbb{R}^n)$, we obtain $\mathcal{P}_0[u^{R,\eta}_{N,M}]\in\mathrm{C}^{0,1}_c(\mathbb{R}^n)$, with $\supp(\mathcal{P}_0[u^{R,\eta}_{N,M}])\subset \overline{\Omega}$, and the estimate
\begin{align*}
    \lVert u -  \mathcal{P}_0[u_{N,M}^{R,\eta}]\rVert_{\dot{\mathrm{H}}^{s_0,p_0}(\mathbb{R}^n)\cap \dot{\mathrm{H}}^{s_1,p_1}(\mathbb{R}^n)} = \lVert  \mathcal{P}_0[u] -  \mathcal{P}_0[u_{N,M}^{R,\eta}]\rVert_{\dot{\mathrm{H}}^{s_0,p_0}(\mathbb{R}^n)\cap \dot{\mathrm{H}}^{s_1,p_1}(\mathbb{R}^n)} \lesssim_{p_0,p_1,s_0,s_1,n} \varepsilon.
\end{align*}
Now, let $h>\zeta>0$, and let $(\Theta_\gamma)_{\gamma>0}\subset\mathrm{C}^\infty_c(\mathbb{R}^n)$ be a mollifying net that satisfies the support property $\supp(\Theta_\gamma)\subset \overline{B(0,\gamma)}$, for all $\gamma >0$. We set $u_{N,M}^{R,\eta,\gamma,h}:=\tau_{-h}\Theta_\gamma\ast \mathcal{P}_0[u_{N,M}^{R,\eta}]$,
\begin{equation*}
    \supp (u_{N,M}^{R,\eta,\gamma,h}) \subset \overline{\Omega_{h,\gamma}}:= \overline{\{\,(x', x_n)\in\mathbb{R}^{n-1}\times \mathbb{R}\,|\, x_n>\phi(x')+h\,\}}+\overline{B(0,\gamma)}\subset {\Omega},
\end{equation*}
so that $\overline{\Omega_{h,\gamma}}\cap \Omega^c=\emptyset$. Therefore, by strong continuity of translations, and mollification, we have $u_{N,M}^{R,\eta,\gamma,h}\in\mathrm{C}^\infty_c(\Omega)$ arbitrarily close to $\mathcal{P}_0[u_{N,M}^{R,\eta}]$, for $0<\gamma<h$ both close enough to $0$.

\textbf{Step 1.3:} Now, we assume $-1+1/p_0<s_0<0$ and $s_1>1$. One notes that $\dot{\mathrm{H}}^{s_0,p_0}_0(\Omega)\cap \dot{\mathrm{H}}^{s_1,p_1}_0(\Omega)$ is a reflexive Banach space. We consider $\tfrac{1}{q_0}:=\tfrac{1}{p_0}-\tfrac{s_0}{n}$, the following embedding is true
\begin{align*}
    \mathrm{L}^{q_0}(\Omega)\cap \mathrm{H}^{s_1,p_1}_0(\Omega) \hookrightarrow \dot{\mathrm{H}}^{s_0,p_0}_0(\Omega)\cap \dot{\mathrm{H}}^{s_1,p_1}_0(\Omega) \hookrightarrow \dot{\mathrm{H}}^{s_0,p_0}_0(\Omega)\hookrightarrow {\mathrm{H}}^{s_0,p_0}_0(\Omega)\text{.}
\end{align*}
We can summarize it has
\begin{align*}
    \mathrm{L}^{q_0}(\Omega)\cap \mathrm{H}^{s_1,p_1}_0(\Omega) \hookrightarrow \dot{\mathrm{H}}^{s_0,p_0}_0(\Omega)\cap \dot{\mathrm{H}}^{s_1,p_1}_0(\Omega) \hookrightarrow {\mathrm{H}}^{s_0,p_0}_0(\Omega)\text{,}
\end{align*}
and each involved space is reflexive. One may dualize it to deduce,
\begin{align*}
     {\mathrm{H}}^{-s_0,p_0'}(\Omega)  \xhookrightarrow{\iota}({\mathrm{H}}^{s_0,p_0}_0(\Omega))'\hookrightarrow (\dot{\mathrm{H}}^{s_0,p_0}_0(\Omega)\cap \dot{\mathrm{H}}^{s_1,p_1}_0(\Omega))'&\hookrightarrow (\mathrm{L}^{q_0}(\Omega)\cap \mathrm{H}^{s_1,p_1}_0(\Omega))'\\&\xhookrightarrow{\iota^{-1}} \mathrm{L}^{q_0'}(\Omega)+ \mathrm{H}^{-s_1,p_1'}(\Omega) \text{.}
\end{align*}
Here, the map $\iota$ denotes the canonical isomorphism that identifies ${\mathrm{H}}^{-s_0,p_0'}(\Omega)$ with $({\mathrm{H}}^{s_0,p_0}_0(\Omega))'$ and $\mathrm{L}^{q_0'}(\Omega)+ \mathrm{H}^{-s_1,p_1'}(\Omega)$ with $(\mathrm{L}^{q_0}(\Omega)\cap \mathrm{H}^{s_1,p_1}_0(\Omega))'$. A direct consequence is that $\iota(\mathcal{S}(\overline{\Omega}))\subset(\dot{\mathrm{H}}^{s_0,p_0}_0(\Omega)\cap \dot{\mathrm{H}}^{s_1,p_1}_0(\Omega))'$, but $\iota^{-1}[\iota(\mathcal{S}(\overline{\Omega}))]=\mathcal{S}(\overline{\Omega})$ is dense in $\mathrm{L}^{q_0'}(\Omega)+ \mathrm{H}^{-s_1,p_1'}(\Omega)$, so $\iota(\mathcal{S}(\overline{\Omega}))$ is dense in $(\mathrm{L}^{q_0}(\Omega)\cap \mathrm{H}^{s_1,p_1}_0(\Omega))'$. Thus, the following embedding is dense
\begin{align*}
    (\dot{\mathrm{H}}^{s_0,p_0}_0(\Omega)\cap \dot{\mathrm{H}}^{s_1,p_1}_0(\Omega))'\hookrightarrow (\mathrm{L}^{q_0}(\Omega)\cap \mathrm{H}^{s_1,p_1}_0(\Omega))'.
\end{align*}
By reflexivity, also is the next one:
\begin{align*}
    \mathrm{L}^{q_0}(\Omega)\cap \mathrm{H}^{s_1,p_1}_0(\Omega)\hookrightarrow\dot{\mathrm{H}}^{s_0,p_0}_0(\Omega)\cap \dot{\mathrm{H}}^{s_1,p_1}_0(\Omega).
\end{align*}
Then, since $\mathrm{C}_c^\infty(\Omega)$ is dense subspace $\mathrm{L}^{q_0}(\Omega)\cap \mathrm{H}^{s_1,p_1}_0(\Omega)$, by successive approximations,  it holds that $\mathrm{C}_c^\infty(\Omega)$ is dense in $\dot{\mathrm{H}}^{s_0,p_0}_0(\Omega)\cap \dot{\mathrm{H}}^{s_1,p_1}_0(\Omega)$.

\textbf{Step 2:} For the density result \ref{eq:densityCcinftyHspNeg}, provided $\frac{1}{q}=\frac{1}{p_0}-\frac{s_0}{n}$, we have the embeddings
\begin{align*}
    \mathrm{L}^q(\Omega) \hookrightarrow \dot{\mathrm{H}}^{s_0,p_0}_0(\Omega) \hookrightarrow {\mathrm{H}}^{s_0,p_0}_0(\Omega).
\end{align*}
Therefore, we may reproduce a simpler version of the arguments presented in the previous Step 1.3. 

\textbf{Step 3:} For the density result \ref{eq:densityCcinftyWk1}, if $k_0=0<k_1$ the result belongs to the range of the standard theory of Sobolev spaces, since it reduces to the density of $\mathrm{C}_c^\infty(\Omega)$ in the inhomogeneous Sobolev space $\mathrm{W}^{k_1,1}_0(\Omega)$. If $1\leqslant k_0\leqslant k_1$, we reproduce \textit{verbatim} the Step 1.1.
\end{proof}

The next corollary is fundamental for a proper theory of Sobolev spaces involving boundary values. This is a direct combination of Proposition \ref{prop:densityCcinftyinhomHspspeLip} and Corollary \ref{cor:HomSobolevMultiplierSpeLip}.
\begin{corollary}\label{cor:Hdp=Hsp0SpeLip} For all $p\in(1,+\infty)$, $s\in(-1+\frac{1}{p},\frac{1}{p})$,
\begin{align*}
    \dot{\mathrm{H}}^{s,p}_0(\Omega)=\dot{\mathrm{H}}^{s,p}(\Omega)\text{.}
\end{align*}
In particular, $\mathrm{C}_c^\infty(\Omega)$ is dense in $\dot{\mathrm{H}}^{s,p}(\Omega)$ for the same range of indices.
\end{corollary}

\begin{proposition}\label{prop:DualitySobolevDomain} Let $p\in(1,+\infty)$, $s\in(-{n}/{p'},+\infty)$, we have
\begin{align*}
    (\dot{\mathrm{H}}^{s,p}_0(\Omega))' = \dot{\mathrm{H}}^{-s,p'}(\Omega) \text{ and }
    (\dot{\mathrm{H}}^{s,p}(\Omega))' = \dot{\mathrm{H}}^{-s,p'}_0(\Omega)\text{. }
\end{align*}
\end{proposition}

\begin{proof}We propose a proof similar but different from \cite[Proposition~3.11]{Gaudin2022}, due to the fact that we reach regularity indices $s\geqslant n/p$.

\textbf{Step 1:} Let $p\in(1,+\infty)$, $s\in(-{n}/{p'},+\infty)$. We want to prove $(\dot{\mathrm{H}}^{s,p}(\Omega))' = \dot{\mathrm{H}}^{-s,p'}_0(\Omega)$. Let $u\in\dot{\mathrm{H}}^{-s,p'}_0(\Omega)\subset\dot{\mathrm{H}}^{-s,p'}(\mathbb{R}^n)$. Then, by the definition of function spaces by restriction, $u$ induces a linear form on $\dot{\mathrm{H}}^{s,p}(\Omega)$ as follows:
\begin{align*}
    v\longmapsto \langle u,\tilde{v} \rangle_{\mathbb{R}^n},
\end{align*}
where $\tilde{v}$ is any $\dot{\mathrm{H}}^{s,p}$-extension of $v$. This map is well-defined and does not depend on the choice of the extension $\Tilde{v}$ of $v$. Indeed, let $v'$ be another $\dot{\mathrm{H}}^{s,p}$-extension of $v$, then $\tilde{v}-v'\in\dot{\mathrm{H}}^{s,p}_0(\overline{\Omega}^c)$. By Proposition \ref{prop:densityCcinftyinhomHspspeLip}, there exists a sequence $(v_\ell)_{\ell\in\mathbb{N}}\subset \mathrm{C}_c^\infty(\overline{\Omega}^c)$ that converges towards $\tilde{v}-v'$. It implies, 
\begin{align*}
   \langle u,\tilde{v} \rangle_{\mathbb{R}^n}- \langle u,v' \rangle_{\mathbb{R}^n} =\langle u,\tilde{v}-v' \rangle_{\mathbb{R}^n} = \lim_{\ell\rightarrow+\infty} \langle u,v_\ell \rangle_{\mathbb{R}^n}=0.
\end{align*}
Therefore, we have a well-defined, injective and bounded map
\begin{align}\label{eq:mapduality1Hsp}
\left\{\begin{array}{cl}
\dot{\mathrm{H}}^{-s,p'}_0(\Omega) &\longrightarrow (\dot{\mathrm{H}}^{s,p}(\Omega))'\\
u &\longmapsto \,\,\, \big\langle u,\tilde{\cdot}\big\rangle_{\mathbb{R}^n}
\end{array}\right.\text{. }
\end{align}
We prove the reverse embedding. Let $U\in(\dot{\mathrm{H}}^{s,p}(\Omega))'$, we use the definition of function spaces by restriction to deduce that $U$ induces an element of $(\dot{\mathrm{H}}^{s,p}(\mathbb{R}^n))'$, by means of
\begin{align*}
    v\longmapsto \langle U, \mathbbm{1}_{\Omega}\tilde{v} \rangle
\end{align*}
where $\tilde{v}$ is, as before, any $\dot{\mathrm{H}}^{s,p}$-extension of $v$ to the whole $\mathbb{R}^n$. Consequently, by Proposition \ref{prop:dualityRieszpotential}, there exists $u\in\dot{\mathrm{H}}^{-s,p'}(\mathbb{R}^n)$, such that for all $\tilde{v}\in\dot{\mathrm{H}}^{s,p}(\mathbb{R}^n)$
\begin{align*}
    \langle U, \mathbbm{1}_{\Omega}\tilde{v} \rangle = \langle u, \tilde{v} \rangle_{\mathbb{R}^n}.
\end{align*}
In particular, for all $v\in\mathrm{C}_c^\infty(\overline{\Omega}^{c})$
\begin{align*}
    \langle U, \mathbbm{1}_{\Omega}\tilde{v} \rangle = \langle u, \tilde{v} \rangle_{\mathbb{R}^n}=0.
\end{align*}
This yields $u\in\dot{\mathrm{H}}^{-s,p'}_0(\Omega)$. The map \eqref{eq:mapduality1Hsp} is onto, and therefore bijective. It is even isometric.

\textbf{Step 2:} We want to show $(\dot{\mathrm{H}}^{s,p}_0(\Omega))' = \dot{\mathrm{H}}^{-s,p'}(\Omega)$. Let $u\in\dot{\mathrm{H}}^{-s,p'}(\Omega)$, then we can consider $\Tilde{u}$ an arbitrary $\dot{\mathrm{H}}^{-s,p'}$-extension of $u$ to the whole $\mathbb{R}^n$, so that we obtain a continuous linear functional on $\dot{\mathrm{H}}^{s,p}_0(\Omega)$ with the map
\begin{align*}
    v\longmapsto \langle \Tilde{u},v\rangle_{\mathbb{R}^n}.
\end{align*}
Again, by Proposition \ref{prop:densityCcinftyinhomHspspeLip}, it does not depend on the choice of the extension. Indeed, if $u'$ is another $\dot{\mathrm{H}}^{-s,p'}$-extension, then $\tilde{u}-u'\in\dot{\mathrm{H}}^{s,p}_0(\overline{\Omega}^c)$, since $\mathrm{C}_c^\infty(\Omega)$ is a dense subspace of $\dot{\mathrm{H}}^{s,p}_0(\Omega)$, for $(v_\ell)_{\ell\in\mathbb{N}}\subset \mathrm{C}_c^\infty({\Omega})$ that converges towards $v$, we deduce
\begin{align*}
    \langle \Tilde{u},v\rangle_{\mathbb{R}^n}-\langle u',v\rangle_{\mathbb{R}^n} = \lim_{\ell\rightarrow+\infty} \langle \Tilde{u}-u',v_\ell\rangle_{\mathbb{R}^n}=0,
\end{align*}
due to support considerations. Thus, we have a well-defined, injective and bounded map
\begin{align}\label{eq:mapduality2Hsp}
\left\{\begin{array}{cl}
\dot{\mathrm{H}}^{-s,p'}(\Omega) &\longrightarrow (\dot{\mathrm{H}}^{s,p}_0(\Omega))'\\
u &\longmapsto \,\,\, \big\langle \tilde{u},{\cdot}\big\rangle_{\mathbb{R}^n}
\end{array}\right.\text{. }
\end{align}
We prove that the map \eqref{eq:mapduality2Hsp} is onto. Let $U\in (\dot{\mathrm{H}}^{s,p}_0(\Omega))'$. We aim to represent $U$ by an element $\mathfrak{u}\in\dot{\mathrm{H}}^{-s,p'}(\mathbb{R}^n)$, and then by its restriction $u:=\mathfrak{u}_{|_{\Omega}}\in\dot{\mathrm{H}}^{-s,p'}(\Omega)$.

Thus, by the Hahn-Banach Theorem, there exists a linear form $\mathfrak{U}\in(\dot{\mathrm{H}}^{s,p}(\mathbb{R}^n))'$ (not unique) such that
\begin{align*}
    \langle U,\varphi\rangle=\langle \mathfrak{U},\varphi\rangle, \forall\varphi\in\dot{\mathrm{H}}^{s,p}_0(\Omega)\text{, and } \lVert U\rVert_{(\dot{\mathrm{H}}^{s,p}_0(\Omega))'}=\lVert \mathfrak{U}\rVert_{(\dot{\mathrm{H}}^{s,p}(\mathbb{R}^n))'}.
\end{align*}
By duality on the whole space, Proposition \ref{prop:dualityRieszpotential}, there exists a unique $\mathfrak{u}\in \dot{\mathrm{H}}^{-s,p'}(\mathbb{R}^n)$ associated to $\mathfrak{U}$ (we recall that, however, $\mathfrak{U}$ is not unique), such that
\begin{align*}
     \langle \mathfrak{U},v\rangle=\langle \mathfrak{u},v\rangle_{\mathbb{R}^n}, \forall v\in\dot{\mathrm{H}}^{s,p}(\mathbb{R}^n)\text{, and } \lVert U\rVert_{(\dot{\mathrm{H}}^{s,p}_0(\Omega))'}=\lVert \mathfrak{U}\rVert_{(\dot{\mathrm{H}}^{s,p}(\mathbb{R}^n))'}=\lVert \mathfrak{u}\rVert_{\dot{\mathrm{H}}^{-s,p'}(\mathbb{R}^n)}.
\end{align*}

Let $\mathfrak{u}'\in \dot{\mathrm{H}}^{-s,p'}(\mathbb{R}^n)$, be such that 
\begin{align*}
    \langle \mathfrak{u}',v\rangle_{\mathbb{R}^n}=\langle \mathfrak{u},v\rangle_{\mathbb{R}^n}, \forall v\in\dot{\mathrm{H}}^{s,p}_0(\Omega).
\end{align*}
Since $\mathrm{C}_c^\infty(\Omega)\subset\dot{\mathrm{H}}^{s,p}_0(\Omega)$, one obtains that $\mathfrak{u}'_{|_\Omega}=\mathfrak{u}_{|_\Omega}$. Thus, $u:=\mathfrak{u}_{|_\Omega}$ is uniquely determined, and we have,
\begin{align*}
    \langle u,\varphi\rangle_{\Omega}=\langle U,\varphi\rangle, \forall \varphi\in\dot{\mathrm{H}}^{s,p}_0(\Omega).
\end{align*}
The map \eqref{eq:mapduality2Hsp} is onto.
\end{proof}

We recall the following result. 
\begin{proposition}[ {\cite[Theorem~3.3]{CostabelTaggartMcIntosh2013}} ]\label{prop:BogovskiCostabelTaggartMcIntosh}For $p\in(1,+\infty)$, $s<n/p$, the following operator is onto
\begin{align*}
    \div\,:\, \dot{\mathrm{H}}^{s,p}_0(\Omega,\mathbb{C}^n)&\longrightarrow\dot{\mathrm{H}}^{s-1,p}_0(\Omega)\\
    u\qquad&\longmapsto \sum_{k=1}^n \partial_{x_k}u_k.
\end{align*}
and admits a bounded right inverse $\eus{B}\,:\,\dot{\mathrm{H}}^{s-1,p}_0(\Omega)\longrightarrow\dot{\mathrm{H}}^{s,p}_0(\Omega,\mathbb{C}^n)$.
\end{proposition}
\begin{remark} The original statement \cite[Theorem~3.3]{CostabelTaggartMcIntosh2013} was actually built in a different setting. However, when $s<n/p$, any construction of homogeneous function spaces $\dot{\mathrm{H}}^{s,p}(\mathbb{R}^n)$  can be identified as a subspace of $\mathcal{S}'_h(\mathbb{R}^n)$. In particular, all the claimed properties remain meaningful here: especially preserving the support in $\overline{\Omega}$ for the operator $\eus{B}$.
\end{remark}

\begin{proposition}\label{prop:EqNormNablakHsp} Let $p\in(1,+\infty)$, $m\in\llb 1,+\infty\llb$, $s\geqslant m-1+1/p$, for all $u\in \dot{\mathrm{H}}^{s,p}(\Omega)$,
\begin{align*}
    \lVert \nabla^m u \rVert_{\dot{\mathrm{H}}^{s-m,p}(\Omega)} \sim_{s,m,p,n,\partial\Omega} \lVert u \rVert_{\dot{\mathrm{H}}^{s,p}(\Omega)}\text{. }
\end{align*}
In particular, $\lVert \nabla^m\cdot \rVert_{\dot{\mathrm{H}}^{s-m,p}(\Omega)}$ is an equivalent norm on $\dot{\mathrm{H}}^{s,p}(\Omega)$. The result remains true if we replace $(\dot{\mathrm{H}}^{s,p},\dot{\mathrm{H}}^{s-m,p},s)$ by $(\dot{\mathrm{W}}^{k,1},\dot{\mathrm{W}}^{k-m,1},k)$, provided $k\in\mathbb{N}$, $m\in\llb0,k\rrb$.
\end{proposition}

\begin{proof} For all $m\in\llb 1,+\infty\llb$, $p\in(1,+\infty)$, $s\in\mathbb{R}$, $u\in \dot{\mathrm{H}}^{s,p}(\Omega)$, the estimate
\begin{align*}
    \lVert \nabla^m u \rVert_{\dot{\mathrm{H}}^{s-m,p}(\Omega)} \lesssim_{s,k,p,n} \lVert u \rVert_{\dot{\mathrm{H}}^{s,p}(\Omega)}
\end{align*}
always holds by the definition of function spaces by restriction. Therefore, it suffices to prove the reverse inequality.

$\bullet$ For $s\geqslant 1$,  following the arguments in the proof of \cite[Proposition~2.18]{JerisonKenig1995}, for $m\in\llb 1,n\rrb$, there is a linear operator $\mathcal{T}_k$ which enjoys exactly the same boundedness properties as $\mathcal{E}$, and which satisfies the commutation property $\partial_{x_k}\mathcal{E} = \mathcal{E}[\partial_{x_k}] + \mathcal{T}_k[\partial_{x_n}]$.  Hence, one obtains,
\begin{align}\label{eq:estHomHigherOrderSteinExt}
    \lVert u \rVert_{\dot{\mathrm{H}}^{s,p}(\Omega)} \leqslant \lVert \mathcal{E}u \rVert_{\dot{\mathrm{H}}^{s,p}(\mathbb{R}^n)} \lesssim_{s,p,n} \lVert \nabla \mathcal{E}u \rVert_{\dot{\mathrm{H}}^{s-1,p}(\mathbb{R}^n)} \lesssim_{s,p,n,\partial\Omega} \lVert \nabla u \rVert_{\dot{\mathrm{H}}^{s-1,p}(\Omega)}\text{. }
\end{align}
This yields the result when $s\geqslant m$, and $m=1$. 

For $m\geqslant 2$, one derives the  result by induction, reproducing the procedure above. The proof remains similar for the spaces $\dot{\mathrm{W}}^{k,1}$, $\dot{\mathrm{W}}^{k-m,1}$.

$\bullet$ For $p\in(1,+\infty)$, $s\in(1/p,1)$, by Proposition \ref{prop:DualitySobolevDomain},
\begin{align*}
    \lVert u \rVert_{\dot{\mathrm{H}}^{s,p}(\Omega)} =\sup_{\substack{\varphi\in\dot{\mathrm{H}}^{-s,p'}_0(\Omega)\\ \lVert \varphi \rVert_{\dot{\mathrm{H}}^{-s,p'}_0(\Omega)} \leqslant 1}} |\langle u, \varphi \rangle|.
\end{align*}
But since $-s+1<1-1/p=1/p'<n/p'$, by Proposition \ref{prop:BogovskiCostabelTaggartMcIntosh}, one has
\begin{align*}
    \lVert u \rVert_{\dot{\mathrm{H}}^{s,p}(\Omega)}  &\lesssim_{p,s,n}^{\partial\Omega}\sup_{\substack{\Psi\in\dot{\mathrm{H}}^{-s+1,p'}_0(\Omega)^n\\ \lVert \Psi \rVert_{\dot{\mathrm{H}}^{-s+1,p'}_0(\Omega)} \leqslant 1}} |\langle u, \div \Psi \rangle|\\
    &\lesssim_{p,s,n}^{\partial\Omega} \sup_{\substack{\Psi\in\dot{\mathrm{H}}^{-s+1,p'}_0(\Omega)^n\\ \lVert \Psi \rVert_{\dot{\mathrm{H}}^{-s+1,p'}_0(\Omega)} \leqslant 1}} |\langle \nabla u, \Psi \rangle| = \lVert \nabla u \rVert_{\dot{\mathrm{H}}^{s-1,p}(\Omega)}.
\end{align*}
Finally, combining the case $m=1$  $s\in(1/p,1)$, and the case $m\in\mathbb{N}$, $s\geqslant1$, in order to deduce the statement.
\end{proof}

Given the description of the dual spaces, we can now state that the restriction operations and the intersection of spaces indeed commute with each other, in the most relevant cases.

\begin{proposition}\label{prop:IntersecRestricHomSobSpaces} Let $p_0,p_1\in(1,+\infty)$, $s_0,s_1\in\mathbb{R}$, such that $s_j>-1+1/p_j$, $j\in\{0,1\}$. We assume that either
\begin{enumerate}
    \item $s_0,s_1\leqslant 1$; or
    \item $s_0,s_1\geqslant 0$.
\end{enumerate}

The following equality of vector spaces holds with equivalence of norms, 
\begin{align*}
    \dot{\mathrm{H}}^{s_0,p_0}(\Omega)\cap \dot{\mathrm{H}}^{s_1,p_1}(\Omega) = [\dot{\mathrm{H}}^{s_0,p_0}\cap \dot{\mathrm{H}}^{s_1,p_1}](\Omega)\text{. }
\end{align*}
In particular, $\dot{\mathrm{H}}^{s_0,p_0}(\Omega)\cap \dot{\mathrm{H}}^{s_1,p_1}(\Omega)$ admits $\mathcal{S}_0(\overline{\Omega})$ and $\mathrm{C}_c^\infty(\overline{\Omega})$ as dense subspaces and is a Banach space whenever $(\mathcal{C}_{s_0,p_0})$ is satisfied.

A similar result holds if, for $j\in\{0,1\}$, we replace $\dot{\mathrm{H}}^{s_j,p_j}$, by $\dot{\mathrm{W}}^{k_j,1}$, $k_j\in\mathbb{N}$.
\end{proposition}

\begin{proof}For $p_0,p_1\in(1,+\infty)$, $s_j>-1+1/p_j$, $j\in\{0,1\}$, by the definition of function spaces by restriction, $[\dot{\mathrm{H}}^{s_0,p_0}\cap \dot{\mathrm{H}}^{s_1,p_1}](\Omega)$ admits  $\mathrm{C}_c^\infty(\overline{\Omega})$ and $\mathcal{S}_0(\overline{\Omega})$ as dense subspaces and the following continuous embedding also holds by construction,
\begin{align*}
    [\dot{\mathrm{H}}^{s_0,p_0}\cap \dot{\mathrm{H}}^{s_1,p_1}](\Omega)\hookrightarrow\dot{\mathrm{H}}^{s_0,p_0}(\Omega)\cap \dot{\mathrm{H}}^{s_1,p_1}(\Omega)\text{. }
\end{align*}
By Proposition \ref{prop:ExtOpIntersecHomHspSpeLip}, if one of the two following condition is satisfied
\begin{enumerate}
    \item $s_j\geqslant 0$, for $j\in\{0,1\}$; or
    \item $s_j\in(-1+1/p_j,1]$, $j\in\{0,1\}$;
\end{enumerate}
for all $u\in\dot{\mathrm{H}}^{s_0,p_0}(\Omega)\cap \dot{\mathrm{H}}^{s_1,p_1}(\Omega)$, one has $\mathbb{E}u\in\dot{\mathrm{H}}^{s_0,p_0}(\mathbb{R}^n)\cap \dot{\mathrm{H}}^{s_1,p_1}(\mathbb{R}^n)$ so that
\begin{align*}
    \lVert u\rVert_{[\dot{\mathrm{H}}^{s_0,p_0}\cap \dot{\mathrm{H}}^{s_1,p_1}](\Omega)}\leqslant \lVert \mathbb{E}u\rVert_{\dot{\mathrm{H}}^{s_0,p_0}(\mathbb{R}^n)}+\lVert \mathbb{E}u\rVert_{\dot{\mathrm{H}}^{s_1,p_1}(\mathbb{R}^n)} \lesssim_{p_0,p_1,s_0,s_1}^{n,\partial\Omega} \lVert u\rVert_{\dot{\mathrm{H}}^{s_0,p_0}(\Omega)}+\lVert u\rVert_{\dot{\mathrm{H}}^{s_1,p_1}(\Omega)}.
\end{align*}
Therefore, we have the reverse inclusion, and the conclusion follows. The arguments remain valid if we replace $\dot{\mathrm{H}}^{s_j,p_j}$, by $\dot{\mathrm{W}}^{k_j,1}$, $k_j\in\mathbb{N}$, $j\in\{0,1\}$.
\end{proof}

Now, we can state our first interpolation result.

\begin{proposition}\label{prop:InterpHomSobSpeLip}Let $p_j\in(1,+\infty)$, $s_j\in(-n/p_{j}',n/p_j)$, for $j\in\{0,1\}$ and for $\theta\in(0,1)$ we set,
\begin{align*}
    \left(s,\frac{1}{p}\right):= (1-\theta)\left(s_0,\frac{1}{p_0}\right)+ \theta\left(s_1,\frac{1}{p_1}\right)\text{.}
\end{align*}
Then, we have the interpolation identities,
\begin{align}
    [\dot{\mathrm{H}}^{s_0,p_0}(\Omega),\dot{\mathrm{H}}^{s_1,p_1}(\Omega)]_{\theta}&=\dot{\mathrm{H}}^{s,p}(\Omega)\text{, } \label{eq:CompInterpHomSobSpelip1}\\
    [\dot{\mathrm{H}}^{s_0,p_0}_0(\Omega),\dot{\mathrm{H}}^{s_1,p_1}_0(\Omega)]_{\theta}&=\dot{\mathrm{H}}^{s,p}_0(\Omega)\label{eq:CompInterpHomSobSpelip2}\text{.}
\end{align}
\end{proposition}

\begin{proof}\textbf{Step 1:} We prove the interpolation equalities \eqref{eq:CompInterpHomSobSpelip1} and \eqref{eq:CompInterpHomSobSpelip2} assuming for now that one of the two following condition is satisfied
\begin{enumerate}
    \item $s_j\geqslant 0$, for $j\in\{0,1\}$;
    \item $s_j\in(-1+1/p_j,1]$, $j\in\{0,1\}$.
\end{enumerate}
In this case, it suffices to assert that, for ${\dot{\mathfrak{H}}}\in\{ \dot{\mathrm{H}},\dot{\mathrm{H}}_0\}$, the couple $\{ {\dot{\mathfrak{H}}}^{s_0,p_0}(\Omega),{\dot{\mathfrak{H}}}^{s_1,p_1}(\Omega)\}$ is a retract of the couple $\{ \dot{\mathrm{H}}^{s_0,p_0}(\mathbb{R}^n),\dot{\mathrm{H}}^{s_1,p_1}(\mathbb{R}^n)\}$, thanks to \cite[Theorem~6.4.2]{BerghLofstrom1976}. Indeed, both retractions and coretractions are given by
\begin{align*}
    \mathbb{E} \,:\, \dot{\mathrm{H}}^{s_j,p_j}(\Omega)\longrightarrow \dot{\mathrm{H}}^{s_j,p_j}(\mathbb{R}^n) &\text{ and } \mathrm{R}_{{\Omega}} \,:\, \dot{\mathrm{H}}^{s_j,p_j}(\mathbb{R}^n)\longrightarrow \dot{\mathrm{H}}^{s_j,p_j}(\Omega) \text{, }\\
    \iota \,:\, \dot{\mathrm{H}}^{s_j,p_j}_0(\Omega)\longrightarrow \dot{\mathrm{H}}^{s_j,p_j}(\mathbb{R}^n) &\text{ and }\,\,\mathcal{P}_{0} \,:\, \dot{\mathrm{H}}^{s_j,p_j}(\mathbb{R}^n)\longrightarrow \dot{\mathrm{H}}^{s_j,p_j}_0(\Omega) \text{. }
\end{align*}
where $\mathbb{E}$ and $\mathcal{P}_0$ are given by Theorem \ref{prop:ExtOpIntersecHomHspSpeLip} and Proposition \ref{prop:ProOpIntersecHomHsp0SpeLip}, respectively. $\mathrm{R}_{{\Omega}}$ stands for the restriction operator, and $\iota$ for the canonical embedding.

\textbf{Step 2:} We want to remove the conditions $\textit{(i)}$ and $\textit{(ii)}$. To do so, thanks to the previous step, let $-1+1/p_0<s_0<0$, $s_1> 1$, and $\lambda,\mu\in(0,1)$, $r_0,r_1\in(1,+\infty)$ such that
\begin{align*}
    [\dot{\mathfrak{H}}^{s_0,p_0}(\Omega),\dot{\mathfrak{H}}^{1,r_1}(\Omega)]_{\lambda}=\mathrm{L}^{r_0}(\Omega) \quad \text{ and }\quad [\mathrm{L}^{r_0}(\Omega),\dot{\mathfrak{H}}^{s_1,{p}_1}(\Omega)]_{\mu}=\dot{\mathfrak{H}}^{1,r_1}(\Omega).
\end{align*}

\begin{figure}[H]
\centering
\begin{tikzpicture}[yscale=0.6,xscale=6]
  \draw[->] (-0.1,0) -- (1.1,0) node[right,yshift=-2mm] {$1/p$};
  \draw[->] (0,-1) -- (0,3.5) node[above] {$s$};

  \draw[domain=0:1,smooth,variable=\x,blue] plot ({\x},{3*\x}) node[right] {$s=n/p$};
  \fill[blue!30,opacity=0.3] (0,0) -- plot[domain=0:1] (0,0) -- (0,-1) -- (1,-1) -- (1,0) -- cycle;
  \fill[blue!30,opacity=0.3] (0,0) -- plot[domain=0:1] (\x,{3*\x}) -- (1,3) -- (1,0) -- cycle;
  \draw[domain=0:1,smooth,variable=\x,red] plot ({\x},{-1+\x}) node[right,yshift=2mm] {$s=-1+1/p$};
  \draw[dashed] (1,3) -- (0,3)  node[left] {$s=n$};
  \draw[dashed] (1,1) -- (0,1)  node[left] {$s=1$};
  \draw[circle,fill,inner sep=1pt] (0,0) node[below left] {$0$};
  \draw[circle,fill,inner sep=1pt] (1,0) node[below right] {$1$};

  \draw[domain=0.10:0.75,smooth,variable=\x,black] plot ({\x},{3.5*\x-1});
  \node[circle,fill,inner sep=1.5pt,label=below:$\mathrm{L}^{r_0}$] at (2/7,0) {};
  \node[circle,fill,inner sep=1.5pt,label=below right:${\dot{\mathfrak{H}}}^{1,r_1}$] at (4/7,1) {};
  \node[circle,fill,inner sep=1.5pt,label=above:${\dot{\mathfrak{H}}}^{s_0,p_0}$] at (0.10,-0.65) {};
  \node[circle,fill,inner sep=1.5pt,label=above right:${\dot{\mathfrak{H}}}^{s_1,p_1}$] at (3/4,13/8) {};
\end{tikzpicture}
\caption{Hardy-Littlewood-Sobolev-Kato diagram: representation of the chosen interpolation scale to apply Wolff's reiteration theorem.}
\end{figure}

We can apply Wolff's reiteration theorem \cite[Theorem~2]{JansonNilssonPeetre1984}, see also \cite[Theorem~2]{Wolff1982} for the original statement, so we deduce for $\theta,\eta\in(0,1)$ such that $\theta=\lambda\eta$ and $1-\eta=(1-\mu)(1-\theta)$:
\begin{align*}
    [\dot{\mathfrak{H}}^{s_0,p_0}(\Omega),\dot{\mathfrak{H}}^{s_1,p_1}(\Omega)]_{\theta}=\mathrm{L}^{r_0}(\Omega) \quad \text{ and }\quad [\dot{\mathfrak{H}}^{s_0,p_0}(\Omega),\dot{\mathfrak{H}}^{s_1,{p}_1}(\Omega)]_{\eta}=\dot{\mathfrak{H}}^{1,r_1}(\Omega).
\end{align*}

We finally apply the reiteration theorem for complex interpolation \cite[Theorem~4.6.1]{BerghLofstrom1976}:
\begin{itemize}
    \item $s\in(0,s_1)$, $1/p=(1-s/s_1)/r_0+ (s/s_1)/p_1$,
    \begin{align*}
     [\dot{\mathfrak{H}}^{s_0,p_0}(\Omega),\dot{\mathfrak{H}}^{s_1,p_1}(\Omega)]_{\frac{s-s_0}{s_1-s_0}} &= \big[ [\dot{\mathfrak{H}}^{s_0,p_0}(\Omega),\dot{\mathfrak{H}}^{s_1,p_1}(\Omega)]_\theta,\dot{\mathfrak{H}}^{s_1,p_1}(\Omega)\big]_{\frac{s}{s_1}}\\ &=[ \mathrm{L}^{r_0}(\Omega),\dot{\mathfrak{H}}^{s_1,p_1}(\Omega)]_{\frac{s}{s_1}} = \dot{\mathfrak{H}}^{s,p}(\Omega).
    \end{align*}
    \item $s\in(s_0,0)\subset (s_0,1)$, $1/p=(1-s/1-s_0)/p_0+ (s-s_0/1-s_0)/r_1$,
    \begin{align*}
     [\dot{\mathfrak{H}}^{s_0,p_0}(\Omega),\dot{\mathfrak{H}}^{s_1,p_1}(\Omega)]_{\frac{s-s_0}{s_1-s_0}} &= \big[ \dot{\mathfrak{H}}^{s_0,p_0}(\Omega),[\dot{\mathfrak{H}}^{s_0,p_0}(\Omega),\dot{\mathfrak{H}}^{s_1,p_1}(\Omega)]_\eta\big]_{\frac{s-s_0}{1-s_0}}\\ &= \big[ \dot{\mathfrak{H}}^{s_0,p_0}(\Omega),\dot{\mathfrak{H}}^{1,r_1}(\Omega)\big]_{\frac{s-s_0}{1-s_0}} = \dot{\mathfrak{H}}^{s,p}(\Omega).
    \end{align*}
\end{itemize}

Thus, the equalities \eqref{eq:CompInterpHomSobSpelip1} and \eqref{eq:CompInterpHomSobSpelip2} hold in all cases whenever $s_j>-1+1/p_j$, $j\in\{0,1\}$.

\textbf{Step 3:} To prove the interpolation identities \eqref{eq:CompInterpHomSobSpelip1} and \eqref{eq:CompInterpHomSobSpelip2}, for $s_j\in(-n/p_j',1/p_j)$, we argue by duality. By construction, $\dot{\mathrm{H}}^{s_0,p_0}(\Omega)\cap\dot{\mathrm{H}}^{s_1,p_1}(\Omega)$ is dense in $\dot{\mathrm{H}}^{s_j,p_j}(\Omega)$, $j\in\{0,1\}$, since it contains $\mathcal{S}_0(\overline{\Omega})$. Proposition \ref{prop:densityCcinftyinhomHspspeLip} gives the density of $\dot{\mathrm{H}}^{s_0,p_0}_0(\Omega)\cap\dot{\mathrm{H}}^{s_1,p_1}_0(\Omega)$ in $\dot{\mathrm{H}}^{s_j,p_j}_0(\Omega)$, $j\in\{0,1\}$. Since all involved spaces are reflexive, by \cite[Corollary~4.5.2]{BerghLofstrom1976} we obtain
\begin{align*}
    [{\dot{\mathfrak{H}}}^{s_0,p_0}(\Omega),{\dot{\mathfrak{H}}}^{s_1,p_1}(\Omega)]_{\theta}={\dot{\mathfrak{H}}}^{s,p}(\Omega)\text{, } s_j\in(-n/p_j',1/p_j)\text{, }j\in\{0,1\}\text{.}
\end{align*}
Again, Wolff's reiteration theorem \cite[Theorem~2]{JansonNilssonPeetre1984}, and the standard reiteration theorem for complex interpolation \cite[Theorem~4.6.1]{BerghLofstrom1976}, applied to Steps 2 and 3, yield the identities \eqref{eq:CompInterpHomSobSpelip1} and \eqref{eq:CompInterpHomSobSpelip2} for $s_j\in(-n/p_j',n/p_j)$.
\end{proof}

\subsection{Homogeneous Besov spaces on special Lipschitz domains}\label{sec:HomBesovSpacesSpeLip}

We want to carry over the behavior of extension and projection operators up to the scale of Besov spaces. We start with the following basic lemma, directly inherited from Lemma \ref{lem:structLemma}, that gives information about the elements that constitute homogeneous Besov spaces defined by restriction.

\begin{lemma}\label{lem:structLemBesovSpeLip} We set $k_s:=\max(0,\lceil s\rceil)$.
We have the inclusions of sets:
\begin{enumerate}
        \item $\dot{\mathrm{B}}^{s}_{p,q}(\Omega) \subset \mathrm{B}^{s}_{p,q}(\Omega)+\mathrm{C}^{k_s}_0(\overline{\Omega})$, $1\leqslant p <+\infty$, $1\leqslant q\leqslant +\infty$, $s\in\mathbb{R}$.
        \item $\dot{\mathrm{B}}^{s,0}_{\infty,q}(\Omega) \subset \mathrm{B}^{s,0}_{\infty,q}(\Omega)+\mathrm{C}^{k_s}_0(\overline{\Omega})$, $1\leqslant q\leqslant +\infty$, $s\in\mathbb{R}$.
        \item $\dot{\mathrm{B}}^{s}_{\infty,q}(\Omega) \subset {\mathrm{B}}^{s}_{\infty,q}(\Omega)\cap \mathcal{S}'_h(\mathbb{R}^n)_{|_{\Omega}}+\mathrm{C}^{k_s}_{b,h}(\overline{\Omega})$, $1\leqslant q\leqslant +\infty$, $s\in\mathbb{R}$.
\end{enumerate}
In particular, for $s>0$ and $q\in[1,+\infty]$, we have the following equalities of sets
\begin{align*}
    \dot{\mathrm{B}}^{s,0}_{\infty,q}(\Omega) = {\mathrm{B}}^{s,0}_{\infty,q}(\Omega) \text{ and } \dot{\mathrm{B}}^{s}_{\infty,q}(\Omega) = {\mathrm{B}}^{s}_{\infty,q}(\Omega)\cap \mathcal{S}'_h(\mathbb{R}^n)_{|_{\Omega}}.
\end{align*}
\end{lemma}

However, we recall that we do not know whether the extension operator $\mathbb{E}$, introduced in Proposition~\ref{prop:ExtOpIntersecHomHspSpeLip}, maps $\mathrm{C}_{b,h}^0(\overline{\Omega})$ to  $\mathrm{C}_{b,h}^0(\mathbb{R}^n)$. Thus, the treatment of end-point Besov spaces $\dot{\mathrm{B}}^{s}_{\infty,q}$ will be lacunary and quite unpleasant. Things will get easier in the case of the spaces $\dot{\mathrm{B}}^{s,0}_{\infty,q}$ and $\dot{\mathcal{B}}^{s,0}_{\infty,\infty}$.

\begin{proposition}\label{prop:ExtProj0HomBspq}Let $p,q\in [1,+\infty]$, $p<+\infty$, and let $s> -1+\tfrac{1}{p}$. Let us  consider the extension operator $\mathbb{E}$ (resp. $\mathcal{P}_0$) as in Theorem \ref{prop:ExtOpIntersecHomHspSpeLip} (resp.  Proposition \ref{prop:ProOpIntersecHomHsp0SpeLip}).
We assume moreover that
\begin{itemize}
    \item $s > 0$, if $\mathbb{E}=\mathcal{E}$; or
    \item $s < 1$, if $\mathbb{E}=\mathrm{E}$.
\end{itemize}
Then $\mathbb{E}u\in \dot{\mathrm{B}}^{s}_{p,q}(\mathbb{R}^n)$, (resp. $\mathcal{P}_0u\in \dot{\mathrm{B}}^{s}_{p,q,0}(\Omega)$) and we have the estimate
\begin{align}\label{eq:HomEstimBesovExtOp}
    \lVert \mathbb{E}u \rVert_{\dot{\mathrm{B}}^{s}_{p,q}(\mathbb{R}^n)} \lesssim_{s,p,n,\partial\Omega}   \lVert  u \rVert_{\dot{\mathrm{B}}^{s}_{p,q}(\Omega)}\textrm{. (resp. } \lVert \mathcal{P}_0 u \rVert_{\dot{\mathrm{B}}^{s}_{p,q}(\mathbb{R}^n)} \lesssim_{s,p,n,\partial\Omega}   \lVert  u \rVert_{\dot{\mathrm{B}}^{s}_{p,q}(\mathbb{R}^n)}\textrm{.)}
\end{align}
The result also holds with $\dot{\mathcal{B}}^{s}_{p,\infty}$ instead of $\dot{\mathrm{B}}^{s}_{p,\infty}$.
\end{proposition}

\begin{proof} We follow the proof \cite[Corollary~3.14]{Gaudin2022}, but we remove the need for the density argument, as well as the need for the completeness of involved function spaces.

Since $\mathcal{P}_0=\mathrm{I}- \mathbb{E}^{-}[\mathbbm{1}_{\overline{\Omega}^c}\cdot]$, $\mathbb{E}^{-}$ is the extension operator from $\overline{\Omega}^c$ to the whole space $\mathbb{R}^n$. Let $p\in[1,+\infty)$, $q\in[1,+\infty]$, $s>-1+1/p$. We deal with the specific case $s<1$. Let $u\in\dot{\mathrm{B}}^{s}_{p,q}(\Omega)$. We assume $p>1$, by Lemma \ref{lem:EmbeddingInterpHomSobspacesSpeLip}, for $-1+1/p<s_0<s<1$, $s=(1-\theta)s_0+\theta\cdot1$, we have $u\in(\dot{\mathrm{H}}^{s_0,p}(\Omega),\dot{\mathrm{W}}^{1,p}(\Omega))_{\theta,q}\subset \dot{\mathrm{H}}^{s_0,p}(\Omega)+\dot{\mathrm{W}}^{1,p}(\Omega)$. Let $(a,b)\in\dot{\mathrm{H}}^{s_0,p}(\Omega)\times\dot{\mathrm{W}}^{1,p}(\Omega)$ such that $u=a+b$. By Theorem \ref{prop:ExtOpIntersecHomHspSpeLip}, we can write
\begin{align*}
    \mathrm{E}u = \mathrm{E}a+\mathrm{E}b\in\dot{\mathrm{H}}^{s_0,p}(\mathbb{R}^n)+\dot{\mathrm{W}}^{1,p}(\mathbb{R}^n)\in\mathcal{S}'_h(\mathbb{R}^n)
\end{align*}
which implies, for all $t>0$, the following estimates:
\begin{align*}
    K(t,\mathrm{E}u,\dot{\mathrm{H}}^{s_0,p}(\mathbb{R}^n),\dot{\mathrm{W}}^{1,p}(\mathbb{R}^n))\leqslant \lVert \mathrm{E}a \rVert_{\dot{\mathrm{H}}^{s_0,p}(\mathbb{R}^n)} + t \lVert \mathrm{E}b \rVert_{\dot{\mathrm{W}}^{1,p}(\mathbb{R}^n)}\lesssim_{s_0,p,n}^{\partial\Omega}\lVert a \rVert_{\dot{\mathrm{H}}^{s_0,p}(\Omega)} + t \lVert b \rVert_{\dot{\mathrm{W}}^{1,p}(\Omega)}.
\end{align*}
One takes the infimum on all such pairs $(a,b)$ to deduce the inequality
\begin{align*}
    K(t,\mathbb{E}u,\dot{\mathrm{H}}^{s_0,p}(\mathbb{R}^n),\dot{\mathrm{W}}^{1,p}(\mathbb{R}^n))\lesssim_{s_0,p,n}^{\partial\Omega} K(t,u,\dot{\mathrm{H}}^{s_0,p}(\Omega),\dot{\mathrm{W}}^{1,p}(\Omega)).
\end{align*}
Therefore, multiplying by $t^{-\theta}$, then one takes the $\mathrm{L}^q_{\ast}$-norms with respect to $t$,  by applying Theorem \ref{thm:InterpHomSpacesRn} and Lemma \ref{lem:EmbeddingInterpHomSobspacesSpeLip},
\begin{align*}
    \lVert \mathbb{E}u\rVert_{\dot{\mathrm{B}}^{s}_{p,q}(\mathbb{R}^n)}\sim_{s,s_0,\theta} \lVert \mathbb{E}u\rVert_{(\dot{\mathrm{H}}^{s_0,p}(\mathbb{R}^n),\dot{\mathrm{W}}^{1,p}(\mathbb{R}^n))_{\theta,q}} \lesssim_{s_0,p,n}^{\partial\Omega,\theta} \lVert u\rVert_{(\dot{\mathrm{H}}^{s_0,p}(\Omega),\dot{\mathrm{W}}^{1,p}(\Omega))_{\theta,q}} \lesssim_{s_0,p,n}^{\partial\Omega,\theta} \lVert u\rVert_{\dot{\mathrm{B}}^{s}_{p,q}(\Omega)}.
\end{align*}
When $p=1$, it suffices to replace $\dot{\mathrm{H}}^{s_0,p}$ by $\mathrm{L}^1$ (noting that $p=1$ implies $s>0$). 

For the case of arbitrary $s>0$, $p\in[1,+\infty)$, the proof remains the same, replacing the couple $(\dot{\mathrm{H}}^{s_0,p},\dot{\mathrm{W}}^{1,p})$ by $({\mathrm{L}}^{p},\dot{\mathrm{W}}^{m,p})$, with $0<s<m$, $m\in\mathbb{N}^\ast$.
\end{proof}

\begin{lemma}\label{lem:GlobalchangeCoordBesov} Let $p,q\in[1,+\infty]$, $p<+\infty$, $s\in (-1,1)$  and $\mathcal{T}\in\{T_\phi,T_\phi^{-1}\}$. For all $u\in \dot{\mathrm{B}}^{s}_{p,q}(\mathbb{R}^n)$, we have $\mathcal{T} u\in \dot{\mathrm{B}}^{s}_{p,q}(\mathbb{R}^n)$ with the estimate
\begin{align*}
    \lVert \mathcal{T} u \rVert_{\dot{\mathrm{B}}^{s}_{p,q}(\mathbb{R}^n)}\lesssim_{p,s,n,\partial\Omega} \left\lVert u \right\rVert_{\dot{\mathrm{B}}^{s}_{p,q}(\mathbb{R}^n)}\text{.}
\end{align*}
Moreover, 
\begin{itemize}
    \item the result also holds for $\dot{\mathcal{B}}^{s}_{p,\infty}(\mathbb{R}^n)$, $\dot{\mathrm{B}}^{s,0}_{\infty,q}(\mathbb{R}^n)$, $\dot{\mathcal{B}}^{s,0}_{\infty,\infty}(\mathbb{R}^n)$, $s\in(-1,1)$;
    \item the result also holds for $\dot{\mathrm{B}}^{0}_{\infty,1}(\mathbb{R}^n)$, $\dot{\mathrm{B}}^{s}_{\infty,q}(\mathbb{R}^n)$, $\dot{\mathcal{B}}^{s}_{\infty,q}(\mathbb{R}^n)$, $s\in(-1,0)$;
    \item if $[s\geqslant 0$, $q>1]$ or $[s>0$, $q\geqslant1]$, for $u\in  \dot{\mathrm{B}}^{s}_{\infty,q}(\mathbb{R}^n)$, we only have $\mathcal{T} u\in  {\mathrm{B}}^{s}_{\infty,q}(\mathbb{R}^n)+\mathrm{C}^{0,1}_{b}(\mathbb{R}^n)$ and the estimate holds with respect to the $\dot{\mathrm{B}}^{s}_{\infty,q}(\mathbb{R}^n)$-norm.
\end{itemize}
The result remains true\footnote{and can be stated in a way simpler way} with $\{\mathrm{B},\, \mathrm{\mathcal{B}}\}$ instead of $\{\dot{\mathrm{B}},\, \dot{\mathcal{B}}\}$.
\end{lemma}

\begin{proof} \textbf{Step 1:} The case $p\in[1,+\infty)$, $q\in[1,+\infty]$, $0<s<1$. By Proposition \ref{prop:globalMapChangeHsp}, and by real interpolation between $\mathrm{L}^p(\mathbb{R}^n)$ and $\dot{\mathrm{W}}^{1,p}(\mathbb{R}^n)$, see Theorem \ref{thm:InterpHomSpacesRn}, one obtains that $T_\phi$, and $T_\phi^{-1}$, maps $\dot{\mathrm{B}}^{s}_{p,q}(\mathbb{R}^n)$ into itself. The same goes for the space $\dot{\mathcal{B}}^{s}_{p,\infty}(\mathbb{R}^n)$.

\textbf{Step 2:} The case $p\in(1,+\infty]$, $q\in[1,+\infty]$, $-1<s<0$. Due to the identity $T_\phi^\ast = T_\phi^{-1}$, by duality, see Proposition \ref{prop:DualityBesovRn}, it holds that $T_\phi$, and $T_\phi^{-1}$, maps $\dot{\mathrm{B}}^{s}_{p,q}(\mathbb{R}^n)$ into itself.

\textbf{Step 3:} The case $p=+\infty$, $s\in[0,1)$. First let $u\in \dot{\mathrm{B}}^{1}_{\infty,1}(\mathbb{R}^n) \subset \mathrm{W}^{1,\infty}(\mathbb{R}^n)$, we have $T_\phi u\in \mathrm{W}^{1,\infty}(\mathbb{R}^n)$, and for $j\in\mathbb{Z}$, by Bernstein's inequality \cite[Lemma~2.1]{bookBahouriCheminDanchin},
\begin{align}\label{eq:estimB1infty1ell1inftyLinfty}
    2^{j}\lVert\dot{\Delta}_{j}T_\phi u\rVert_{\mathrm{L}^\infty(\mathbb{R}^n)}\lesssim_{n} \lVert\nabla T_\phi u\rVert_{\mathrm{L}^\infty(\mathbb{R}^n)} \lesssim_{p,n,\partial\Omega} \lVert\nabla u\rVert_{\mathrm{L}^\infty(\mathbb{R}^n)} \lesssim_{p,n,\partial\Omega} \lVert u\rVert_{\dot{\mathrm{B}}^{1}_{\infty,1}(\mathbb{R}^n)}.
\end{align}
Thus, by this and Step 2, one has the boundedness properties
\begin{align*}
    (\dot{\Delta}_{j}T_\phi)_{j\in\mathbb{Z}}\,:\, \dot{\mathrm{B}}^{1}_{\infty,1}(\mathbb{R}^n) &\longrightarrow \ell^{\infty}_{1}(\mathbb{Z},\mathrm{L}^\infty(\mathbb{R}^n)),\\
    \,:\,\dot{\mathrm{B}}^{s_0}_{\infty,1}(\mathbb{R}^n) &\longrightarrow \ell^{\infty}_{s_0}(\mathbb{Z},\mathrm{L}^\infty(\mathbb{R}^n)),
\end{align*}
for any $-1<s_0<0$. By real interpolation, thanks to Theorem \ref{thm:InterpHomSpacesRn}, for all $q\in[1,+\infty]$, all $s\in(-1,1)$,
\begin{align*}
    (\dot{\Delta}_{j}T_\phi)_{j\in\mathbb{Z}}\,:\, \dot{\mathrm{B}}^{s}_{\infty,q}(\mathbb{R}^n) &\longrightarrow  \ell^{q}_{s}(\mathbb{Z},\mathrm{L}^\infty(\mathbb{R}^n))
\end{align*}
is well-defined and bounded the result still holds for $(\dot{\Delta}_{j}T_\phi^{-1})_{j\in\mathbb{Z}}$.

\textbf{Step 4:} $T_\phi$ preserves $\dot{\mathrm{B}}^{s,0}_{\infty,q}(\mathbb{R}^n)$ and  $\dot{\mathcal{B}}^{s,0}_{\infty,\infty}(\mathbb{R}^n)$ for $s\in(-1,1)$. For $u\in\dot{\mathrm{B}}^{s,0}_{\infty,q}(\mathbb{R}^n)$, we have $T_\phi u\in \mathcal{S}'(\mathbb{R}^n)$.

We assume $q<+\infty$. Let $(u_\ell)_{\ell\in\mathbb{N}}\subset \mathrm{C}_c^\infty(\mathbb{R}^n)$ such that in converges to $u$ in $\dot{\mathrm{B}}^{s,0}_{\infty,q}(\mathbb{R}^n)$. One obtains $T_\phi u_\ell \in\mathrm{C}^{0,1}_c(\mathbb{R}^n)$, so as $\ell$ tends to infinity, it holds that $\dot{\Delta}_{j} T_\phi u\in\mathrm{C}^{0}_0(\mathbb{R}^n)$ for all $j\in\mathbb{Z}$. When $s<0$ or $q=1, s\leqslant 0$, this also implies $T_\phi u \in\dot{\mathrm{B}}^{s,0}_{\infty,q}(\mathbb{R}^n)$.

It remains to show $T_\phi u\in \mathcal{S}'_h(\mathbb{R}^n)$ when $s\geqslant 0$. First, if $s>0$, for $u\in\dot{\mathrm{B}}^{s,0}_{\infty,q}(\mathbb{R}^n)$, $\dot{\Delta}_ju\in \mathrm{C}^{0}_0(\mathbb{R}^n)$, so we obtain
\begin{align*}
    u = [\mathrm{I}-\dot{S}_0]u+ \dot{S}_0 u \in {\mathrm{B}}^{s,0}_{\infty,q}(\mathbb{R}^n) + \mathrm{C}^{0}_0(\mathbb{R}^n) \subset \mathrm{C}^{0}_0(\mathbb{R}^n).
\end{align*}
Hence, $T_\phi u\in \mathrm{C}^{0}_0(\mathbb{R}^n) \subset\mathcal{S}'_h(\mathbb{R}^n)$.

Thus for $q\in[1,+\infty)$, $s\in(-1,0)\cup(0,1)$, we have a bounded operator
\begin{align*}
    T_\phi\,:\, \dot{\mathrm{B}}^{s,0}_{\infty,q}(\mathbb{R}^n) &\longrightarrow \dot{\mathrm{B}}^{s,0}_{\infty,q}(\mathbb{R}^n).
\end{align*}

By real interpolation, $T_\phi$ maps boundedly $\dot{\mathcal{B}}^{s,0}_{\infty,\infty}(\mathbb{R}^n)$ into itself, as well as  $\dot{\mathrm{B}}^{s,0}_{\infty,q}(\mathbb{R}^n)$, for all $q\in[1,+\infty]$, $s\in(-1,1)$.

Everything still holds for $T_\phi^{-1}$.

\textbf{Step 5:} The case $p\in[1,+\infty)$, $q\in[1,+\infty]$, $s=0$. By duality, see Proposition \ref{prop:DualityBesovRn}, one obtains that $T_\phi$ maps $\dot{\mathrm{B}}^{s}_{1,q}(\mathbb{R}^n)$ into itself, for all $s\in(-1,1)$, $q\in[1,+\infty]$. Thus, with Steps 1 and 2, we have obtained a bounded operator
\begin{align*}
    T_\phi\,:\, \dot{\mathrm{B}}^{s}_{p,q}(\mathbb{R}^n) &\longrightarrow  \dot{\mathrm{B}}^{s}_{p,q}(\mathbb{R}^n)\text{, } p\in[1,+\infty)\text{, }q\in[1,+\infty]\text{, } s\in(-1,0)\cup(0,1),
\end{align*}
and similarly for $T_\phi^{-1}$.

For those remaining cases, $s=0$ or $\dot{\mathcal{B}}^{s}_{p,\infty}$, $p\in[1,+\infty)$, $s\in(-1,1)$, the result follows from real interpolation, see Theorem \ref{thm:InterpHomSpacesRn}.
\end{proof}

Below, we deal separately with the case $p=+\infty$. Before that, we need the fundamental multiplier result in the case of Besov spaces.

\begin{proposition}\label{prop:HomBesovMultiplierSpeLip}For all $p,q\in[1,+\infty]$, for all $s\in (-1+\frac{1}{p},\frac{1}{p})$, for all $u\in\dot{\mathrm{B}}^{s}_{p,q}(\mathbb{R}^n)$,
\begin{align*}
    \lVert \mathbbm{1}_{\Omega} u \rVert_{\dot{\mathrm{B}}^{s}_{p,q}(\mathbb{R}^{n})} \lesssim_{s,p,n,\partial\Omega} \lVert u \rVert_{\dot{\mathrm{B}}^{s}_{p,q}(\mathbb{R}^{n})}
\end{align*}
Moreover, for ${\dot{\mathfrak{B}}}\in\{\dot{\mathcal{B}}^{s}_{p,\infty},\dot{\mathrm{B}}^{s,0}_{\infty,q},\dot{\mathcal{B}}^{s,0}_{\infty,\infty}\}$, we do have $\mathbbm{1}_{\Omega} {\dot{\mathfrak{B}}}(\mathbb{R}^n)\subset {\dot{\mathfrak{B}}}(\mathbb{R}^n)$.
\end{proposition}

\begin{proof} It suffices to write $\mathbbm{1}_{\Omega}=T_\phi^{-1}\mathbbm{1}_{\mathbb{R}^n_+}T_\phi$, and to apply Proposition \ref{prop:SobolevMultiplier} and Lemma \ref{lem:GlobalchangeCoordBesov}. The fact that $\dot{\mathcal{B}}^{s}_{p,\infty}(\mathbb{R}^n)$ is preserved through the multiplication by $\mathbbm{1}_{\Omega}$ follows by real interpolation.

$\bullet$ It preserves $\dot{\mathrm{B}}^{s,0}_{\infty,q}(\mathbb{R}^n)$, $q\in[1,+\infty)$. Let $u\in\dot{\mathrm{B}}^{s,0}_{\infty,q}(\mathbb{R}^n)$, it is sufficient to prove $\dot{\Delta}_j[\mathbbm{1}_{\Omega}u]\in\mathrm{C}^0_0(\mathbb{R}^n)$, for $j\in\mathbb{Z}$. Since $q<+\infty$, there exists $(u_\ell)_{\ell\in\mathbb{N}}\subset\mathrm{C}^\infty_c(\mathbb{R}^n)$ converging to $u$ in $\dot{\mathrm{B}}^{s}_{\infty,q}(\mathbb{R}^n)$. The multiplication by $\mathbbm{1}_{\Omega}$ is bounded on $\dot{\mathrm{B}}^{s}_{\infty,q}(\mathbb{R}^n)$, and one has
\begin{align*}
    2^{js}\lVert\dot{\Delta}_j[\mathbbm{1}_{\Omega}u]- \dot{\Delta}_j[\mathbbm{1}_{\Omega}u_\ell]\rVert_{\mathrm{L}^\infty(\mathbb{R}^n)}\leqslant \lVert \mathbbm{1}_{\Omega}[u-u_\ell] \rVert_{\dot{\mathrm{B}}^{s}_{\infty,q}(\mathbb{R}^n)} \lesssim_{s,\partial\Omega}\lVert u-u_\ell\rVert_{\dot{\mathrm{B}}^{s}_{\infty,q}(\mathbb{R}^n)}\xrightarrow[\ell\rightarrow+\infty]{} 0.
\end{align*}
Due to the compact support of $\mathbbm{1}_{\Omega}u_\ell$, we have $\dot{\Delta}_j[\mathbbm{1}_{\Omega}u_\ell]\in\mathrm{C}^0_0(\mathbb{R}^n)$. Since $\mathrm{C}^0_0(\mathbb{R}^n)$  is closed in $\mathrm{L}^\infty(\mathbb{R}^n)$, it gives $\dot{\Delta}_j[\mathbbm{1}_{\Omega}u]\in\mathrm{C}^0_0(\mathbb{R}^n)$.

$\bullet$ By real interpolation from the case $\dot{\mathrm{B}}^{s,0}_{\infty,q}(\mathbb{R}^n)$, $1\leqslant q <+\infty$, see Theorem \ref{thm:InterpHomSpacesRn}, we obtain that the multiplication by $\mathbbm{1}_{\Omega}$ preserves respectively $\dot{\mathrm{B}}^{s,0}_{\infty,\infty}(\mathbb{R}^n)$ and $\dot{\mathcal{B}}^{s,0}_{\infty,\infty}(\mathbb{R}^n)$.
\end{proof}

\begin{proposition}\label{prop:ExtProj0HomBspinfty} Let $q\in[1,+\infty]$, $s\in(-1,+\infty)$. Let us  consider the extension operator $\mathbb{E}$ (resp. $\mathcal{P}_0$) as in Proposition \ref{prop:ExtProj0HomBspq}. For all $u\in\dot{\mathrm{B}}^{s}_{\infty,q}(\Omega)$ (resp. $\dot{\mathrm{B}}^{s}_{\infty,q}(\mathbb{R}^n)$), we have the homogeneous estimate
\begin{align}\label{eq:HomEstimBesovExtOppinfty}
    \lVert \mathbb{E}u \rVert_{\dot{\mathrm{B}}^{s}_{\infty,q}(\mathbb{R}^n)} \lesssim_{s,n,\partial\Omega}   \lVert  u \rVert_{\dot{\mathrm{B}}^{s}_{\infty,q}(\Omega)}\textrm{. (resp. } \lVert \mathcal{P}_0 u \rVert_{\dot{\mathrm{B}}^{s}_{\infty,q}(\mathbb{R}^n)} \lesssim_{s,n,\partial\Omega}   \lVert  u \rVert_{\dot{\mathrm{B}}^{s}_{\infty,q}(\mathbb{R}^n)}\textrm{.)}
\end{align}
More precisely,
\begin{itemize}
    \item if $s>0$, $\mathbb{E}u\in{\mathrm{B}}^{s}_{\infty,q}(\mathbb{R}^n)$ and $\mathcal{P}_0u\in{\mathrm{B}}^{s}_{\infty,q,0}(\Omega)$;
    \item if $s=0$, $q>1$, $\mathbb{E}u\in{\mathrm{B}}^{0}_{\infty,q}(\mathbb{R}^n)+\mathrm{C}_b^0(\mathbb{R}^n)$, and $\mathcal{P}_0u\in{\mathrm{B}}^{s}_{\infty,q,0}(\Omega)+\mathrm{C}_b^0(\mathbb{R}^n)$;
    \item if $(\mathcal{C}_{s,\infty,q})$, one has $\mathbb{E}u\in \dot{\mathrm{B}}^{s}_{\infty,q}(\mathbb{R}^n)$ and $\mathcal{P}_0 u\in \dot{\mathrm{B}}^{s}_{\infty,q,0}(\Omega)$;
    \item for all $s\in(-1,+\infty)$, one has $\mathbb{E}(\dot{\mathrm{B}}^{s,0}_{\infty,q}(\Omega))\subset\dot{\mathrm{B}}^{s,0}_{\infty,q}(\mathbb{R}^n)$ and $\mathcal{P}_0(\dot{\mathrm{B}}^{s,0}_{\infty,q}(\mathbb{R}^n))\subset \dot{\mathrm{B}}^{s,0}_{\infty,q,0}(\Omega)$. Similarly, with $\dot{\mathcal{B}}^{s,0}_{\infty,\infty}$ instead of $\dot{\mathrm{B}}^{s,0}_{\infty,q}$.
\end{itemize}
\end{proposition}

\begin{remark}\label{rmk:RemarkNonOptimalitypequalinfty} We point out several phenomenons concerning the statements in Lemma \ref{lem:GlobalchangeCoordBesov} and Proposition \ref{prop:ExtProj0HomBspinfty}.
\begin{itemize}
    \item The main issue in those results is about to know if $T_\phi u, \mathbb{E}u, \mathcal{P}_0u\in\mathcal{S}'_h(\mathbb{R}^n)$ or not. This is only known when $(\mathcal{C}_{s,\infty,q})$ is satisfied, or when the corresponding spaces on $\mathbb{R}^n$ are the strong closures of $\mathcal{S}_0(\mathbb{R}^n)$.
    \item When $s>0$, one has the norm inequalities, for all $w\in\mathcal{S}'(\mathbb{R}^n)$,
    \begin{align*}
        \lVert w\rVert_{{\mathrm{B}}^{-s}_{\infty,q}(\mathbb{R}^n)}\lesssim_{s,n,q}\lVert w\rVert_{\dot{\mathrm{B}}^{-s}_{\infty,q}(\mathbb{R}^n)} \text{, and }\lVert w\rVert_{\dot{\mathrm{B}}^{s}_{\infty,q}(\mathbb{R}^n)}\lesssim_{s,n,q}\lVert w\rVert_{{\mathrm{B}}^{s}_{\infty,q}(\mathbb{R}^n)}.
    \end{align*}
    So estimates like \eqref{eq:HomEstimBesovExtOppinfty} are still a refinement of the statement $T_\phi u, \mathbb{E}u, \mathcal{P}_0 u\in {\mathrm{B}}^{s}_{\infty,q}(\mathbb{R}^n)$ in all cases.
\end{itemize}
\end{remark}

\begin{proof}[of Proposition \ref{prop:ExtProj0HomBspinfty}]\textbf{Step 1:} The case $s\in(-1,1)$, $\mathbb{E}=\mathrm{E}$ the operator by rough reflection around the boundary.

\textbf{Step 1.1:} The boundedness of $\mathrm{E}$. We reintroduce the operator $\tilde{\mathrm{E}}$ as in \eqref{eq:FlatExtensionOpRn+}.

We write $\tilde{\mathrm{E}}=\tilde{\mathrm{E}}[\mathbbm{1}_{\mathbb{R}^n_+} \cdot]= \mathbbm{1}_{\mathbb{R}^n_+} + \sigma(\mathbbm{1}_{\mathbb{R}^n_+}[\cdot])$, where $\sigma(f)(x',x_n):=f(x',-x_n)$ for all measurable functions $f\,:\,\mathbb{R}^n\longrightarrow \mathbb{C}$, all $(x',x_n)\in\mathbb{R}^{n-1}\times\mathbb{R}$.

Thus, the dual operator of $\tilde{\mathrm{E}}$ is $\tilde{\mathrm{E}}^\ast=\mathbbm{1}_{\mathbb{R}^n_+}[\mathrm{I} -\sigma]$. By Proposition \ref{prop:SobolevMultiplier}, $\tilde{\mathrm{E}}^\ast$ is well-defined and bounded on $\dot{\mathrm{B}}^{s}_{1,q}(\mathbb{R}^n)$, $q\in[1,+\infty]$, $0<s<1$. By duality, see Proposition \ref{prop:DualityBesovRn}, and real interpolation, we have a bounded operator
\begin{align*}
    \tilde{\mathrm{E}}[\mathbbm{1}_{\mathbb{R}^n_+}\cdot]\,:\,\dot{\mathrm{B}}^{s}_{\infty,q}(\mathbb{R}^n)\longrightarrow\dot{\mathrm{B}}^{s}_{\infty,q}(\mathbb{R}^n)\text{, }-1<s<0\text{, }q\in[1,+\infty].
\end{align*}

Lemma \ref{lem:GlobalchangeCoordBesov} implies that ${\mathrm{E}}[\mathbbm{1}_\Omega\cdot]=T_\phi^{-1}\tilde{\mathrm{E}}T_\phi\mathbbm{1}_\Omega$ is bounded on  $\dot{\mathrm{B}}^{s}_{\infty,q}(\mathbb{R}^n)$, $s\in(-1,0)$, $q\in[1,+\infty]$.

Now let $U\in\dot{\mathrm{B}}^{1}_{\infty,1}(\mathbb{R}^n)\subset\mathrm{W}^{1,\infty}(\mathbb{R}^n)$, we have ${\mathrm{E}}[\mathbbm{1}_\Omega U]\in\mathrm{W}^{1,\infty}(\mathbb{R}^n)$ accompanied by the estimates, obtained as in \eqref{eq:estimB1infty1ell1inftyLinfty},
\begin{align*}
   \lVert (\dot{\Delta}_j {\mathrm{E}}[\mathbbm{1}_\Omega U])_{j\in\mathbb{Z}}\rVert_{\ell_1^\infty(\mathbb{Z},\mathrm{L}^\infty(\mathbb{R}^n))}\lesssim_{n,\partial\Omega} \lVert U \rVert_{\dot{\mathrm{B}}^{1}_{\infty,1}(\mathbb{R}^n)}.
\end{align*}
The operator $(\dot{\Delta}_j {\mathrm{E}}[\mathbbm{1}_\Omega U])_{j\in\mathbb{Z}}$ is bounded from $\dot{\mathrm{B}}^{1}_{\infty,1}(\mathbb{R}^n)$ to $\ell_1^\infty(\mathbb{Z},\mathrm{L}^\infty(\mathbb{R}^n))$, and by the same trick, from $\dot{\mathrm{B}}^{s_0}_{\infty,1}(\mathbb{R}^n)$ to $\ell_{s_0}^\infty(\mathbb{Z},\mathrm{L}^\infty(\mathbb{R}^n))$, where $s_0\in(-1,0)$. By real interpolation, for $s\in(-1,1)$, $q\in[1,+\infty]$, all $U\in\dot{\mathrm{B}}^{s}_{\infty,q}(\mathbb{R}^n)$, we obtain the estimate
\begin{align*}
   \lVert {\mathrm{E}}[\mathbbm{1}_\Omega U] \rVert_{\dot{\mathrm{B}}^{s}_{\infty,q}(\mathbb{R}^n)}\lesssim_{n,\partial\Omega} \lVert U \rVert_{\dot{\mathrm{B}}^{s}_{\infty,q}(\mathbb{R}^n)}.
\end{align*}

\textbf{Step 1.2:} The lifted extension operator $\mathbb{E}=\mathrm{E}$ preserves $\dot{\mathrm{B}}^{s,0}_{\infty,q}(\mathbb{R}^n)$ and $\dot{\mathcal{B}}^{s,0}_{\infty,\infty}(\mathbb{R}^n)$. We assume $U\in\dot{\mathrm{B}}^{s,0}_{\infty,q}(\mathbb{R}^n)$, where $q\in[1,+\infty)$, for $s\neq0$. For $(U_\ell)_{\ell\in\mathbb{N}}\subset\mathrm{C}^\infty_c(\mathbb{R}^n)$ converging to $U$ in $\dot{\mathrm{B}}^{s}_{\infty,q}(\mathbb{R}^n)$,
\begin{align*}
    2^{js}\lVert\dot{\Delta}_j{\mathrm{E}}[\mathbbm{1}_\Omega U]- \dot{\Delta}_j{\mathrm{E}}[\mathbbm{1}_\Omega U_\ell]\rVert_{\mathrm{L}^\infty(\mathbb{R}^n)}\lesssim_{s,\partial\Omega} \lVert U-U_\ell \rVert_{\dot{\mathrm{B}}^{s}_{\infty,q}(\mathbb{R}^n)} \xrightarrow[\ell\rightarrow+\infty]{} 0.
\end{align*}
Due to the compact support of ${\mathrm{E}}[\mathbbm{1}_\Omega U_\ell]$, we have $\dot{\Delta}_j{\mathrm{E}}[\mathbbm{1}_\Omega U_\ell]\in\mathrm{C}^0_0(\mathbb{R}^n)$, and therefore $\dot{\Delta}_j{\mathrm{E}}[\mathbbm{1}_\Omega U]$ is an element of $\mathrm{C}^0_0(\mathbb{R}^n)$.

If $s<0$, then we automatically have ${\mathrm{E}}[\mathbbm{1}_\Omega U]\in\eus{S}'_h(\mathbb{R}^n)$. If $s>0$,
\begin{align*}
    {\mathrm{E}}[\mathbbm{1}_\Omega U]\in {\mathrm{E}}[\mathbbm{1}_\Omega {\mathrm{B}}^{s,0}_{\infty,q}(\mathbb{R}^n)] + {\mathrm{E}}[\mathbbm{1}_\Omega {\mathrm{C}}^{0}_{0}(\mathbb{R}^n)] \subset{\mathrm{E}}[\mathbbm{1}_\Omega {\mathrm{C}}^{0}_{0}(\mathbb{R}^n)] \subset  {\mathrm{C}}^{0}_{0}(\mathbb{R}^n)\subset\eus{S}'_h(\mathbb{R}^n).
\end{align*}
So, for $s\in(-1,0)\cup(0,1)$, $q\in[1,+\infty)$, for all $ U\in \dot{\mathrm{B}}^{s,0}_{\infty,q}(\mathbb{R}^n)$, we have ${\mathrm{E}}[\mathbbm{1}_\Omega U]\in \dot{\mathrm{B}}^{s,0}_{\infty,q}(\mathbb{R}^n)$. By real interpolation, it remains true when $s=0$, and for the spaces $\dot{\mathcal{B}}^{s,0}_{\infty,\infty}(\mathbb{R}^n)$, $s\in(-1,1)$.

\textbf{Step 1.3:} The boundedness of $\mathrm{E}$ as an extension operator. Let $u\in\dot{\mathrm{B}}^{s}_{\infty,q}(\Omega)$, then for all $U\in\dot{\mathrm{B}}^{s}_{\infty,q}(\mathbb{R}^n)$ such that $U_{|_\Omega}=u$, we have $\mathrm{E}u={\mathrm{E}}[\mathbbm{1}_\Omega U]$ and the estimates
\begin{align*}
    \lVert {\mathrm{E}}u \rVert_{\dot{\mathrm{B}}^{s}_{\infty,q}(\mathbb{R}^n)}=\lVert {\mathrm{E}}[\mathbbm{1}_\Omega U] \rVert_{\dot{\mathrm{B}}^{s}_{\infty,q}(\mathbb{R}^n)}\lesssim_{n,\partial\Omega} \lVert U \rVert_{\dot{\mathrm{B}}^{s}_{\infty,q}(\mathbb{R}^n)}.
\end{align*}
We take the infimum on all such $U$, which yields
\begin{align*}
    \lVert {\mathrm{E}}u \rVert_{\dot{\mathrm{B}}^{s}_{\infty,q}(\mathbb{R}^n)}\lesssim_{n,\partial\Omega} \lVert u \rVert_{\dot{\mathrm{B}}^{s}_{\infty,q}(\Omega)}.
\end{align*}

This finishes Step 1, since this procedure remains valid for the spaces $\dot{\mathrm{B}}^{s,0}_{\infty,q}$ and $\dot{\mathcal{B}}^{s,0}_{\infty,\infty}$, thanks to Step 1.2. In particular, for $u\in\dot{\mathrm{B}}^{s,0}_{\infty,q}(\Omega)$, $\mathrm{E}u$ is both $\dot{\mathrm{B}}^{s,0}_{\infty,q}$ and $\dot{\mathrm{B}}^{s}_{\infty,q}$-extension of $u$. Thus,
\begin{align*}
    \lVert u\rVert_{\dot{\mathrm{B}}^{s,0}_{\infty,q}(\Omega)}\sim_{s,n,\partial\Omega}\lVert u\rVert_{\dot{\mathrm{B}}^{s}_{\infty,q}(\Omega)}.
\end{align*}
The argument remains valid for the spaces $\dot{\mathrm{B}}^{s}_{\infty,\infty}$, $\dot{\mathrm{B}}^{s,0}_{\infty,\infty}$, $\dot{\mathcal{B}}^{s}_{\infty,\infty}$ and $\dot{\mathcal{B}}^{s,0}_{\infty,\infty}$.

\textbf{Step 2:} The case $s>0$, with Stein's extension operator $\mathbb{E}=\mathcal{E}$.

\textbf{Step 2.1:} The spaces $\dot{\mathrm{B}}^{s,0}_{\infty,q}$ and $\dot{\mathcal{B}}^{s,0}_{\infty,\infty}$. We interpolate the boundedness of
\begin{align*}
    \mathcal{E}\,:\,\dot{\mathrm{C}}^k_0(\overline{\Omega})\longrightarrow \dot{\mathrm{C}}^k_0(\mathbb{R}^n)\text{, } k\in\mathbb{N}\text{, }
\end{align*}
provided by Theorem \ref{thm:SteinsExtensionOp}. By Theorem \ref{thm:InterpHomSpacesRn} and Lemma \ref{lem:EmbeddingInterpHomSobspacesSpeLip}, for all  $u\in\dot{\mathrm{B}}^{s,0}_{\infty,q}(\Omega)\hookrightarrow
({\mathrm{C}}^0_0(\overline{\Omega}),\dot{\mathrm{C}}^k_0(\overline{\Omega}))_{\theta,q}$, where $\theta\in(0,1)$, $k\in\mathbb{N}^\ast$, $s=\theta k$, $q\in[1,+\infty]$, we have
\begin{align*}
    \lVert \mathcal{E}u\rVert_{\dot{\mathrm{B}}^{s}_{\infty,q}(\mathbb{R}^n)}\lesssim_{p,k,n,\theta} \lVert \mathcal{E}u \rVert_{({\mathrm{C}}^0_0(\mathbb{R}^n),\dot{\mathrm{C}}^k_0(\mathbb{R}^n))_{\theta,q}}\lesssim_{p,n,k}^{\partial\Omega,\theta} \lVert u \rVert_{({\mathrm{C}}^0_0(\overline{\Omega}),\dot{\mathrm{C}}^k_0(\overline{\Omega}))_{\theta,q}}\lesssim_{p,n,k}^{\partial\Omega,\theta} \lVert u\rVert_{\dot{\mathrm{B}}^{s,0}_{\infty,q}(\Omega)}.
\end{align*}

The same proofs apply for the spaces $\dot{\mathcal{B}}^{s,0}_{\infty,\infty}$, $s>0$. We can conclude as in Step 1.3.

\textbf{Step 2.2:} The spaces $\dot{\mathrm{B}}^{s}_{\infty,q}$. For all $u\in\dot{\mathrm{B}}^{s}_{\infty,q}(\Omega)\subset {\mathrm{B}}^{s}_{\infty,q}(\Omega)$, thus $\mathcal{E}u\in {\mathrm{B}}^{s}_{\infty,q}(\mathbb{R}^n)$. We set $q=1$, $s=k\in\mathbb{N}$. As in Step 1.1, and the definition of function spaces by restriction, the linear operator
\begin{align*}
    (\dot{\Delta}_{j}\mathcal{E})_{j\in\mathbb{Z}}\,:\, \dot{\mathrm{B}}^{k}_{\infty,1}(\Omega) &\longrightarrow \ell^{\infty}_{k}(\mathbb{Z},\mathrm{L}^\infty(\mathbb{R}^n))\text{, }k\in\mathbb{N}.
\end{align*}
The boundedness with respect to the $\dot{\mathrm{B}}^{s}_{\infty,q}$-norms follows again from real interpolation.

To conclude properly the Step 2, we claim $\mathcal{E}$ also preserves $\mathrm{B}^{s,0}_{\infty,q}$-spaces as in Step 1.3. This yields for all $s>0$, all $u\in\mathrm{B}^{s,0}_{\infty,q}(\Omega)$,
\begin{align*}
    \lVert u\rVert_{\dot{\mathrm{B}}^{s,0}_{\infty,q}(\Omega)}\sim_{s,n,\partial\Omega}\lVert u\rVert_{\dot{\mathrm{B}}^{s}_{\infty,q}(\Omega)}.
\end{align*}

\textbf{Step 3:} Finally, the properties for $\mathcal{P}_0$ are carried over by the formula $\mathcal{P}_0=\mathrm{I}-\mathbb{E}^{-}[\mathbbm{1}_{\overline{\Omega}^c}\cdot]$.
\end{proof}

This allows to improve the boundedness properties of the extension operator by rough reflection and its subordinated projection operator by duality.

\begin{lemma}\label{lem:RoughExtensionProjBesov} Let $p,q\in[1,+\infty]$, $p<+\infty$, $s\in(-1,1)$. We consider the extension operator $\mathrm{E}$ and its subordinated projection $\mathcal{P}_0$ as in Proposition \ref{prop:ExtProj0HomBspq}. Each operator maps boundedly
\begin{align*}
    \mathrm{E}\,:\,\dot{\mathrm{B}}^{s}_{p,q}(\Omega)\longrightarrow \dot{\mathrm{B}}^{s}_{p,q}(\mathbb{R}^n),\text{ and }
    \mathcal{P}_0\,:\,\dot{\mathrm{B}}^{s}_{p,q}(\mathbb{R}^n)\longrightarrow \dot{\mathrm{B}}^{s}_{p,q,0}(\Omega).
\end{align*}
Moreover, when $q=+\infty$, the result still holds with  $\mathcal{B}$ instead of $\mathrm{B}$.
\end{lemma}

\begin{proof} If $s>-1+1/p$, the result is already contained in Proposition \ref{prop:ExtProj0HomBspq}. Therefore, we assume $s\in(-1,0)$, and let $q<+\infty$.

We let $u\in\dot{\mathrm{B}}^{s}_{p,q}(\Omega)$, and consider $U\in \dot{\mathrm{B}}^{s}_{p,q}(\mathbb{R}^n)$ such that $u = U_{|_{\Omega}}$. With the notation introduced in the Step 1.1 of the proof of Proposition \ref{prop:ExtProj0HomBspinfty}, since $\mathrm{E}=T_{\phi}^{-1}\Tilde{\mathrm{E}}T_{\phi}$, one has for $\varphi\in\mathcal{S}(\mathbb{R}^n)$,
\begin{align*}
    \big\langle \mathrm{E}u, \varphi\big\rangle_{\mathbb{R}^n} &=\big\langle U,  \left[T_{\phi}^{-1}\mathbbm{1}_{\mathbb{R}^n_+}[\mathrm{I}-\sigma]T_{\phi}\right]\varphi\big\rangle_{\mathbb{R}^n}\\
    &= \big\langle U,  \mathcal{P}_0\varphi\big\rangle_{\mathbb{R}^n}.
\end{align*}
Hence, by Proposition \ref{prop:ExtProj0HomBspq} (and \ref{prop:ExtProj0HomBspinfty} if $p=1$, noticing that $\mathcal{S}(\mathbb{R}^n)\subset\dot{\mathrm{B}}^{-s,0}_{\infty,q'}(\mathbb{R}^n)$), we deduce
\begin{align*}
    \big\lvert \big\langle \mathrm{E}u, \varphi\big\rangle_{\mathbb{R}^n} \big\rvert \lesssim_{p,s,n}^{\partial\Omega} \lVert U \rVert_{\dot{\mathrm{B}}^{s}_{p,q}(\mathbb{R}^n)}\lVert \varphi \rVert_{\dot{\mathrm{B}}^{-s}_{p',q'}(\mathbb{R}^n)}
\end{align*}
Now, by Proposition \ref{prop:DualityBesovRn} and the definition of function spaces by restriction, we obtain
\begin{align*}
    \lVert \mathrm{E}u \rVert_{\dot{\mathrm{B}}^{s}_{p,q}(\mathbb{R}^n)}\lesssim_{p,s,n}^{\partial\Omega}\lVert u \rVert_{\dot{\mathrm{B}}^{s}_{p,q}(\Omega)}.
\end{align*}
The remaining cases $(s,p)=(0,1)$ and $q=+\infty$ are deduced as in the proof of Proposition \ref{prop:ExtProj0HomBspq}, by a manual real interpolation procedure thanks to Lemma \ref{lem:EmbeddingInterpHomSobspacesSpeLip}.
\end{proof}

The previous results, such as Lemma \ref{lem:GlobalchangeCoordBesov} or Proposition \ref{prop:ExtProj0HomBspinfty}, are not really satisfying with regard to the spaces $\dot{\mathrm{B}}^{s}_{\infty,q}$. Consequently, the focus should primarily be on the spaces $\dot{\mathrm{B}}^{s,0}_{\infty,q}$ in the remainder of this paper. When it comes to the behavior of intersection of homogeneous Besov spaces, things get better.

\begin{proposition}\label{prop:IntersecExtHomBspqProjHomBspq0}Let $p_j,q_j\in[1,+\infty]$, $s_j>-1$, $j\in\{0,1\}$, and consider the extension operator $\mathbb{E}$ given by Proposition \ref{prop:ExtProj0HomBspq}. We assume either $s_0,s_1>0$, or $s_0,s_1<1$, and that,
\begin{enumerate}
    \item $(p_0,p_1)\neq(+\infty,+\infty)$; or
    \item $(\mathcal{C}_{s_0,p_0,q_0})$ is satisfied.
\end{enumerate}
It holds that
\begin{align*}
    [\dot{\mathrm{B}}^{s_0}_{p_0,q_0}\cap \dot{\mathrm{B}}^{s_1}_{p_1,q_1}](\Omega)=\dot{\mathrm{B}}^{s_0}_{p_0,q_0}(\Omega)\cap \dot{\mathrm{B}}^{s_1}_{p_1,q_1}(\Omega),
\end{align*}
and for all $u\in \dot{\mathrm{B}}^{s_0}_{p_0,q_0}(\Omega)\cap \dot{\mathrm{B}}^{s_1}_{p_1,q_1}(\Omega)$, we have $\mathbb{E}u\in \dot{\mathrm{B}}^{s_j}_{p_j,q_j}(\mathbb{R}^n)$, $j\in\{0,1\}$, with the estimate
\begin{align*}
    \lVert \mathbb{E} u \rVert_{\dot{\mathrm{B}}^{s_j}_{p_j,q_j}(\mathbb{R}^n)} \lesssim_{s_j,p_j,n}^{\partial\Omega}   \lVert  u \rVert_{\dot{\mathrm{B}}^{s_j}_{p_j,q_j}(\Omega)}\text{. }
\end{align*}
Moreover,
\begin{itemize}
    \item when $p_j=+\infty$, one can replace $\dot{\mathrm{B}}^{s_j}_{\infty,q_j}$ by $\dot{\mathrm{B}}^{s_j,0}_{\infty,q_j}$ and remove the assumptions \textit{(i)} and \textit{(ii)};
    \item when $q_j=+\infty$, one can replace $\dot{\mathrm{B}}^{s_j}_{p_j,\infty}$ by $\dot{\mathcal{B}}^{s_j}_{p_j,\infty}$;
    \item when $p_j=q_j=+\infty$, one can replace $\dot{\mathrm{B}}^{s_j}_{\infty,\infty}$ by $\dot{\mathrm{B}}^{s_j,0}_{\infty,\infty}$ and remove the assumptions \textit{(i)} and \textit{(ii)}.
\end{itemize}
\end{proposition}

\begin{remark}We have the following similar result: for all $p,q\in[1,+\infty]$, $p<+\infty$, $s>0$, one has with equivalence of norms
\begin{align*}
    \mathrm{L}^p(\Omega)\cap\dot{\mathrm{B}}^{s}_{p,q}(\Omega)={\mathrm{B}}^{s}_{p,q}(\Omega),\text{ and }
    \mathrm{C}^0_0(\overline{\Omega})\cap\dot{\mathrm{B}}^{s,0}_{\infty,q}(\Omega)={\mathrm{B}}^{s,0}_{\infty,q}(\Omega).
\end{align*}
(The equality of sets ${\mathrm{B}}^{s,0}_{\infty,q}(\Omega)=\dot{\mathrm{B}}^{s,0}_{\infty,q}(\Omega)$ is already known.)
\end{remark}

\begin{proof} One just has to reproduce the proof of Proposition \ref{prop:ExtOpIntersecHomHspSpeLip} and Proposition \ref{prop:IntersecRestricHomSobSpaces}. This is possible thanks to Propositions \ref{prop:ExtProj0HomBspq} and \ref{prop:ExtProj0HomBspinfty}.

We just make a comment on  the case $p_0<+\infty$, $p_1=+\infty$, which also applies to the case $p_0=p_1=+\infty$ under the assumption $(\mathcal{C}_{s_0,\infty,q_0})$. 

If $u\in\dot{\mathrm{B}}^{s_0}_{p_0,q_0}(\Omega)\cap\dot{\mathrm{B}}^{s_1}_{p_1,q_1}(\Omega)$, The biggest issue is about to know if $\mathbb{E}u\in \dot{\mathrm{B}}^{s_1}_{p_1,q_1}(\mathbb{R}^n)$, which reduces to show $\mathbb{E}u\in\mathcal{S}'_h(\mathbb{R}^n)$. Indeed, by Proposition \ref{prop:ExtProj0HomBspinfty}, it turns out that
\begin{align*}
    (\dot{\Delta}_j\mathbb{E}u)_{j\in\mathbb{Z}} \in \ell^{q_1}_{s_1}(\mathbb{Z},\mathrm{L}^{p_1}(\mathbb{R}^n)).
\end{align*}
Since we do have $\mathbb{E}u\in\dot{\mathrm{B}}^{s_0}_{p_0,q_0}(\mathbb{R}^n)\subset\mathcal{S}'_h(\mathbb{R}^n)$, this yields
\begin{align*}
    \mathbb{E}u\in \dot{\mathrm{B}}^{s_1}_{p_1,q_1}(\mathbb{R}^n).
\end{align*}
Therefore, the result still holds when  $p_0<+\infty$ and $p_1=+\infty$, or when $p_0=p_1=+\infty$ under the assumption $(\mathcal{C}_{s_0,\infty,q_0})$.

In all the remaining cases, one always has $\mathbb{E}u\in\mathcal{S}'_h(\mathbb{R}^n)$. This ends the proof.
\end{proof}

\begin{proposition}\label{prop:densityCcinftyBspq0SpeLip}For $1\leqslant p,q<+\infty$, $s>-1$, the space $\mathrm{C}_c^\infty(\Omega)$ is a dense subspace of 
\begin{enumerate}
    \item $\dot{\mathrm{B}}^{s}_{p,q,0}(\Omega)$, $s>-n/p'$;
    \item $\dot{\mathcal{B}}^{s}_{p,\infty,0}(\Omega)$, $s\geqslant -n/p'$;
    \item $\dot{\mathrm{B}}^{s,0}_{\infty,q,0}(\Omega)$, $\dot{\mathcal{B}}^{s,0}_{\infty,\infty,0}(\Omega)$.
\end{enumerate}
\end{proposition}

\begin{remark} Again, the approximation procedure is universal if either $s>0$ or $-1<s<1$.
\end{remark}

\begin{proof} Thanks to the boundedness of the operator $\mathcal{P}_0$ on the scale of Besov spaces provided by Proposition \ref{prop:ExtProj0HomBspq} and Lemma \ref{lem:RoughExtensionProjBesov}, one may reproduce the Steps 1.1 and 1.2 in the proof of Proposition \ref{prop:densityCcinftyinhomHspspeLip}.
\end{proof}

\begin{proposition}\label{prop:Bspq=Bspq0SpeLip} For all $p,q\in[1,+\infty]$, $s\in(-1+\frac{1}{p},\frac{1}{p})$,
\begin{align*}
    \dot{\mathrm{B}}^{s}_{p,q,0}(\Omega)=\dot{\mathrm{B}}^{s}_{p,q}(\Omega)\text{.}
\end{align*}
In particular, if $p,q<+\infty$, the space $\mathrm{C}_c^\infty(\Omega)$ is dense in $\dot{\mathrm{B}}^{s}_{p,q}(\Omega)$ and $\dot{\mathcal{B}}^{s}_{p,\infty}(\Omega)$. The result still holds for $\dot{\mathrm{B}}^{s,0}_{\infty,q}$ and $\dot{\mathcal{B}}^{s,0}_{\infty,\infty}$, $s\in(-1,0)$.
\end{proposition}

\begin{proof} The equality of function spaces is straightforward from Proposition \ref{prop:HomBesovMultiplierSpeLip} and the previous Proposition \ref{prop:densityCcinftyBspq0SpeLip}.
\end{proof}

We recover the standard and well-know equivalence of norms, for which we give an easy proof.

\begin{proposition}\label{prop:EquivNormsBesovSpeLip} Let $1\leqslant p,q\leqslant +\infty$, $p<+\infty$, and $s\in\mathbb{R}$. The following equivalences of norms hold:
\begin{enumerate}
    \item if $m\in\mathbb{N}$ and $s>0$, for all $u\in\dot{\mathrm{B}}^{s}_{p,q}(\Omega)$,
    \begin{align*}
        \lVert u\rVert_{\dot{\mathrm{B}}^{s+m}_{p,q}(\Omega)}\sim_{p,s,n}^{m,\partial\Omega}\lVert \nabla^m u\rVert_{\dot{\mathrm{B}}^{s}_{p,q}(\Omega)};
    \end{align*}
    A similar result holds for $\dot{\mathrm{B}}^{s+m,0}_{\infty,q}(\Omega)$.
    \item if $0<s<1$, for all $u\in\dot{\mathrm{B}}^{s}_{p,p}(\Omega)$,
    \begin{align*}
        \lVert u\rVert_{\dot{\mathrm{B}}^{s}_{p,p}(\Omega)}\sim_{p,s,n}^{\partial\Omega} \left(\int_{\Omega}\int_{\Omega} \frac{| u(x)- u(y)|^p}{|x-y|^{ps+n}} \mathrm{~d}x\mathrm{d}y\right)^\frac{1}{p};
    \end{align*}
    A similar result holds for $u\in\dot{\mathrm{B}}^{s,0}_{\infty,\infty}(\Omega)+\dot{\mathcal{B}}^{s}_{\infty,\infty}(\Omega)\subset\dot{\mathrm{B}}^{s}_{\infty,\infty}(\Omega)$ with
    \begin{align*}
        \lVert u\rVert_{\dot{\mathrm{B}}^{s}_{\infty,\infty}(\Omega)}\sim_{s,n}^{\partial\Omega} \sup_{\substack{x,y\in\Omega\\ x\neq y}}\frac{| u(x)- u(y)|}{|x-y|^{s}}.
    \end{align*}
\end{enumerate}
\end{proposition}

\begin{proof} \textbf{Step 1:} The proof of point \textit{(i)} follows the lines in the proof of Proposition \ref{prop:EqNormNablakHsp}.

\textbf{Step 2:} The proof of \textit{(ii)}. Let $0<s<1$, $1\leqslant p\leqslant+\infty$.

\textbf{Step 2.1:} The case $1\leqslant p<+\infty$.
Let $u\in\dot{\mathrm{B}}^{s}_{p,p}(\Omega)$. By the definition of function spaces by restriction, let ${U}\in\dot{\mathrm{B}}^{s}_{p,p}(\mathbb{R}^n)$ such that $U_{|_{\Omega}}=u$, and $\lVert U \rVert_{\dot{\mathrm{B}}^{s}_{p,p}(\mathbb{R}^n)}\leqslant2\lVert u \rVert_{\dot{\mathrm{B}}^{s}_{p,p}(\Omega)}$. By \cite[Theorem~2.36]{bookBahouriCheminDanchin},
\begin{align*}
    \left(\int_{\Omega}\int_{\Omega} \frac{|u(x)-u(y)|^p}{|x-y|^{ps+n}} \mathrm{~d}x\mathrm{d}y\right)^\frac{1}{p} &\leqslant \left(\int_{\mathbb{R}^n}\int_{\mathbb{R}^n} \frac{|U(x)-U(y)|^p}{|x-y|^{ps+n}} \mathrm{~d}x\mathrm{d}y\right)^\frac{1}{p}\\ &\lesssim_{p,s,n} \lVert U \rVert_{\dot{\mathrm{B}}^{s}_{p,p}(\mathbb{R}^n)} \lesssim_{p,s,n} \lVert u \rVert_{\dot{\mathrm{B}}^{s}_{p,p}(\Omega)}.
\end{align*}
When $p=+\infty$, the proof remains similar.

For the reverse inequality when $p<+\infty$, we consider $\mathrm{E}=T_{\phi}^{-1}\Tilde{\mathrm{E}}T_\phi$, the extension operator as in Proposition \ref{prop:ExtProj0HomBspq}. By the definition of function spaces by restriction, \cite[Theorem~2.36]{bookBahouriCheminDanchin} and Lemma \ref{lem:GlobalchangeCoordBesov}, one obtains the inequalities
\begin{align*}
    \lVert u \rVert_{\dot{\mathrm{B}}^{s}_{p,p}(\Omega)} \leqslant \lVert \mathrm{E}u \rVert_{\dot{\mathrm{B}}^{s}_{p,p}(\mathbb{R}^n)}\lesssim_{p,s,n}^{\partial\Omega}  \left(\int_{\mathbb{R}^n}\int_{\mathbb{R}^n} \frac{|\Tilde{\mathrm{E}}T_\phi u(x)-\Tilde{\mathrm{E}}T_\phi u(y)|^p}{|x-y|^{ps+n}} \mathrm{~d}x\mathrm{d}y\right)^\frac{1}{p}.
\end{align*}
We recall that $\tilde{\mathrm{E}}$ is the extension operator from $\mathbb{R}^n_+$ to $\mathbb{R}^n$ by even reflection, see \eqref{eq:FlatExtensionOpRn+}. 
\begin{align*}
    \lVert u \rVert_{\dot{\mathrm{B}}^{s}_{p,p}(\Omega)} \lesssim_{p,s,n}^{\partial\Omega}   &2 \left(\int_{\mathbb{R}^n_+}\int_{\mathbb{R}^n_+} \frac{|T_\phi u(x)-T_\phi u(y)|^p}{|x-y|^{ps+n}} \mathrm{~d}x\mathrm{d}y\right)^\frac{1}{p}\\ &+ 2 \left(\int_{\mathbb{R}^n_+}\int_{\mathbb{R}^n_+} \frac{|T_\phi u(x)-T_\phi u(y)|^p}{|(x'-y', x_n+y_n)|^{ps+n}} \mathrm{~d}x'\mathrm{d}x_n\mathrm{d}y'\mathrm{d}y_n\right)^\frac{1}{p}\\
    \lesssim_{p,s,n}^{\partial\Omega} &\left(\int_{\mathbb{R}^n_+}\int_{\mathbb{R}^n_+} \frac{|T_\phi u(x)-T_\phi u(y)|^p}{|x-y|^{ps+n}} \mathrm{~d}x\mathrm{d}y\right)^\frac{1}{p}.
\end{align*}
Here, the last line was obtained by means of $|(x'-y', x_n-y_n)|\leqslant|(x'-y', x_n+y_n)|$, provided $x_n,y_n\geqslant 0$. We also have the inequality, valid for all $(x',x_n),(y',y_n)\in\mathbb{R}^n_+$,
\begin{align*}
    |(x'-y', x_n-y_n)|&\leqslant \big|\big(x'-y', x_n+\phi(x')-(y_n+\phi(y'))\big)\big| + \lVert \nabla' \phi\rVert_{\mathrm{L}^\infty(\mathbb{R}^{n-1})}|x'-y'|\\&\lesssim_{\partial\Omega} \big|\big(x'-y', x_n+\phi(x')-(y_n+\phi(y'))\big)\big|\\
    &\lesssim_{\partial\Omega} |(x'-y', x_n-y_n)|.
\end{align*}
So that, by the change of variable $(x',x_n)\mapsto (x',x_n+\phi(x'))$, with Jacobian determinant $1$,
\begin{align*}
    \lVert u \rVert_{\dot{\mathrm{B}}^{s}_{p,p}(\Omega)} &\lesssim_{p,s,n}^{\partial\Omega}\left(\int_{\mathbb{R}^n_+}\int_{\mathbb{R}^n_+} \frac{|u(x',x_n+\phi(x'))-u(y',y_n+\phi(y'))|^p}{\big|\big(x'-y', x_n+\phi(x')-(y_n+\phi(y'))\big)\big|^{ps+n}} \mathrm{~d}x'\mathrm{d}x_n\mathrm{d}y'\mathrm{d}y_n\right)^\frac{1}{p}\\
    &\lesssim_{p,s,n}^{\partial\Omega}\left(\int_{\Omega}\int_{\Omega} \frac{|u(x)-u(y)|^p}{|x-y|^{ps+n}} \mathrm{~d}x\mathrm{d}y\right)^\frac{1}{p}.
\end{align*}

\textbf{Step 2.2:} The case  $p=+\infty$. As for the Step 2.1, the definition of function spaces by restriction and \cite[Theorem~2.36]{bookBahouriCheminDanchin} yield directly for all $u\in \dot{\mathrm{B}}^{s}_{\infty,\infty}(\Omega)$,
\begin{align}\label{eq:EstBsinftyinftyFiniteDiff1}
        \sup_{\substack{x,y\in\Omega\\ x\neq y}}\frac{| u(x)- u(y)|}{|x-y|^{s}} \lesssim_{s,n} \lVert u\rVert_{\dot{\mathrm{B}}^{s}_{\infty,\infty}(\Omega)}.
\end{align}
For the reverse inequality, let $s<s_1<1$. By Theorem \ref{thm:InterpHomSpacesRn}, \cite[Theorem~3.4.2]{BerghLofstrom1976}, Proposition~\ref{prop:IntersecExtHomBspqProjHomBspq0}, and the definition of function spaces by restriction imply that $\dot{\mathrm{B}}^{0}_{\infty,1}(\Omega)\cap \dot{\mathrm{B}}^{s_1}_{\infty,\infty}(\Omega)$ is strongly dense in $\dot{\mathcal{B}}^{s}_{\infty,\infty}(\Omega)$. So, let $(u_\ell)_{\ell\in\mathbb{N}} \subset \dot{\mathrm{B}}^{0}_{\infty,1}(\Omega)\cap \dot{\mathrm{B}}^{s_1}_{\infty,\infty}(\Omega)$ such that it converges to $u$ in $\dot{\mathcal{B}}^{s}_{\infty,\infty}(\Omega)$. For $\ell\in\mathbb{N}$, Proposition \ref{prop:IntersecExtHomBspqProjHomBspq0} implies $\mathrm{E}u_\ell\in \dot{\mathrm{B}}^{s}_{\infty,\infty}(\mathbb{R}^n)$, and by the definition of function spaces by restriction, we can proceed as in Step 2.1, yielding
\begin{align*}
    \lVert u_\ell \rVert_{\dot{\mathrm{B}}^{s}_{\infty,\infty}(\Omega)}\lesssim_{s,n}^{\partial\Omega}\sup_{\substack{x,y\in\Omega\\ x\neq y}}\frac{| u_\ell(x)- u_\ell(y)|}{|x-y|^{s}}.
\end{align*}
Now, by the continuity provided by \eqref{eq:EstBsinftyinftyFiniteDiff1}, one can pass to the limit, in order to obtain
\begin{align*}
    \lVert u \rVert_{\dot{\mathrm{B}}^{s}_{\infty,\infty}(\Omega)}\lesssim_{s,n}^{\partial\Omega}\sup_{\substack{x,y\in\Omega\\ x\neq y}}\frac{| u(x)- u(y)|}{|x-y|^{s}}.
\end{align*}

For $u\in\dot{\mathrm{B}}^{s,0}_{\infty,\infty}(\Omega)$, the proof is similar to Step 2.1. Finally, it implies that the results holds for $u\in\dot{\mathrm{B}}^{s,0}_{\infty,\infty}(\Omega)+\dot{\mathcal{B}}^{s}_{\infty,\infty}(\Omega)$.
\end{proof}

\begin{theorem}\label{thm:RealInterpHomSpacesSpeLip} Let $p,q,q_0,q_1\in[1,+\infty]$, $p<+\infty$, $s_0,s_1\in\mathbb{R}$, and $\theta\in(0,1)$.

We assume $s_0\neq s_1$ and we set $s:=(1-\theta)s_0+\theta s_1$. It holds that
\begin{enumerate}
    \item $(\dot{\mathrm{B}}^{s_0}_{p,q_0}(\Omega),\dot{\mathrm{B}}^{s_1}_{p,q_1}(\Omega))_{\theta,q}=(\dot{\mathrm{H}}^{s_0,p}(\Omega),\dot{\mathrm{H}}^{s_1,p}(\Omega))_{\theta,q}=\dot{\mathrm{B}}^{s}_{p,q}(\Omega)$, $s_j>-1$, $j\in\{0,1\}$;\label{eq:InterpBesovOmega1}
    \item $(\dot{\mathrm{B}}^{s_0,0}_{\infty,q_0}(\Omega),\dot{\mathrm{B}}^{s_1,0}_{\infty,q_1}(\Omega))_{\theta,q}=\dot{\mathrm{B}}^{s,0}_{\infty,q}(\Omega)$, $s_j>-1$, $j\in\{0,1\}$;\label{eq:InterpBesovOmega2}
    \item $(\dot{\mathrm{B}}^{s_0}_{\infty,q_0}(\Omega),\dot{\mathrm{B}}^{s_1}_{\infty,q_1}(\Omega))_{\theta,q}=\dot{\mathrm{B}}^{s}_{\infty,q}(\Omega)$, $-1<s_j<1$, $j\in\{0,1\}$ and $(\mathcal{C}_{s,\infty,q})$;\label{eq:InterpBesovOmega3}
    \item $(\dot{\mathrm{W}}^{k_0,1}(\Omega),\dot{\mathrm{W}}^{k_1,1}(\Omega))_{\theta,q}=\dot{\mathrm{B}}^{s}_{1,q}(\Omega)$, $k_j=s_j\in\mathbb{N}$, $j\in\{0,1\}$;
\end{enumerate}
Moreover,
\begin{itemize}
    \item when $q=+\infty$, one can replace $((\cdot,\cdot)_{\theta,\infty}, \mathrm{B}^{s})$ by $((\cdot,\cdot)_{\theta}, \mathcal{B}^{s})$;
    \item the result remains true if we replace  $(\dot{\mathrm{W}},\dot{\mathrm{H}},\dot{\mathrm{B}},\dot{\mathcal{B}})$ by $(\dot{\mathrm{W}}_0,\dot{\mathrm{H}}_0,\dot{\mathrm{B}}_{\cdot,\cdot,0},\dot{\mathcal{B}}_{\cdot,\cdot,0})$.
\end{itemize}
\end{theorem}

\begin{remark} The complex interpolation result corresponding to \ref{eq:complexInterpHomBspqRn} from Theorem \ref{thm:InterpHomSpacesRn} is also available under the additional assumption $s_j > -1$.
\end{remark}

\begin{proof}\textbf{Step 1:} We prove \ref{eq:InterpBesovOmega1} in the case of Besov spaces. More precisely,
\begin{equation}\label{eq:InterpBesovOmega1Proof}
    (\dot{\mathrm{B}}^{s_0}_{p,q_0}(\Omega),\dot{\mathrm{B}}^{s_1}_{p,q_1}(\Omega))_{\theta,q}=\dot{\mathrm{B}}^{s}_{p,q}(\Omega) ,\,1\leqslant p<+\infty,\,s_j>-1,\,j\in\{0,1\}.
\end{equation}

\textbf{Step 1.1:} We prove \eqref{eq:InterpBesovOmega1Proof} assuming for now that one of the following condition is satisfied
\begin{enumerate}
    \item $s_j > 0$, for $j\in\{0,1\}$;
    \item $s_j\in(-1,1)$, $j\in\{0,1\}$;
    \item $(\mathcal{C}_{s_j,p,q_j})$, for $j\in\{0,1\}$ (implying $(\mathcal{C}_{s,p,q})$).
\end{enumerate}

Under the conditions \textit{(i)} and \textit{(ii)}, this is a direct consequence of Proposition \ref{prop:ExtProj0HomBspq}, since it asserts that one has a retraction and a coretraction given by the operators
\begin{align*}
    \mathbb{E} \,:\, \dot{\mathrm{B}}^{\tilde{s}}_{p,\tilde{q}}(\Omega)\longrightarrow \dot{\mathrm{B}}^{\tilde{s}}_{p,\tilde{q}}(\mathbb{R}^n) &\text{ and } \mathrm{R}_{{\Omega}} \,:\, \dot{\mathrm{B}}^{\tilde{s}}_{p,\tilde{q}}(\mathbb{R}^n)\longrightarrow \dot{\mathrm{B}}^{\tilde{s}}_{p,\tilde{q}}(\Omega) \ \text{. }
\end{align*}
where $(\tilde{s},\Tilde{q}) \in\{ (s_0,q_0),(s_1,q_1),(s,q)\}$.

Under the condition \textit{(iii)}, this is a consequence of Wolff's reiteration theorem \cite[Theorem~1]{JansonNilssonPeetre1984}, since all involved spaces are complete.

This gives the full result when $p=1$.

\textbf{Step 1.2:} We assume $1<p<+\infty$, we want to show \eqref{eq:InterpBesovOmega1Proof} for $s_0\leqslant 0$ and $s_1\geqslant 1$. Let $\tilde{s}_0\leqslant 0<\alpha_0<\alpha_1<n/p$, we assume $0< \tilde{s}_0+1\leqslant \min(1,n/p)$, and we also let $m\in\mathbb{N}$ such that $\tilde{s}_0+m> n/p$.

\begin{figure}[H]
\centering
\begin{tikzpicture}[yscale=0.6,xscale=6]
  \draw[->] (-0.1,0) -- (1.1,0) node[right,yshift=-2mm] {$1/p$};
  \draw[->] (0,-1) -- (0,4.5) node[above] {$s$};

  \draw[domain=0:1,smooth,variable=\x,blue] plot ({\x},{3*\x}) node[right] {$s=n/p$};
  \fill[blue!30,opacity=0.3] (0,0) -- plot[domain=0:1] (0,0) -- (0,-1) -- (1,-1) -- (1,0) -- cycle;
  \fill[blue!30,opacity=0.3] (0,0) -- plot[domain=0:1] (\x,{3*\x}) -- (1,3) -- (1,0) -- cycle;
  \draw[domain=0:1,smooth,variable=\x,red] plot ({\x},{-1+\x}) node[right,yshift=2mm] {$s=-1+1/p$};
  \draw[domain=0:1/3,smooth,variable=\x,dashed] plot ({\x},{3*\x-1}) ;
  \draw[domain=0:1/3,smooth,variable=\x,dashed] plot ({\x},{3*\x}) ;
  \draw[dashed] (1,1) -- (0,1)  node[left] {$s=1$};
  \draw[domain=0:1/3,smooth,variable=\x,dashed] plot ({\x},{3*\x+1}) ;
  \draw[dashed] (1,2) -- (0,2);
  \draw[domain=0:1/3,smooth,variable=\x,dashed] plot ({\x},{3*\x+2}) ;
  \draw[dashed] (1,3) -- (0,3);  %node[left] {$s=n$};%
  \draw[domain=0:1/3,smooth,variable=\x,dashed] plot ({\x},{3*\x+3}) ;
  \draw[dashed] (1,4) -- (1/3,4);
  \draw[circle,fill,inner sep=1pt] (0,0) node[below left] {$0$};
  \draw[circle,fill,inner sep=1pt] (1,0) node[below right] {$1$};

  \draw[domain=-0.2:3.8,smooth,variable=\y,black] plot ({0.63},{\y});
  \node[circle,fill,inner sep=1.5pt,label=below right:$\dot{\mathrm{B}}^{\tilde{s}_0}_{p,1}$] at (0.63,0-0.2) {};
  \node[circle,fill,inner sep=1.5pt,label=left:$\dot{\mathrm{B}}^{\alpha_0}_{p,\infty}$] at (0.63,0.38) {};
  \node[circle,fill,inner sep=1.5pt,label=right:$\dot{\mathrm{B}}^{\alpha_1}_{p,\infty}$] at (0.63,1.45) {};
  \node[circle,fill,inner sep=1.5pt,label=above :$\dot{\mathrm{B}}^{\tilde{s}_0+m}_{p,1}$] at (0.63,3.8) {};
\end{tikzpicture}
\caption{Hardy-Littlewood-Sobolev-Kato diagram: representation of the chosen interpolation scale to apply Wolff's reiteration theorem.}
\label{fig:BesovspacesWolffsReiter}
\end{figure}
By the definition of function spaces by restriction, Proposition \ref{prop:EmbeddSobBesovRn}, since $\tilde{s}_0\leqslant0$ and $\alpha_1>\alpha_0>0$, one has the chain of inclusions\footnote{The first inclusion holds because $\tilde{s}_0\leqslant 0$ and because of the definition of function spaces by restriction:
\begin{align*}
    \lVert u\rVert_{{\mathrm{B}}^{\tilde{s}_0+m}_{p,1}(\Omega)} \sim \lVert u\rVert_{{\mathrm{B}}^{\tilde{s}_0}_{p,1}(\Omega)}+\lVert \nabla^m u\rVert_{{\mathrm{B}}^{\tilde{s}_0}_{p,1}(\Omega)}\lesssim \lVert u\rVert_{\dot{\mathrm{B}}^{\tilde{s}_0}_{p,1}(\Omega)}+\lVert \nabla^m u\rVert_{\dot{\mathrm{B}}^{\tilde{s}_0}_{p,1}(\Omega)}\lesssim \lVert u\rVert_{\dot{\mathrm{B}}^{\tilde{s}_0}_{p,1}(\Omega)}+\lVert  u\rVert_{\dot{\mathrm{B}}^{\tilde{s}_0+m}_{p,1}(\Omega)}.
\end{align*}}
\begin{align*}
\dot{\mathrm{B}}^{\tilde{s}_0}_{p,1}(\Omega)\cap\dot{\mathrm{B}}^{\tilde{s}_0+m}_{p,1}(\Omega) \subset {\mathrm{B}}^{\tilde{s}_0+m}_{p,1}(\Omega) \subset {\mathrm{B}}^{\alpha_1}_{p,\infty}(\Omega) \subset \dot{\mathrm{B}}^{\alpha_0}_{p,\infty}(\Omega)\cap\dot{\mathrm{B}}^{\alpha_1}_{p,\infty}(\Omega).
\end{align*}
Moreover, one can check that, for $\mathfrak{a}=\alpha_0,\alpha_1$, the space
\begin{align*}
(\dot{\mathrm{B}}^{\tilde{s}_0}_{p,1}(\Omega)+ \dot{\mathrm{B}}^{\mathfrak{a}}_{p,\infty}(\Omega))\cap \dot{\mathrm{B}}^{\tilde{s}_0+m}_{p,1}(\Omega)
\end{align*}
is complete, this by means of the assumptions on $\tilde{s}_0,\mathfrak{a}<n/p$, and $0\leqslant\tilde{s}_0+1 \leqslant \min(1,n/p)$ thanks to the Proposition \ref{prop:EquivNormsBesovSpeLip}. Therefore by Wolff's reiteration Theorem, Lemma \ref{lem:WOlffReitThmnonComp}, by the extremal reiteration property, Lemma \ref{lem:ExtrmReitPropnonComp}, and with Step 1.1, we obtain for all $-1<\tilde{s}_0<s<\tilde{s}_0+m$, $q\in[1,+\infty]$, satisfying $\tilde{s}_0\leqslant 0$, $\tilde{s}_0+m\in\mathbb{N}^\ast + [0,\min(1,n/p)]$,
\begin{align*}
     \dot{\mathrm{B}}^{s}_{p,q}(\Omega) = (\dot{\mathrm{B}}^{\tilde{s}_0}_{p,1}(\Omega),\dot{\mathrm{B}}^{\tilde{s}_0+m}_{p,1}(\Omega))_{\frac{s-\tilde{s}_0}{m},q}.
\end{align*}
Now, we perform a bootstrap procedure twice. First, for any $\tilde{s}_0<s_0<s$, $s_0\leqslant0$, we can write, thanks to the extremal reiteration property, Lemma \ref{lem:ExtrmReitPropnonComp},
\begin{align*}
(\dot{\mathrm{B}}^{s_0}_{p,q_0}(\Omega),\dot{\mathrm{B}}^{\tilde{s}_0+m}_{p,1}(\Omega))_{\frac{s-s_0}{\tilde{s}_1-{s}_0},q} &= \left((\dot{\mathrm{B}}^{\tilde{s}_0}_{p,1}(\Omega),\dot{\mathrm{B}}^{\tilde{s}_0+m}_{p,1}(\Omega))_{\frac{s_0-\tilde{s}_0}{m},q_0},\dot{\mathrm{B}}^{\tilde{s}_0+m}_{p,1}(\Omega)\right)_{\frac{s-s_0}{m},q}\\
&= (\dot{\mathrm{B}}^{\tilde{s}_0}_{p,1}(\Omega),\dot{\mathrm{B}}^{\tilde{s}_0+m}_{p,1}(\Omega))_{\frac{s-\tilde{s}_0}{m},q}=\dot{\mathrm{B}}^{s}_{p,q}(\Omega).
\end{align*}
Now for $s_1>0$, one can always find $m\in\mathbb{N}$ such that $s_0<s_1<\tilde{s_0}+m$, and then for $q_1\in[1,+\infty]$, similarly, but we apply the preceding interpolation identity in order to obtain
\begin{align*}
(\dot{\mathrm{B}}^{s_0}_{p,q_0}(\Omega),\dot{\mathrm{B}}^{s_1}_{p,q_1}(\Omega))_{\frac{s-s_0}{s_1-s_0},q} &= \left(\dot{\mathrm{B}}^{s_0}_{p,q_0}(\Omega),(\dot{\mathrm{B}}^{s_0}_{p,q_0}(\Omega),\dot{\mathrm{B}}^{\tilde{s}_1}_{p,1}(\Omega))_{\frac{s_1-s_0}{\tilde{s}_0+m-s_0},q_1}\right)_{\frac{s-s_0}{s_1-s_0},q}\\
&=(\dot{\mathrm{B}}^{s_0}_{p,q_0}(\Omega),\dot{\mathrm{B}}^{\tilde{s}_1}_{p,1}(\Omega))_{\frac{s-s_0}{\tilde{s}_0+m-{s}_0},q}=\dot{\mathrm{B}}^{s}_{p,q}(\Omega).
\end{align*}
Thus \eqref{eq:InterpBesovOmega1Proof} holds in any cases.

\textbf{Step 2:} We prove \ref{eq:InterpBesovOmega2}, namely
\begin{align}\label{eq:InterpBesovOmega2Proof}
(\dot{\mathrm{B}}^{s_0,0}_{\infty,q_0}(\Omega),\dot{\mathrm{B}}^{s_1,0}_{\infty,q_1}(\Omega))_{\theta,q}=\dot{\mathrm{B}}^{s,0}_{\infty,q}(\Omega)\text{, } s_0,s_1>-1\text{, } 1\leqslant q_0,q_1,q\leqslant+\infty.
\end{align}
\textbf{Step 2.1:} As in Step 1.1, \eqref{eq:InterpBesovOmega2Proof} holds in the following subcases:
\begin{enumerate}
\item $s_j>0$, for $j\in\{0,1\}$;
\item $-1<s_j<1$ , for $j\in\{0,1\}$.
\end{enumerate}

Moreover, the same retraction-coretraction arguments give the identities:
\begin{align}\label{eq:InterpBesovOmega2Proof-2}
\dot{\mathrm{B}}^{s,0}_{\infty,q}(\Omega):=\left\{\begin{array}{cl}
(\dot{\mathrm{C}}^{k_0}_{0}(\overline{\Omega}),\dot{\mathrm{C}}^{k_1}_{0}(\overline{\Omega}))_{\theta,q} \text{, }& k_0,k_1\in\mathbb{N}\text{, } 1\leqslant q\leqslant+\infty,\\
(\dot{\mathrm{C}}^{k_0}_{0}(\overline{\Omega}),\dot{\mathrm{B}}^{s_1,0}_{\infty,q_1}({\Omega}))_{\theta,q} \text{, }& k_0\in\mathbb{N}\text{, }s_1>0\text{, } 1\leqslant q,q_1\leqslant+\infty.\\
(\dot{\mathrm{B}}^{0,0}_{\infty,1}({\Omega}),\dot{\mathrm{B}}^{s_1,0}_{\infty,q_1}({\Omega}))_{\theta,q}\text{, }& s_1>0\text{, } 1\leqslant q,q_1\leqslant+\infty
\end{array}\right.
\end{align}
The third interpolation equality follows from the second one in combination with Lemma \ref{lem:EmbeddingInterpHomSobspacesSpeLip} and the embedding $\dot{\mathrm{B}}^{0,0}_{\infty,1}({\Omega})\hookrightarrow {\mathrm{C}}^{0}_{0}(\overline{\Omega})$.

\textbf{Step 2.2:} The case $s_1\geqslant 1$, $s_0\leqslant0$. We assume $s_0<0$, and wet let $0<\alpha<1$. We recall that, one has the sets equalities and inclusions $$\dot{\mathrm{B}}^{s_1,0}_{\infty,q_1}({\Omega})={\mathrm{B}}^{s_1,0}_{\infty,q_1}({\Omega})\subset {\mathrm{B}}^{\alpha,0}_{\infty,\infty}({\Omega})=\dot{\mathrm{B}}^{\alpha,0}_{\infty,\infty}({\Omega})\subset \mathrm{C}_0^0(\overline{\Omega}).$$

By previous Step 2.1, and by definition of function spaces by restriction
\begin{align*}
({\mathrm{C}}^{0}_{0}(\overline{\Omega}),\dot{\mathrm{B}}^{s_1,0}_{\infty,q_1}({\Omega}))_{\frac{\alpha}{s_1},\infty} = \dot{\mathrm{B}}^{\alpha,0}_{\infty,\infty}(\Omega)\text{, and } (\dot{\mathrm{B}}^{s_0,0}_{\infty,q_0}({\Omega}),\dot{\mathrm{B}}^{\alpha,0}_{\infty,\infty}({\Omega}))_{\frac{-s_0}{\alpha-s_0},1} = \dot{\mathrm{B}}^{0,0}_{\infty,1}({\Omega})\hookrightarrow {\mathrm{C}}^{0}_{0}(\overline{\Omega}).
\end{align*}

Thus, by Lemmas \ref{lem:EmbeddingInterpHomSobspacesSpeLip}, \ref{lem:ExtrmReitPropnonComp} and  \ref{lem:WOlffReitThmnonComp2}, for $0<s<s_1$,
\begin{align*}
    \dot{\mathrm{B}}^{s,0}_{\infty,q}(\Omega)\hookrightarrow(\dot{\mathrm{B}}^{s_0,0}_{\infty,q_0}({\Omega}),\dot{\mathrm{B}}^{s_1,0}_{\infty,q_1}({\Omega}))_{\frac{s-s_0}{s_1-s_0},q} &= \left((\dot{\mathrm{B}}^{s_0,0}_{\infty,q_0}({\Omega}),\dot{\mathrm{B}}^{s_1,0}_{\infty,q_1}({\Omega}))_{\frac{-s_0}{s_1-s_0},1},\dot{\mathrm{B}}^{s_1,0}_{\infty,q_1}({\Omega})\right)_{\frac{s}{s_1},q}
    \\
    &\hookrightarrow({\mathrm{C}}^{0}_{0}(\overline{\Omega}),\dot{\mathrm{B}}^{s_1,0}_{\infty,q_1}({\Omega}))_{\frac{s}{s_1},q} =\dot{\mathrm{B}}^{s,0}_{\infty,q}(\Omega)
\end{align*}
where the last equality is given by \eqref{eq:InterpBesovOmega2Proof-2}. Now for $s_0<s\leqslant0$, $0<\beta<1$ and $r\in[1,+\infty]$, Step 2.1 and Lemma \ref{lem:ExtrmReitPropnonComp} yield
\begin{align*}
    \dot{\mathrm{B}}^{s,0}_{\infty,q}(\Omega) &=(\dot{\mathrm{B}}^{s_0,0}_{\infty,q_0}({\Omega}),\dot{\mathrm{B}}^{\beta,0}_{\infty,r}({\Omega}))_{\frac{s-s_0}{\beta-s_0},q}\\
    &=\left(\dot{\mathrm{B}}^{s_0,0}_{\infty,q_0}({\Omega}),(\dot{\mathrm{B}}^{s_0,0}_{\infty,q_0}({\Omega}),\dot{\mathrm{B}}^{s_1,0}_{\infty,q_1}({\Omega}))_{\frac{\beta-s_0}{s_1-s_0},r}\right)_{\frac{s-s_0}{\beta-s_0},q}\\ &= (\dot{\mathrm{B}}^{s_0,0}_{\infty,q_0}({\Omega}),\dot{\mathrm{B}}^{s_1,0}_{\infty,q_1}({\Omega}))_{\frac{s-s_0}{s_1-s_0},q}.
\end{align*}
The case $s_0=0$ can be deduced with the same argument. Thus, \eqref{eq:InterpBesovOmega2Proof} holds in any cases.

\textbf{Step 3:} We prove \ref{eq:InterpBesovOmega3}. As in Step 1.1, by Proposition \ref{prop:ExtProj0HomBspinfty}, \ref{eq:InterpBesovOmega3} holds whenever $(\mathcal{C}_{s_j,\infty,q_j})$ is satisfied for $j\in\{0,1\}$.

We recall that, by Lemma \ref{lem:EmbeddingInterpHomSobspacesSpeLip}, one always has
\begin{align*}
    \dot{\mathrm{B}}^{s}_{\infty,q}(\Omega)\hookrightarrow(\dot{\mathrm{B}}^{s_0}_{\infty,q_0}(\Omega),\dot{\mathrm{B}}^{s_1}_{\infty,q_1}(\Omega))_{\theta,q}.
\end{align*}
So we just have to prove the reverse embedding.

We assume $(\mathcal{C}_{s_0,\infty,q_0})$ and temporarily assume $q<+\infty$. The Theorem \ref{thm:InterpHomSpacesRn}, \cite[Theorem~3.4.2]{BerghLofstrom1976}, Proposition~\ref{prop:IntersecExtHomBspqProjHomBspq0}, and the definition of function spaces by restriction imply that $\dot{\mathrm{B}}^{s_0}_{\infty,q_0}(\Omega)\cap \dot{\mathrm{B}}^{s_1}_{\infty,q_1}(\Omega)$ is strongly dense in $\dot{\mathrm{B}}^{s}_{\infty,q}(\Omega)$. Let $u\in\dot{\mathrm{B}}^{s_0}_{\infty,q_0}(\Omega)\cap \dot{\mathrm{B}}^{s_1}_{\infty,q_1}(\Omega)$, and $(a,b)\in\dot{\mathrm{B}}^{s_0}_{\infty,q_0}(\Omega)\times \dot{\mathrm{B}}^{s_1}_{\infty,q_1}(\Omega)$ such that
\begin{align*}
    u=a+b.
\end{align*}
By Proposition \ref{prop:IntersecExtHomBspqProjHomBspq0}, one has $\mathrm{E}u\in \dot{\mathrm{B}}^{s_0}_{\infty,q_0}(\mathbb{R}^n)\cap \dot{\mathrm{B}}^{s_1}_{\infty,q_1}(\mathbb{R}^n)$ and $\mathrm{E}a\in\dot{\mathrm{B}}^{s_0}_{\infty,q_0}(\mathbb{R}^n)$, so that
\begin{align*}
    \mathrm{E}b =\mathrm{E}u - \mathrm{E}a\in &\left[\dot{\mathrm{B}}^{s_0}_{\infty,q_0}(\mathbb{R}^n)\cap \dot{\mathrm{B}}^{s_1}_{\infty,q_1}(\mathbb{R}^n) + \dot{\mathrm{B}}^{s_0}_{\infty,q_0}(\mathbb{R}^n)\right]\cap\dot{\mathrm{B}}^{s_1}_{\infty,q_1}(\mathbb{R}^n)\\ \qquad & \subset \dot{\mathrm{B}}^{s_0}_{\infty,q_0}(\mathbb{R}^n)\cap \dot{\mathrm{B}}^{s_1}_{\infty,q_1}(\mathbb{R}^n).
\end{align*}
So, by Proposition \ref{prop:IntersecExtHomBspqProjHomBspq0} again, for $t>0$ and the definition of the $K$-functional:
\begin{align*}
    K(t,\mathrm{E}u,\dot{\mathrm{B}}^{s_0}_{\infty,q_0}(\mathbb{R}^n), \dot{\mathrm{B}}^{s_1}_{\infty,q_1}(\mathbb{R}^n)) \lesssim_{s_0,s_1,n}^{\partial\Omega} \lVert a\rVert_{\dot{\mathrm{B}}^{s_0}_{\infty,q_0}(\Omega)} + t\lVert b\rVert_{\dot{\mathrm{B}}^{s_1}_{\infty,q_1}(\Omega)}.
\end{align*}
Taking the infimum over all such pairs $(a,b)$ yields for all $t>0$,
\begin{align*}
    K(t,\mathrm{E}u,\dot{\mathrm{B}}^{s_0}_{\infty,q_0}(\mathbb{R}^n), \dot{\mathrm{B}}^{s_1}_{\infty,q_1}(\mathbb{R}^n)) \lesssim_{s_0,s_1,n}^{\partial\Omega} K(t,u,\dot{\mathrm{B}}^{s_0}_{\infty,q_0}(\Omega), \dot{\mathrm{B}}^{s_1}_{\infty,q_1}(\Omega)).
\end{align*}
One can multiply by $t^{-\theta}$, take the $\mathrm{L}^q_{\ast}$-norm on both sides, so that by the definition of function spaces by restriction, Theorem \ref{thm:InterpHomSpacesRn} and Lemma \ref{lem:EmbeddingInterpHomSobspacesSpeLip},
\begin{align*}
    \lVert u\rVert_{\dot{\mathrm{B}}^{s}_{\infty,q}(\Omega)}\leqslant \lVert \mathrm{E}u\rVert_{\dot{\mathrm{B}}^{s}_{\infty,q}(\mathbb{R}^n)} &\lesssim_{s_0,s_1,n}^{\theta} \lVert \mathrm{E}u\rVert_{(\dot{\mathrm{B}}^{s_0}_{\infty,q_0}(\mathbb{R}^n),\dot{\mathrm{B}}^{s_1}_{\infty,q_1}(\mathbb{R}^n))_{\theta,q}}\\
    &\lesssim_{s_0,s_1,n}^{\theta,\partial\Omega}\lVert u\rVert_{(\dot{\mathrm{B}}^{s_0}_{\infty,q_0}(\Omega),\dot{\mathrm{B}}^{s_1}_{\infty,q_1}(\Omega))_{\theta,q}}\lesssim_{s_0,s_1,n}^{\theta,\partial\Omega}\lVert u\rVert_{\dot{\mathrm{B}}^{s}_{\infty,q}(\Omega)}.
\end{align*}
Therefore, for all $u\in\dot{\mathrm{B}}^{s_0}_{\infty,q_0}(\Omega)\cap \dot{\mathrm{B}}^{s_1}_{\infty,q_1}(\Omega)$,
\begin{align}\label{eq:equivNormsInterpBesovInftySpeLip}
    \lVert u\rVert_{(\dot{\mathrm{B}}^{s_0}_{\infty,q_0}(\Omega),\dot{\mathrm{B}}^{s_1}_{\infty,q_1}(\Omega))_{\theta,q}}\sim_{s_0,s_1,n}^{\theta,\partial\Omega}\lVert u\rVert_{\dot{\mathrm{B}}^{s}_{\infty,q}(\Omega)}.
\end{align}
We mention again that $\dot{\mathrm{B}}^{s}_{\infty,q}(\Omega)\hookrightarrow(\dot{\mathrm{B}}^{s_0}_{\infty,q_0}(\Omega),\dot{\mathrm{B}}^{s_1}_{\infty,q_1}(\Omega))_{\theta,q}$ and that $\dot{\mathrm{B}}^{s_0}_{\infty,q_0}(\Omega)\cap \dot{\mathrm{B}}^{s_1}_{\infty,q_1}(\Omega)$ is strongly dense in both spaces since $q<+\infty$. Thus, \eqref{eq:equivNormsInterpBesovInftySpeLip} holds for all $u\in\dot{\mathrm{B}}^{s}_{\infty,q}(\Omega)$.

Now, if $q=+\infty$ and $s<0$, the result follows by the standard reiteration Theorem \cite[Theorem~3.5.3]{BerghLofstrom1976}.

\textbf{Step 4:} Similar arguments to those used in the previous Steps 1, 2 and 3 allow deducing the result for $\dot{\mathrm{B}}_{\cdot,\cdot,0}$ instead of $\dot{\mathrm{B}}$. However, to reproduce Step 2, one has to be aware that the natural embedding is reversed, that is
\begin{align*}
    (\dot{\mathrm{B}}^{s_0,0}_{\infty,q_0,0}({\Omega}),\dot{\mathrm{B}}^{s_1,0}_{\infty,q_1,0}({\Omega}))_{\theta,q} \hookrightarrow \dot{\mathrm{B}}^{s,0}_{\infty,q,0}(\Omega)
\end{align*}
and that we need to replace the space $\dot{\mathrm{C}}^k_0(\overline{\Omega})$ by
\begin{align*}
\dot{\mathrm{C}}^k_{0,0}({\Omega}) := \{\, u\in{\mathrm{C}}^k_{0}({\mathbb{R}^n})\,|\,\supp u \subset \overline{\Omega}\,\}=\mathcal{P}_0[{\mathrm{C}}^k_{0}({\mathbb{R}^n})].
\end{align*}
endowed with the (semi-)norm $\lVert\cdot\rVert_{\dot{\mathrm{W}}^{k,\infty}(\mathbb{R}^n)}$. Instead of Lemma \ref{lem:WOlffReitThmnonComp2}, one has to use Lemma \ref{lem:WOlffReitThmnonComp3} with  $\lambda = \frac{-s_0}{\alpha-s_0}$, $r=+\infty$,
\begin{align*}
\mathrm{X}_1=\dot{\mathrm{B}}^{s_0,0}_{\infty,q_0,0}(\Omega),\,\mathrm{X}_2=\mathrm{C}^0_{0,0}(\Omega),\,\mathrm{X}_3=\dot{\mathrm{B}}^{\alpha,0}_{\infty,1,0}(\Omega),\,\mathrm{X}_4=\dot{\mathrm{B}}^{s_1,0}_{\infty,q_1,0}(\Omega),
\end{align*}
where $-1<s_0<0<\alpha<1\leqslant s_1$. The details are left to the reader.

\textbf{Step 5:} Now, the remaining interpolation identities follow directly from the natural embeddings, for $({\dot{\mathfrak{W}}},{\dot{\mathfrak{H}}},{\dot{\mathfrak{B}}})\in\{(\dot{\mathrm{W}},\dot{\mathrm{H}},\dot{\mathrm{B}}),(\dot{\mathrm{W}}_0,\dot{\mathrm{H}}_0,\dot{\mathrm{B}}_{\cdot,\cdot,0})\}$,
\begin{align*}
    \dot{\mathfrak{B}}^{s}_{p,1}(\Omega)\hookrightarrow &\dot{\mathfrak{H}}^{s,p}(\Omega)\hookrightarrow \dot{\mathfrak{B}}^{s}_{p,\infty}(\Omega)\text{, } s\in\mathbb{R}\text{, }p\in[1,+\infty)\text{;}\\
    \dot{\mathfrak{B}}^{k}_{1,1}(\Omega)\hookrightarrow &\dot{\mathfrak{W}}^{k,1}(\Omega)\hookrightarrow \dot{\mathfrak{B}}^{k}_{1,\infty}(\Omega)\text{, } k\in\mathbb{N},
\end{align*}
obtained by the definition of function spaces by restriction and Proposition \ref{prop:EmbeddSobBesovRn}.
\end{proof}

\begin{corollary}\label{cor:weakastdensityBspinfty}Let $p\in(1,+\infty)$, let $-1<s<n/p$ such that $s\geqslant -n/p'$, then the space $\mathrm{C}_c^\infty(\Omega)$ is weakly-$\ast$ dense in $\dot{\mathrm{B}}^{s}_{p,\infty,0}(\Omega)$.
\end{corollary}

\begin{proof}We follow a modified version of the argument given in the proof of \cite[Corollary~3.19]{Gaudin2022}. For $-1<s_0<s<s_1<n/p$, by Theorem \ref{thm:RealInterpHomSpacesSpeLip}, \cite[Theorems~3.4.2~\&~3.7.1]{BerghLofstrom1976} and Proposition \ref{prop:DualitySobolevDomain}, we have the weakly-$\ast$ dense embedding
\begin{align*}
    \dot{\mathcal{B}}^{s}_{p,\infty,0}(\Omega) = ( \dot{\mathrm{H}}^{s_0,p}_0(\Omega),\dot{\mathrm{H}}^{s_1,p}_0(\Omega))_{\theta} \hookrightarrow &( \dot{\mathrm{H}}^{s_0,p}_0(\Omega),\dot{\mathrm{H}}^{s_1,p}_0(\Omega))_{\theta}''\\
    \hookrightarrow&( \dot{\mathrm{H}}^{s_0,p}_0(\Omega),\dot{\mathrm{H}}^{s_1,p}_0(\Omega))_{\theta,\infty} =  \dot{\mathrm{B}}^{s}_{p,\infty,0}(\Omega).
\end{align*}
By Proposition \ref{prop:densityCcinftyBspq0SpeLip}, the space $\mathrm{C}^\infty_c(\Omega)$ in a strongly dense subspace of $\dot{\mathcal{B}}^{s}_{p,\infty,0}(\Omega)$, and therefore is a weakly-$\ast$ dense subspace of $\dot{\mathrm{B}}^{s}_{p,\infty,0}(\Omega)$.
\end{proof}

\begin{theorem}\label{thm:dualityBesovSpeLip}Let $p,q\in[1,+\infty]$, $s\in\mathbb{R}$, if \eqref{AssumptionCompletenessExponents} is satisfied then the following canonical isomorphisms hold
\begin{align*}
    (\dot{\mathrm{B}}^{-s}_{p',q',0}(\Omega))'=\dot{\mathrm{B}}^{s}_{p,q}(\Omega)\, \text{, }&\, (\dot{\mathrm{B}}^{-s}_{p',q'}(\Omega))'=\dot{\mathrm{B}}^{s}_{p,q,0}(\Omega)\text{, } p,q>1\text{;}\\
    (\dot{\mathcal{B}}^{-s}_{p',\infty,0}(\Omega))'=\dot{\mathrm{B}}^{s}_{p,1}(\Omega)\, \text{, }&\, (\dot{\mathcal{B}}^{-s}_{p',\infty}(\Omega))'=\dot{\mathrm{B}}^{s}_{p,1,0}(\Omega)\text{, }p>1,q=1\text{;}\\
    (\dot{\mathrm{B}}^{-s,0}_{\infty,q',0}(\Omega))'=\dot{\mathrm{B}}^{s}_{1,q}(\Omega)\, \text{, }&\, (\dot{\mathrm{B}}^{-s,0}_{\infty,q'}(\Omega))'=\dot{\mathrm{B}}^{s}_{1,q,0}(\Omega)\text{, }p=1,q>1\text{;}\\
    (\dot{\mathcal{B}}^{-s,0}_{\infty,\infty,0}(\Omega))'=\dot{\mathrm{B}}^{s}_{1,1}(\Omega)\, \text{, }&\, (\dot{\mathcal{B}}^{-s,0}_{\infty,\infty}(\Omega))'=\dot{\mathrm{B}}^{s}_{1,1,0}(\Omega)\text{, } p=q=1\text{. }
\end{align*}
\end{theorem}

\begin{remark} Above Theorem \ref{thm:dualityBesovSpeLip}, \cite[Theorem~3.7.1~\&~Remark]{BerghLofstrom1976} and \cite[Theorem~1]{JansonNilssonPeetre1984} can be used to improve Theorem \ref{thm:RealInterpHomSpacesSpeLip}. These results allow us to refine the lower bounds "$s_j > -1$, $j \in \{0,1\}$" in \ref{eq:InterpBesovOmega1}-\ref{eq:InterpBesovOmega3} into
\begin{align*}
    s_j > \min(-1, n/p'), \, j \in \{0,1\}. \text{ (even } s_j \geqslant -n/p' \text{ if } q_j = +\infty \text{).}
\end{align*}
Since Steps 1 and 2 of the proof of Theorem \ref{thm:RealInterpHomSpacesSpeLip} must also be reproduced and readapted, we do not present the proof of such improvement.
\end{remark}

\begin{proof} We only prove
\begin{align*}
   (\dot{\mathrm{B}}^{-s}_{p',q'}(\Omega))'=\dot{\mathrm{B}}^{s}_{p,q,0}(\Omega)\text{, }p,q>1; 
\end{align*}
under the assumption $p=+\infty$ or $q=+\infty$. Other cases admit a similar proof.

\textbf{Step 1:} The case $p=+\infty$, $q>1$, $s<0$:
\begin{align*}
    (\dot{\mathrm{B}}^{-s}_{1,q'}(\Omega))'=\dot{\mathrm{B}}^{s}_{\infty,q,0}(\Omega)
\end{align*}

Let $u\in \dot{\mathrm{B}}^{s}_{\infty,q,0}(\Omega)$, by the definition of function spaces by restriction, $u$ induces a linear form on $\dot{\mathrm{B}}^{-s}_{1,q'}(\Omega)$ as follows:
\begin{align*}
    v\longmapsto \langle u,\tilde{v} \rangle_{\mathbb{R}^n},
\end{align*}
where $\tilde{v}$ is any $\dot{\mathrm{B}}^{-s}_{1,q'}$-extension of $v$. This map is well-defined and does not depend on the choice of the extension $\Tilde{v}$ of $v$. Indeed, let $v'$ be another $\dot{\mathrm{B}}^{-s}_{1,q'}$-extension of $v$, then $\tilde{v}-v'\in\dot{\mathrm{B}}^{-s}_{1,q',0}(\overline{\Omega}^c)$. Since $-s>0$, by Proposition \ref{prop:densityCcinftyBspq0SpeLip}, there exists a sequence $(v_\ell)_{\ell\in\mathbb{N}}\subset \mathrm{C}_c^\infty(\overline{\Omega}^c)$ that converges towards $\tilde{v}-v'$. It implies, 
\begin{align*}
   \langle u,\tilde{v} \rangle_{\mathbb{R}^n}- \langle u,v' \rangle_{\mathbb{R}^n} =\langle u,\tilde{v}-v' \rangle_{\mathbb{R}^n} = \lim_{\ell\rightarrow+\infty} \langle u,v_\ell \rangle_{\mathbb{R}^n}=0.
\end{align*}
Therefore, we have a well-defined, injective and bounded map
\begin{align}\label{eq:mapduality1BspqSpeLip}
\left\{\begin{array}{cl}
\dot{\mathrm{B}}^{-s}_{\infty,q',0}(\Omega) &\longrightarrow (\dot{\mathrm{B}}^{s}_{1,q}(\Omega))'\\
u &\longmapsto \,\,\, \big\langle u,\tilde{\cdot}\big\rangle_{\mathbb{R}^n}
\end{array}\right.\text{. }
\end{align}
We show that the map is surjective. Let $\mathfrak{U}\in (\dot{\mathrm{B}}^{s}_{1,q}(\Omega))'$, it induces a bounded linear functional on $\dot{\mathrm{B}}^{s}_{1,q}(\mathbb{R}^n)$ by means of
\begin{align*}
    \tilde{v}\longmapsto\langle \mathfrak{U}, \mathbbm{1}_{\Omega} \tilde{v}\rangle.
\end{align*}
Therefore by Proposition \ref{prop:DualityBesovRn}, there exists $\mathfrak{u}\in \dot{\mathrm{B}}^{-s}_{\infty,q'}(\mathbb{R}^n)$, such that for all $v\in\dot{\mathrm{B}}^{s}_{1,q}(\Omega)$ and all $\tilde{v}\in\dot{\mathrm{B}}^{s}_{1,q}(\mathbb{R}^n)$ such that $\Tilde{v}_{|_\Omega}=v$,
\begin{align}\label{eq:proofDualityBracketEquality1}
\langle \mathfrak{U}, v\rangle = \langle\mathfrak{U}, \mathbbm{1}_{\Omega} \tilde{v}\rangle = \langle\mathfrak{u}, \tilde{v}\rangle_{\mathbb{R}^n}
\end{align}
with the estimate $\lVert\mathfrak{U}\rVert_{(\dot{\mathrm{B}}^{s}_{1,q}(\Omega))'}\sim_{s,n} \lVert \mathfrak{u}\rVert_{\dot{\mathrm{B}}^{-s}_{\infty,q'}(\mathbb{R}^n)}$.

Testing the equality \eqref{eq:proofDualityBracketEquality1} with $\Tilde{v}\in\mathrm{C}^\infty_c(\overline{\Omega}^c)$ implies $\supp \mathfrak{u} \subset \overline{\Omega}$. Therefore, $\mathfrak{u}$ is an element of $\dot{\mathrm{B}}^{-s}_{\infty,q',0}(\Omega)$, and consequently the map \eqref{eq:mapduality1BspqSpeLip} is surjective.

\textbf{Step 2:} The case $q=+\infty$, $p>1$, $s<n/p$:
\begin{align*}
    (\dot{\mathrm{B}}^{-s}_{p',1}(\Omega))'=\dot{\mathrm{B}}^{s}_{p,\infty,0}(\Omega).
\end{align*}
As before, we have a well-defined injective and bounded map
\begin{align}\label{eq:mapduality2BspqSpeLip}
\left\{\begin{array}{cl}
\dot{\mathrm{B}}^{s}_{p,\infty,0}(\Omega) &\longrightarrow (\dot{\mathrm{B}}^{-s}_{p',1}(\Omega))'\\
u &\longmapsto \,\,\, \big\langle u,\tilde{\cdot}\big\rangle_{\mathbb{R}^n}
\end{array}\right.\text{. }
\end{align}
Indeed, for $u\in\dot{\mathrm{B}}^{s}_{p,\infty,0}(\Omega)$, and $v\in\dot{\mathrm{B}}^{-s}_{p',1}(\Omega)$, $\Tilde{v},v'\in\dot{\mathrm{B}}^{-s}_{p',1}(\mathbb{R}^n)$ such that $\Tilde{v}_{|_{\Omega}}=v'_{|_{\Omega}}=v$, one has $\supp (\Tilde{v}-v')\subset {\Omega}^{c}$, and in particular $\Tilde{v}-v'\in \dot{\mathrm{B}}^{-s}_{p',1,0}(\overline{\Omega}^c)$.
\begin{itemize}
    \item If $s<\min(1,n/p)$, by Proposition \ref{prop:densityCcinftyBspq0SpeLip} one can find a sequence $(w_\ell)_{\ell\in\mathbb{N}}\subset \mathrm{C}_c^\infty(\overline{\Omega}^c)$ such that it converges strongly to $\Tilde{v}-v'$ in $\dot{\mathrm{B}}^{-s}_{p',1}(\mathbb{R}^n)$, and therefore
    \begin{align*}
   \langle u,\tilde{v} \rangle_{\mathbb{R}^n}- \langle u,v' \rangle_{\mathbb{R}^n} =\langle u,\tilde{v}-v' \rangle_{\mathbb{R}^n} = \lim_{\ell\rightarrow+\infty} \langle u,w_\ell \rangle_{\mathbb{R}^n}=0.
    \end{align*}
    \item if $s\geqslant0$\footnote{Note that it implies necessarily $p<+\infty$, due to the conditions $q=+\infty$ and \eqref{AssumptionCompletenessExponents}.}, by Corollary \ref{cor:weakastdensityBspinfty} there exists a sequence $(u_\ell)_{\ell\in\mathbb{N}}\subset \mathrm{C}_c^\infty({\Omega})$ that converges weakly-$\ast$ to $u$ in $\dot{\mathrm{B}}^{s}_{p,\infty}(\mathbb{R}^n)$. Hence,
    \begin{align*}
   \langle u,\tilde{v} \rangle_{\mathbb{R}^n}- \langle u,v' \rangle_{\mathbb{R}^n} =\langle u,\tilde{v}-v' \rangle_{\mathbb{R}^n} = \lim_{\ell\rightarrow+\infty} \langle u_\ell,\tilde{v}-v'  \rangle_{\mathbb{R}^n}=0.
    \end{align*}
\end{itemize}
The map \eqref{eq:mapduality2BspqSpeLip} does not depend on the choice of the extension $\tilde{v}\in \dot{\mathrm{B}}^{-s}_{p',1}(\mathbb{R}^n)$ of $v\in\dot{\mathrm{B}}^{-s}_{p',1}(\Omega)$. One can conclude as in the end of Step 1 to obtain the ontoness of the map \eqref{eq:mapduality2BspqSpeLip}.
\end{proof}

We finish this section with a result that will be useful to build the trace theorem in the next Section \ref{sec:TraceTHM}. This is a direct consequence of Lemma \ref{lem:GlobalchangeCoordBesov} and the definition of function spaces by restriction. To keep it short, and since they won't be of use in the next Section, we omit the mention of end-point function spaces.

\begin{lemma}\label{lem:IsomHomBesovSpacesRn+SpeLip} Let $p,q\in[1,+\infty]$, and $s\in (-1,1)$. For all $u\in \dot{\mathrm{B}}^{s}_{p,q}(\Omega)$, we have $T_\phi u\in \dot{\mathrm{B}}^{s}_{p,q}(\mathbb{R}^n_+)$ with the estimate
\begin{align*}
    \lVert T_\phi u \rVert_{\dot{\mathrm{B}}^{s}_{p,q}(\mathbb{R}^n_+)}\lesssim_{p,s,n,\partial\Omega} \left\lVert u \right\rVert_{\dot{\mathrm{B}}^{s}_{p,q}(\Omega)}\text{.}
\end{align*}
The result still holds if we replace $(\Omega,\mathbb{R}^n_+,T_\phi)$ by $(\mathbb{R}^n_+,\Omega,T_\phi^{-1})$.
\end{lemma}

%----------------------------------------------------
%------------------- Section 4 ----------------------
%----------------------------------------------------
%----------------------------------------------------
%------------------- Section 4 ----------------------
%----------------------------------------------------

\section{The trace theorem for homogeneous function spaces}\label{sec:TraceTHM}

In the previous section, an appropriate construction of homogeneous Sobolev and Besov spaces on special Lipschitz domains was given with their interpolation properties. Now, we aim to make sense of boundary values in homogeneous function spaces, which constitutes the second main focus of this present paper.

The first subsection is devoted to the study of function spaces on the boundary.

The second one concerns the transference of properties on the flat upper half-space to the bent one via the global change of coordinates. However, we want to reach the sharp range of regularity $(1/p,1+1/p)$ for the trace result. The main issues occur when $s\in[1,1+1/p)$, since we do not have more than one full gradient under the action of the global change of coordinates. To circumvent this issue, we introduce an anisotropic trace result inherited from the recent work of the author \cite[Theorem~4.7]{Gaudin2023}. This result is obtained from the $\dot{\mathrm{H}}^{s-1,p}(\mathrm{L}^p)$-maximal regularity for the Poisson semigroup $(e^{-t(-\Delta')^{1/2}})_{t\geqslant 0}$ on $\mathbb{R}^{n-1}$.

The last section is devoted to the statement of the main theorem, and several straightforward consequences.

Before we dive into the heart of the matter, the author thinks it is worth highlighting a few points:
\begin{itemize}
    \item We do not assume full knowledge of trace theory for inhomogeneous Sobolev and Besov spaces on Lipschitz domains. One only requires knowledge of the trace result for Sobolev spaces over the flat half-space $\mathrm{H}^{s,p}(\mathbb{R}^n_+)$, $1<p<+\infty$, $1/p<s\leqslant 1$, see Proposition~\ref{prop:TraceEmbeddingInhomHspRn+} and Remark \ref{rmk:traceWellDefSobSpaces} below. 
    \item From our main trace theorem below, Theorem~\ref{thm:TraceSpeLipopti}, one may deduce the general trace result for inhomogeneous function spaces over special Lipschitz domains, such as in Theorem~\ref{thm:genericinhomTrace}. As usual, the case of inhomogeneous function spaces on bounded Lipschitz domains follows from the standard localization procedure as described, \textit{e.g.}, in \cite[Proof of Theorem~1]{Ding1996}.
    \item However, to reach the trace Theorem for higher regularities, say for the space $\dot{\mathrm{H}}^{s,p}(\Omega)$ with indices $1<p<+\infty$, $1< s<1+1/p$, we use a few advanced results from operator theory and from the theory of vector-valued Sobolev spaces in Subsection~\ref{sec:AnisotropicTrcThm}. We claim that one can partially circumvent the use of such advanced technology, but circumventing it would require much more tedious work to obtain the estimates, especially the one in Corollary~\ref{cor:homAnisotropicTraceEst} below. We chose not to present this approach here to be concise and because, otherwise, it would substantially increase the length of this paper.
\end{itemize}

\subsection{Function spaces on the boundary}

To define the trace as in the case of inhomogeneous function spaces, we have to define first (homogeneous) Besov spaces on the boundary $\partial \Omega$. To do so, since we have the definition of our special Lipschitz domain $$\Omega=\left\{\left. (x',x_n)\in\mathbb{R}^{n-1}\times\mathbb{R}\right|x_n>\phi(x')\right\},$$ where $\phi\,:\,\mathbb{R}^{n-1}\longrightarrow \mathbb{R}$ is uniformly Lipschitz, we recall that the surface measure on the boundary $\partial\Omega=\left\{ (x',\phi(x'))\text{, } x'\in\mathbb{R}^{n-1}\right\}\subset\mathbb{R}^n$ is defined as
\begin{align*}
    \sigma (A) := \int_{\mathbb{R}^{n-1}} \mathbbm{1}_{A}(x',\phi(x')){\sqrt{1+\lvert\nabla' \phi(x')\rvert^2}}\,{\mathrm{d}x'}\text{,}
\end{align*}
where $A$ is any Lebesgue-measurable set of $\partial\Omega$.

We also recall that $\sigma$ is the unique Borel measure on $\partial\Omega$ so that we have the integration by parts formula
\begin{align}\label{eq:intbyPartsFormulae}
    \int_{\Omega} \partial_{x_k} u(x) v(x) \,\mathrm{d}x= -\int_{\Omega}  u(x) \partial_{x_k} v(x) \,\mathrm{d}x + \int_{\partial\Omega} u(x)v(x)\nu_k(x)\,\mathrm{d}\sigma_x\text{, } k\in\llb 1,n \rrb\text{,} 
\end{align}
provided $u,v\in\mathrm{C}^{0,1}_c(\mathbb{R}^n)$, the space of complex-valued compactly supported Lipschitz functions. And in \eqref{eq:intbyPartsFormulae}, $\nu_j$ stands for the $j$-th component of the outward unit normal of $\Omega$, defined almost everywhere on $\partial\Omega$ by 
\begin{align*}
    \nu := \frac{1}{{\sqrt{\lvert\nabla' \phi\rvert^2+1}}}(\nabla' \phi,-1)\text{.}
\end{align*}

We introduce the pushforward map from $\partial\Omega$ to $\mathbb{R}^{n-1}$ for any measurable function $u\,:\,\partial\Omega \longrightarrow \mathbb{C}$,
\begin{align}
    S_\phi u (x') := u(x',\phi(x'))\text{, } x'\in\mathbb{R}^{n-1} \text{.}
\end{align}
We also have the pullback map defined for any measurable function $v\,:\,\mathbb{R}^{n-1} \longrightarrow \mathbb{C}$,
\begin{align}\label{eq:pullbackBoundaryRn+}
    S_\phi^{-1} v (y) := v(y')\text{, } y\in\partial\Omega \text{.}
\end{align}

 To construct, and investigate the properties of, the homogeneous function spaces on the boundary, we are going to follow and elaborate the ideas given in \cite[Chapter~2,~Section~2.2]{DanchinMucha2015} and \cite[Section~2]{Ding1996}.

\begin{definition} For $p\in[1,+\infty]$, $s\in(0,1)$, for any measurable function $f$ on $\partial\Omega$, we define the following quantities
\begin{align*}
    \lVert f \rVert_{\mathrm{L}^p(\partial\Omega)}^p := \int_{\partial\Omega}\lvert f(x)\rvert^p ~\mathrm{d}\sigma_x &\text{, }\qquad\lVert f \rVert_{\dot{\mathrm{B}}^{s}_{p,p}(\partial\Omega)}^p:= \int_{\partial\Omega}\int_{\partial\Omega} \frac{\lvert f(x)-f(y) \rvert^p}{\lvert x-y\rvert^{ps+n-1}} ~\mathrm{d}\sigma_x\mathrm{d}\sigma_y\text{, }
\end{align*}
\begin{align*}
    \lVert u \rVert_{\dot{\mathrm{W}}^{1,p}(\partial\Omega)} :=\lVert S_\phi^{-1} [\nabla' S_\phi u] \rVert_{\mathrm{L}^p(\partial\Omega)},
\end{align*}
with the usual modification when $p=+\infty$. We set,
\begin{itemize}
    \item $\mathrm{L}^p(\partial\Omega):=\{\, u\,:\,\partial\Omega\longrightarrow \mathbb{C}\,\text{ meas.}\,|\,\lVert u \rVert_{\mathrm{L}^p(\partial\Omega)}<+\infty\,\}$;
    \item $\dot{\mathrm{B}}^s_{p,p}(\partial\Omega):=\{\, u\in\mathrm{L}^1_{\text{loc}}(\partial\Omega)\,|\, S_\phi u \in \eus{S}'_h(\mathbb{R}^{n-1})\,\&\,\lVert u \rVert_{\dot{\mathrm{B}}^s_{p,p}(\partial\Omega)}<+\infty\,\}$;
    \item $\dot{\mathrm{W}}^{1,p}(\partial\Omega):=\{\, u\in\mathrm{L}^1_{\text{loc}}(\partial\Omega)\,|\, S_\phi u \in \eus{S}'_h(\mathbb{R}^{n-1})\,\&\,\lVert u \rVert_{\dot{\mathrm{W}}^{1,p}(\partial\Omega)} <+\infty\,\}$;
    \item $\mathrm{C}_b^0(\partial\Omega)$ and $\mathrm{C}_0^0(\partial\Omega)$ stands respectively for the normed vector spaces of bounded continuous functions, and continuous functions that vanish at infinity. Both are endowed with the $\mathrm{L}^\infty$-norm.
    \item For $\mathcal{C}\in\{\mathrm{C}_b,\mathrm{C}_0\}$, we define
    \begin{align*}
        \mathcal{C}^{0,1}(\partial\Omega):=\Big\{\,u\in\mathcal{C}^0(\partial\Omega)\,\Big|\, \sup_{\substack{x,y\in\partial\Omega\\x\neq y}} \frac{|u(x)-u(y)|}{|x-y|}<+\infty \,\Big\}.
    \end{align*}
    We write $\dot{\mathcal{C}}^{0,1}(\partial\Omega)$ when endowed with the (semi-)norms,
    \begin{align*}
        u\mapsto \sup_{\substack{x,y\in\partial\Omega\\x\neq y}} \frac{|u(x)-u(y)|}{|x-y|}\text{, or equivalently } u\mapsto \lVert u\rVert_{\dot{\mathrm{W}}^{1,\infty}(\partial\Omega)};
    \end{align*}
    \item $\mathrm{C}_{b,h}^{0}(\partial\Omega) := \{ u\in \mathrm{C}^0_{b}(\partial\Omega)\,|\, S_\phi u \in \eus{S}'_h(\mathbb{R}^{n-1})\,\};$
    \item $\mathrm{C}_{b,h}^{0,1}(\partial\Omega) := \{ u\in \mathrm{C}^{0,1}_{b}(\partial\Omega)\,|\, S_\phi u \in \eus{S}'_h(\mathbb{R}^{n-1})\,\}$, written $\dot{\mathrm{C}}_{b,h}^{0,1}(\partial\Omega)$ when endowed with the $\dot{\mathrm{W}}^{1,\infty}$-norm.
\end{itemize}
\end{definition}

The following lemmas demonstrate that these definitions are meaningful.

\begin{lemma}\label{lem:LpspacesPushForwardBoundary} Let $p\in[1,+\infty]$ and $\mathrm{X}\in\{\mathrm{L}^p,\mathrm{C}_b^0,\mathrm{C}_{b,h}^0,\mathrm{C}_0^0\}$. The map
\begin{align*}
    S_\phi\,:\, &\mathrm{X}(\partial\Omega)\longrightarrow \mathrm{X}(\mathbb{R}^{n-1}),
\end{align*}
is well-defined and is a continuous isomorphism of normed vector spaces. 
\end{lemma}

\begin{proof} By direct computations, we obtain for all $p\in[1,+\infty]$, $u\in\mathrm{L}^p(\partial\Omega)$,
\begin{align*}
    \lVert S_\phi u \rVert_{\mathrm{L}^p(\mathbb{R}^{n-1})}\leqslant\lVert u \rVert_{\mathrm{L}^p(\partial\Omega)}\leqslant (1+\lVert \nabla'\phi\rVert_{\mathrm{L}^{\infty}(\mathbb{R}^{n-1})}^2)^{\frac{1}{2p}}\lVert S_\phi u \rVert_{\mathrm{L}^p(\mathbb{R}^{n-1})}\text{.}
\end{align*}
A similar argument applies for $S_\phi^{-1}$.

It remains to see that $S_\phi(\mathrm{C}_b^0(\partial\Omega))=\mathrm{C}_b^0(\mathbb{R}^{n-1})$ and $S_\phi(\mathrm{C}_0^0(\partial\Omega))=\mathrm{C}_0^0(\mathbb{R}^{n-1})$.
\end{proof}

The following corollary is a direct consequence of Lemma \ref{lem:LpspacesPushForwardBoundary}.
\begin{corollary}\label{cor:H1pDOmegaH1pRn-1} Let $p\in[1,+\infty)$. For $u\in \dot{\mathrm{W}}^{1,p}(\partial\Omega)$, we have $S_\phi u \in \dot{\mathrm{W}}^{1,p}(\mathbb{R}^{n-1})$, with the estimate
\begin{align*}
    \lVert u\rVert_{\dot{\mathrm{W}}^{1,p}(\partial\Omega)}\sim_{s,p,n,\partial\Omega} \lVert S_\phi u\rVert_{\dot{\mathrm{W}}^{1,p}(\mathbb{R}^{n-1})}\text{.}
\end{align*}
Conversely, if $v\in\dot{\mathrm{W}}^{1,p}(\mathbb{R}^{n-1})$, then $S_{\phi}^{-1}v \in\dot{\mathrm{W}}^{1,p}(\partial\Omega)$ with the corresponding estimate.

The result still holds with $\mathrm{C}_b^{0,1}$, $\mathrm{C}_{b,h}^{0,1}$ and $\mathrm{C}_0^{0,1}$, all endowed with the $\dot{\mathrm{W}}^{1,\infty}$-(semi-)norm.
\end{corollary}

The idea behind the definition of Besov spaces on the boundary lies in the fact that for all $u\in\mathrm{L}^1_{\text{loc}}(\mathbb{R}^{n-1})\cap\eus{S}'_h(\mathbb{R}^{n-1})$, when $s\in(0,1)$, $p\in[1,+\infty)$,
\begin{align}\label{eq:equivalentnorm}
    \lVert f \rVert_{\dot{\mathrm{B}}^{s}_{p,p}(\mathbb{R}^{n-1})}^p\sim_{p,s,n} \int_{\mathbb{R}^{n-1}}\int_{\mathbb{R}^{n-1}} \frac{\lvert f(x)-f(y) \rvert^p}{\lvert x-y\rvert^{ps+n-1}} ~\mathrm{d}x\mathrm{d}y \text{, }
\end{align}
see \cite[Theorem~2.36]{bookBahouriCheminDanchin} for a proof. The case $p=+\infty$ is treated via usual modification, giving the homogeneous H\"{o}lder (semi-)norms.

\begin{lemma}\label{lem:BsppDOmegaBsppRn-1}Let $p\in[1,+\infty]$, $s\in(0,1)$. For all $u\in \dot{\mathrm{B}}^{s}_{p,p}(\partial\Omega)$, $S_\phi u\in\dot{\mathrm{B}}^{s}_{p,p}(\mathbb{R}^{n-1})$ with the estimate
\begin{align*}
    \lVert u\rVert_{\dot{\mathrm{B}}^{s}_{p,p}(\partial\Omega)}\sim_{s,p,n,\partial\Omega} \lVert S_\phi u\rVert_{\dot{\mathrm{B}}^{s}_{p,p}(\mathbb{R}^{n-1})}
\end{align*}
Conversely, for $v\in \dot{\mathrm{B}}^{s}_{p,p}(\mathbb{R}^{n-1})$, one has $S_{\phi}^{-1}v \in\dot{\mathrm{B}}^{s}_{p,p}(\partial\Omega)$ with the corresponding estimate. 
\end{lemma}

\begin{proof} For $u\in\mathrm{L}^1_{\text{loc}}(\partial\Omega)$, if $p<+\infty$,
\begin{align*}
    \lVert  u \rVert_{\dot{\mathrm{B}}^{s}_{p,p}(\partial\Omega)}^p &= \int_{\partial\Omega}\int_{\partial\Omega} \frac{\lvert u(x)-u(y) \rvert^p}{\lvert x-y\rvert^{ps+n-1}} ~\mathrm{d}\sigma_x\mathrm{d}\sigma_y\\
    &= \int_{\mathbb{R}^{n-1}}\int_{\mathbb{R}^{n-1}} \frac{\lvert u(x',\phi(x'))-u(y',\phi(y')) \rvert^p}{\lvert (x'-y',\phi(x')-\phi(y'))\rvert^{ps+n-1}} \sqrt{\lvert \nabla' \phi(x')\rvert^2 +1}\sqrt{\lvert \nabla' \phi(y')\rvert^2 +1}~\mathrm{d}x'\mathrm{d}y'\\
    &\leqslant [1+\lVert \nabla'\phi\rVert_{\mathrm{L}^{\infty}(\mathbb{R}^{n-1})}^2] \int_{\mathbb{R}^{n-1}}\int_{\mathbb{R}^{n-1}} \frac{\lvert S_\phi u(x')- S_\phi u(y') \rvert^p}{\lvert x'-y'\rvert^{ps+n-1}} ~\mathrm{d}x'\mathrm{d}y'\\
    &\lesssim_{p,s,n,\partial\Omega}\,\lVert  S_\phi u \rVert_{\dot{\mathrm{B}}^{s}_{p,p}(\mathbb{R}^{n-1})}^p\text{.}
\end{align*}
The last estimate comes from \eqref{eq:equivalentnorm}.

For the reverse estimate, we start with \eqref{eq:equivalentnorm} then, we obtain
\begin{align*}
    \lVert  S_\phi u \rVert_{\dot{\mathrm{B}}^{s}_{p,p}(\mathbb{R}^{n-1})}^p &\lesssim_{p,s,n} \int_{\mathbb{R}^{n-1}}\int_{\mathbb{R}^{n-1}} \frac{\lvert S_\phi u(x')- S_\phi u(y') \rvert^p}{\lvert x'-y'\rvert^{ps+n-1}} ~\mathrm{d}x'\mathrm{d}y'\\
    &\lesssim_{p,s,n} [1+\lVert \nabla'\phi\rVert_{\mathrm{L}^{\infty}(\mathbb{R}^{n-1})}^2]^{\frac{(n-1)+ps}{2}}\int_{\mathbb{R}^{n-1}}\int_{\mathbb{R}^{n-1}} \frac{\lvert u(x',\phi(x'))- u(y',\phi(y')) \rvert^p}{\lvert (x'-y',\phi(x')-\phi(y'))\rvert^{ps+n-1}} ~\mathrm{d}x'\mathrm{d}y'\\
    &\lesssim_{p,s,n,\partial\Omega}\int_{\partial\Omega}\int_{\partial\Omega} \frac{\lvert u(x)- u(y) \rvert^p}{\lvert x-y\rvert^{ps+n-1}} ~\mathrm{d}\sigma_x\mathrm{d}\sigma_y =\lVert  u \rVert_{\dot{\mathrm{B}}^{s}_{p,p}(\partial\Omega)}^p \text{.}
\end{align*}
The case $p=+\infty$ is similar and left to the reader with
\begin{align*}
    \lVert  w \rVert_{\dot{\mathrm{B}}^{s}_{\infty,\infty}(\partial\Omega)} = \sup_{\substack{(x,y)\in \partial\Omega^2\text{,}\\x\neq y\text{.}}} \frac{\lvert w(x)- w(y) \rvert}{\lvert x-y\rvert^{s}}\text{.}
\end{align*}
\end{proof}

\begin{proposition}\label{prop:realInterpolationDiagonalBesovBoundary} Let $p\in[1,+\infty)$, $s\in(0,1)$. The following equalities hold with equivalence of norms
\begin{align*}
    (\mathrm{L}^p(\partial\Omega),\dot{\mathrm{W}}^{1,p}(\partial\Omega))_{s,p} &= \dot{\mathrm{B}}^{s}_{p,p}(\partial\Omega)\text{,}\\
    (\mathrm{C}^0_0(\partial\Omega),\dot{\mathrm{C}}^{0,1}_0(\partial\Omega))_{s,\infty} &= \dot{\mathrm{B}}^{s,0}_{\infty,\infty}(\partial\Omega)\text{,}\\
    (\mathrm{C}^0_{b,h}(\partial\Omega),\dot{\mathrm{C}}^{0,1}_{b,h}(\partial\Omega))_{s,\infty} &= \dot{\mathrm{B}}^{s}_{\infty,\infty}(\partial\Omega).
\end{align*}
\end{proposition}

\begin{proof}We focus on the first equality, the second and third ones admit a similar proof.

Let $u\in \mathrm{L}^{p}(\partial\Omega)+\dot{\mathrm{W}}^{1,p}(\partial\Omega)$, then for $(a,b)\in\mathrm{L}^{p}(\partial\Omega)\times\dot{\mathrm{W}}^{1,p}(\partial\Omega)$ such that $u=a+b$, by Corollary \ref{cor:H1pDOmegaH1pRn-1}, we have
\begin{align}\label{eq:SuinsumLpH1pRn-1}
    S_{\phi}u = S_{\phi}a+S_{\phi}b \in \mathrm{L}^{p}(\mathbb{R}^{n-1})+\dot{\mathrm{W}}^{1,p}(\mathbb{R}^{n-1})\text{.}
\end{align}
Therefore, by the definition of the $K$-functional and Corollary \ref{cor:H1pDOmegaH1pRn-1}, we obtain
\begin{align*}
    K(t,S_{\phi}u,\mathrm{L}^{p}(\mathbb{R}^{n-1}),\dot{\mathrm{W}}^{1,p}(\mathbb{R}^{n-1})) &\leqslant \lVert S_{\phi}a \rVert_{\mathrm{L}^{p}(\mathbb{R}^{n-1})} + t\lVert S_{\phi}b \rVert_{\dot{\mathrm{W}}^{1,p}(\mathbb{R}^{n-1})}\\
    &\lesssim_{p,n,\partial\Omega} \lVert a \rVert_{\mathrm{L}^{p}(\partial\Omega)} + t\lVert b \rVert_{\dot{\mathrm{W}}^{1,p}(\partial\Omega)}\text{.}
\end{align*}
Looking at the infimum on all such pair $(a,b)$ yields
\begin{align*}
    K(t,S_{\phi}u,\mathrm{L}^{p}(\mathbb{R}^{n-1}),\dot{\mathrm{W}}^{1,p}(\mathbb{R}^{n-1}))\lesssim_{p,n,\partial\Omega} K(t,u,\mathrm{L}^{p}(\partial\Omega),\dot{\mathrm{W}}^{1,p}(\partial\Omega))\text{.}
\end{align*}
Now, for the reverse estimate from \eqref{eq:SuinsumLpH1pRn-1}, let $(A,B)\in\mathrm{L}^{p}(\mathbb{R}^{n-1})\times\dot{\mathrm{W}}^{1,p}(\mathbb{R}^{n-1})$, such that one has $S_{\phi}u=A+B$, it follows that, by Corollary \ref{cor:H1pDOmegaH1pRn-1},
\begin{align*}
    u = S_{\phi}^{-1}A+S_{\phi}^{-1}B\in \mathrm{L}^{p}(\partial\Omega)+\dot{\mathrm{W}}^{1,p}(\partial\Omega)\text{.}
\end{align*}
So as before, we obtain,
\begin{align*}
    K(t,u,\mathrm{L}^{p}(\partial\Omega),\dot{\mathrm{W}}^{1,p}(\partial\Omega)) \lesssim_{p,n,\partial\Omega}K(t,S_{\phi}u,\mathrm{L}^{p}(\mathbb{R}^{n-1}),\dot{\mathrm{W}}^{1,p}(\mathbb{R}^{n-1})) \text{.}
\end{align*}
In the end we have obtained for all $u\in \mathrm{L}^{p}(\partial\Omega)+\dot{\mathrm{W}}^{1,p}(\partial\Omega)$ and all $t>0$:
\begin{align}\label{eq:EqKfuncRn-1DOmega}
    K(t,u,\mathrm{L}^{p}(\partial\Omega),\dot{\mathrm{W}}^{1,p}(\partial\Omega)) \sim_{p,n,\partial\Omega}K(t,S_{\phi}u,\mathrm{L}^{p}(\mathbb{R}^{n-1}),\dot{\mathrm{W}}^{1,p}(\mathbb{R}^{n-1})) \text{.}
\end{align}
Finally, if one multiplies \ref{eq:EqKfuncRn-1DOmega} by $t^{-s}$, then take its $\mathrm{L}^q_{\ast}$-norm, thanks to \ref{eq:realInterpHomBspqRn} and Lemma \ref{lem:BsppDOmegaBsppRn-1} we obtain
\begin{align*}
    \lVert u \rVert_{(\mathrm{L}^p(\partial\Omega),\dot{\mathrm{W}}^{1,p}(\partial\Omega))_{s,p}} \sim_{s,p,n,\partial\Omega} \lVert S_\phi u\rVert_{\dot{\mathrm{B}}^{s}_{p,p}(\mathbb{R}^{n-1})} \sim_{s,p,n,\partial\Omega} \lVert u\rVert_{\dot{\mathrm{B}}^{s}_{p,p}(\partial\Omega)}
\end{align*}
which ends the proof.
\end{proof}

Now, we introduce the following definition of homogeneous Besov space on the boundary with third index $q\neq p$, consistent with the case $q=p$.

\begin{definition}\label{def:HomBesovspaceBoundaryBspq} For $p,q\in[1,+\infty]$, $p<+\infty$, $s\in(0,1)$, we define
\begin{align*}
    \dot{\mathrm{B}}^{s}_{p,q}(\partial\Omega)&:= (\mathrm{L}^p(\partial\Omega),\dot{\mathrm{W}}^{1,p}(\partial\Omega))_{s,q};\\
    \dot{\mathrm{B}}^{s,0}_{\infty,q}(\partial\Omega)&:= (\mathrm{C}^0_0(\partial\Omega),\dot{\mathrm{C}}^{0,1}_0(\partial\Omega))_{s,q};\\
    \dot{\mathrm{B}}^{s}_{\infty,q}(\partial\Omega)&:= (\mathrm{C}^0_{b,h}(\partial\Omega),\dot{\mathrm{C}}^{0,1}_{b,h}(\partial\Omega))_{s,q}.
\end{align*}
When $q=+\infty$, we can replace $((\cdot,\cdot)_{s,\infty},\mathrm{B}^s)$ by $((\cdot,\cdot)_{s},\mathcal{B}^s)$.
\end{definition}

The following results are then a direct consequence of the estimate \eqref{eq:EqKfuncRn-1DOmega} and usual results for homogeneous Sobolev and Besov spaces on $\mathbb{R}^{n-1}$.

\begin{corollary}\label{lem:BspqDOmegaBspqRn-1}Let $p,q\in[1,+\infty]$, $s\in(0,1)$. For all $u\in \dot{\mathrm{B}}^{s}_{p,q}(\partial\Omega)$, one has $S_\phi u\in\dot{\mathrm{B}}^{s}_{p,q}(\mathbb{R}^{n-1})$ with the estimate
\begin{align*}
    \lVert u\rVert_{\dot{\mathrm{B}}^{s}_{p,q}(\partial\Omega)}\sim_{s,p,n,\partial\Omega} \lVert S_\phi u\rVert_{\dot{\mathrm{B}}^{s}_{p,q}(\mathbb{R}^{n-1})}.
\end{align*}
Conversely, for $v\in \dot{\mathrm{B}}^{s}_{p,q}(\mathbb{R}^{n-1})$, one has $S_{\phi}^{-1}v \in\dot{\mathrm{B}}^{s}_{p,q}(\partial\Omega)$ with corresponding estimate. When $p=+\infty$, the result still holds with  $\dot{\mathrm{B}}^{s,0}_{\infty,q}$ instead of $\dot{\mathrm{B}}^{s}_{\infty,q}$. When $q=+\infty$, one can replace $\dot{\mathrm{B}}^{s}_{p,\infty}$ by $\dot{\mathcal{B}}^{s}_{p,\infty}$.
\end{corollary}

\begin{proposition}Let $p,q\in[1,+\infty]$, $p<+\infty$, $s\in(0,1)$. The following assertions are true.
\begin{enumerate}
    \item $\dot{\mathrm{B}}^{s}_{p,q}(\partial\Omega)$ is a Banach space whenever $(\mathcal{C}_{s+\frac{1}{p},p,q})$ is satisfied.
    \item If $s\in(0,\frac{n-1}{p})$, for $\frac{1}{r}:= \frac{1}{p}-\frac{s}{n-1}$, if $q\in [1,r]$, we have the continuous embedding
    \begin{align*}
        \dot{\mathrm{B}}^{s}_{p,q}(\partial\Omega) \hookrightarrow \mathrm{L}^r(\partial\Omega)\text{.}
    \end{align*}
    \item When $p>n-1$, we have the continuous embedding
    \begin{align*}
        \dot{\mathrm{B}}^{\frac{n-1}{p}}_{p,1}(\partial\Omega) \hookrightarrow \mathrm{C}^{0}_{0}(\partial\Omega)\text{.}
    \end{align*}
\end{enumerate}
\end{proposition}

\begin{remark}From there, it is straightforward to check that one can recover usual and very well-known function spaces ${\mathrm{W}}^{1,p}(\partial\Omega) = \mathrm{L}^p(\partial\Omega)\cap\dot{\mathrm{W}}^{1,p}(\partial\Omega)$, ${\mathrm{B}}^{s}_{p,q}(\partial\Omega)=\mathrm{L}^p(\partial\Omega)\cap\dot{\mathrm{B}}^{s}_{p,q}(\partial\Omega)$, $s\in(0,1)$, $p,q\in[1,+\infty]$, $p<+\infty$.

One could also check that the intersection space $\dot{\mathrm{B}}^{s_0}_{p_0,q_0}(\partial\Omega)\cap \dot{\mathrm{B}}^{s_1}_{p_1,q_1}(\partial\Omega)$ is complete whenever $(\mathcal{C}_{s_0+\frac{1}{p_0},p_0,q_0})$ is satisfied.
\end{remark}

\subsection{Definitions and preliminary results for the trace theorem}\label{sec:AnisotropicTrcThm}

The strategy of the proof will mainly arise from a standard flattening procedure of the boundary, but then, to use anisotropic estimates, as done in \cite[Lemmas~1~\&~2]{Ding1996}. For the reader's convenience we recall, from \eqref{eq:globalChangecoordSpeLip}, \eqref{eq:globalChangecoordJac} and \eqref{eq:Tmapsglobalchangeofcoodinates}, that for any measurable function $u\,:\,\Omega \longrightarrow\mathbb{C}$,
\begin{align} \label{eq:globalChangecoordSpeLip2}
    T_\phi u(x',x_n) = u(x',x_n+\phi(x')),\quad (x',x_n)\in\mathbb{R}^{n-1}\times[0,+\infty)\text{.}
\end{align}

In this section, we will need to use Banach-valued (anisotropic) homogeneous Sobolev spaces with non-negative regularity index and with values in a (reflexive) Lebesgue space. See the previous work of the author \cite[Section~3.1]{Gaudin2023} for an elementary construction of homogeneous vector-valued Riesz potential spaces and references therein for a more general review of vector-valued Sobolev (Bessel potential) spaces and their properties.

\begin{definition}For $p\in(1,+\infty)$, provided $0\leqslant\alpha<1/p$, we define $\frac{1}{r}:=\frac{1}{p}-\alpha$,
\begin{align*}
    \dot{\mathrm{H}}^{\alpha,p}(\mathbb{R},\mathrm{L}^p(\mathbb{R}^{n-1})) := \left\{\,u\in\mathrm{L}^{r}(\mathbb{R},\mathrm{L}^p(\mathbb{R}^{n-1}))\,\big|\, (-\partial_{x_n}^2)^{\frac{\alpha}{2}}u\in\mathrm{L}^{p}(\mathbb{R},\mathrm{L}^p(\mathbb{R}^{n-1}))=\mathrm{L}^p(\mathbb{R}^{n})\,\right\}\text{.}
\end{align*}
We also define by restriction, in the sense of distributions, the corresponding space on the half line
\begin{align*}
    \dot{\mathrm{H}}^{\alpha,p}(\mathbb{R}_+,\mathrm{L}^p(\mathbb{R}^{n-1})) :=  \dot{\mathrm{H}}^{\alpha,p}(\mathbb{R},\mathrm{L}^p(\mathbb{R}^{n-1}))_{|_{\mathbb{R}_+}}\text{.}
\end{align*}
This is a Banach space with respect to the quotient norm
\begin{align*}
    \lVert u \rVert_{\dot{\mathrm{H}}^{\alpha,p}(\mathbb{R}_+,\mathrm{L}^p(\mathbb{R}^{n-1}))} := \inf_{\substack{U_{|_{\mathbb{R}_+}}=u\text{,}\\ U\in \dot{\mathrm{H}}^{\alpha,p}(\mathbb{R},\mathrm{L}^p(\mathbb{R}^{n-1}))\text{.}}} \lVert U\rVert_{\dot{\mathrm{H}}^{\alpha,p}(\mathbb{R},\mathrm{L}^p(\mathbb{R}^{n-1}))}\text{.}
\end{align*}
\end{definition}
For more details, see \cite[Subsection~3.1]{Gaudin2023}.

We recall that considering the unbounded operator $$(-\partial_{x_n},\mathrm{D}_{p}(-\partial_{x_n}))=(-\partial_{x_n},\mathrm{H}^{1,p}(\mathbb{R}_+,\mathrm{L}^p(\mathbb{R}^{n-1}))),$$ acting on ${\mathrm{L}}^{p}(\mathbb{R}_+,\mathrm{L}^p(\mathbb{R}^{n-1}))=\mathrm{L}^p(\mathbb{R}^{n}_+),$ then its fractional powers, one has the following equivalence of norms, for all $p\in(1,+\infty)$, all $\alpha\in[0,1/p)$ and all $u\in\dot{\mathrm{H}}^{\alpha,p}(\mathbb{R}_+,\mathrm{L}^p(\mathbb{R}^{n-1}))$,
\begin{align}\label{eq:EquivNormVectorValuedBesselNeumannDeriv}
    \lVert (-\partial_{x_n})^{\alpha} u \rVert_{\mathrm{L}^p(\mathbb{R}^n_+)} \sim_{p,\alpha} \lVert u \rVert_{\dot{\mathrm{H}}^{\alpha,p}(\mathbb{R}_+,\mathrm{L}^p(\mathbb{R}^{n-1}))}.
\end{align}
For more details, see \cite[Subsection~3.2]{Gaudin2023}.

\begin{lemma}\label{lem:EstimatefromHsptoAnisotropicHspLp}Let $p\in(1,+\infty)$ and $\alpha\in[0,1/p)$. For all $u\in\dot{\mathrm{H}}^{\alpha,p}(\mathbb{R}^n_+)$, we have the estimate
\begin{align*}
    \lVert u\rVert_{\dot{\mathrm{H}}^{\alpha,p}_{x_n}(\mathbb{R}_+,\mathrm{L}^p_{x'}(\mathbb{R}^{n-1}))} \lesssim_{p,\alpha,n} \lVert u\rVert_{\dot{\mathrm{H}}^{\alpha,p}(\mathbb{R}^n_+)}\text{.}
\end{align*}
\end{lemma}

\begin{proof}On $\mathbb{R}^n=\mathbb{R}^{n-1}\times\mathbb{R}$, the result follows from the $\mathrm{L}^p$-boundedness of Riesz transforms. Indeed for $u\in\dot{\mathrm{H}}^{\alpha,p}(\mathbb{R}^n)$, one can write
\begin{align*}
(-\partial_{x_n}^2)^\frac{\alpha}{2} u = (-\partial_{x_n}^2)^\frac{\alpha}{2}(-\Delta)^{-\frac{\alpha}{2}} (-\Delta)^{\frac{\alpha}{2}} u = |R_n|^\alpha (-\Delta)^{\frac{\alpha}{2}} u,
\end{align*}
where $R_n$ is the Riesz transform with respect to variable $x_n$. So that
\begin{align*}
    \lVert u\rVert_{\dot{\mathrm{H}}^{\alpha,p}_{x_n}(\mathbb{R},\mathrm{L}^p_{x'}(\mathbb{R}^{n-1}))} = \lVert (-\partial_{x_n}^2)^\frac{\alpha}{2}u\rVert_{\mathrm{L}^p(\mathbb{R}^{n})}\lesssim_{p,\alpha,n} \lVert (-\Delta)^\frac{\alpha}{2}u\rVert_{\mathrm{L}^p(\mathbb{R}^{n})}=\lVert u\rVert_{\dot{\mathrm{H}}^{\alpha,p}(\mathbb{R}^n)}\text{.}
\end{align*}
The case of the half-space follows from the definition of function spaces by restriction.
\end{proof}

For $p\in(1,+\infty)$, $s\in[1,2]$, we introduce the function space
\begin{align*}
    \mathcal{K}^{s,p}(\mathbb{R}^{n}_+):=\mathrm{H}^{s,p}(\mathbb{R}_+,\mathrm{L}^p(\mathbb{R}^{n-1}))\cap \mathrm{H}^{s-1,p}(\mathbb{R}_+,\mathrm{H}^{1,p}(\mathbb{R}^{n-1}))
\end{align*}
with its natural norm. We also introduce the homogeneous (semi-)norm
\begin{align*}
    \lVert u\rVert_{\dot{\mathcal{K}}^{s,p}(\mathbb{R}^{n}_+)} := \lVert (\partial_{x_n}u,\nabla'u)\rVert_{\dot{\mathrm{H}}^{s-1,p}(\mathbb{R}_+,\mathrm{L}^p(\mathbb{R}^{n-1}))}\text{.}
\end{align*}
We mention that $\mathcal{K}^{1,p}(\mathbb{R}^{n}_+)=\mathrm{H}^{1,p}(\mathbb{R}^n_+)$, and $\lVert \cdot \rVert_{\dot{\mathcal{K}}^{1,p}(\mathbb{R}^{n}_+)}\sim_{p,n} \lVert \nabla \cdot \rVert_{\mathrm{L}^{p}(\mathbb{R}^{n}_+)}$.

\begin{lemma}\label{lem:TphiHspKsp} For $p\in(1,+\infty)$, $s\in[1,1+1/p)$. The linear operator
\begin{align*}
    T_\phi\,:\, \mathrm{H}^{s,p}(\Omega) \longrightarrow  \mathcal{K}^{s,p}(\mathbb{R}^{n}_+)
\end{align*}
is well-defined and bounded.

Moreover, for all $u\in\mathrm{H}^{s,p}(\Omega)$ we have the homogeneous estimate
\begin{align}\label{eq:homAnisotropicEstTphi}
    \lVert T_\phi u\rVert_{\dot{\mathcal{K}}^{s,p}(\mathbb{R}^{n}_+)}\lesssim_{p,s,n,\partial\Omega} \lVert u\rVert_{\dot{\mathrm{H}}^{s,p}(\Omega)}\text{.}
\end{align}
\end{lemma}

\begin{proof} For the boundedness of $T_\phi$ from $\mathrm{H}^{s,p}(\Omega)$ to  $\mathcal{K}^{s,p}(\mathbb{R}^{n}_+)$, it suffices to follow the proof of \cite[Lemma~2]{Ding1996}.
One may check the boundedness properties
\begin{align*}
    T_\phi \,:\,    &\mathrm{H}^{1,p}(\Omega)\longrightarrow {\mathrm{H}}^{1,p}(\mathbb{R}_+,\mathrm{L}^p(\mathbb{R}^{n-1})),\\
                    &\mathrm{H}^{2,p}(\Omega)\longrightarrow {\mathrm{H}}^{2,p}(\mathbb{R}_+,\mathrm{L}^p(\mathbb{R}^{n-1})).
\end{align*}
By complex interpolation, it is implied that
\begin{align}\label{eq:BoundednessTphiKsp1}
    T_\phi \,:\,    &\mathrm{H}^{s,p}(\Omega)\longrightarrow {\mathrm{H}}^{s,p}(\mathbb{R}_+,\mathrm{L}^p(\mathbb{R}^{n-1}))
\end{align}
is well-defined and bounded for all $s\in[1,2]$. Similarly, the boundedness of
\begin{align*}
    T_\phi \,:\,    &\mathrm{H}^{1,p}(\Omega)\longrightarrow {\mathrm{L}}^{p}(\mathbb{R}_+,{\mathrm{H}}^{1,p}(\mathbb{R}^{n-1})),\\
                    &\mathrm{H}^{2,p}(\Omega)\longrightarrow {\mathrm{H}}^{1,p}(\mathbb{R}_+,{\mathrm{H}}^{1,p}(\mathbb{R}^{n-1}))\text{,}
\end{align*}
for $s\in[1,2]$, implies, by complex interpolation, that the following operator is well-defined and bounded 
\begin{align}\label{eq:BoundednessTphiKsp2}
    T_\phi \,:\,    &\mathrm{H}^{s,p}(\Omega)\longrightarrow {\mathrm{H}}^{s-1,p}(\mathbb{R}_+,{\mathrm{H}}^{1,p}(\mathbb{R}^{n-1})).
\end{align}
Thus, \eqref{eq:BoundednessTphiKsp1} and \eqref{eq:BoundednessTphiKsp2} yield the boundedness of $T_\phi$. Now, we prove the estimate~\eqref{eq:homAnisotropicEstTphi}. For $u\in\mathrm{H}^{s,p}(\Omega)\subset \dot{\mathrm{H}}^{s,p}(\Omega)$, we have $T_\phi u \in \mathcal{K}^{s,p}(\mathbb{R}^n_+)$, and since
\begin{align*}
    \partial_{x_k}(T_\phi u) = T_\phi(\partial_{x_k}u) + \partial_{x_k}\phi T_\phi(\partial_{x_n}u)\text{, }\qquad \partial_{x_n}(T_\phi u)=T_\phi(\partial_{x_n}u)\text{, } k\in\llb 1,n-1\rrb\text{,}
\end{align*}
we obtain,
\begin{align*}
    \lVert T_\phi u\rVert_{\dot{\mathcal{K}}^{s,p}(\mathbb{R}^{n}_+)} &= \lVert (\partial_{x_n} T_\phi u,\nabla' T_\phi u)\rVert_{\dot{\mathrm{H}}^{s-1,p}(\mathbb{R}_+,\mathrm{L}^p(\mathbb{R}^{n-1}))}\\
    &\leqslant \lVert T_\phi \partial_{x_n}u \rVert_{\dot{\mathrm{H}}^{s-1,p}(\mathbb{R}_+,\mathrm{L}^p(\mathbb{R}^{n-1}))} + \lVert \nabla' T_\phi u\rVert_{\dot{\mathrm{H}}^{s-1,p}(\mathbb{R}_+,\mathrm{L}^p(\mathbb{R}^{n-1}))}\\
    &\leqslant (1+(n-1)\lVert \nabla'\phi \lVert_{\mathrm{L}^{\infty}(\mathbb{R}^{n-1})})\lVert T_\phi \nabla' u \lVert_{\dot{\mathrm{H}}^{s-1,p}(\mathbb{R}_+,\mathrm{L}^p(\mathbb{R}^{n-1}))} +\lVert T_\phi \partial_{x_n} u \lVert_{\dot{\mathrm{H}}^{s-1,p}(\mathbb{R}_+,\mathrm{L}^p(\mathbb{R}^{n-1}))}\\
    &\lesssim_{p,s,n,\partial\Omega} \lVert T_\phi \nabla u \lVert_{\dot{\mathrm{H}}^{s-1,p}(\mathbb{R}_+,\mathrm{L}^p(\mathbb{R}^{n-1}))}\text{.}
\end{align*}
The estimate \eqref{eq:homAnisotropicEstTphi} is then a consequence of Lemma \ref{lem:EstimatefromHsptoAnisotropicHspLp} and Proposition \ref{prop:globalMapChangeHsp}
\begin{align*}
    \lVert T_\phi u\rVert_{\dot{\mathcal{K}}^{s,p}(\mathbb{R}^{n}_+)} &\lesssim_{p,s,n,\partial\Omega} \lVert T_\phi \nabla u \lVert_{\dot{\mathrm{H}}^{s-1,p}(\mathbb{R}_+,\mathrm{L}^p(\mathbb{R}^{n-1}))}\\ &\lesssim_{p,s,n,\partial\Omega} \lVert T_\phi \nabla u \lVert_{\dot{\mathrm{H}}^{s-1,p}(\mathbb{R}^n_+)}\\ & \lesssim_{p,s,n,\partial\Omega} \lVert \nabla u \lVert_{\dot{\mathrm{H}}^{s-1,p}(\Omega)}\\ & \lesssim_{p,s,n,\partial\Omega} \lVert u \lVert_{\dot{\mathrm{H}}^{s,p}(\Omega)}\text{.}
\end{align*}
This is the desired statement.
\end{proof}

We recall the following result on the flat half-space.
\begin{proposition}\label{prop:TraceEmbeddingInhomHspRn+} Let $p\in(1,+\infty)$, $s\in(1/p,1]$, one has the embedding
\begin{align}
    \mathrm{H}^{s,p}(\mathbb{R}^n_+) &\hookrightarrow \mathrm{C}^0_{0,x_n}(\overline{\mathbb{R}_+},\mathrm{B}^{s-1/p}_{p,p}(\mathbb{R}^{n-1})).
\end{align}
In particular, the map
\begin{align}\label{eq:traceOp0Rn+}
    \gamma_0\,:\,u\longmapsto u(\cdot,0)
\end{align}
is well-defined and bounded from $\mathrm{H}^{s,p}(\mathbb{R}^n_+)$ to $\mathrm{B}^{s-1/p}_{p,p}(\mathbb{R}^{n-1})$.
\end{proposition}

Note that, for, say, all $u\in\mathrm{C}^{0}(\overline{\Omega})$, the following equality holds pointwise:
\begin{align*}
    S_\phi^{-1}[T_\phi u(\cdot,0)] =   u_{|_{\partial\Omega}}.
\end{align*}
Here, $u_{|_{\partial\Omega}}$ stands for the restriction of $u$ to $\partial\Omega$.

\begin{definition}\label{def:TraceOperator} We define the \textbf{trace operator}, on the boundary $\partial\Omega$, still denoted by $[\cdot]_{|_{\partial\Omega}}$, as
\begin{align*}
    [\cdot]_{|_{\partial\Omega}}:= S_\phi^{-1} \gamma_0 T_\phi .
\end{align*}
The operators $ S_\phi^{-1}$, $\gamma_0$ and $T_\phi$ are defined respectively through \eqref{eq:pullbackBoundaryRn+}, \eqref{eq:traceOp0Rn+} and \eqref{eq:globalChangecoordSpeLip2}.
\end{definition}

\begin{remark}\label{rmk:traceWellDefSobSpaces} For $p\in(1,+\infty)$, $s\in(1/p,1+1/p)$, writing $\sigma=\min(1,s)$, the trace operator is well-defined on $\dot{\mathrm{H}}^{s,p}(\Omega)$. Indeed, since $\dot{\mathrm{H}}^{s,p}(\Omega)\subset {\mathrm{H}}^{s,p}(\Omega)+\mathrm{C}^0_0(\overline{\Omega})$, by Corollary~\ref{cor:IsomHomSobSpacesRn+SpeLip}, Proposition~\ref{prop:TraceEmbeddingInhomHspRn+} and Lemmas \ref{lem:LpspacesPushForwardBoundary} and \ref{lem:BsppDOmegaBsppRn-1}
\begin{align*}
    [\dot{\mathrm{H}}^{s,p}(\Omega)]_{|_{\partial\Omega}} \subset [{\mathrm{H}}^{\sigma,p}(\Omega)]_{|_{\partial\Omega}} + [\mathrm{C}^0_0(\overline{\Omega})]_{|_{\partial\Omega}}\subset {\mathrm{B}}^{\sigma-1/p}_{p,p}(\partial\Omega) + \mathrm{C}^0_0(\partial\Omega).
\end{align*}
Therefore, again, it suffices to show continuity with respect to the appropriate homogeneous Sobolev and Besov norms.
\end{remark}

\begin{lemma}\label{lem:TangentialContinuityTraceHspOmegaBsppRn-1} Let $p\in(1,+\infty)$, $s\in[1,1+1/p)$. For all $u\in{\mathrm{H}}^{s,p}(\Omega)$, we have
\begin{align*}
    T_\phi u \in \mathrm{C}_0^0(\overline{\mathbb{R}_+},\mathrm{B}^{s-1/p}_{p,p}(\mathbb{R}^{n-1}))\text{.}
\end{align*}
\end{lemma}

\begin{proof} For $u\in {\mathrm{H}}^{s,p}(\Omega)\subset {\mathrm{H}}^{1,p}(\Omega) $, then $T_\phi u \in {\mathrm{H}}^{1,p}(\mathbb{R}^n_+)\cap \mathcal{K}^{s,p}(\mathbb{R}^n_+) $, and by \cite[Theorem~4.3]{Gaudin2022},
\begin{align*}
    T_\phi u \in \mathrm{C}_0^0(\overline{\mathbb{R}_+},\mathrm{B}^{1-1/p}_{p,p}(\mathbb{R}^{n-1}))\text{,}
\end{align*}
and $T_\phi u(\cdot,0) = S_\phi [ u_{|_{\partial\Omega}}]$ in $\mathrm{B}^{1-1/p}_{p,p}(\mathbb{R}^{n-1})$.

Therefore, if we set $v(t,x'):= T_\phi u (x',t)$, $t\geqslant 0$, $x'\in\mathbb{R}^{n-1}$, we have
\begin{align*}
    &F:=\partial_t v + (\mathrm{I}-\Delta')^\frac{1}{2}v \in {\mathrm{H}}^{s-1,p}(\mathbb{R}_+,\mathrm{L}^p(\mathbb{R}^{n-1})) \subset {\mathrm{L}}^{p}(\mathbb{R}_+,\mathrm{L}^p(\mathbb{R}^{n-1}))\text{, }\\
     \text{ and } & v(0,\cdot) = \left[T_\phi u\right]_{|_{\partial\mathbb{R}^n_+}} = S_\phi [ u_{|_{\partial\Omega}}] \in \mathrm{B}^{1-1/p}_{p,p}(\mathbb{R}^{n-1})\subset \mathrm{L}^p(\mathbb{R}^{n-1}) \text{.}
\end{align*}
By uniqueness of the mild solution, for all $t\geqslant 0$
\begin{align*}
    v(t) = e^{-t(\mathrm{I}-\Delta')^\frac{1}{2}}v(0) + \int_{0}^{t} e^{-(t-s)(\mathrm{I}-\Delta')^\frac{1}{2}}F(s) \,\mathrm{d}s\text{.}
\end{align*}
However, one has $v\in \mathcal{K}^{s,p}(\mathbb{R}^n_+)$ as well as 
\begin{align*}
    t\mapsto \int_{0}^{t} e^{-(t-s)(\mathrm{I}-\Delta')^\frac{1}{2}}F(s) \,\mathrm{d}s \in \mathcal{K}^{s,p}(\mathbb{R}^n_+),
\end{align*}
by construction and maximal regularity, see \cite[Theorem~4.7]{Gaudin2023}, noting that $F\in \mathrm{H}^{s-1,p}(\mathbb{R}_+,\mathrm{L}^p(\mathbb{R}^{n-1}))$.

This implies
\begin{align*}
    t\mapsto e^{-t(\mathrm{I}-\Delta')^\frac{1}{2}}v(0)= v(t) - \int_{0}^{t} e^{-(t-s)(\mathrm{I}-\Delta')^\frac{1}{2}}F(s) \,\mathrm{d}s \in \mathcal{K}^{s,p}(\mathbb{R}^n_+).
\end{align*}

Since $(\mathrm{I}-\Delta')^\frac{1}{2}$ is invertible on $\mathrm{L}^p(\mathbb{R}^{n-1})$ with domain $\mathrm{D}_p((\mathrm{I}-\Delta')^\frac{1}{2})=\mathrm{H}^{1,p}(\mathbb{R}^{n-1})$, we claim that this implies
\begin{align*}
    v(0)\in \mathrm{B}^{s-1/p}_{p,p}(\mathbb{R}^{n-1}).
\end{align*}

Indeed, by the identity $(-\partial_{t})^{s-1} e^{-t(\mathrm{I}-\Delta')^\frac{1}{2}} = (\mathrm{I}-\Delta')^\frac{s-1}{2}e^{-t(\mathrm{I}-\Delta')^\frac{1}{2}}$, the equivalence of norms~\eqref{eq:EquivNormVectorValuedBesselNeumannDeriv}, and from the fact $t\mapsto e^{-t(\mathrm{I}-\Delta')^\frac{1}{2}}v(0)\in \mathcal{K}^{s,p}(\mathbb{R}^n_+)$, we deduce
\begin{align*}
    \lVert \partial_t e^{-(\cdot)(\mathrm{I}-\Delta')^\frac{1}{2}}v(0) \rVert_{\mathrm{H}^{s-1,p}(\mathbb{R}_+,\mathrm{L}^p(\mathbb{R}^{n-1}))} &\sim \big\lVert t\mapsto (\mathrm{I}-\Delta')^\frac{s}{2} e^{-t(\mathrm{I}-\Delta')^\frac{1}{2}}v(0) \big\rVert_{\mathrm{L}^p(\mathbb{R}_+,\mathrm{L}^p(\mathbb{R}^{n-1}))}\\
    &\sim\left(\int_{0}^{+\infty} \lVert t^{1-(1-1/p)}(\mathrm{I}-\Delta')^\frac{s}{2} e^{-t(\mathrm{I}-\Delta')^\frac{1}{2}}v(0)\big\rVert_{\mathrm{L}^p(\mathbb{R}^{n-1})}^p \frac{\mathrm{d}t}{t}\right)^\frac{1}{p}\\
    &\sim_{p,n} \big\lVert (\mathrm{I}-\Delta')^\frac{s-1}{2} v(0) \big\rVert_{(\mathrm{L}^p,\mathrm{D}_p((\mathrm{I}-\Delta')^\frac{1}{2}))_{1-\frac{1}{p},p}}\\
    &\sim_{p,n} \big\lVert (\mathrm{I}-\Delta')^\frac{s-1}{2} v(0) \big\rVert_{\mathrm{B}^{1-1/p}_{p,p}(\mathbb{R}^{n-1})}\\
    &\sim_{p,s,n} \big\lVert v(0) \big\rVert_{\mathrm{B}^{s-1/p}_{p,p}(\mathbb{R}^{n-1})}
.\end{align*}
The last estimates are consequences of \cite[Theorem~6.2.9~\&~Corollary~6.5.5]{bookHaase2006}, and the identity $(\mathrm{I}-\Delta')^\frac{s-1}{2}\mathrm{B}^{s-1/p}_{p,p}(\mathbb{R}^{n-1}) = \mathrm{B}^{1-1/p}_{p,p}(\mathbb{R}^{n-1})$.

Applying \cite[Theorem~4.7]{Gaudin2023} again, we obtain the following maximal regularity estimate
\begin{align*}
    \lVert v \rVert_{\mathrm{L}^{\infty}(\mathbb{R}_+,\mathrm{B}^{s-1/p}_{p,p}(\mathbb{R}^{n-1}))} \lesssim_{p,s,n} \lVert (\partial_t v,(\mathrm{I}-\Delta')^\frac{1}{2}v)\rVert_{ \dot{\mathrm{H}}^{s-1,p}(\mathbb{R}_+,\mathrm{L}^p(\mathbb{R}^{n-1}))} \lesssim_{p,s,n}& \lVert F\rVert_{ \dot{\mathrm{H}}^{s-1,p}(\mathbb{R}_+,\mathrm{L}^p(\mathbb{R}^{n-1}))}\\ &+ \lVert v(0) \rVert_{\mathrm{B}^{s-1/p}_{p,p}(\mathbb{R}^{n-1})}\text{,}
\end{align*}
and $v\in \mathrm{C}_0^0(\overline{\mathbb{R}_+},\mathrm{B}^{s-1/p}_{p,p}(\mathbb{R}^{n-1}))$.
\end{proof}

\begin{corollary}\label{cor:homAnisotropicTraceEst} Let $p\in(1,+\infty)$, $s\in[1,1+1/p)$. For all $u\in{\mathrm{H}}^{s,p}(\Omega)$,
\begin{align*}
    \lVert T_\phi u \rVert_{\mathrm{L}^{\infty}(\mathbb{R}_+,\dot{\mathrm{B}}^{s-1/p}_{p,p}(\mathbb{R}^{n-1}))} \lesssim_{p,s,n,\partial\Omega} \lVert T_\phi u\rVert_{\dot{\mathcal{K}}^{s,p}(\mathbb{R}^n_+)}\text{.}
\end{align*}
\end{corollary}

\begin{proof}By Lemma \ref{lem:TangentialContinuityTraceHspOmegaBsppRn-1},
\begin{align*}
    T_\phi u \in \mathrm{C}_0^0(\overline{\mathbb{R}_+},\mathrm{B}^{s-1/p}_{p,p}(\mathbb{R}^{n-1}))\subset \mathrm{C}_0^0(\overline{\mathbb{R}_+},\dot{\mathrm{B}}^{s-1/p}_{p,p}(\mathbb{R}^{n-1}))\text{.}
\end{align*}
As in the proof of Lemma \ref{lem:TangentialContinuityTraceHspOmegaBsppRn-1}, for $v(t,x'):= T_\phi u (x',t)$, $x'\in\mathbb{R}^{n-1}$ and $t\geqslant 0$, we have
\begin{align*}
    &f:=\partial_t v + (-\Delta')^\frac{1}{2}v \in {\mathrm{H}}^{s-1,p}(\mathbb{R}_+,\mathrm{L}^p(\mathbb{R}^{n-1}))\subset\dot{\mathrm{H}}^{s-1,p}(\mathbb{R}_+,\mathrm{L}^p(\mathbb{R}^{n-1}))\text{, }\\
     \text{ and } & v(0,\cdot) = \left[T_\phi u\right]_{|_{\partial\mathbb{R}^n_+}} = S_\phi [ u_{|_{\partial\Omega}}] \in \mathrm{B}^{s-1/p}_{p,p}(\mathbb{R}^{n-1})\subset \dot{\mathrm{B}}^{s-1/p}_{p,p}(\mathbb{R}^{n-1}) \text{.}
\end{align*}
Therefore, by \cite[Theorem~4.7]{Gaudin2023}, since the operator $(-\Delta')^\frac{1}{2}$ on $\mathrm{L}^p(\mathbb{R}^{n-1})$ has a homogeneous domain $\mathrm{D}_p(\mathring{(-\Delta')}^\frac{1}{2})=\dot{\mathrm{H}}^{1,p}(\mathbb{R}^{n-1})$, we have the following maximal regularity estimate
\begin{align*}
    \lVert v \rVert_{\mathrm{L}^{\infty}(\mathbb{R}_+,\dot{\mathrm{B}}^{s-1/p}_{p,p}(\mathbb{R}^{n-1}))} \lesssim_{p,s,n} \lVert (\partial_t v,(-\Delta')^\frac{1}{2}v)\rVert_{ \dot{\mathrm{H}}^{s-1,p}(\mathbb{R}_+,\mathrm{L}^p(\mathbb{R}^{n-1}))}\sim_{p,s,n} \lVert v\rVert_{\dot{\mathcal{K}}^{s,p}(\mathbb{R}^n_+)}\text{,}
\end{align*}
as desired.
\end{proof}

\subsection{The trace theorem and related results}

We state and prove the trace theorem. We mention for the reader that the definition of the trace operator was given above, see Definition \ref{def:TraceOperator}.

\begin{theorem}\label{thm:TraceSpeLipopti}Let $p,q\in[1,+\infty]$, $p<+\infty$, $s\in(1/p,1+1/p)$. The following statements are true:
\begin{enumerate}[label=($\roman*$)]
    \item provided $1<p<+\infty$, for all $u\in \dot{\mathrm{H}}^{s,p}(\Omega)$, we have $u_{|_{\partial\Omega}}\in \dot{\mathrm{B}}^{s-\frac{1}{p}}_{p,p}(\partial\Omega)$ with the estimate
    \begin{align*}
        \lVert  u_{|_{\partial\Omega}} \rVert_{\dot{\mathrm{B}}^{s-\frac{1}{p}}_{p,p}(\partial\Omega)}  \lesssim_{p,s,n,\partial\Omega} \lVert u \rVert_{\dot{\mathrm{H}}^{s,p}(\Omega)} ;
    \end{align*}
    \item for all $u\in \dot{\mathrm{B}}^{s}_{p,q}(\Omega)$, we have $u_{|_{\partial\Omega}}\in \dot{\mathrm{B}}^{s-\frac{1}{p}}_{p,q}(\partial\Omega)$ with the estimate
    \begin{align*}
        \lVert u_{|_{\partial\Omega}} \rVert_{\dot{\mathrm{B}}^{s-\frac{1}{p}}_{p,q}(\partial\Omega)} \lesssim_{p,s,n,\partial\Omega} \lVert u \rVert_{\dot{\mathrm{B}}^{s}_{p,q}(\Omega)}  ;
    \end{align*}
    \item for all $u\in \dot{\mathrm{B}}^{s,0}_{\infty,q}(\Omega)$, we have $u_{|_{\partial\Omega}}\in \dot{\mathrm{B}}^{s,0}_{\infty,q}(\partial\Omega)$ with the estimate
    \begin{align*}
        \lVert u_{|_{\partial\Omega}} \rVert_{\dot{\mathrm{B}}^{s}_{\infty,q}(\partial\Omega)} \lesssim_{p,s,n,\partial\Omega} \lVert u \rVert_{\dot{\mathrm{B}}^{s}_{\infty,q}(\Omega)}  ;
    \end{align*}
    \item provided $k\in\llb0,1\rrb$, for all $u\in \dot{\mathrm{B}}^{k+\frac{1}{p}}_{p,1}(\Omega)$, we have $u_{|_{\partial\Omega}}\in\dot{\mathrm{W}}^{k,p}(\partial\Omega)$ with the estimate
    \begin{align*}
        \lVert u_{|_{\partial\Omega}} \rVert_{\dot{\mathrm{W}}^{k,p}(\partial\Omega)} \lesssim_{p,n,\partial\Omega} \lVert u \rVert_{\dot{\mathrm{B}}^{k+\frac{1}{p}}_{p,1}(\Omega)}.
    \end{align*}
    \item provided $k\in\llb0,1\rrb$, for all $u\in \dot{\mathrm{W}}^{k+1,1}(\Omega)$, we have $u_{|_{\partial\Omega}}\in\dot{\mathrm{W}}^{k,1}(\partial\Omega)$ with the estimate
    \begin{align*}
        \lVert u_{|_{\partial\Omega}} \rVert_{\dot{\mathrm{W}}^{k,1}(\partial\Omega)} \lesssim_{n,\partial\Omega} \lVert u \rVert_{\dot{\mathrm{W}}^{k+1,1}(\Omega)}.
    \end{align*}
\end{enumerate}
Moreover, when $p=+\infty$, the estimates in \textit{(ii)} and \textit{(iv)} remain true while it is not clear that $S_{\phi}[u_{|_{\partial\Omega}}]$ belongs to $\mathcal{S}'_h(\mathbb{R}^{n-1})$.
\end{theorem}

\begin{proof}
\textbf{Step 1:} homogeneous Sobolev spaces with $1<p<+\infty$.
    
    Let $u\in \mathrm{H}^{s,p}(\Omega)$. We assume first that $s\in(\frac{1}{p},1]$. By Proposition \ref{prop:globalMapChangeHsp}, we have $T_\phi u \in \mathrm{H}^{s,p}(\mathbb{R}^n_+)$. The standard trace theorem with homogeneous estimates \cite[Theorem~4.3]{Gaudin2022}\footnote{which is just Proposition \ref{prop:TraceEmbeddingInhomHspRn+} to which was applied a dilation argument.} yields that, $$x_n\mapsto T_\phi u(\cdot,x_n)\in \mathrm{C}^0_b(\overline{\mathbb{R}_+},\dot{\mathrm{B}}^{s-1/p}_{p,p}(\mathbb{R}^{n-1})),$$ with the estimates
    \begin{align*}
        \lVert   T_\phi u (\cdot,0) \rVert_{\dot{\mathrm{B}}^{s-\frac{1}{p}}_{p,p}(\mathbb{R}^{n-1})}  \lesssim_{s,p,n} \lVert T_\phi u \rVert_{\dot{\mathrm{H}}^{s,p}(\mathbb{R}^n_+)}\lesssim_{s,p,n} \lVert u \rVert_{\dot{\mathrm{H}}^{s,p}(\Omega)}\text{.}
    \end{align*}

    Now, we consider $u\in \mathrm{H}^{s,p}(\Omega)$, with $s\in[1,1+1/p)$, it follows from the successive use of Lemma~\ref{lem:EstimatefromHsptoAnisotropicHspLp} and Corollary~\ref{cor:homAnisotropicTraceEst}, that,
    \begin{align*}
        \lVert  T_\phi u(\cdot,0) \rVert_{\dot{\mathrm{B}}^{s-\frac{1}{p}}_{p,p}(\mathbb{R}^{n-1})} \lesssim_{s,p,n} \lVert T_\phi u \rVert_{\dot{\mathcal{K}}^{s,p}(\mathbb{R}^n_+)}\lesssim_{s,p,n,\partial\Omega} \lVert u \rVert_{\dot{\mathrm{H}}^{s,p}(\Omega)}\text{.}
    \end{align*}
    We want to relax the estimate for all $u\in\dot{\mathrm{H}}^{s,p}(\Omega)$. We cannot conclude directly with a density argument, since the involved spaces are not necessarily complete.

    Let $1/p<s<1+1/p$, $u\in\dot{\mathrm{H}}^{s,p}(\Omega)$. We consider $U\in \dot{\mathrm{H}}^{s,p}(\mathbb{R}^n)$ such that $U_{|_\Omega}=u$, we also do have $U_{|_{\partial\Omega}}=u_{|_{\partial\Omega}}$. We set
    $U_+:=[\mathrm{I}-\dot{S}_0]U$, and
    $U_-:=\dot{S}_0U$, $U_{N,-}:=[\dot{S}_0-\dot{S}_{-N}]U$. we have for all $N\in\mathbb{N}$,
\begin{align*}
    U_{-,N},\,U_{-}\in\mathrm{C}^1_0(\mathbb{R}^n)\text{ and }U_{-,N}, U_+\in\mathrm{H}^{s,p}(\mathbb{R}^n),
\end{align*}
with the property
\begin{align*}
    \lVert U_{-}-U_{-,N}\rVert_{\mathrm{L}^{\infty}(\mathbb{R}^n)}\xrightarrow[N\rightarrow+\infty]{}0.
\end{align*}
In particular, for all $N\in\mathbb{N}$,
\begin{align*}
    &T_{\phi}U_{-,N},T_{\phi}U_{-}\in\mathrm{C}^{0,1}_0(\mathbb{R}^n)\text{, and }T_{\phi}U_{-,N}, T_{\phi}U_+\in\mathrm{H}^{s,p}(\mathbb{R}^n),\\
    &T_{\phi}U_{-,N}(\cdot,0),T_{\phi}U_{-}(\cdot,0)\in\mathrm{C}^{0,1}_0(\mathbb{R}^{n-1})\text{, and }T_{\phi}U_{-,N}(\cdot,0), T_{\phi}U_+(\cdot,0)\in\mathrm{B}^{s-1/p}_{p,p}(\mathbb{R}^{n-1}),
\end{align*}
with the property
\begin{align*}
    \lVert T_{\phi}U_{-}(\cdot,0)-T_{\phi}U_{-,N}(\cdot,0)\rVert_{\mathrm{L}^{\infty}(\mathbb{R}^{n-1})}\xrightarrow[N\rightarrow+\infty]{}0.
\end{align*}
It has to be mentioned that $T_{\phi}u(\cdot,0)\in\mathcal{S}'_h(\mathbb{R}^{n-1})$.

Since $(T_{\phi}U_{-,N}(\cdot,0))_{N\in\mathbb{N}}$ converges everywhere to $T_{\phi}U_{-}(\cdot,0)$, by the Fatou Lemma, and the previous estimates
    \begin{align*}
        \lVert  T_\phi u(\cdot,0) \rVert_{\dot{\mathrm{B}}^{s-\frac{1}{p}}_{p,p}(\mathbb{R}^{n-1})} = \lVert  T_\phi U(\cdot,0) \rVert_{\dot{\mathrm{B}}^{s-\frac{1}{p}}_{p,p}(\mathbb{R}^{n-1})} &\leqslant \liminf_{N\rightarrow+\infty}\, \lVert  T_\phi U_{-,N}(\cdot,0) \rVert_{\dot{\mathrm{B}}^{s-\frac{1}{p}}_{p,p}(\mathbb{R}^{n-1})} \\ & \qquad \qquad \qquad + \lVert  T_\phi U_+(\cdot,0) \rVert_{\dot{\mathrm{B}}^{s-\frac{1}{p}}_{p,p}(\mathbb{R}^{n-1})}\\
        &\lesssim_{s,p,n,\partial\Omega} \lVert U \rVert_{\dot{\mathrm{H}}^{s,p}(\mathbb{R}^n)}\text{.}
    \end{align*}
    Taking the infimum on all such $U$ yields
    \begin{align*}
        \lVert  T_\phi u(\cdot,0) \rVert_{\dot{\mathrm{B}}^{s-\frac{1}{p}}_{p,p}(\mathbb{R}^{n-1})} \lesssim_{s,p,n,\partial\Omega} \lVert u \rVert_{\dot{\mathrm{H}}^{s,p}(\Omega)}\text{.}
    \end{align*}
    But for almost every $x'\in\mathbb{R}^{n-1}$, we recall that
    \begin{align*}
        T_\phi u (x',0) = u(x',0+\phi(x'))= S_\phi[u_{|_{\partial\Omega}}](x')\text{.}
    \end{align*}
    Thus, one applies Lemma \ref{lem:BsppDOmegaBsppRn-1}:
    \begin{align*}
        \lVert  u_{|_{\partial\Omega}} \rVert_{\dot{\mathrm{B}}^{s-\frac{1}{p}}_{p,p}(\partial\Omega)} \lesssim_{p,s,n,\partial\Omega}\lVert  S_\phi[u_{|_{\partial\Omega}}] \rVert_{\dot{\mathrm{B}}^{s-\frac{1}{p}}_{p,p}(\mathbb{R}^{n-1})} = \lVert T_{\phi}u(\cdot,0) \rVert_{\dot{\mathrm{B}}^{s-\frac{1}{p}}_{p,p}(\mathbb{R}^{n-1})}  \lesssim_{s,p,n,\partial\Omega} \lVert u \rVert_{\dot{\mathrm{H}}^{s,p}(\Omega)}\text{.}
    \end{align*}

    Hence, we have obtained \textit{(i)}.
    
    \textbf{Step 2:} homogeneous Besov spaces with $1<p<+\infty$.
    
    For $u\in \dot{\mathrm{B}}^{\frac{1}{p}}_{p,1}(\Omega)$, by Lemma \ref{lem:IsomHomBesovSpacesRn+SpeLip} and \cite[Theorem~4.3]{Gaudin2022}, it holds
    \begin{align*}
        T_\phi u \in \dot{\mathrm{B}}^{\frac{1}{p}}_{p,1}(\mathbb{R}^n_+)\subset \mathrm{C}^0_0(\overline{\mathbb{R}_+},\mathrm{L}^p(\mathbb{R}^{n-1}))
    \end{align*}
    with the estimates
    \begin{align}\label{eq:estimateTraceBesovinLebesgue}
        \lVert u_{|_{\partial\Omega}}\rVert_{\mathrm{L}^p(\partial\Omega)}\sim_{p,n,\partial\Omega}\lVert T_\phi u(\cdot,0)\rVert_{\mathrm{L}^p(\mathbb{R}^{n-1})} \lesssim_{p,n,\partial\Omega} \lVert T_\phi u\rVert_{\dot{\mathrm{B}}^{\frac{1}{p}}_{p,1}(\mathbb{R}^n_+)} \lesssim_{p,n,\partial\Omega} \lVert u\rVert_{\dot{\mathrm{B}}^{\frac{1}{p}}_{p,1}(\Omega)}\text{.}
    \end{align}
    
    Now, for $u\in \dot{\mathrm{B}}^{1+\frac{1}{p}}_{p,1}(\Omega)$, since $\nabla u \in \dot{\mathrm{B}}^{\frac{1}{p}}_{p,1}(\Omega)$, we may use the estimate \eqref{eq:estimateTraceBesovinLebesgue},
    \begin{align*}
        \lVert [\nabla u]_{|_{\partial\Omega}}\rVert_{\mathrm{L}^p(\partial\Omega)} \lesssim_{p,n,\partial\Omega} \lVert \nabla u\rVert_{\dot{\mathrm{B}}^{\frac{1}{p}}_{p,1}(\Omega)} \lesssim_{p,n,\partial\Omega} \lVert u\rVert_{\dot{\mathrm{B}}^{1+\frac{1}{p}}_{p,1}(\Omega)} \text{.}
    \end{align*}
    But one may check that
    \begin{align*}
        \lVert u_{|_{\partial\Omega}}\rVert_{\dot{\mathrm{W}}^{1,p}(\partial\Omega)}=\lVert S_\phi^{-1}\nabla'[S_\phi u_{|_{\partial\Omega}}]\rVert_{\mathrm{L}^p(\partial\Omega)}\lesssim_{p,n,\partial\Omega}\lVert [\nabla u]_{|_{\partial\Omega}}\rVert_{\mathrm{L}^p(\partial\Omega)} \lesssim_{p,n,\partial\Omega}\lVert u\rVert_{\dot{\mathrm{B}}^{1+\frac{1}{p}}_{p,1}(\Omega)} \text{.}
    \end{align*}
    Real interpolation, Theorem \ref{thm:RealInterpHomSpacesSpeLip}, yields \textit{(ii)} and \textit{(iv)} whenever $1<p<+\infty$.
    
    \textbf{Step 3:} homogeneous Sobolev and Besov spaces with $p=1$.

    Let $u\in\mathcal{S}(\mathbb{R}^n)$, one has
    \begin{align*}
        \lVert  u_{|_{\partial\Omega}}\rVert_{\mathrm{L}^1(\partial\Omega)} \lesssim_{n,\partial\Omega} \lVert T_\phi u(\cdot,0)\rVert_{\mathrm{L}^1(\mathbb{R}^{n-1})} \lesssim_{n,\partial\Omega} \lVert \partial_{x_n} T_\phi u\rVert_{\mathrm{L}^{1}(\mathbb{R}^n_+)} \lesssim_{n,\partial\Omega} \lVert u\rVert_{\dot{\mathrm{W}}^{1,1}(\Omega)},
    \end{align*}
    where we recall that $\partial_{x_n}$ and $ T_\phi$ commute. Similarly,
    \begin{align*}
        \lVert  u_{|_{\partial\Omega}}\rVert_{\dot{\mathrm{W}}^{1,1}(\partial\Omega)} \lesssim_{n,\partial\Omega} \lVert \nabla' T_\phi u(\cdot,0)\rVert_{\mathrm{L}^1(\mathbb{R}^{n-1})} \lesssim_{n,\partial\Omega} \lVert \partial_{x_n} \nabla' T_\phi u\rVert_{\mathrm{L}^{1}(\mathbb{R}^n_+)} \lesssim_{n,\partial\Omega} \lVert u\rVert_{\dot{\mathrm{W}}^{2,1}(\Omega)}.
    \end{align*}
    Again, this follows from the fact that $\partial_{x_n}$ commutes with $T_\phi$ and the multiplication by $\phi$ (that does not depend on $x_n$). Each space involved is complete, therefore \textit{(v)} holds by density.

    For $k\in\llb 0,1\rrb$, one has $\dot{\mathrm{B}}^{k+1}_{1,1}(\Omega)\hookrightarrow \dot{\mathrm{W}}^{k+1,1}(\Omega)$ which yields \textit{(iv)} when $p=1$. Consequently, by real interpolation, Theorem \ref{thm:RealInterpHomSpacesSpeLip}, \textit{(ii)} holds for $p=1$.

    \textbf{Step 4:} homogeneous Besov spaces with $p=+\infty$.

    Let $k=0,1$, and $u\in\dot{\mathrm{B}}^{k}_{\infty,1}(\Omega)$, by the definition of function spaces by restriction, one has $u\in\mathrm{C}^{k}_{b,h}(\overline{\Omega})$ and then $T_\phi u \in\mathrm{C}^{0,k}_{b}(\overline{\mathbb{R}^{n}_+})$. It follows that
    \begin{align*}
        \lVert (\Delta_jT_\phi u(\cdot,0))_{j\in\mathbb{Z}}\rVert_{\ell^\infty_k(\mathbb{Z},\mathrm{L}^\infty(\mathbb{R}^{n-1}))} \leqslant \lVert T_\phi u(\cdot,0) \rVert_{\dot{\mathrm{W}}^{k,\infty}(\mathbb{R}^{n-1})} &\leqslant \lVert T_\phi u \rVert_{\dot{\mathrm{W}}^{k,\infty}(\mathbb{R}^{n}_+)}\\ &\lesssim_{n,\partial\Omega}\lVert  u \rVert_{\dot{\mathrm{W}}^{k,\infty}(\Omega)} \lesssim_{n,\partial\Omega}  \lVert  u \rVert_{\dot{\mathrm{B}}^{k}_{\infty,1}(\Omega)}.
    \end{align*}
    This yields \textit{(iv)} when $p=+\infty$.

    The map $u\mapsto (\Delta_jT_\phi u(\cdot,0))_{j\in\mathbb{Z}}$ is then well-defined and bounded as a linear operator from $\dot{\mathrm{B}}^{k}_{\infty,1}(\Omega)$ to $\ell^\infty_k(\mathbb{Z},\mathrm{L}^\infty(\mathbb{R}^{n-1}))$.
    
    By real interpolation and Lemma \ref{lem:EmbeddingInterpHomSobspacesSpeLip}, one obtains for all $u\in\dot{\mathrm{B}}^{s}_{\infty,q}(\Omega)\hookrightarrow \big(\dot{\mathrm{B}}^{0}_{\infty,1}(\Omega),\dot{\mathrm{B}}^{1}_{\infty,1}(\Omega)\big)_{s,q}$, $s\in(0,1)$, $q\in[1,+\infty]$,
    \begin{align*}
        \lVert (\Delta_jT_\phi u(\cdot,0))_{j\in\mathbb{Z}}\rVert_{\ell^q_s(\mathbb{Z},\mathrm{L}^\infty(\mathbb{R}^{n-1}))} \lesssim_{p,s,n}^{\partial\Omega}  \lVert  u \rVert_{\big(\dot{\mathrm{B}}^{0}_{\infty,1}(\Omega),\dot{\mathrm{B}}^{1}_{\infty,1}(\Omega)\big)_{s,q}}\lesssim_{p,s,n}^{\partial\Omega}  \lVert  u \rVert_{\dot{\mathrm{B}}^{s}_{\infty,q}(\Omega)}
    \end{align*}
    We recall that, since $s>0$, by Lemma \ref{lem:structLemBesovSpeLip}, $\dot{\mathrm{B}}^{s}_{\infty,q}(\Omega) = {\mathrm{B}}^{s}_{\infty,q}(\Omega)\cap \mathcal{S}'_h(\mathbb{R}^n)_{|_{\Omega}}$. Therefore, $u_{|_{\partial\Omega}}\in{\mathrm{B}}^{s}_{\infty,q}(\partial\Omega)$ is already well-defined, so the estimate
    \begin{align*}
        \lVert  u_{|_{\partial\Omega}} \rVert_{\dot{\mathrm{B}}^{s}_{\infty,q}(\partial\Omega)} \lesssim_{p,s,n}^{\partial\Omega}  \lVert  u \rVert_{\dot{\mathrm{B}}^{s}_{\infty,q}(\Omega)}
    \end{align*}
    holds, even if it is not clear that $S_\phi[u_{|_{\partial\Omega}}]$ belongs to $\mathcal{S}'_h(\mathbb{R}^{n-1})$. This yields the case $p=+\infty$. In the case of elements $u \in \dot{\mathrm{B}}^{s,0}_{\infty,q}(\Omega) \subset {\mathrm{B}}^{s,0}_{\infty,q}(\Omega) \subset \mathrm{C}^{0}_0(\overline{\Omega})$, one has $S_\phi[u_{|_{\partial\Omega}}]\in\mathrm{C}^{0}_0(\mathbb{R}^{n-1})\subset\mathcal{S}'_h(\mathbb{R}^{n-1})$. Therefore, \textit{(iii)} holds.
\end{proof}

We state interesting consequences on regularity and integrability properties of traces in the case of intersection spaces, and we provide the classical identification of function spaces for functions that vanish on the boundary. The proofs are very similar to the corresponding ones in \cite[Section~4B]{Gaudin2022} for $\Omega$ to be the flat half-space.

\begin{proposition}\label{prop:TraceSpeLipIntersec1}Let $p,q\in[1,+\infty]$, and $\theta\in(0,1)$, $-1+\frac{1}{p} < s_0 < \frac{1}{p} < s_1<1+\frac{1}{p}$ such that
\begin{align*}
    \frac{1}{p}=(1-\theta)s_0+\theta s_1\text{.}
\end{align*}
\begin{enumerate}[label=($\roman*$)]
    \item Provided $1<p<+\infty$, for all $u\in \dot{\mathrm{H}}^{s_0,p}(\Omega)\cap \dot{\mathrm{H}}^{s_1,p}(\Omega)$, we have $u_{|_{\partial\Omega}} \in \mathrm{B}^{s_1-\frac{1}{p}}_{p,p}(\partial\Omega)$, with the estimate
    \begin{align*}
        \lVert  u_{|_{\partial\Omega}} \rVert_{\mathrm{B}^{s_1-\frac{1}{p}}_{p,p}(\partial\Omega)}  \lesssim_{s_0,s_1,p,n,\partial\Omega} \lVert u \rVert_{\dot{\mathrm{H}}^{s_0,p}(\Omega)}^{1-\theta}\lVert u \rVert_{\dot{\mathrm{H}}^{s_1,p}(\Omega)}^\theta +  \lVert u \rVert_{\dot{\mathrm{H}}^{s_1,p}(\Omega)} \text{. }
    \end{align*}
    \item For all $u\in \dot{\mathrm{B}}^{s_0}_{p,q}(\Omega)\cap \dot{\mathrm{B}}^{s_1}_{p,q}(\Omega)$, we have $u_{|_{\partial\Omega}} \in \mathrm{B}^{s_1-\frac{1}{p}}_{p,q}(\partial\Omega)$, with the estimate
    \begin{align*}
        \lVert  u_{|_{\partial\Omega}} \rVert_{\mathrm{B}^{s_1-\frac{1}{p}}_{p,q}(\partial\Omega)}  \lesssim_{s_0,s_1,p,n,\partial\Omega} \lVert u \rVert_{\dot{\mathrm{B}}^{s_0}_{p,q}(\Omega)}^{1-\theta}\lVert u \rVert_{\dot{\mathrm{B}}^{s_1}_{p,q}(\Omega)}^\theta +  \lVert u \rVert_{\dot{\mathrm{B}}^{s_1}_{p,q}(\Omega)} \text{. }
    \end{align*}
\end{enumerate}
\end{proposition}

\begin{proof} We mention the result \cite[Proposition~4.4]{Gaudin2022} for the case $\Omega=\mathbb{R}^n_+$, where the proof only relies on interpolation inequalities and the appropriate trace estimates. Everything has been made in order to recover the same interpolation inequalities, the result then follows from Theorem~\ref{thm:TraceSpeLipopti}.
\end{proof}

\begin{lemma}\label{lem:IdentHsp0and0trace} Let $p_j\in(1,+\infty)$, $s_j\in(1/p_j,1+1/p_j)$, $j\in\{0,1\}$. We have the following canonical isomorphism of normed spaces
\begin{align*}
    \{\, u\in[\dot{\mathrm{H}}^{s_0,p_0}\cap\dot{\mathrm{H}}^{s_1,p_1}](\Omega) \,|\, u_{|_{\partial\Omega}} =0 \,\} \simeq [\dot{\mathrm{H}}^{s_0,p_0}_0\cap\dot{\mathrm{H}}^{s_1,p_1}_0](\Omega) \text{. }
\end{align*}
The result still holds replacing $\dot{\mathrm{H}}^{s_j,p_j}$ by $\dot{\mathrm{B}}^{s_j}_{p_j,q_j}$, $\dot{\mathcal{B}}^{s_j}_{p_j,\infty}$, $\dot{\mathrm{B}}^{s_j,0}_{\infty,q_j}$, $\dot{\mathrm{W}}^{1,1}$, ${\mathrm{C}}^{0}_0$, $p_j,q_j\in[1,+\infty]$, $j\in\{0,1\}$.
\end{lemma}

\begin{proof} We follow \cite[Lemma~4.8]{Gaudin2022}. Let $u\in [\dot{\mathrm{H}}^{s_0,p_0}\cap\dot{\mathrm{H}}^{s_1,p_1}](\Omega,\mathbb{C})$ such that $u_{|_{\partial\Omega}} =0$, then for all $\phi\in [\dot{\mathrm{H}}^{1-s_j,p_j'}\cap\mathcal{S}](\mathbb{R}^n,\mathbb{C}^n)$, we have
\begin{align*}
    \int_{\Omega} \nabla u \cdot \phi = - \int_{\Omega} u\,\div (\phi) \text{. }
\end{align*}
So that introducing the extensions by $0$ to $\mathbb{R}^n$,  $\tilde{u}$ and $\widetilde{\nabla u}$,
\begin{align*}
    \int_{\mathbb{R}^n} \widetilde{\nabla u} \cdot \phi = \int_{\Omega} \nabla u \cdot \phi = - \int_{\Omega} u\,\div (\phi)  = - \int_{\mathbb{R}^n} \tilde{u}\,\div (\phi)  = \big\langle \nabla\tilde{ u},\phi \big\rangle_{\mathbb{R}^n}\text{. }
\end{align*}
Therefore, for all $\phi\in [\dot{\mathrm{H}}^{1-s_j,p_j'}\cap\eus{S}](\mathbb{R}^n,\mathbb{C}^n)$,
\begin{align*}
    \int_{\mathbb{R}^n} \widetilde{\nabla u} \cdot \phi= \big\langle \nabla\tilde{ u},\phi \big\rangle_{\mathbb{R}^n}\text{. }
\end{align*}
Hence $\widetilde{\nabla u} = \nabla\tilde{ u}$ in $\eus{S}'(\mathbb{R}^n,\mathbb{C}^n)$. Thus, by Proposition \ref{prop:dualityRieszpotential} and Corollary \ref{cor:HomSobolevMultiplierSpeLip}, we deduce that
\begin{align*}
    \big\lvert\big\langle\nabla\tilde{ u},\phi \big\rangle _{\mathbb{R}^n}\big\rvert &\leqslant \lVert \phi \rVert_{\dot{\mathrm{H}}^{1-s_j,p_j'}(\mathbb{R}^n)}\lVert \widetilde{\nabla u} \rVert_{\dot{\mathrm{H}}^{s_j-1,p_j}(\mathbb{R}^n)}\\
    &\lesssim_{p_j,n,s_j}^{\partial\Omega} \lVert {\phi} \rVert_{\dot{\mathrm{H}}^{1-s_j,p_j'}(\mathbb{R}^n)}\lVert {\nabla u} \rVert_{\dot{\mathrm{H}}^{s_j-1,p_j}(\Omega)}\\
    &\lesssim_{p_j,n,s_j}^{\partial\Omega} \lVert {\phi} \rVert_{\dot{\mathrm{H}}^{1-s_j,p_j'}(\mathbb{R}^n)}\lVert {u} \rVert_{\dot{\mathrm{H}}^{s_j,p_j}(\Omega)}\text{. }
\end{align*}
In the case of inhomogeneous function spaces, the proof should conclude here.
In order to conclude the case of homogeneous function spaces, it remains to show, thanks to Proposition \ref{prop:dualityRieszpotential} and Proposition \ref{prop:EqNormNablakHsp}, that $\Tilde{u}$ belongs to $\mathcal{S}'_h(\mathbb{R}^n)$. To do so, we present three different strategies here, so that the proof for the remaining function spaces can benefit from at least one of them.

If $s_j<n/p_j$, for $j=0$ or $j=1$, then at least one space is complete and there is nothing to achieve. If $s_j>n/p_j$, for $j=0$ or $j=1$, Sobolev embeddings yield 
\begin{align*}
    u \in \dot{\mathrm{H}}^{s_j,p_j}(\Omega) \subset {\mathrm{H}}^{s_j,p_j}(\Omega) + \mathrm{C}^{0}_0(\overline{\Omega}) \subset \mathrm{C}_0^{0}(\overline{\Omega}).
\end{align*}
Since $u$ is continuous up to the boundary with trace $0$, it holds that
\begin{align*}
    \tilde{u}\in\mathrm{C}_0^0(\mathbb{R}^n)\subset \mathcal{S}'_h(\mathbb{R}^n).
\end{align*}
Now, it remains to deal with the case where one has $s_j=n/p_j$ for $j=0$ and $j=1$. We set $\sigma_j:=\min(1,n/p_j)>1/p_j$. We do have
\begin{align*}
    u \in \dot{\mathrm{H}}^{n/p_j,p_j}(\Omega) \subset {\mathrm{H}}^{\sigma_j,p_j}(\Omega) + \mathrm{C}^{0}_0(\overline{\Omega}).
\end{align*}
We consider $(a,b)\in {\mathrm{H}}^{\sigma_j,p_j}(\Omega) + \mathrm{C}^{0}_0(\overline{\Omega})$ such that $u=a+b$, since $u_{|_{\partial\Omega}}=0$, by Lemma \ref{lem:ProjDirBinHomSpaces}, we have the equality
\begin{align*}
    u = \mathcal{P}_\mathcal{D} u = \mathcal{P}_\mathcal{D} a + \mathcal{P}_\mathcal{D} b.
\end{align*}
We do have $\mathcal{P}_\mathcal{D} a\in\mathrm{H}^{\sigma_j,p_j}(\Omega)$, $\mathcal{P}_\mathcal{D} b\in \mathrm{C}^0_0(\overline{\Omega})$ with
\begin{align*}
    \mathcal{P}_\mathcal{D} a_{|_{\partial\Omega}}=\mathcal{P}_\mathcal{D} b_{|_{\partial\Omega}}=0.
\end{align*}
Consequently, since $\sigma_j>1/p_j$,
\begin{align*}
    \tilde{u} = \widetilde{\mathcal{P}_\mathcal{D} a}+\widetilde{\mathcal{P}_\mathcal{D} b} \in \mathrm{H}^{\sigma_j,p_j}_0(\Omega)+ \mathrm{C}_{0,0}^0(\Omega) \subset \mathcal{S}'_h(\mathbb{R}^{n-1}).
\end{align*}
The case of Besov spaces and other function spaces follows the same lines. The isomorphism is then a direct consequence.
\end{proof}

\begin{remark}We conclude this paper with few remarks.
    \begin{enumerate}
    \item It has been brought to the attention of the author Jones' extensions operators \cite{Jones1981}. This extension operator admits global homogeneous estimates for $(\varepsilon,\infty)$-domains $\Omega$. Moreover, those also have the good property of preserving the local character of extended functions, see \cite[Theorem~10.2]{BechtelBrownHallerDintelTolksdorf2021}. Then, in this case, one can reproduce the proof of Proposition~\ref{prop:ExtOpHomSobSpaces}. 
    
    \item In Theorem \ref{thm:TraceSpeLipopti}, we don't give any claim about a right bounded inverse. The naive composition with Poisson's extension
    \begin{align*}
        f\longmapsto T_{\phi}^{-1}[(y',y_n)\mapsto e^{-y_n(-\Delta')^{1/2}}S_{\phi} f(y')]
    \end{align*}
    only yields right bounded inverse for regularity indices  $1/p < s \leqslant 1$ ($1/p<s<1$ in the case of Besov spaces). 
    \item For a right bounded inverse for regularity indices between $1/p$ and $1+1/p$ in Theorem \ref{thm:TraceSpeLipopti}, personal discussions with Patrick Tolksdorf and Moritz Egert persuaded the author that it should be possible to adapt Jonsson and Wallin's extension operator, \cite[Chapter~VII, Theorem~3]{JonssonWallin1984} from the boundary from the whole domain, in a way so that it preserves homogeneous norms. 
    \item One may also use the same Jonsson and Wallin's (usual) extension operator to reprove exactly the same way \cite[Theorem~A.2]{Gaudin2023Hodge}, making sense of weak partial traces of differentials forms, and in this case $\mathbb{R}^n_+$ by $\Omega$ a special Lipschitz domain. The result is still far from being optimal, by the way. For the case of inhomogeneous function spaces on bounded Lipschitz domains, one may find an optimal statement in \cite{MitreaMitreaShaw2008}.
\end{enumerate}
\end{remark}

\section*{Acknowledgments} \addcontentsline{toc}{section}{Acknowledgments}

This work was partially supported by the ANR project RAGE ANR-18-CE40-0012.

\section*{Data Availability} \addcontentsline{toc}{section}{Data Availability}

The manuscript has no associated data. No new data were created
during the study.

%----------------------------------------------------
%----------------- Bibliographie --------------------
%----------------------------------------------------
\typeout{}                                
\bibliographystyle{alpha}
{\footnotesize
\bibliography{BiblioAnatole}}

%----------------------------------------------------
%------------------- Appendices ----------------------
%----------------------------------------------------
%----------------------------------------------------
%------------------- Appendices ----------------------
%----------------------------------------------------

\appendix
\section*{Appendices}
\addcontentsline{toc}{section}{Appendices}

\renewcommand{\thesection}{\Alph{section}}
\section{Dual spaces of some vector-valued sequence spaces}\label{Appendix:VectorValHomSpaces}

We present basic results on the dual of Lebesgue vector-valued sequence spaces, for which no concise and adapted statement appears to exist, to the best of the author's knowledge. The closest reference in spirit is \cite[Chapter~2, Section~11, Proposition, p.177]{bookTriebel1983}. Although this is basic and very well-known among specialists, we provide the explicit result and a proof for a particular case here, arguing that the remaining ones can be proven as in the last reference.

We recall that $\mathcal{M}(\mathbb{R}^n)=(\mathrm{C}^{0}_{0}(\mathbb{R}^n))'$ is the space of finite Radon measures over $\mathbb{R}^n$, for which we have $\mathcal{M}(\mathbb{R}^n)\ast\mathrm{L}^1(\mathbb{R}^n)\hookrightarrow\mathrm{L}^1(\mathbb{R}^n)$. We recall that the $\mathcal{M}(\mathbb{R}^n)$-norm of $\mu\in\mathcal{M}(\mathbb{R}^n)$ is
\begin{align*}
    \lVert \mu \rVert_{\mathcal{M}(\mathbb{R}^n)} := \sup_{\substack{g\in\mathrm{C}^0_0(\mathbb{R}^n),\\ \lVert g\rVert_{\mathrm{L}^\infty(\mathbb{R}^n)}\leqslant 1}} \,\left\lvert \int_{\mathbb{R}^n} g(x)\mathrm{~d}\mu(x) \right\rvert.
\end{align*}
We have the additional property $\lVert \mu \rVert_{\mathcal{M}(\mathbb{R}^n)}=|\mu|(\mathbb{R}^n)$.

\begin{lemma}\label{lem:DualitySequencespaces} Let $p,q\in[1,+\infty]$, $s\in\mathbb{R}$, the following duality identities hold:
\begin{enumerate}
    \item $\ell^{q}_{s}(\mathbb{Z},\mathrm{L}^p(\mathbb{R}^n)) = (\ell^{q'}_{-s}(\mathbb{Z},\mathrm{L}^{p'}(\mathbb{R}^n)) )'$, $p,q\in(1,+\infty]$;
    \item $\ell^{1}_{s}(\mathbb{Z},\mathrm{L}^p(\mathbb{R}^n)) = (c^{0}_{-s}(\mathbb{Z},\mathrm{L}^{p'}(\mathbb{R}^n)))'$, $p\in(1,+\infty]$;
    \item $\ell^{q}_{s}(\mathbb{Z},\mathcal{M}(\mathbb{R}^n)) = (\ell^{q'}_{-s}(\mathbb{Z},\mathrm{C}^{0}_{0}(\mathbb{R}^n)))'$, $q\in(1,+\infty]$;
    \item $\ell^{1}_{s}(\mathbb{Z},\mathcal{M}(\mathbb{R}^n)) = (c^{0}_{-s}(\mathbb{Z},\mathrm{C}^{0}_{0}(\mathbb{R}^n)))'$.
\end{enumerate}
\end{lemma}

\begin{proof} The identity \textit{(i)} for $s=0$ has been proved in \cite[Chapter~2,~Section~11,~Proposition,~p.177]{bookTriebel1983}. For each identity, the embedding from the left-hand side to the right-hand side always holds. This is a direct consequence of Hölder's inequality. The action is given by the map
\begin{align*}
\left\{\begin{array}{cl}
\ell^{q}_{s}(\mathbb{Z},\mathrm{L}^p(\mathbb{R}^n)) &\longrightarrow (\ell^{q'}_{-s}(\mathbb{Z},\mathrm{L}^{p'}(\mathbb{R}^n)) )'\\
(g_j)_{j\in\mathbb{Z}} &\longmapsto \left[ (f_j)_{j\in\mathbb{Z}}\longmapsto \sum_{j\in\mathbb{Z}}\int_{\mathbb{R}^n} f_j(x)\,g_j(x)\mathrm{~d}x\right]
\end{array}\right.
\end{align*}
and similarly for $\ell^{1}_{s}(\mathbb{Z},\mathrm{L}^p(\mathbb{R}^n)) \hookrightarrow (c^{0}_{-s}(\mathbb{Z},\mathrm{L}^{p'}(\mathbb{R}^n)))$.

In the case of $\mathrm{C}^{0}_{0}(\mathbb{R}^n)$ instead of $\mathrm{L}^p(\mathbb{R}^n)$, Hölder's inequality is replaced by the duality relation $\mathcal{M}(\mathbb{R}^n)=(\mathrm{C}^{0}_{0}(\mathbb{R}^n))'$: the action is given by the map
\begin{align*}
\left\{\begin{array}{cl}
\ell^{1}_{s}(\mathbb{Z},\mathcal{M}(\mathbb{R}^n)) &\longrightarrow (c^{0}_{-s}(\mathbb{Z},\mathrm{C}^{0}_{0}(\mathbb{R}^n) )'\\
(\mu_j)_{j\in\mathbb{Z}} &\longmapsto \left[ (f_j)_{j\in\mathbb{Z}}\longmapsto \sum_{j\in\mathbb{Z}}\int_{\mathbb{R}^n} f_j(x)\mathrm{~d}\mu_{j}(x)\right]
\end{array}\right.
\end{align*}
and similarly for $\ell^{q}_{s}(\mathbb{Z},\mathcal{M}(\mathbb{R}^n)) \hookrightarrow (\ell^{q'}_{-s}(\mathbb{Z},\mathrm{C}^{0}_{0}(\mathbb{R}^n)))'$.

Now, we focus on the ontoness of the canonical map $\ell^{1}_{s}(\mathbb{Z},\mathcal{M}(\mathbb{R}^n)) \longrightarrow (c^{0}_{-s}(\mathbb{Z},\mathrm{C}^{0}_{0}(\mathbb{R}^n))'$. The remaining cases admit a similar proof, thanks to the arguments one can find in the proof of \cite[Chapter~2,~Section~11,~Proposition,~p.177]{bookTriebel1983}.

We set $(\mathrm{e}_j)_{j\in\mathbb{Z}}$ to be the canonical Schauder basis of sequence spaces. For all $k\in\mathbb{Z}$, all $f\in\mathrm{C}_0^0(\mathbb{R}^n)$,  we set $\iota_k(f):= f\otimes\mathrm{e}_k \in c^{0}_{-s}(\mathbb{Z},\mathrm{C}^{0}_{0}(\mathbb{R}^n))$. For all $N\in\mathbb{N}$, all $\mathfrak{f}\in c^{0}_{-s}(\mathbb{Z},\mathrm{C}^{0}_{0}(\mathbb{R}^n))$, we define the projection $\pi_{N}\mathfrak{f} := \sum_{-N\leqslant j\leqslant N} \mathfrak{f}_j\otimes\mathrm{e}_j$.

Let $\Gamma \in (c^{0}_{-s}(\mathbb{Z},\mathrm{C}^{0}_{0}(\mathbb{R}^n))'$. For all $k\in\mathbb{Z}$, we deduce by composition that $\Gamma\circ\iota_k \in (\mathrm{C}^{0}_{0}(\mathbb{R}^n))'=\mathcal{M}(\mathbb{R}^n)$, we define $\mu_k\in\mathcal{M}(\mathbb{R}^n)$, through the identity
\begin{align*}
    \Gamma\circ\iota_k(\varphi) =: \int_{\mathbb{R}^n} \varphi(x)\mathrm{~d}\mu_k(x),\, \forall \varphi\in\mathrm{C}^0_0(\mathbb{R}^n).
\end{align*}
We want to show that $\mathfrak{\mu} = (\mu_k)_{k\in\mathbb{Z}}\in\ell^{1}_{s}(\mathbb{Z},\mathcal{M}(\mathbb{R}^n))$. Let $\varepsilon>0$, by definition of the $\mathcal{M}(\mathbb{R}^n)$-norm, for all $k\in\mathbb{Z}$, there exists $\psi_k\in\mathrm{C}^{0}_{0}(\mathbb{R}^n)$, $|\psi_k|\leqslant 1$, such that\footnote{Up to multiply $\psi_k$ by $\arg(\int \psi_k \mathrm{d}\mu_k)$}
\begin{align*}
    \lVert \mu_k\rVert_{\mathcal{M}(\mathbb{R}^n)} \leqslant \varepsilon + \int_{\mathbb{R}^n} \psi_k(x)\mathrm{~d}\mu_k(x).
\end{align*}

For all $N\in\mathbb{N}$, one has the equality
\begin{align*}
    \Gamma\circ\pi_{N} = \sum_{-N\leqslant k\leqslant N} \Gamma\circ\iota_k.
\end{align*}
Therefore,
\begin{align*}
    \sum_{j\in\llb-N,N\rrb} 2^{js} \lVert \mu_j\rVert_{\mathcal{M}(\mathbb{R}^n)} &\leqslant \varepsilon\left(\sum_{j\in\llb-N,N\rrb} 2^{js}\right) + \sum_{j\in\llb-N,N\rrb} \int_{\mathbb{R}^n} (2^{js}\psi_j(x))\mathrm{~d}\mu_j(x)\\
    &\leqslant C_{N,s}\cdot\varepsilon + \Gamma\left[\pi_{\llb-N,N\rrb}(2^{js}\psi_j)_{j\in\mathbb{Z}}\right]\\
    &\leqslant C_{N,s}\cdot\varepsilon + \lVert\Gamma\rVert_{(c^{0}_{-s}(\mathbb{Z},\mathrm{C}^{0}_{0}(\mathbb{R}^n)))'}\lVert(\psi_j)_{j\in\mathbb{Z}}\rVert_{\ell^{\infty}(\mathbb{Z},\mathrm{L}^{\infty}(\mathbb{R}^n))}\\
    &\leqslant C_{N,s}\cdot\varepsilon + \lVert\Gamma\rVert_{(c^{0}_{-s}(\mathbb{Z},\mathrm{C}^{0}_{0}(\mathbb{R}^n)))'}.
\end{align*}
The last inequality comes from $|\psi_k|\leqslant 1$, $k\in\mathbb{Z}$. Therefore, as $\varepsilon$ tends to 0, we obtain, for all $N\in\mathbb{N}$, the inequality
\begin{align*}
     \sum_{j\in\llb-N,N\rrb} 2^{js} \lVert \mu_j\rVert_{\mathcal{M}(\mathbb{R}^n)} \leqslant \lVert\Gamma\rVert_{(c^{0}_{-s}(\mathbb{Z},\mathrm{C}^{0}_{0}(\mathbb{R}^n)))'}.
\end{align*}
Thus, as $N$ tends to infinity, it yields $(\mu_j)_{j\in\mathbb{Z}}\in{\ell^{1}_{s}(\mathbb{Z},\mathcal{M}(\mathbb{R}^n))}$ with the estimate
\begin{align*}
    \lVert (\mu_j)_{j\in\mathbb{Z}}\rVert_{\ell^{1}_{s}(\mathbb{Z},\mathcal{M}(\mathbb{R}^n))} \leqslant \lVert\Gamma\rVert_{(c^{0}_{-s}(\mathbb{Z},\mathrm{C}^{0}_{0}(\mathbb{R}^n)))'}.
\end{align*}
\end{proof}

\renewcommand{\thesection}{\Alph{section}}
\section{Some additional interpolation results}\label{Appendix:Interp}

Here we present a number of results concerning the reiteration property for real interpolation. These well-known results are stated in a way so that it does not (systematically) require the completeness of the normed vector spaces involved.

\begin{lemma}[Extremal reiteration property]\label{lem:ExtrmReitPropnonComp} Let $\mathrm{X}_0$ and $\mathrm{X}_1$ be two compatible normed vector spaces. For $\theta,\theta_0,\theta_1\in(0,1)$, $r,q\in[1,+\infty]$, we have with equivalence of norms
\begin{align*}
    \left( (\mathrm{X}_0,\mathrm{X}_1)_{\theta_0,r},\mathrm{X}_1\right)_{\theta,q} = (\mathrm{X}_0,\mathrm{X}_1)_{(1-\theta)\theta_0+\theta,q} \text{ and } \left(\mathrm{X}_0,(\mathrm{X}_0,\mathrm{X}_1)_{\theta_1,r}\right)_{\theta,q} = (\mathrm{X}_0,\mathrm{X}_1)_{\theta_1\theta,q}.
\end{align*}
\end{lemma}

\begin{proof} We prove the identity
\begin{align*}
    \left(\mathrm{X}_0,(\mathrm{X}_0,\mathrm{X}_1)_{\theta_1,r}\right)_{\theta,q} = (\mathrm{X}_0,\mathrm{X}_1)_{\theta_1\theta,q},
\end{align*}
the other one can be treated with similar arguments.

$\bullet$ The embedding $\left(\mathrm{X}_0,(\mathrm{X}_0,\mathrm{X}_1)_{\theta_1,r}\right)_{\theta,q} \hookrightarrow (\mathrm{X}_0,\mathrm{X}_1)_{\theta_1\theta,q}$ follows from the first arguments in the proof of \cite[Theorem~3.5.3]{BerghLofstrom1976}.

$\bullet$ For the reverse embedding $(\mathrm{X}_0,\mathrm{X}_1)_{\theta_1\theta,q}\hookrightarrow \left(\mathrm{X}_0,(\mathrm{X}_0,\mathrm{X}_1)_{\theta_1,r}\right)_{\theta,q} $, it suffices to focus on the case $r=1$.

For $u\in\mathrm{X}_0+\mathrm{X}_1$, the Holmstedt estimate for the $K$-functional is given in \cite[Corollary~3.6.2]{BerghLofstrom1976}, and is stated as
\begin{align*}
    K(t,u,\mathrm{X}_0,\mathrm{X}_{\theta_1,1}) \sim_{\theta_1} t\int_{t^{1/\theta_1}}^{+\infty} \tau^{-\theta_1}K(\tau,u,\mathrm{X}_0,\mathrm{X}_{1}) \frac{\mathrm{d}\tau}{\tau}\text{, } t>0\text{,}
\end{align*}
where $\mathrm{X}_{\theta_1,1}:=(\mathrm{X}_0,\mathrm{X}_1)_{\theta_1,1}$. We can multiply by $t^{-\theta}$ and take the $\mathrm{L}^q_\ast$-norm to obtain,
\begin{align*}
    \lVert u\rVert_{(\mathrm{X}_0,\mathrm{X}_{\theta_1,1})_{\theta,q}} &\sim_{\theta_1} \left(\int_{0}^{+\infty}\left(t^{1-\theta}\int_{t^{1/\theta_1}}^{+\infty} \tau^{-\theta_1}K(\tau,u,\mathrm{X}_0,\mathrm{X}_{1}) \frac{\mathrm{d}\tau}{\tau}\right)^q \frac{\mathrm{d}t}{t}\right)^\frac{1}{q}\\
    &\sim_{\theta_1} \left(\int_{0}^{+\infty}\left(t^{(1-\theta)\theta_1}\int_{t}^{+\infty} \tau^{-\theta_1}K(\tau,u,\mathrm{X}_0,\mathrm{X}_{1}) \frac{\mathrm{d}\tau}{\tau}\right)^q \frac{\mathrm{d}t}{t}\right)^\frac{1}{q}.
\end{align*}
The last line is obtained by means of the change of variables $t\mapsto t^{\theta_1}$. Hardy's inequality \cite[Lemma~6.2.6]{bookHaase2006} implies
\begin{align*}
    \lVert u\rVert_{(\mathrm{X}_0,\mathrm{X}_{\theta_1,1})_{\theta,q}} &\lesssim_{\theta_1} \lVert u\rVert_{(\mathrm{X}_0,\mathrm{X}_{1})_{\theta\theta_1,q}}.
\end{align*}
We have a similar proof if $q=+\infty$.
\end{proof}

Thanks to the previous Lemma, one can relax the completeness assumptions in the first Wolff's reiteration Theorem \cite[Theorem~1]{Wolff1982}, by a close inspection of the proof.

\begin{lemma}[Wolff's reiteration theorem for real interpolation]\label{lem:WOlffReitThmnonComp} Let $\mathrm{X}_1$, $\mathrm{X}_2$, $\mathrm{X}_3$ and $\mathrm{X}_4$ be four compatible normed vector spaces. Let $\lambda,\mu\in(0,1)$, and $q,r\in[1,+\infty]$, we assume
\begin{align*}
    (\mathrm{X}_1,\mathrm{X}_3)_{\lambda,r}=\mathrm{X}_2 \text{ and }(\mathrm{X}_2,\mathrm{X}_4)_{\mu,q}=\mathrm{X}_3.
\end{align*}
Then, for $\theta_2=\frac{\lambda \mu}{1-\mu+\mu\lambda}$ and $\theta_3=\frac{\lambda}{1-\mu+\mu\lambda}$,
\begin{enumerate}
    \item if $\mathrm{X}_2$ is a complete space, and $\mathrm{X}_1\cap\mathrm{X}_4\subset\mathrm{X}_2\cap \mathrm{X}_3$, then
    \begin{align*}
        (\mathrm{X}_1,\mathrm{X}_4)_{\theta_2,1}\hookrightarrow\mathrm{X}_2;
    \end{align*}
    \item if $\mathrm{X}_3$ is a complete space, and $\mathrm{X}_1\cap\mathrm{X}_4\subset\mathrm{X}_2\cap \mathrm{X}_3$, then
    \begin{align*}
        (\mathrm{X}_1,\mathrm{X}_4)_{\theta_3,1}\hookrightarrow\mathrm{X}_3;
    \end{align*}
    \item if $\mathrm{X}_1$ and $(\mathrm{X}_1+\mathrm{X}_2)\cap\mathrm{X}_4$ equipped with their natural norms are complete spaces, then
    \begin{align*}
        \mathrm{X}_2\hookrightarrow (\mathrm{X}_1,\mathrm{X}_4)_{\theta_2,\infty};
    \end{align*}
    \item if $\mathrm{X}_1$ and $(\mathrm{X}_1+\mathrm{X}_3)\cap\mathrm{X}_4$ equipped with their natural norms are complete spaces, then
    \begin{align*}
        \mathrm{X}_3\hookrightarrow (\mathrm{X}_1,\mathrm{X}_4)_{\theta_3,\infty}.
    \end{align*}
\end{enumerate}
In particular,
\begin{align*}
     (\mathrm{X}_1,\mathrm{X}_4)_{\theta_3,q}&=\mathrm{X}_3\text{, under the assumptions in \textit{(i)} and \textit{(iii)};}\\
     (\mathrm{X}_1,\mathrm{X}_4)_{\theta_2,r}&=\mathrm{X}_2\text{, under the assumptions in \textit{(ii)} and \textit{(iv)}.}
\end{align*}
\end{lemma}

In particular, we can also derive the following lemmas.
\begin{lemma}\label{lem:WOlffReitThmnonComp2} Let $\mathrm{X}_1$, $\mathrm{X}_2$, $\mathrm{X}_3$ and $\mathrm{X}_4$ be four compatible normed vector spaces. Let $\lambda,\mu\in(0,1)$, and $q,r\in[1,+\infty]$, we assume
\begin{align*}
    (\mathrm{X}_1,\mathrm{X}_3)_{\lambda,r}\hookrightarrow\mathrm{X}_2 \text{ and }(\mathrm{X}_2,\mathrm{X}_4)_{\mu,q}\hookrightarrow\mathrm{X}_3.
\end{align*}
For $\theta_2=\frac{\lambda \mu}{1-\mu+\mu\lambda}$, if $\mathrm{X}_2$ is a complete space, and $\mathrm{X}_1\cap\mathrm{X}_4\subset\mathrm{X}_2\cap \mathrm{X}_3$, then
\begin{align*}
        (\mathrm{X}_1,\mathrm{X}_4)_{\theta_2,1}\hookrightarrow\mathrm{X}_2.
\end{align*}
One has a similar result for $X_3$.
\end{lemma}

\begin{lemma}\label{lem:WOlffReitThmnonComp3} Let $\mathrm{X}_1$, $\mathrm{X}_2$, $\mathrm{X}_3$ and $\mathrm{X}_4$ be four compatible normed vector spaces. Let $\lambda,\mu\in(0,1)$, and $q,r\in[1,+\infty]$, we assume
\begin{align*}
    \mathrm{X}_2\hookrightarrow(\mathrm{X}_1,\mathrm{X}_3)_{\lambda,r} \text{ and }\mathrm{X}_3\hookrightarrow(\mathrm{X}_2,\mathrm{X}_4)_{\mu,q}.
\end{align*}
For $\theta_2=\frac{\lambda \mu}{1-\mu+\mu\lambda}$, if $\mathrm{X}_1$ and $(\mathrm{X}_1+\mathrm{X}_2)\cap\mathrm{X}_4$ endowed with their natural norms are complete spaces, then
\begin{align*}
        \mathrm{X}_2\hookrightarrow(\mathrm{X}_1,\mathrm{X}_4)_{\theta_2,\infty}.
\end{align*}
One has a similar result for $X_3$.
\end{lemma}

\renewcommand{\thesection}{\Alph{section}}
\section{A projection operator on the kernel of the trace operator in inhomogeneous function spaces}\label{Appendix:projNulTrace}

\begin{lemma}\label{lem:ProjDirBinHomSpaces} Let $\chi\in\mathrm{C}_c^\infty(\mathbb{R})$, satisfying $0\leqslant \chi\leqslant 1$, such that $\supp \chi \subset [-1,1]$ and $\chi(0)=1$. We introduce
\begin{align*}
    P_{\chi}\,:\, f   &   \longmapsto    \left[    (x',x_n)\mapsto \chi(x_n)e^{-x_n(-\Delta')^\frac{1}{2}}f(x') \right] .
\end{align*}

The linear operator, given formally by
\begin{align*}
    \mathcal{P}_\mathcal{D} = \mathrm{I} - T^{-1}_{\phi}P_\chi S_{\phi}([\cdot]_{|_{\partial\Omega}}),
\end{align*}
is a linear projection such that it maps boundedly
\begin{enumerate}
    \item $\mathrm{H}^{s,p}(\Omega)$ into itself,  $1/p<s\leqslant1$, $1<p<+\infty$;
    \item $\mathrm{B}^{s}_{p,q}(\Omega)$  into itself,  $1/p<s<1$, $1\leqslant p,q\leqslant +\infty$;
    \item $\mathrm{C}^{0}_{0}(\overline{\Omega})$ into itself.
\end{enumerate}
For all measurable function $u$ in any of the spaces above, it satisfies
\begin{align*}
    [\mathcal{P}_\mathcal{D}u]_{|_{\partial\Omega}} =0.
\end{align*}
Moreover, for $\mathfrak{B}\in\{\mathcal{B}^{s}_{p,\infty},\mathrm{B}^{s,0}_{\infty,q}\}$, we have $\mathcal{P}_\mathcal{D} \mathfrak{B}(\Omega) \subset \mathfrak{B}(\Omega)$.
\end{lemma}

\begin{proof} Up to conjugate $\mathcal{P}_\mathcal{D}$ by $T_\phi$, one can assume, without loss of generality, that $\Omega=\mathbb{R}^n_+$.

By the trace theorem, \cite[Theorem~6.6.1]{BerghLofstrom1976}, it suffices to show that $P_\chi$ maps boundedly $\mathrm{B}^{s-1/p}_{p,p}(\mathbb{R}^{n-1})$ to $\mathrm{H}^{s,p}(\mathbb{R}^n_{+})$, $\mathrm{B}^{s-1/p}_{p,q}(\mathbb{R}^{n-1})$ to $\mathrm{B}^{s}_{p,q}(\mathbb{R}^n_{+})$ and $\mathrm{C}^{0}_{0}(\mathbb{R}^{n-1})$ to $\mathrm{C}^{0}_0(\overline{\mathbb{R}^{n}_+})$, provided $s\in(0,1]$. 

\textbf{Step 1:} Let $1\leqslant p<+\infty$. Let $f\in\mathrm{B}^{-\frac{1}{p}}_{p,p}(\mathbb{R}^{n-1}) = \mathrm{L}^p(\mathbb{R}^{n-1}) + \dot{\mathrm{B}}^{-\frac{1}{p}}_{p,p}(\mathbb{R}^{n-1})$. We write $f=a+b$, with $(a,b)\in\mathrm{L}^p(\mathbb{R}^{n-1})\times \dot{\mathrm{B}}^{-\frac{1}{p}}_{p,p}(\mathbb{R}^{n-1})$. It holds that
\begin{align*}
    \left\lVert P_\chi f\right\rVert_{\mathrm{L}^p(\mathbb{R}^n_+)} &= \left(\int_{0}^{+\infty}\lVert \chi(x_n)e^{-x_n(-\Delta')^\frac{1}{2}}f\rVert_{\mathrm{L}^p(\mathbb{R}^{n-1})}^p\mathrm{~d}x_n\right)^\frac{1}{p} \\
    &= \left(\int_{0}^{+\infty}\left(t^{\frac{1}{p}}\lVert \chi(t)e^{-t(-\Delta')^\frac{1}{2}}f\rVert_{\mathrm{L}^p(\mathbb{R}^{n-1})}\right)^p\frac{\mathrm{d}t}{t}\right)^\frac{1}{p}\\
    &\leqslant \left(\int_{0}^{1}\lVert e^{-t(-\Delta')^\frac{1}{2}}a\rVert_{\mathrm{L}^p(\mathbb{R}^{n-1})}^p{\mathrm{d}t}\right)^\frac{1}{p} + \left(\int_{0}^{+\infty}\left(t^{\frac{1}{p}}\lVert e^{-t(-\Delta')^\frac{1}{2}}b\rVert_{\mathrm{L}^p(\mathbb{R}^{n-1})}\right)^p\frac{\mathrm{d}t}{t}\right)^\frac{1}{p}\\
    &\lesssim_{p,n} \left\lVert a \right\rVert_{{\mathrm{L}}^{p}(\mathbb{R}^{n-1})}+ \left\lVert b \right\rVert_{\dot{\mathrm{B}}^{-\frac{1}{p}}_{p,p}(\mathbb{R}^{n-1})}\text{.}
\end{align*}
The last inequality is a consequence of \cite[Lemma~B.1]{Gaudin2022}. Taking the infimum on all such pairs $(a,b)$ yields
\begin{align}\label{eq:estPchi1Sob}
    \left\lVert P_\chi f\right\rVert_{\mathrm{L}^p(\mathbb{R}^n_+)} \lesssim_{p,n} \left\lVert f \right\rVert_{{\mathrm{B}}^{-\frac{1}{p}}_{p,p}(\mathbb{R}^{n-1})}.
\end{align}

Now, for $f\in\mathrm{W}^{1,p}(\mathbb{R}^n_+)$, by the Leibniz rule, one can proceed as before to obtain
\begin{align*}
    \lVert \nabla P_\chi f\rVert_{\mathrm{L}^p(\mathbb{R}^n_+)} &= \left(\int_{0}^{+\infty}\lVert \nabla [\chi(x_n)e^{-x_n(-\Delta')^\frac{1}{2}}f]\rVert_{\mathrm{L}^p(\mathbb{R}^{n-1})}^p\mathrm{~d}x_n\right)^\frac{1}{p}\\
    &\leqslant \left(\int_{0}^{+\infty}\lVert e^{-t(-\Delta')^\frac{1}{2}}\nabla'f\rVert_{\mathrm{L}^p(\mathbb{R}^{n-1})}^p\mathrm{~d}t\right)^\frac{1}{p} +\left(\int_{0}^{1}\lVert  \chi'(t)e^{-t(-\Delta')^\frac{1}{2}}f\rVert_{\mathrm{L}^p(\mathbb{R}^{n-1})}^p\mathrm{~d}t\right)^\frac{1}{p}\\ \quad\quad&+ \left(\int_{0}^{+\infty}\lVert (-\Delta')^\frac{1}{2}e^{-t(-\Delta')^\frac{1}{2}}f\rVert_{\mathrm{L}^p(\mathbb{R}^{n-1})}^p\mathrm{~d}t\right)^\frac{1}{p}\\
    &\lesssim_{p,n,\chi}  \left\lVert f \right\rVert_{{\mathrm{L}}^{p}(\mathbb{R}^{n-1})} + \left\lVert \nabla' f\right\rVert_{\dot{\mathrm{B}}^{-\frac{1}{p}}_{p,p}(\mathbb{R}^{n-1})}+\lVert (-\Delta')^\frac{1}{2}f \rVert_{\dot{\mathrm{B}}^{-\frac{1}{p}}_{p,p}(\mathbb{R}^{n-1})}\\
    &\lesssim_{p,n,\chi} \left\lVert f \right\rVert_{{\mathrm{B}}^{1-\frac{1}{p}}_{p,p}(\mathbb{R}^{n-1})}.
\end{align*}
Since ${\mathrm{B}}^{1-\frac{1}{p}}_{p,p}(\mathbb{R}^{n-1})\hookrightarrow {\mathrm{B}}^{-\frac{1}{p}}_{p,p}(\mathbb{R}^{n-1})$, one can take the sum with \eqref{eq:estPchi1Sob} giving
\begin{align}\label{eq:estPchi2Sob}
    \left\lVert P_\chi f\right\rVert_{\mathrm{W}^{1,p}(\mathbb{R}^n_+)} \lesssim_{p,n} \left\lVert f \right\rVert_{{\mathrm{B}}^{-\frac{1}{p}}_{p,p}(\mathbb{R}^{n-1})}.
\end{align}

When $1<p<+\infty$, complex interpolation between \eqref{eq:estPchi1Sob} and \eqref{eq:estPchi2Sob}, yields the boundedness of
\begin{align*}
    P_\chi\,:\,\mathrm{B}^{s-1/p}_{p,p}(\mathbb{R}^{n-1})\longrightarrow \mathrm{H}^{s,p}(\mathbb{R}^n_{+}),\, 0\leqslant s \leqslant 1,\, 1< p<+\infty.
\end{align*}
In particular, $\mathcal{P}_\mathcal{D}$ is well-defined and bounded on $\mathrm{H}^{s,p}(\Omega)$, for $1/p<s\leqslant 1$, $1<p<+\infty$.

One proceeds the same way when $1\leqslant p<+\infty$: real interpolation between \eqref{eq:estPchi1Sob} and \eqref{eq:estPchi2Sob} yields the boundedness of
\begin{align*}
    P_\chi\,:\,\mathrm{B}^{s-1/p}_{p,q}(\mathbb{R}^{n-1})\longrightarrow \mathrm{B}^{s}_{p,q}(\mathbb{R}^{n}_+),\, 0< s < 1,\, 1\leqslant p<+\infty, 1\leqslant q\leqslant +\infty.
\end{align*}
In particular, $\mathcal{P}_\mathcal{D}$ is well-defined and bounded on $\mathrm{B}^{s}_{p,q}(\mathbb{R}^n_{+})$, for $1/p<s< 1$, $1<p<+\infty$, $q\in[1,+\infty]$.

\textbf{Step 2:} The case $p=+\infty$. Let $f\in\mathrm{C}_{b}^0(\mathbb{R}^{n-1})$, then, by the mapping properties of the convolution $\mathrm{L}^1(\mathbb{R}^{n-1})\ast \mathrm{C}_{b}^0(\mathbb{R}^{n-1})\hookrightarrow \mathrm{C}_{b}^0(\mathbb{R}^{n-1})$ and by the Dominated Convergence Theorem, one has $P_\chi f\in\mathrm{C}^0_b(\overline{\mathbb{R}^n_+})$ with the estimate
\begin{align}\label{eq:LinftyPchi1}
    \lVert P_\chi f \rVert_{\mathrm{L}^\infty(\mathbb{R}^n_+)} \leqslant \lVert  f \rVert_{\mathrm{L}^\infty(\mathbb{R}^n_+)}.
\end{align}
Additionally, if $f\in\mathrm{C}_0^0(\mathbb{R}^{n-1})$, one also has $P_\chi f\in\mathrm{C}^0_0(\overline{\mathbb{R}^n_+})$ which gives \textit{(iii)}.

Now, if $f\in\mathrm{B}^1_{\infty,1}(\mathbb{R}^{n-1})\subset \mathrm{W}_{1,\infty}(\mathbb{R}^{n-1})\subset\mathrm{C}_{b}(\mathbb{R}^{n-1})$,
\begin{align}
    \lVert \nabla P_\chi f \rVert_{\mathrm{L}^\infty(\mathbb{R}^n_+)} &\leqslant \lVert \chi' \rVert_{\mathrm{L}^\infty([0,1])} \lVert  f \rVert_{\mathrm{L}^\infty(\mathbb{R}^{n-1})} + \lVert  \nabla' f \rVert_{\mathrm{L}^\infty(\mathbb{R}^{n-1})} + \lVert  (-\Delta')^\frac{1}{2} f \rVert_{\mathrm{L}^\infty(\mathbb{R}^{n-1})}\nonumber\\
    &\lesssim_{n,\chi} \lVert f \rVert_{\mathrm{B}^1_{\infty,1}(\mathbb{R}^{n-1})}.\label{eq:LinftyPchi2}
\end{align}
One can sum up \eqref{eq:LinftyPchi1} and \eqref{eq:LinftyPchi2} to deduce,
\begin{align}\label{eq:LinftyPchi3}
    \lVert P_\chi f \rVert_{\mathrm{W}^{1,\infty}(\mathbb{R}^n_+)} \lesssim_{n,\chi} \lVert f \rVert_{\mathrm{B}^1_{\infty,1}(\mathbb{R}^{n-1})}.
\end{align}
Now, we perform real interpolation between \eqref{eq:LinftyPchi1} and \eqref{eq:LinftyPchi3}, which yields, for all $0 < s < 1$, all $q \in [1, +\infty]$, and all $f \in \mathrm{B}^{s}_{\infty,q}(\mathbb{R}^{n}_+)$,
\begin{align*}
    \lVert P_\chi f \rVert_{\mathrm{B}^{s}_{\infty,q}(\mathbb{R}^n_+)} \lesssim_{n,\chi} \lVert f \rVert_{\mathrm{B}^s_{\infty,q}(\mathbb{R}^{n-1})}.
\end{align*}
Furthermore, if $f \in \mathrm{B}^{s,0}_{\infty,q}(\mathbb{R}^{n-1}) \subset \mathrm{C}^0_0(\mathbb{R}^{n-1})$, then one can check $P_\chi f \in \mathrm{B}^{s,0}_{\infty,q}(\mathbb{R}^{n}_+)$, completing the proof.
\end{proof}

\end{document}